\documentclass[reqno,11pt]{amsart}
\usepackage{amsmath}
\usepackage{amsxtra}
\usepackage{amscd}
\usepackage{amsthm}
\usepackage{amsfonts}
\usepackage{amssymb}
\usepackage{eucal}
\usepackage{scalerel}
\usepackage{verbatim}
\usepackage{bm}
\usepackage[matrix,arrow,curve]{xy}
\usepackage{tikz}

\usepackage{color}
\usepackage{colortbl}
\usepackage{multirow}

\usepackage{mathdots}
\usepackage{easybmat}
\usepackage{a4wide}

\usepackage[T1]{fontenc}

\usepackage{xr-hyper}
\usepackage[
pdftex,
bookmarks=false,
colorlinks=true,
debug=true,
pdfnewwindow=true]{hyperref}

\theoremstyle{plain}
\newtheorem{Thm}[equation]{Theorem}
\newtheorem{Cor}[equation]{Corollary}
\newtheorem{Lem}[equation]{Lemma}
\newtheorem{Prop}[equation]{Proposition}
\newtheorem{Conj}[equation]{Conjecture}

\theoremstyle{definition}
\newtheorem{Def}[equation]{Definition}

\theoremstyle{remark}

\newtheorem{Rem}[equation]{Remark}

\newtheorem{conj}{Conjecture}

\numberwithin{equation}{section}
\renewcommand{\rm}{\normalshape}

\newif\ifShowLabels
\ShowLabelstrue
\newdimen\theight
\def\TeXref#1{%
    \leavevmode\vadjust{\setbox0=\hbox{{\tt
        \quad\quad  {\small \rm #1}}}%
    \theight=\ht0
    \advance\theight by \lineskip
    \kern -\theight \vbox to
    \theight{\rightline{\rlap{\box0}}%
    \vss}%
    }}%

\ShowLabelsfalse% comment this out if labels should be printed

\newenvironment{thm}[1]%
    { \begin{Thm} \label{T:#1}  \ifShowLabels \TeXref{T:#1} \fi }%
    { \end{Thm} }

\renewcommand{\th}[1]{\begin{thm}{#1} \sl }
\renewcommand{\eth}{\end{thm} }

\newenvironment{lemma}[1]%
    { \begin{Lem} \label{L:#1}  \ifShowLabels \TeXref{L:#1} \fi }%
    { \end{Lem} }
\newcommand{\lem}[1]{\begin{lemma}{#1} \sl}
\newcommand{\elem}{\end{lemma}}

\newenvironment{propos}[1]%
    { \begin{Prop} \label{P:#1}  \ifShowLabels \TeXref{P:#1} \fi }%
    { \end{Prop} }
\newcommand{\prop}[1]{\begin{propos}{#1}\sl }
\newcommand{\eprop}{\end{propos}}

\newenvironment{corol}[1]%
    { \begin{Cor} \label{C:#1}  \ifShowLabels \TeXref{C:#1} \fi }%
    { \end{Cor} }
\newcommand{\cor}[1]{\begin{corol}{#1} \sl }
\newcommand{\ecor}{\end{corol}}

\newenvironment{defeni}[1]%
    { \begin{Def} \label{D:#1}  \ifShowLabels \TeXref{D:#1} \fi }%
    { \end{Def} }
\newcommand{\defe}[1]{\begin{defeni}{#1} \sl }
\newcommand{\edefe}{\end{defeni}}

\newenvironment{remark}[1]%
    { \begin{Rem} \label{R:#1}  \ifShowLabels \TeXref{R:#1} \fi }%
    { \end{Rem} }
\newcommand{\rem}[1]{\begin{remark}{#1}}
\newcommand{\erem}{\end{remark}}

\newenvironment{conjec}[1]%
    { \begin{Conj} \label{Co:#1}  \ifShowLabels \TeXref{Co:#1} \fi }%
    { \end{Conj} }
\renewcommand{\conj}[1]{\begin{conjec}{#1} \sl }
\newcommand{\econj}{\end{conjec}}

\newcommand{\eq}[1]%
    { \ifShowLabels \TeXref{E:#1} \fi
       \begin{equation} \label{E:#1} }
\newcommand{\eeq}{ \end{equation} }

\newcommand{\prf}{ \begin{proof} }
\newcommand{\epr}{ \end{proof} }

\newcommand\nc{\newcommand}
\nc{\unl}{\underline}
\nc{\ol}{\overline}
\nc{\on}{\operatorname}

\nc{\BA}{{\mathbb{A}}}
\nc{\BC}{{\mathbb{C}}}
\nc{\BD}{{\mathbb{D}}}
\nc{\BF}{{\mathbb{F}}}
\nc{\BG}{{\mathbb{G}}}
\nc{\BM}{{\mathbb{M}}}
\nc{\BN}{{\mathbb{N}}}
\nc{\BO}{{\mathbb{O}}}
\nc{\BQ}{{\mathbb{Q}}}
\nc{\BP}{{\mathbb{P}}}
\nc{\BR}{{\mathbb{R}}}
\nc{\BZ}{{\mathbb{Z}}}
\nc{\BS}{{\mathbb{S}}}
\nc{\BK}{{\mathbb{K}}}

\nc{\CA}{{\mathcal{A}}} \nc{\CB}{{\mathcal{B}}} \nc{\CalC}{{\mathcal
C}} \nc{\CalD}{{\mathcal D}} \nc{\CE}{{\mathcal{E}}}
\nc{\CF}{{\mathcal{F}}} \nc{\CG}{{\mathcal{G}}}
\nc{\CH}{{\mathcal{H}}} \nc{\CI}{{\mathcal{I}}}
\nc{\CK}{{\mathcal{K}}} \nc{\CL}{{\mathcal{L}}}
\nc{\CM}{{\mathcal{M}}} \nc{\CN}{{\mathcal{N}}}
\nc{\CO}{{\mathcal{O}}} \nc{\CP}{{\mathcal{P}}}
\nc{\CQ}{{\mathcal{Q}}} \nc{\CR}{{\mathcal{R}}}
\nc{\CS}{{\mathcal{S}}} \nc{\CT}{{\mathcal{T}}}
\nc{\CU}{{\mathcal{U}}} \nc{\CV}{{\mathcal{V}}}
\nc{\CW}{{\mathcal{W}}} \nc{\CX}{{\mathcal{X}}}
\nc{\CY}{{\mathcal{Y}}} \nc{\CZ}{{\mathcal{Z}}}

\nc{\fa}{{\mathfrak{a}}}
\nc{\fb}{{\mathfrak{b}}}
\nc{\fg}{{\mathfrak{g}}}
\nc{\fgl}{{\mathfrak{gl}}}
\nc{\fh}{{\mathfrak{h}}}
\nc{\fj}{{\mathfrak{j}}}
\nc{\fl}{{\mathfrak{l}}}
\nc{\fm}{{\mathfrak{m}}}
\nc{\fn}{{\mathfrak{n}}}
\nc{\fu}{{\mathfrak{u}}}
\nc{\fp}{{\mathfrak{p}}}
\nc{\frr}{{\mathfrak{r}}}
\nc{\fs}{{\mathfrak{s}}}
\nc{\ft}{{\mathfrak{t}}}
\nc{\fw}{{\mathfrak{w}}}
\nc{\fz}{{\mathfrak{z}}}

\nc{\fA}{{\mathfrak{A}}}
\nc{\fB}{{\mathfrak{B}}}
\nc{\fD}{{\mathfrak{D}}}
\nc{\fE}{{\mathfrak{E}}}
\nc{\fF}{{\mathfrak{F}}}
\nc{\fG}{{\mathfrak{G}}}
\nc{\fI}{{\mathfrak{I}}}
\nc{\fJ}{{\mathfrak{J}}}
\nc{\fK}{{\mathfrak{K}}}
\nc{\fL}{{\mathfrak{L}}}
\nc{\fM}{{\mathfrak{M}}}
\nc{\fN}{{\mathfrak{N}}}
\nc{\frP}{{\mathfrak{P}}}
\nc{\fQ}{{\mathfrak Q}}
\nc{\fR}{{\mathfrak R}}
\nc{\fS}{{\mathfrak S}}
\nc{\fT}{{\mathfrak{T}}}
\nc{\fU}{{\mathfrak{U}}}
\nc{\fW}{{\mathfrak{W}}}
\nc{\fY}{{\mathfrak{Y}}}
\nc{\fZ}{{\mathfrak{Z}}}

\nc{\ba}{{\mathbf{a}}}
\nc{\bb}{{\mathbf{b}}}
\nc{\bc}{{\mathbf{c}}}
\nc{\bd}{{\mathbf{d}}}
\nc{\be}{{\mathbf{e}}}
\nc{\bi}{{\mathbf{i}}}
\nc{\bj}{{\mathbf{j}}}
\nc{\bn}{{\mathbf{n}}}
\nc{\bp}{{\mathbf{p}}}
\nc{\bq}{{\mathbf{q}}}
\nc{\bu}{{\mathbf{u}}}
\nc{\bv}{{\mathbf{v}}}
\nc{\bw}{{\mathbf{w}}}
\nc{\bx}{{\mathbf{x}}}
\nc{\by}{{\mathbf{y}}}
\nc{\bz}{{\mathbf{z}}}

\nc{\bA}{{\mathbf{A}}}
\nc{\bB}{{\mathbf{B}}}
\nc{\bC}{{\mathbf{C}}}
\nc{\bD}{{\mathbf{D}}}
\nc{\bE}{{\mathbf{E}}}
\nc{\bF}{{\mathbf{F}}}
\nc{\bI}{{\mathbf{I}}}
\nc{\bK}{{\mathbf{K}}}
\nc{\bH}{{\mathbf{H}}}
\nc{\bM}{{\mathbf{M}}}
\nc{\bN}{{\mathbf{N}}}
\nc{\bO}{{\mathbf{O}}}
\nc{\bQ}{{\mathbf Q}}
\nc{\bS}{{\mathbf{S}}}
\nc{\bT}{{\mathbf{T}}}
\nc{\bV}{{\mathbf{V}}}
\nc{\bW}{{\mathbf{W}}}
\nc{\bX}{{\mathbf{X}}}
\nc{\bP}{{\mathbf{P}}}
\nc{\bY}{{\mathbf{Y}}}
\nc{\bZ}{{\mathbf{Z}}}

\nc{\sA}{{\mathsf{A}}}
\nc{\sB}{{\mathsf{B}}}
\nc{\sC}{{\mathsf{C}}}
\nc{\sD}{{\mathsf{D}}}
\nc{\sF}{{\mathsf{F}}}
\nc{\sH}{{\mathsf{H}}}
\nc{\sK}{{\mathsf{K}}}
\nc{\sM}{{\mathsf{M}}}
\nc{\sO}{{\mathsf{O}}}
\nc{\sQ}{{\mathsf{Q}}}
\nc{\sP}{{\mathsf{P}}}
\nc{\sT}{{\mathsf{T}}}
\nc{\sV}{{\mathsf{V}}}
\nc{\sW}{{\mathsf{W}}}
\nc{\sX}{{\mathsf{X}}}
\nc{\sZ}{{\mathsf{Z}}}
\nc{\sU}{{\mathsf{U}}}
\nc{\sS}{{\mathsf{S}}}
\nc{\sG}{{\mathsf{G}}}

\nc{\sfb}{{\mathsf{b}}}
\nc{\sfc}{{\mathsf{c}}}
\nc{\sd}{{\mathsf{d}}}
\nc{\sg}{{\mathsf{g}}}
\nc{\sk}{{\mathsf{k}}}
\nc{\sfl}{{\mathsf{l}}}
\nc{\sfp}{{\mathsf{p}}}
\nc{\sr}{{\mathsf{r}}}
\nc{\st}{{\mathsf{t}}}
\nc{\sfu}{{\mathsf{u}}}
\nc{\sw}{{\mathsf{w}}}
\nc{\sz}{{\mathsf{z}}}
\nc{\sx}{{\mathsf{x}}}
\nc{\se}{{\mathsf{e}}}
\nc{\sff}{{\mathsf{f}}}
\nc{\sfv}{{\mathsf{v}}}

\nc{\bLambda}{{\boldsymbol{\Lambda}}}
\nc{\vv}{{\boldsymbol{v}}}
\nc{\Fl}{{{\mathcal F}\ell}}
\nc{\Gr}{{\on{Gr}}}
\nc{\CHH}{{\CH\!\!\CH}}
\nc{\lambdavee}{{\lambda^{\!\scriptscriptstyle\vee}}}
\nc{\alphavee}{\alpha^{\!\scriptscriptstyle\vee}}
\nc{\rhovee}{{\rho^{\!\scriptscriptstyle\vee}}}
\newcommand\iso{\,\vphantom{j^{X^2}}\smash{\overset{\sim}{\vphantom{\rule{0pt}{0.20em}}\smash{\longrightarrow}}}\,}
\nc{\oQM}{\vphantom{j^{X^2}}\smash{\overset{\circ}{\vphantom{\vstretch{0.7}{A}}\smash{\QM}}}}
\nc{\oZ}{{}^\dagger\!\vphantom{j^{X^2}}\smash{\overset{\circ}{\vphantom{\vstretch{0.7}{A}}\smash{Z}}}}
\nc{\odZ}{{}^\dagger\!\vphantom{j^{X^2}}\smash{\overset{\circ}{\vphantom{\vstretch{0.7}{A}}\smash{\mathfrak Z}}}^{c',c}}
\nc{\bdZ}{{}^\dagger\!\vphantom{j^{X^2}}\smash{\overset{\bullet}{\vphantom{\vstretch{0.7}{A}}\smash{\mathfrak Z}}}^{c',c}}
\nc{\oS}{\vphantom{j^{X^2}}\smash{\overset{\circ}{\vphantom{\vstretch{0.7}{A}}\smash{S}}}}
\nc{\buM}{\vphantom{j^{X^2}}\smash{\overset{\bullet}{\vphantom{\vstretch{0.7}{A}}\smash{M}}}}
\nc{\dW}{{}^\dagger\ol\CW{}}
\nc{\hW}{{}^\dagger\hat\CW{}}
\nc{\wW}{{}^\dagger\widetilde\CW{}}
\nc{\dZ}{{}^\dagger\!\fZ^{c',c}}
\nc{\dZc}{{}^\dagger\!\fZ^{c,c}}
\nc{\tZ}{{}^\dagger\!\tilde{Z}{}}
\nc{\hZ}{{}^\dagger\!\hat{Z}{}}

\nc{\ssl}{\mathfrak{sl}} \nc{\gl}{\mathfrak{gl}}
\nc{\wt}{\widetilde} \nc{\Sym}{\mathrm{Sym}} \nc{\Res}{\mathrm{Res}}
\nc{\sE}{{\mathsf{E}}} \nc{\bs}{{\mathbf{s}}}
\nc{\trig}{\mathrm{trig}} \nc{\rat}{\mathrm{rat}}
\nc{\sign}{\mathrm{sign}} \nc{\sL}{{\mathsf{L}}}
\nc{\fv}{{\mathfrak{v}}} \nc{\ad}{\mathrm{ad}}
\nc{\spsi}{{\mathsf{\psi}}} \nc{\sh}{{\mathsf{h}}}
\nc{\rtt}{\mathrm{rtt}} \nc{\qdet}{\mathrm{qdet}} \nc{\pt}{{\operatorname{pt}}}
\nc{\M}{\mathrm{M}} \nc{\Ker}{\mathrm{Ker}} \nc{\ssc}{\mathrm{sc}}
\nc{\loc}{\mathrm{loc}} \nc{\fra}{\mathrm{frac}}
\nc{\ddj}{\mathrm{DJ}} \nc{\End}{\mathrm{End}} \nc{\ev}{\mathrm{ev}}
\nc{\GL}{\mathrm{GL}}
\nc{\ext}{\mathrm{ext}}
\nc{\Ad}{\mathrm{Ad}}
\nc{\jlm}{\mathrm{{JLM}}}

\nc{\blambda}{\boldsymbol{\lambda}}
\nc{\bmu}{\boldsymbol{\mu}}
\nc{\spp}{\mathfrak{p}}
\nc{\sll}{\mathfrak{l}}
\nc{\snn}{\mathfrak{n}}
\nc{\sso}{\mathfrak{so}}
\nc{\ssp}{\mathfrak{sp}}

\nc{\rrr}{{\mathsf{r}}}
\nc{\sss}{{\mathsf{s}}}
\nc{\supp}{\mathrm{supp}}
\nc{\cl}{\mathrm{cl}}
\nc{\op}{\mathrm{op}}

\nc{\Pop}{\mathrm{P}}
\nc{\Qop}{\mathrm{Q}}
\nc{\ID}{\mathrm{I}}
\nc{\rk}{r}
\nc{\extt}{X(so_{2\rk})}

\newcommand{\osc}[1]{\mathbf{#1}}
\renewcommand{\wp}{\bar{\mathbf{w}}}
\newcommand{\wm}{\mathbf{w}}
\newcommand{\oad}{\osc{\bar{a}}}
\newcommand{\oa}{\osc{a}}
\newcommand{\idb}{\mathrm{J}}
\newcommand{\id}{{\mathrm{I}}}
\newcommand{\ap}{\bar{\mathbf{A}}}
\newcommand{\am}{{\mathbf{A}}}

\begin{document}
\title[Rational Lax matrices: BCD types]
{Rational Lax matrices from\\ antidominantly shifted extended Yangians:\\ BCD types}

\author{Rouven Frassek}
 \address{R.F.: University of Modena and Reggio Emilia, FIM, 41125 Modena, Italy}
 \email{rouven.frassek@unimore.it}

\author{Alexander Tsymbaliuk}
 \address{A.T.:  Purdue University, Department of Mathematics, West Lafayette, IN 47907, USA}
 \email{sashikts@gmail.com}

\begin{abstract}
Generalizing~\cite{fpt}, we construct a family of $SO(2r),Sp(2r),SO(2r+1)$ rational Lax matrices $T_D(z)$,
polynomial in the spectral parameter $z$, parametrized by $\Lambda^+$-valued divisors $D$ on $\BP^1$.
To this end, we provide the RTT realization of the antidominantly shifted extended Drinfeld Yangians
of $\fg=\sso_{2r}, \ssp_{2r}, \sso_{2r+1}$, and of their coproduct homomorphisms.
\end{abstract}

\maketitle
\tableofcontents

    %%%%%%%%%%%%%%%%%%%%%%%%%%%%%%%%%%%%%%%%%%%%%%%%%%%%%%%%%%%%%%%%%%%%%%%%%%%%%%%
    %%%%%%%%%%%%%%%%%%%%%%%%%%%%%% INTRODUCTION %%%%%%%%%%%%%%%%%%%%%%%%%%%%%%%%%%%
    %%%%%%%%%%%%%%%%%%%%%%%%%%%%%%%%%%%%%%%%%%%%%%%%%%%%%%%%%%%%%%%%%%%%%%%%%%%%%%%

\section{Introduction}

    %%%%%%%%%%%%%%%%%%%%%%%%%%%%%%%%%%%%%%%%%%%%%%%%%%%%%%%%%%%%%%%%%%%%%%%%%%%%%%%
    %%%%%%%%%%%%%%%%%%%%%%%%%%%%%%%%%%%% SUMMARY %%%%%%%%%%%%%%%%%%%%%%%%%%%%%%%%%%
    %%%%%%%%%%%%%%%%%%%%%%%%%%%%%%%%%%%%%%%%%%%%%%%%%%%%%%%%%%%%%%%%%%%%%%%%%%%%%%%

\subsection{Summary}\label{ssec summary}
\

\noindent
The main results of the present paper are:
\begin{itemize}
\item[$\bullet$]
  The RTT realization of the antidominantly shifted (extended) Yangians associated to the simple Lie algebras $\fg$
  of the classical types $\sso_{2r},\ssp_{2r},\sso_{2r+1}$, generalizing the recent isomorphisms
  of~\cite{jlm1} in the non-shifted case. This naturally equips those algebras with the coproduct homomorphisms,
  which as we show do coincide with those of~\cite{fkp} (obtained by rather lengthy computations in the Drinfeld realization).
\item[$\bullet$]
  The construction of a family of (rational) Lax matrices, regular in the spectral parameter, of the corresponding type,
  parametrized by the divisors on the projective line $\BP^1$ with coefficients in $\Lambda^+$, the dominant integral cone
  of the coweight lattice of $\fg$. In the simplest cases, this recovers recent constructions in the physics literature~\cite{ikk,f,kk}.
\end{itemize}
Our exposition follows closely that of our previous joint work with V.~Pestun~\cite{fpt},
where both above constructions were carried out for $\fg=\ssl_n$ (extended version corresponding to $\gl_n$).

\medskip
\noindent
The original definition of Yangians $Y(\fg)$ associated to any simple Lie algebra $\fg$ is due
to~\cite{d1}, where these algebras are realized as Hopf algebras with a finite set of generators.
The representation theory of such algebras is best developed using their alternative
\emph{Drinfeld realization} proposed in~\cite{d2}, though the Hopf algebra
structure is much more involved in this presentation (e.g.\ the coproduct formula has been
known since the 90s, see~\cite[(2.8)--(2.11)]{kt}, but its proof has never appeared in the literature
until the very recent paper~\cite{gnw}).

\medskip
\noindent
For $\fg=\gl_n$, a closely related algebra was studied earlier in the work of L.~Faddeev's school
(see e.g.~\cite{frt}), where the algebra generators were encoded into an $n\times n$ square matrix
$T(z)$ subject to a single \emph{RTT relation}
\begin{equation}\label{eq:rtt intro}
  R_{12}(z-w)T_1(z)T_2(w) = T_2(w)T_1(z)R_{12}(z-w)
\end{equation}
involving the rational $R$-matrix $R(z)$ satisfying the \emph{Yang-Baxter equation}
\begin{equation}\label{eq:YB intro}
   R_{12}(z)R_{13}(z+w)R_{23}(w) = R_{23}(w)R_{13}(z+w)R_{12}(z)
\end{equation}
(note that the $\ssl_n$-version is recovered by imposing an extra relation $\mathrm{qdet}\, T(z)=1$).
The Hopf algebra structure is extremely simple in this RTT realization, which is suitable both for
the development of the representation theory and study of the corresponding integrable systems.

\medskip
\noindent
An explicit isomorphism from the new Drinfeld to the RTT realizations of type $A$ Yangians is
constructed using the Gauss decomposition of $T(z)$, a complete proof provided in~\cite{bk}
(the trigonometric version of this result established earlier in~\cite{df}).
A similar explicit isomorphism for the remaining classical types $B, C, D$
was only recently provided in~\cite{jlm1}, where it was again constructed using the Gauss decomposition of the
generating matrices $T(z)$ which are subject to the RTT relations~\eqref{eq:rtt intro} with the rational
solutions of~\eqref{eq:YB intro} first discovered in~\cite{zz}.
However, let us emphasize that the formulas recovering the matrix $T(z)$ through
the Drinfeld currents in $B,C,D$ types are significantly harder than their counterparts in type $A$, see our
Lemmas \ref{lem:all-H}, \ref{lem:all-E-known}, \ref{lem:all-E-new}, \ref{lem:all-F-known}, \ref{lem:all-F-new},
\ref{lem:known type C}, \ref{lem:new type C}, \ref{lem:known type B}, \ref{lem:new type B}, generalizing partial
results of~\cite{jlm1}.
We note that a non-constructive
existence of such an isomorphism for any $\fg$ was noted by V.~Drinfeld back in 80s, while
a detailed proof of this result was only recently provided in~\cite{wen}.

\medskip
\noindent
In the present paper, we are mostly interested with the shifted versions of the algebras above.
Historically, the shifted Yangians $Y_\nu(\fg)$ were first introduced for $\fg=\gl_n$
and dominant shifts $\nu$ in~\cite{bk2}, where their certain quotients were identified
with type $A$ finite $W$-algebras, the latter being natural quantizations of type $A$ Slodowy slices.
This construction was further generalized to any semisimple $\fg$ still with dominant
shifts $\nu\in \Lambda^+$ in~\cite{kwwy}, where it was shown that their ``GKLO-type'' quotients
(called \emph{truncated shifted Yangians}) quantize slices in the affine Grassmannians.
To this end, the authors constructed a family of algebra homomorphisms
\begin{equation}\label{eq:gklo homs}
  \Phi^{\lambda,\underline{x}}_\nu\colon\ Y_{\nu}(\fg)\longrightarrow \CA
\end{equation}
to the (localized) oscillator algebra $\CA$ (generalizing the construction of~\cite{gklo} for $\nu=0$)
parametrized by $\lambda\in \Lambda^+$ and an associated collection of points $\underline{x}\in \BC^N$.
The generalization to arbitrary shifts $\nu\in \Lambda$ was finally carried out in~\cite[Appendix B]{bfnb}
for simply-laced $\fg$ and later in~\cite[\S5]{nw} for non-simply-laced types, where it was
also shown (using earlier arguments of A.~Weekes) that their images quantize
\emph{generalized slices in the affine Grassmannians}.

\medskip
\noindent
In contrast to~\cite{bk2, kwwy}, we consider the opposite case of antidominantly shifted
Yangians (note that any shifted Yangian $Y_\nu(\fg)$ may be embedded into the antidominantly
shifted one $Y_{-\mu}(\fg),\ \mu\in \Lambda^+,$ via the \emph{shift homomorphisms} of~\cite{fkp}).
For $\fg=\sso_{2r},\ssp_{2r},\sso_{2r+1}$, we introduce the shifted extended Drinfeld Yangians
$X_\mu(\fg)$ related to $Y_{\nu}(\fg)$ via isomorphisms
\begin{equation}\label{eq:ext vs nonext}
  X_{\mu}(\fg)\simeq Y_{\bar{\mu}}(\fg)\otimes ZX_\mu(\fg)
\end{equation}
where the center $ZX_\mu(\fg)$ of $X_\mu(\fg)$ can be explicitly described via a central Cartan current.
For $\mu\in \Lambda^+$ and $\fg$ as above, we also introduce the shifted extended RTT Yangians
$X^\rtt_{-\mu}(\fg)$, whose generators are encoded in a single matrix $T(z)$ (the shift is reflected in the
powers of~$z$) subject to the relation~\eqref{eq:rtt intro}.
Based on and generalizing~\cite{jlm1}, we construct isomorphisms
\begin{equation}\label{eq:key isom}
  \Upsilon_{-\mu}\colon\ X_{-\mu}(\fg)\iso X^\rtt_{-\mu}(\fg) \qquad \mathrm{for\ any}\quad \mu\in \Lambda^+ \, .
\end{equation}
The construction of $\Upsilon_{-\mu}\colon X_{-\mu}(\fg)\twoheadrightarrow X^\rtt_{-\mu}(\fg)$ is
exactly the same as in~\cite{jlm1}, but the proof of its injectivity is different, since the arguments
of~\emph{loc.cit}.\ do not apply in the shifted setup.

\medskip
\noindent
To this end, we construct a family of $\CA((z^{-1}))$-valued Lax matrices $T_D(z)$, parametrized
by $\Lambda^+$-valued divisors $D$ on the projective line $\BP^1$, which can be equivalently thought of as algebra homomorphisms
$\Theta_D\colon X^{\rtt}_{-\mu}(\fg)\to \CA$ with $\mu=D|_\infty$, the coefficient of $[\infty]$. The compositions
\begin{equation}\label{eq:composed  maps}
  \Psi_D=\Theta_D\circ \Upsilon_{-\mu}\colon\ X_{-\mu}(\fg)\longrightarrow \CA
\end{equation}
coincide with extended versions of~\eqref{eq:gklo homs}. Combining this with the recent result of~\cite{w},
asserting that the intersection of kernels of~\eqref{eq:gklo homs} as $\lambda$ varies is trivial,
implies the injectivity of~$\Upsilon_{-\mu}$.

\medskip
\noindent
The aforementioned Lax matrices $T_D(z)$ are defined explicitly by providing the lower-triangular,
diagonal, and upper-triangular factors in their Gauss decomposition. The exact defining formulas
are exactly engineered (utilizing the new explicit formulas for the inverse of the isomorphism $\Upsilon_0$
constructed in~\cite{jlm1}) to allow us match the resulting homomorphisms $\Psi_D$ of~\eqref{eq:composed  maps} with extended
versions of~\eqref{eq:gklo homs}. Meanwhile, the fact that thus constructed matrices are Lax, i.e.\ satisfy~\eqref{eq:rtt intro},
follows from a simple \emph{renormalized limit} argument as we shall explain now
(expected from the physics of $\mathcal{N}=2$ ADE quiver gauge theories as explained in~\cite[p.~3]{fpt}).
To this end, we show that if the divisor $D$ contains a summand $\omega_i[x]$
(with $x\in \BP^1$ and $\omega_i$ being the $i$-th fundamental coweight of $\fg$)
and $D'$ is defined as $D'=D-\omega_i[x]+\omega_i[\infty]$, then
\begin{equation}\label{eq:limit construction intro}
  T_{D'}(z)\ =\underset{x\to \infty}{\lim}\, \Big\{(-x)^{\omega_i}\cdot\, T_D(z)\Big\}
\end{equation}
realizing $T_{D'}(z)$ as an $x\to \infty$ limit of $T_D(z)$ multiplied on the left by a $z$-independent
diagonal factor $(-x)^{\omega_i}$, the latter preserving the RTT relation~\eqref{eq:rtt intro}. Therefore, it suffices to prove that $T_D(z)$
satisfies the RTT relation for the divisors $D$ whose support does not contain $\infty\in \BP^1$.
However, the latter follows from the fact that $\Upsilon_0$ is indeed an isomorphism as proved in~\cite{jlm1}.

\medskip
\noindent
Similar to the type $A$ case treated in~\cite{fpt}, the Lax matrices $T_D(z)$ are actually regular in the spectral parameter $z$
(up to a rational factor). This provides a shortcut to the explicit formulas of all linear (in $z$) Lax
matrices $T_D(z)$, which we classify explicitly for each of the $B,C,D$ types. We also show that some of our simplest
linear and quadratic Lax matrices, after nontrivial canonical transformations, recover the recent constructions
in the physics literature~\cite{ikk,f,kk} (see also~\cite{r}).
The latter results were obtained by making an ansatz for the Lax matrices and subsequently solving
the conditions that arise from the RTT relation. We would like to point out that our formalism provides a recipe
to write down Lax matrices of any degree in the spectral parameter
(with the leading term not necessarily proportional to the identity matrix) without making such an ansatz.

%%%%%%%%%%%%%%%%%%%%%%%%%%%%%%%%%%%%%%% Comment %%%%%%%%%%%%%%%%%%%%%%%%%%%%%%%%%%%%%%%%%
% While the latter have the benefit of being \emph{polynomial}
% (as they take values in non-localized oscillator algebras), but already
% the degeneration procedure~\eqref{eq:limit construction intro} is completely
% non-trivial for them. An interesting question is to understand which of our
% Lax matrices can be transformed (up to canonical transformations) to polynomial ones.
%%%%%%%%%%%%%%%%%%%%%%%%%%%%%%%%%%%%%%%%%%%%%%%%%%%%%%%%%%%%%%%%%%%%%%%%%%%%%%%%%%%%%%%%%

\medskip
\noindent
The algebras $X^\rtt_{-\mu}(\fg)$ are naturally equipped with coassociative coproduct homomorphisms
\begin{equation}\label{eq:rtt coproduct intro}
  \Delta^\rtt_{-\mu_1,-\mu_2}\colon\
  X^\rtt_{-\mu_1-\mu_2}(\fg)\longrightarrow X^\rtt_{-\mu_1}(\fg)\otimes X^\rtt_{-\mu_2}(\fg)\, ,
  \qquad T(z)\mapsto T(z)\otimes T(z) \, .
\end{equation}
Evoking the isomorphisms~\eqref{eq:key isom} and the embeddings
$Y_{-\bar{\mu}}(\fg)\hookrightarrow X_{-\mu}(\fg)$, cf.~\eqref{eq:ext vs nonext}, we obtain
\begin{equation}\label{eq:Drinf coproduct intro}
  \Delta_{-\nu_1,-\nu_2}\colon\
  Y_{-\nu_1-\nu_2}(\fg)\longrightarrow Y_{-\nu_1}(\fg)\otimes Y_{-\nu_2}(\fg) \, .
\end{equation}
We show that the homomorphisms~\eqref{eq:Drinf coproduct intro} precisely coincide with
the coproduct homomorphisms of~\cite[Theorem 4.8]{fkp} provided in~\emph{loc.cit}.\ via lengthy formulas
(but suitable for any $\fg$).

\noindent
We note that both the isomorphism~\eqref{eq:key isom} and the identifications
of~(\ref{eq:rtt coproduct intro},~\ref{eq:Drinf coproduct intro}) with~\cite{fkp} were conjectured
recently (for a general $\fg$) in the physics literature~\cite[\S7-8]{cgy} (see also~\cite{dg}).

    %%%%%%%%%%%%%%%%%%%%%%%%%%%%%%%%%%%%%%%%%%%%%%%%%%%%%%%%%%%%%%%%%%%%%%%%%%%%%%%
    %%%%%%%%%%%%%%%%%%%%%%%%%%%%%% OUTLINE of the PAPER %%%%%%%%%%%%%%%%%%%%%%%%%%%
    %%%%%%%%%%%%%%%%%%%%%%%%%%%%%%%%%%%%%%%%%%%%%%%%%%%%%%%%%%%%%%%%%%%%%%%%%%%%%%%

\subsection{Outline of the paper}
\

The structure of the present paper is the following:

\medskip
\noindent
$\bullet$
In Section~\ref{sec Rational Lax matrices}, we present our results relevant to the classical type $D_r$ ($\fg=\sso_{2r}$)
in full details.

\medskip
\noindent
$\bullet$
In Section~\ref{sec C-type}, we provide our results relevant to the classical type $C_r$ (that is, for $\fg=\ssp_{2r}$).
Since this is very similar to the type $D_r$, we only highlight the few technical differences.

\medskip
\noindent
$\bullet$
In Section~\ref{sec B-type}, we provide our results relevant to the classical type $B_r$ (that is, for $\fg=\sso_{2r+1}$).
Since this is very similar to the type $D_r$, we only highlight the few technical differences.

\medskip
\noindent
$\bullet$
In Section~\ref{sec further directions}, we briefly discuss the further directions.

\medskip
\noindent
$\bullet$
In Appendix~\ref{Appendix A: Lax explicitly}, we provide explicit formulas for the Lax matrices in type $D_r$.

\medskip
\noindent
$\bullet$
In Appendix~\ref{Appendix B: shuffle realization}, we provide the shuffle algebra realization of the
key homomorphisms~\eqref{eq:gklo homs}, which allows us to derive the explicit formulas for the Lax
matrices in types $C_r$ and $B_r$.

\medskip

    %%%%%%%%%%%%%%%%%%%%%%%%%%%%%%%%%%%%%%%%%%%%%%%%%%%%%%%%%%%%%%%%%%%%%%%%%%%%%%%
    %%%%%%%%%%%%%%%%%%%%%%%%%%%%%%%% ACKNOWLEDGMENTS %%%%%%%%%%%%%%%%%%%%%%%%%%%%%%
    %%%%%%%%%%%%%%%%%%%%%%%%%%%%%%%%%%%%%%%%%%%%%%%%%%%%%%%%%%%%%%%%%%%%%%%%%%%%%%%

\subsection{Acknowledgments}
\

We are indebted to Vasily Pestun for bringing us together to the main subject of the current paper,
a natural BCD-type generalization of previous work~\cite{fpt} in type A, joint with Vasily.
We are also very grateful to the anonymous referees for useful suggestions.
R.F.\ acknowledges the support of the DFG Research Fellowships Programme No.~$416527151$.
A.T.\ is grateful to Boris Feigin, Michael Finkelberg, Igor Frenkel, Alexander Molev, Andrei Negu\c{t}, Leonid Rybnikov for discussions on the subject;
to Kevin Costello for a correspondence on~\cite{cgy}; to IHES (Bures-sur-Yvette) for the hospitality and wonderful working conditions.
A.T.\ gratefully acknowledges the support from NSF Grants DMS-$1821185$, DMS-$2001247$, and DMS-$2037602$.

    %%%%%%%%%%%%%%%%%%%%%%%%%%%%%%%%%%%%%%%%%%%%%%%%%%%%%%%%%%%%%%%%%%%%%%%%%%%%%%%
    %%%%%%%%%%%%%%%%%%%%%%%%%%%%%%%%%% Type D %%%%%%%%%%%%%%%%%%%%%%%%%%%%%%%%%%%%%
    %%%%%%%%%%%%%%%%%%%%%%%%%%%%%%%%%%%%%%%%%%%%%%%%%%%%%%%%%%%%%%%%%%%%%%%%%%%%%%%

\section{Type D}\label{sec Rational Lax matrices}

Consider the lattice $\bar{\Lambda}^\vee=\bigoplus_{j=1}^r \BZ\epsilon^\vee_j$, endowed
with the bilinear form with $(\epsilon^\vee_i,\epsilon^\vee_j)=\delta_{i,j}$. We realize
the simple positive roots $\{\alphavee_i\}_{i=1}^r$ of the Lie algebra $\sso_{2r}$ via:
\begin{equation}\label{eq:alpha-vee}
  \alphavee_1=\epsilon^\vee_1-\epsilon^\vee_2\, ,\
  \alphavee_2=\epsilon^\vee_2-\epsilon^\vee_3\, ,\ \ldots\, ,\
  \alphavee_{r-1}=\epsilon^\vee_{r-1}-\epsilon^\vee_r\, ,\
  \alphavee_r=\epsilon^\vee_{r-1}+\epsilon^\vee_r\, ,
\end{equation}
so that the Cartan matrix $A=(a_{ij})_{i,j=1}^r$ is symmetric and is given by
$a_{ij}=(\alphavee_i,\alphavee_j)$.

    %%%%%%%%%%%%%%%%%%%%%%%%%%%%%%%%%%%%%%%%%%%%%%%%%%%%%%%%%%%%%%%%%%%%%%%%%%%%%%%
    %%%%%%%%%%%%%%%%%%%%%%%%%%%% Unshisfted setup %%%%%%%%%%%%%%%%%%%%%%%%%%%%%%%%%
    %%%%%%%%%%%%%%%%%%%%%%%%%%%%%%%%%%%%%%%%%%%%%%%%%%%%%%%%%%%%%%%%%%%%%%%%%%%%%%%

\subsection{Classical (unshifted) story}
\label{ssec: unshifted D}
\

To motivate our constructions in the shifted setting, as well as to carry out the explicit
computation of the corresponding Lax matrices, we start by recalling the unshifted setup.

    %%%%%%%%%%%%%%%%%%%%%%%%%%%%%%%%%%%%%%%%%%%%%%%%%%%%%%%%%%%%%%%%%%%%%%%%%%%%%%%
    %%%%%%%%%%%%%%%%%%%%% Unshisfted Drinfeld Yangian D %%%%%%%%%%%%%%%%%%%%%%%%%%%
    %%%%%%%%%%%%%%%%%%%%%%%%%%%%%%%%%%%%%%%%%%%%%%%%%%%%%%%%%%%%%%%%%%%%%%%%%%%%%%%

\subsubsection{Drinfeld Yangian $Y(\sso_{2r})$ and its extended version $X(\sso_{2r})$}
\label{sssec: Drinfeld-unshifted-D}
\

The \emph{Drinfeld Yangian} of $\sso_{2r}$, denoted by $Y(\sso_{2r})$,
is the associative $\BC$-algebra generated by
  $\{\sE_i^{(k)},\sF_i^{(k)},\sH_i^{(k)}\}_{1\leq i\leq r}^{k\geq 1}$
with the following defining relations:\footnote{We note that our conventions $k\geq 1$ instead
of $k\geq 0$ are in charge of perceiving the Yangian
as a \emph{QFSHA} (quantum formal series Hopf algebra) which is related to a more standard
viewpoint of it as a \emph{QUEA} (quantum universal enveloping algebra) via the so-called
Drinfeld-Gavarini quantum duality principle, see~\cite{d3} and~\cite{ga}.}
\begin{equation}\label{Y0}
  [\sH_i^{(k)}, \sH_j^{(\ell)}]=0 \, ,
\end{equation}
\begin{equation}\label{Y1}
  [\sE_i^{(k)}, \sF_j^{(\ell)}]=
  \delta_{i,j}\, \sH_{i}^{(k+\ell-1)} \, ,
\end{equation}
\begin{equation}\label{Y2}
  [\sH_i^{(k'+1)}, \sE_j^{(\ell)}] - [\sH_i^{(k')}, \sE_j^{(\ell+1)}]=
  \frac{(\alphavee_i,\alphavee_j)}{2}\, \{\sH_i^{(k')},\sE_j^{(\ell)}\} \, ,
\end{equation}
\begin{equation}\label{Y3}
  [\sH_i^{(k'+1)}, \sF_j^{(\ell)}] - [\sH_i^{(k')}, \sF_j^{(\ell+1)}]=
  -\frac{(\alphavee_i,\alphavee_j)}{2}\, \{\sH_i^{(k')},\sF_j^{(\ell)}\} \, ,
\end{equation}
\begin{equation}\label{Y4}
  [\sE_i^{(k+1)}, \sE_j^{(\ell)}] - [\sE_i^{(k)}, \sE_j^{(\ell+1)}]=
  \frac{(\alphavee_i,\alphavee_j)}{2}\, \{\sE_i^{(k)},\sE_j^{(\ell)}\} \, ,
\end{equation}
\begin{equation}\label{Y5}
  [\sF_i^{(k+1)}, \sF_j^{(\ell)}] - [\sF_i^{(k)}, \sF_j^{(\ell+1)}]=
  -\frac{(\alphavee_i,\alphavee_j)}{2}\, \{\sF_i^{(k)},\sF_j^{(\ell)}\} \, ,
\end{equation}
\begin{equation}\label{Y6}
  \sum_{\sigma\in S(1-a_{ij})}
    [\sE_i^{(k_{\sigma(1)})},[\sE_i^{(k_{\sigma(2)})}, \,\cdots, [\sE_i^{(k_{\sigma(1-a_{ij})})},\sE_j^{(\ell)}]\cdots]]=0
  \quad \mathrm{for}\ i\ne j \, ,
\end{equation}
\begin{equation}\label{Y7}
  \sum_{\sigma\in S(1-a_{ij})}
    [\sF_i^{(k_{\sigma(1)})},[\sF_i^{(k_{\sigma(2)})}, \,\cdots, [\sF_i^{(k_{\sigma(1-a_{ij})})},\sF_j^{(\ell)}]\cdots]]=0
  \quad \mathrm{for}\ i\ne j \, ,
\end{equation}
for $i,j\in \{1,\ldots, r\}$, $k,\ell,k_s\in \BZ_{>0}$, and $k'\in \BZ_{\geq 0}$, where we set:
\begin{equation}\label{eq:conventions}
  \sH_i^{(0)}=1 \qquad \mathrm{and} \qquad \{a,b\}=ab+ba \, .
\end{equation}
Considering the generating series:
\begin{equation}\label{eq:EFH series}
  \sE_i(z):=\sum_{k\geq 1} \sE_i^{(k)}z^{-k}\, ,\
  \sF_i(z):=\sum_{k\geq 1} \sF_i^{(k)}z^{-k}\, ,\
  \sH_i(z):=\sum_{k\geq 0} \sH_i^{(k)}z^{-k} = 1 + \sum_{k\geq 1} \sH_i^{(k)}z^{-k} \, ,
\end{equation}
the defining relations~(\ref{Y0})--(\ref{Y7}) are easily seen to be equivalent to
(cf.~\cite[(6.1)--(6.5)]{jlm1}):
\begin{equation}\label{gY0}
  [\sH_i(z), \sH_j(w)]=0 \, ,
\end{equation}
\begin{equation}\label{gY1}
  [\sE_i(z), \sF_j(w)]=
  -\delta_{i,j}\, \frac{\sH_i(z)-\sH_i(w)}{z-w} \, ,
\end{equation}
\begin{equation}\label{gY2}
  [\sH_i(z), \sE_j(w)]=
  -\frac{(\alphavee_i,\alphavee_j)}{2}\, \frac{\{\sH_i(z),\sE_j(z)-\sE_j(w)\}}{z-w} \, ,
\end{equation}
\begin{equation}\label{gY3}
  [\sH_i(z), \sF_j(w)]=
  \frac{(\alphavee_i,\alphavee_j)}{2}\, \frac{\{\sH_i(z),\sF_j(z)-\sF_j(w)\}}{z-w} \, ,
\end{equation}
\begin{equation}\label{gY4}
  [\sE_i(z), \sE_j(w)] + [\sE_j(z), \sE_i(w)]=
  -\frac{(\alphavee_i,\alphavee_j)}{2}\, \frac{\{\sE_i(z)-\sE_i(w),\sE_j(z)-\sE_j(w)\}}{z-w} \, ,
\end{equation}
\begin{equation}\label{gY5}
  [\sF_i(z), \sF_j(w)] + [\sF_j(z), \sF_i(w)]=
  \frac{(\alphavee_i,\alphavee_j)}{2}\, \frac{\{\sF_i(z)-\sF_i(w),\sF_j(z)-\sF_j(w)\}}{z-w} \, ,
\end{equation}
\begin{equation}\label{gY6}
  \sum_{\sigma\in S(1-a_{ij})}
    [\sE_i(z_{\sigma(1)}),[\sE_i(z_{\sigma(2)}), \,\cdots, [\sE_i(z_{\sigma(1-a_{ij})}),\sE_j(w)]\cdots]]=0\quad
  \mathrm{for}\ i\ne j \, ,
\end{equation}
\begin{equation}\label{gY7}
  \sum_{\sigma\in S(1-a_{ij})}
    [\sF_i(z_{\sigma(1)}),[\sF_i(z_{\sigma(2)}), \,\cdots, [\sF_i(z_{\sigma(1-a_{ij})}),\sF_j(w)]\cdots]]=0
  \quad \mathrm{for}\ i\ne j \, .
\end{equation}

\medskip
\noindent
Likewise, following~\cite[Theorem 5.14]{jlm1}, the \emph{extended Drinfeld Yangian} of $\sso_{2r}$,
denoted by $X(\sso_{2r})$, is defined as the associative $\BC$-algebra generated by
  $\{E_i^{(k)},F_i^{(k)}\}_{1\leq i\leq r}^{k\geq 1}\cup \{D_i^{(k)}\}_{1\leq i\leq r+1}^{k\geq 1}$
with the following defining relations:
\begin{equation}\label{eY0}
  [D_i(z), D_j(w)]=0 \, ,
\end{equation}
\begin{equation}\label{eY1}
  [E_i(z), F_j(w)]=
  -\delta_{i,j}\, \frac{K_i(z)-K_i(w)}{z-w} \, ,
\end{equation}
\begin{equation}\label{eY2.1}
  [D_i(z), E_j(w)] = (\epsilon^\vee_i,\alphavee_j)\, \frac{D_i(z)(E_j(z)-E_j(w))}{z-w}
  \quad \mathrm{if}\ i\leq r \, ,
\end{equation}
\begin{equation}\label{eY2.2}
  [D_{r+1}(z), E_j(w)] =
  \begin{cases}
    -(\epsilon^\vee_r,\alphavee_r)\, \frac{D_{r+1}(z)(E_r(z)-E_r(w))}{z-w} & \mbox{if } j=r \\
    \frac{D_{r+1}(z)(E_{r-1}(z)-E_{r-1}(w))}{z-w} & \mbox{if } j=r-1 \\
    0 & \mbox{if } j<r-1
  \end{cases} \, ,
\end{equation}
\begin{equation}\label{eY3.1}
  [D_i(z), F_j(w)] = -(\epsilon^\vee_i,\alphavee_j)\, \frac{(F_j(z)-F_j(w))D_i(z)}{z-w}
  \quad \mathrm{if}\ i\leq r \, ,
\end{equation}
\begin{equation}\label{eY3.2}
  [D_{r+1}(z), F_j(w)] =
  \begin{cases}
    (\epsilon^\vee_r,\alphavee_r)\, \frac{(F_r(z)-F_r(w))D_{r+1}(z)}{z-w} & \mbox{if } j=r \\
    -\frac{(F_{r-1}(z)-F_{r-1}(w))D_{r+1}(z)}{z-w} & \mbox{if } j=r-1 \\
    0 & \mbox{if } j<r-1
  \end{cases} \, ,
\end{equation}
\begin{equation}\label{eY4.1}
  [E_i(z),E_i(w)] = -\frac{(\alphavee_i,\alphavee_i)}{2}\, \frac{(E_i(z)-E_i(w))^2}{z-w} \, ,
\end{equation}
\begin{equation}\label{eY4.2}
  z[E^\circ_i(z),E_j(w)] - w[E_i(z),E_j^\circ(w)] = (\alphavee_i,\alphavee_j)\, E_i(z)E_j(w)\quad \mathrm{for}\ i\ne j \, ,
\end{equation}
\begin{equation}\label{eY5.1}
  [F_i(z),F_i(w)] = \frac{(\alphavee_i,\alphavee_i)}{2}\, \frac{(F_i(z)-F_i(w))^2}{z-w} \, ,
\end{equation}
\begin{equation}\label{eY5.2}
  z[F^\circ_i(z),F_j(w)] - w[F_i(z),F_j^\circ(w)] = -(\alphavee_i,\alphavee_j)\, F_j(w)F_i(z)\quad \mathrm{for}\ i\ne j \, ,
\end{equation}
\begin{equation}\label{eY6}
  \sum_{\sigma\in S(1-a_{ij})}
    [E_i(z_{\sigma(1)}),[E_i(z_{\sigma(2)}), \,\cdots, [E_i(z_{\sigma(1-a_{ij})}),E_j(w)]\cdots]]=0
  \quad \mathrm{for}\ i\ne j \, ,
\end{equation}
\begin{equation}\label{eY7}
  \sum_{\sigma\in S(1-a_{ij})}
    [F_i(z_{\sigma(1)}),[F_i(z_{\sigma(2)}), \,\cdots, [F_i(z_{\sigma(1-a_{ij})}),F_j(w)]\cdots]]=0
  \quad \mathrm{for}\ i\ne j \, ,
\end{equation}
where the generating series are defined via:
\begin{equation}\label{eq:extended EF generating series}
\begin{split}
  & E_i(z):=\sum_{k\geq 1} E_i^{(k)}z^{-k} \, , \qquad
    E^\circ_i(z):=\sum_{k\geq 2} E_i^{(k)}z^{-k} \, , \\
  & F_i(z):=\sum_{k\geq 1} F_i^{(k)}z^{-k} \, , \qquad
    F^\circ_i(z):=\sum_{k\geq 2} F_i^{(k)}z^{-k} \, ,
\end{split}
\end{equation}
as well as:
\begin{equation}\label{eq:extended DK generating series}
\begin{split}
  & D_i(z):=\sum_{k\geq 0} D_i^{(k)}z^{-k} = 1 + \sum_{k\geq 1} D_i^{(k)}z^{-k} \, , \\
  & K_i(z):=
    \begin{cases}
      D_i(z)^{-1}D_{i+1}(z) & \mbox{if } i<r \\
      D_{r-1}(z)^{-1} D_{r+1}(z) & \mbox{if } i=r
    \end{cases}.
\end{split}
\end{equation}

\medskip
\noindent
Let us define the elements $\{C_r^{(k)}\}_{k\geq 1}$ of $X(\sso_{2r})$ via:
\begin{equation}\label{eq:central C}
  C_r(z)=1+\sum_{k\geq 1} C_r^{(k)}z^{-k}:=
  \prod_{i=1}^{r-1}\frac{D_i(z+i-r)}{D_i(z+i-r+1)} \cdot D_r(z) D_{r+1}(z).
\end{equation}
The following result follows from~\cite[Main Theorem, Theorem 5.8]{jlm1}:

\medskip

\begin{Lem}\label{lem:C is central}
The elements $\{C_r^{(k)}\}_{k\geq 1}$ are in the center of $X(\sso_{2r})$.
\end{Lem}

\medskip
\noindent
This result is actually an immediate corollary of the defining relations~(\ref{eY0},~\ref{eY2.1}--\ref{eY3.2}),
as the proof below shows. This will allow us to generalize it to the shifted setup in Subsection~\ref{sssec shifted Drinfeld extended}.

\medskip

\begin{proof}
$C_r(z)$ obviously commutes with all $\{D_i(w)\}_{i=1}^{r+1}$, due to~\eqref{eY0}.
We shall now verify that it also commutes with all $\{E_i(w)\}_{i=1}^r$
(cf.~\cite[Theorem 7.2]{bk} for the type $A$ counterpart); the commutativity
with $\{F_i(w)\}_{i=1}^r$ is completely analogous and is left to the interested reader.

\medskip
\noindent
$\bullet$
For $i\leq r-2$, the relations~(\ref{eY2.1},~\ref{eY2.2}) imply:
\begin{equation}\label{eq:DE1}
  (z-w+1)D_i(z)E_i(w)-D_i(z)E_i(z)=(z-w)E_i(w)D_i(z) \, ,
\end{equation}
\begin{equation}\label{eq:DE2}
  (z-w-1)D_{i+1}(z)E_i(w)+D_{i+1}(z)E_i(z)=(z-w)E_i(w)D_{i+1}(z) \, .
\end{equation}
Setting $w=z-1$ in~\eqref{eq:DE2}, we find:
\begin{equation}\label{eq:DE3}
  E_i(z-1)D_{i+1}(z)=D_{i+1}(z)E_i(z) \, .
\end{equation}
Now, calculating $(z-w)E_i(w)D_i(z)D_{i+1}(z+1)$ using~(\ref{eq:DE1})--(\ref{eq:DE3}), we find
that it equals $(z-w)D_i(z)D_{i+1}(z+1)E_i(w)$. Hence, $E_i(w)$ commutes with $D_i(z)D_{i+1}(z+1)$.
But it also commutes with $D_j(z)$ for $j\ne i,i+1$, due to~(\ref{eY2.1},~\ref{eY2.2}).
Thus, $[C_r(z),E_i(w)]=0$ for $i\leq r-2$.

\medskip
\noindent
$\bullet$
For $i=r-1$, applying the same arguments we see that $E_{r-1}(w)$ commutes both
with $D_{r-1}(z)D_r(z+1)$ and $D_{r+1}(z)D_r(z+1)$, hence, it also commutes with
$$
  \frac{D_{r-1}(z-1)}{D_{r-1}(z)}D_r(z)D_{r+1}(z)=
  \frac{(D_{r-1}(z-1)D_{r}(z))\cdot (D_r(z+1)D_{r+1}(z))}{D_{r-1}(z)D_{r}(z+1)} \, .
$$
As $[E_{r-1}(w),D_j(w)]=0$ for $j<r-1$ by~\eqref{eY2.1}, we thus get the equality
$[C_r(z),E_{r-1}(w)]=0$.

\medskip
\noindent
$\bullet$
For $i=r$, applying the same arguments we see that $E_{r}(w)$ commutes with
$D_{r}(z)D_{r+1}(z+1)$ as well as with $D_{r-1}(z)D_{r+1}(z+1)$, hence, it also commutes with
$$
  \frac{D_{r-1}(z-1)}{D_{r-1}(z)}D_r(z)D_{r+1}(z)=
  \frac{(D_{r-1}(z-1)D_{r+1}(z))\cdot (D_r(z)D_{r+1}(z+1))}{D_{r-1}(z)D_{r+1}(z+1)} \, .
$$
As $[E_{r}(w),D_j(w)]=0$ for $j<r-1$ by~\eqref{eY2.1}, we thus get the equality $[C_r(z),E_{r}(w)]=0$.
\end{proof}

\medskip
\noindent
On the other hand, comparing the defining relations of $Y(\sso_{2r})$ and $X(\sso_{2r})$,
it is easy to check (see~\cite[Proposition 6.2]{jlm1}) that there is a natural homomorphism
\begin{equation}\label{eq:iota-null}
  \iota_0\colon Y(\sso_{2r})\longrightarrow X(\sso_{2r}) \, ,
\end{equation}
determined by:
\begin{equation}\label{eq:iota-null explicitly}
\begin{split}
  & \sE_i(z)\mapsto
    \begin{cases}
      E_{i}(z+\frac{i-1}{2}) & \mbox{if } i<r \\
      E_{r}(z+\frac{r-2}{2}) & \mbox{if } i=r
    \end{cases} \, , \qquad
    \sF_i(z)\mapsto
    \begin{cases}
      F_i(z+\frac{i-1}{2}) & \mbox{if } i<r \\
      F_r(z+\frac{r-2}{2}) & \mbox{if } i=r
    \end{cases} \, , \\
  & \sH_i(z)\mapsto
    \begin{cases}
      D_i(z+\frac{i-1}{2})^{-1}D_{i+1}(z+\frac{i-1}{2}) & \mbox{if } i<r \\
      D_{r-1}(z+\frac{r-2}{2})^{-1}D_{r+1}(z+\frac{r-2}{2}) & \mbox{if } i=r
    \end{cases} \, .
\end{split}
\end{equation}

\medskip

\begin{Lem}\label{lem:embedding}
$\iota_0$~\eqref{eq:iota-null} is an embedding and we have a tensor product algebra decomposition:
\begin{equation}\label{eq:decomposition of X}
  X(\sso_{2r})\simeq Y(\sso_{2r})\otimes_{\BC} \BC[\{C_r^{(k)}\}_{k\geq 1}] \, .
\end{equation}
\end{Lem}

\medskip

\begin{proof}
Given an abstract polynomial algebra
  $\mathcal{B}=\BC[\{D_i^{(k)}\}_{1\leq i\leq r+1}^{k\geq 1}]$,
define the elements $\{\bar{D}_i^{(k)}\}_{1\leq i\leq r}^{k\geq 1}$
and $\{C_r^{(k)}\}_{k \geq 1}$ of $\mathcal{B}$ via
\begin{equation*}
  \bar{D}_i(z):=1+\sum_{k\geq 1}\bar{D}_i^{(k)}z^{-k}=D_i(z)^{-1}D_{i+1}(z) \, , \qquad 1\leq i<r \, ,
\end{equation*}
\begin{equation*}
  \bar{D}_r(z):=1+\sum_{k\geq 1}\bar{D}_r^{(k)}z^{-k}=D_{r-1}(z)^{-1}D_{r+1}(z) \, ,
\end{equation*}
\begin{equation*}
  C_r(z):=1+\sum_{k\geq 1}C_r^{(k)}z^{-k}=
  \prod_{i=1}^{r-1}\frac{D_i(z+i-r)}{D_i(z+i-r+1)} \cdot D_r(z) D_{r+1}(z) \, ,
\end{equation*}
where $D_i(z):=1+\sum_{k\geq 1} D_i^{(k)}z^{-k}$.
It is clear that
  $\{\bar{D}_i^{(k)}\}_{1\leq i\leq r}^{k\geq 1}\cup \{C_{r}^{(k)}\}_{k\geq 1}$
provide an alternative collection of generators of the polynomial algebra $\mathcal{B}$,
so that we have:
\begin{equation*}
  \mathcal{B}\simeq
  \BC[\{C_r^{(k)}\}_{k\geq 1}]\otimes_\BC \BC[\{\bar{D}_i^{(k)}\}_{1\leq i\leq r}^{k\geq 1}].
\end{equation*}
Applying this in our setup, we get a tensor product decomposition of vector spaces:
\begin{equation}\label{eq:tensor decomp vspace}
   X(\sso_{2r})\simeq Z\otimes_{\BC} X'(\sso_{2r}) \, ,
\end{equation}
where $Z$ is a $\BC$-subalgebra generated by
$\{C_r^{(k)}\}_{k\geq 1}$ and $X'(\sso_{2r})$ is the $\BC$-subalgebra generated by
$\{E_{i}^{(k)},F_i^{(k)},\bar{D}_i^{(k)}\}_{1\leq i\leq r}^{k\geq 1}$. Moreover,
the defining relations~(\ref{eY0})--(\ref{eY3.2}) are equivalent to $Z$ being central
(as explained above) and the commutation relations between $\bar{D}_i^{(k)}$ and
$E_i^{(k)},F_i^{(k)}$ exactly matching those of $Y(\sso_{2r})$ through~(\ref{eq:iota-null explicitly}).
Thus, $\iota_0$ of~(\ref{eq:iota-null},~\ref{eq:iota-null explicitly}) is indeed injective,
and furthermore~\eqref{eq:tensor decomp vspace} precisely recovers the tensor product
algebra decomposition~\eqref{eq:decomposition of X}.
\end{proof}

\medskip

    %%%%%%%%%%%%%%%%%%%%%%%%%%%%%%%%%%%%%%%%%%%%%%%%%%%%%%%%%%%%%%%%%%%%%%%%%%%%%%%
    %%%%%%%%%%%%%%%%%%%%%%%% Unshisfted RTT Yangian D %%%%%%%%%%%%%%%%%%%%%%%%%%%%%
    %%%%%%%%%%%%%%%%%%%%%%%%%%%%%%%%%%%%%%%%%%%%%%%%%%%%%%%%%%%%%%%%%%%%%%%%%%%%%%%

\subsubsection{RTT Yangian $Y^\rtt(\sso_{2r})$ and its extended version $X^\rtt(\sso_{2r})$}
\label{sssec: RTT-unshifted-D}
\

It will be convenient to use the following notations:
\begin{equation}\label{eq:N,kappa,prime}
\begin{split}
  & N=2r \, , \qquad \kappa=r-1 \, , \\
  & i'=N+1-i \quad \mathrm{for} \quad 1\leq i\leq N \, .
\end{split}
\end{equation}
Following~\cite{zz}, we consider the rational \emph{$R$-matrix} $R(z)$ given by:
\begin{equation}\label{eq:R-matrix}
  R(z)=\mathrm{Id}+\frac{\Pop}{z}-\frac{\Qop}{z+\kappa}
\end{equation}
with $\Pop,\Qop\in \mathrm{End}\, \BC^N \otimes \mathrm{End}\, \BC^N$ defined via:
\begin{equation}
  \Pop=\sum_{i,j=1}^N E_{ij}\otimes E_{ji} \, , \qquad
  \Qop=\sum_{i,j=1}^N E_{ij}\otimes E_{i'j'} \, .
\end{equation}
We note the following relations:
\begin{equation*}
  \Pop^2=\mathrm{Id} \, ,\qquad \Qop^2=N \Qop \, , \qquad \Pop\Qop=\Qop\Pop=\Qop \, ,
\end{equation*}
which imply that $R(z)$ of~\eqref{eq:R-matrix} satisfies the Yang-Baxter equation with a spectral parameter:
\begin{equation}\label{rYB}
  R_{12}(z)R_{13}(z+w)R_{23}(w) = R_{23}(w)R_{13}(z+w)R_{12}(z) \, .
\end{equation}

\medskip
\noindent
The \emph{extended RTT Yangian} of $\sso_{2r}$, denoted by $X^\rtt(\sso_{2r})$, is the
associative $\BC$-algebra generated by $\{t_{ij}^{(k)}\}_{1\leq i,j\leq N}^{k\geq 1}$
with the following defining relation (the so-called \emph{RTT relation}):
\begin{equation}\label{eq:rtt}
  R_{12}(z-w)T_1(z)T_2(w) = T_2(w)T_1(z)R_{12}(z-w) \, ,
\end{equation}
where
  $T(z)\in X^\rtt(\sso_{2r})[[z^{-1}]]\otimes_\BC \End\, \BC^N$
is defined via:
\begin{equation}\label{eq:T-matrix}
  T(z)=\sum_{i,j=1}^N t_{ij}(z)\otimes E_{ij} \qquad \mathrm{with} \qquad
  t_{ij}(z):=\sum_{k\geq 0} t^{(k)}_{ij}z^{-k} = \delta_{i,j} + \sum_{k\geq 1} t^{(k)}_{ij}z^{-k} \, ,
\end{equation}
where we set $t_{ij}^{(0)}=\delta_{i,j}$. Thus,~\eqref{eq:rtt} is an equality in
  $X^\rtt(\sso_{2r})[[z^{-1},w^{-1}]]\otimes_\BC (\End\, \BC^N)^{\otimes 2}$,
which can be explicitly written as:
\begin{equation}\label{eq:rtt explicit}
\begin{split}
  [t_{ij}(z),t_{k\ell}(w)]
  & = \frac{1}{z-w}\Big(t_{kj}(w)t_{i\ell}(z)-t_{kj}(z)t_{i\ell}(w)\Big)\, + \\
  & \quad \frac{1}{z-w+\kappa}
    \left(\delta_{k,i'}\sum_{p=1}^N t_{pj}(z)t_{p'\ell}(w)-\delta_{\ell,j'}\sum_{p=1}^N t_{kp'}(w)t_{ip}(z)\right)\, .
\end{split}
\end{equation}
These formulas immediately imply the following simple result, which will be needed later:

\medskip

\begin{Cor}\label{cor:YBE invariance}
If $\mathsf{T}^\circ(z)$ satisfies~\eqref{eq:rtt} and
  $\mathsf{T}=\mathrm{diag}(\mathsf{t}_1,\ldots,\mathsf{t}_{2r})$
is a diagonal $z$-independent matrix such that
  $\mathsf{t}_1\mathsf{t}_{2r}=\mathsf{t}_{2}\mathsf{t}_{2r-1}=\ldots=\mathsf{t}_{r}\mathsf{t}_{r+1}$,
then $\bar{\mathsf{T}}^\circ(z):=\mathsf{T}\cdot \mathsf{T}^\circ(z)$ also satisfies~\eqref{eq:rtt}.
\end{Cor}

\medskip
\noindent
The \emph{RTT Yangian} of $\sso_{2r}$, denoted by $Y^\rtt(\sso_{2r})$, is the subalgebra
of $X^\rtt(\sso_{2r})$ which consists of the elements stable under the automorphisms:
\begin{equation}\label{eq:f-autom}
  \mu_f\colon T(z)\mapsto f(z)T(z) \, , \qquad \forall\,
  f(z)=1+f_1z^{-1}+f_2z^{-2}+\ldots \in \BC[[z^{-1}]] \, .
\end{equation}
At the same time, $Y^\rtt(\sso_{2r})$ may also be viewed as a quotient of $X^\rtt(\sso_{2r})$.
To this end, we recall the following tensor product decomposition (see~\cite[Theorem 3.1, Corollary 3.9]{amr}):
\begin{equation}\label{eq:extended to usual}
  X^\rtt(\sso_{2r})\simeq ZX^\rtt(\sso_{2r})\otimes_{\BC} Y^\rtt(\sso_{2r}) \, ,
\end{equation}
where $ZX^\rtt(\sso_{2r})$ is the center of $X^\rtt(\sso_{2r})$.
Explicitly, $ZX^\rtt(\sso_{2r})$ is a polynomial algebra in the coefficients
$\{\sz_N^{(k)}\}_{k\geq 1}$ of the series
\begin{equation}\label{eq:z-series}
  \sz_N(z)=1+\sum_{k\geq 1} \sz_N^{(k)}z^{-k} \,,
\end{equation}
determined from (with $\ID_N$ denoting the $N\times N$ identity matrix):
\begin{equation}\label{eq:zenter}
  T'(z-\kappa)T(z)= T(z)T'(z-\kappa)=\sz_N(z)\ID_N \, ,
\end{equation}
where the prime denotes the matrix transposition along the antidiagonal, that is:
\begin{equation}\label{eq:prime}
  (X')_{ij}=X_{j'i'} \, \quad  \mathrm{for\ any} \quad N\times N \ \mathrm{matrix}\ X \, .
\end{equation}
Therefore, the \emph{RTT Yangian} $Y^\rtt(\sso_{2r})$ may also be realized
as a quotient of $X^\rtt(\sso_{2r})$ by:
\begin{equation}\label{eq:killing center}
  \sz_N(z)=1+\sum_{k\geq 1} b_kz^{-k}\qquad \mathrm{for\ any\ collection\ of}\ b_k\in \BC \, ,
\end{equation}
though it is common (\cite[Corollary 3.2]{amr}) to choose $b_{\geq 1}=0$, so that~\eqref{eq:killing center} reads $\sz_N(z)=1$.

\medskip

    %%%%%%%%%%%%%%%%%%%%%%%%%%%%%%%%%%%%%%%%%%%%%%%%%%%%%%%%%%%%%%%%%%%%%%%%%%%%%%%
    %%%%%%%%%%%%%%%%%%%%%%%%%%%%%% From RTT to Drinfeld %%%%%%%%%%%%%%%%%%%%%%%%%%%
    %%%%%%%%%%%%%%%%%%%%%%%%%%%%%%%%%%%%%%%%%%%%%%%%%%%%%%%%%%%%%%%%%%%%%%%%%%%%%%%

\subsubsection{From RTT to Drinfeld realization}
\label{sssec: RTT-to-Drinfeld-D}
\

Consider the Gauss decomposition of the matrix $T(z)$ of~\eqref{eq:T-matrix}:
\begin{equation}\label{eq:Gauss-D}
  T(z)=F(z)\cdot H(z)\cdot E(z) \, ,
\end{equation}
where $H(z)$ is diagonal:
\begin{equation}\label{eq:diagonal factor}
  H(z)=\left(\begin{array}{cccc}
        h_1(z) & 0 & \cdots & 0 \\
        0 & h_2(z) & \cdots & 0 \\
        \vdots & \ddots & \ddots & \vdots \\
        0 & \cdots & 0 & h_{N}(z)
       \end{array}\right) \, ,
\end{equation}
$F(z)$ is lower-triangular:
\begin{equation}\label{eq:lower triangular}
  F(z)=\left(\begin{array}{cccc}
        1 & 0 & \cdots & 0 \\
        f_{2,1}(z) & 1 & \ddots & \vdots \\
        \vdots & \ddots & \ddots & 0 \\
        f_{N,1}(z) & \cdots & f_{N,N-1}(z) & 1
       \end{array}\right) \, ,
\end{equation}
and $E(z)$ is upper-triangular:
\begin{equation}\label{eq:upper triangular}
  E(z)=\left(\begin{array}{cccc}
        1 & e_{1,2}(z) & \cdots & e_{1,N}(z) \\
        0 & 1 & \ddots & \vdots \\
        \vdots & \ddots & \ddots & e_{N-1,N}(z)\\
        0 & \cdots & 0 & 1
       \end{array}\right) \, .
\end{equation}

\medskip
\noindent
The following explicit identification of the Drinfeld and RTT extended Yangians
of $\sso_{2r}$ constitutes the key result of~\cite{jlm1}:

\medskip

\begin{Thm}[{\cite[Theorem 5.14]{jlm1}}]\label{thm:Dr=RTT-unshifted-D}
There is a $\BC$-algebra isomorphism:
\begin{equation}\label{eq:unshifted isomorphism}
  \Upsilon_0\colon X(\sso_{2r})\iso X^\rtt(\sso_{2r}) \, ,
\end{equation}
defined by:
\begin{equation}\label{eq:explicit identification 1}
  E_i(z)\mapsto
    \begin{cases}
      e_{i,i+1}(z) & \mbox{if\, } i<r \\
      e_{r-1,r+1}(z) & \mbox{if\, } i=r
    \end{cases} \, , \qquad
  F_i(z)\mapsto
    \begin{cases}
      f_{i+1,i}(z) & \mbox{if\, } i<r \\
      f_{r+1,r-1}(z) & \mbox{if\, } i=r
    \end{cases}
\end{equation}
and
\begin{equation}\label{eq:explicit identification 2}
  D_j(z)\mapsto h_j(z) \qquad \mathrm{for}\quad 1\leq j\leq r+1 \, .
\end{equation}
\end{Thm}

\medskip
\noindent
Combining the Theorem above with Lemma~\ref{lem:embedding}, we obtain the following
explicit identification of the Drinfeld and RTT Yangians of $\sso_{2r}$:

\medskip

\begin{Thm}[{\cite[Main Theorem]{jlm1}}]\label{thm:JLM Main thm}
The composition of the algebra embedding $\iota_0$~\eqref{eq:iota-null} and the algebra isomorphism
$\Upsilon_0$~\eqref{eq:unshifted isomorphism} gives rise to a $\BC$-algebra isomorphism:
\begin{equation*}
  \Upsilon_0\circ \iota_0 \colon Y(\sso_{2r}) \iso Y^\rtt(\sso_{2r}) \, .
\end{equation*}
Explicitly, it is given by:
\begin{equation}\label{eq:explicit identification}
\begin{split}
  & \sE_i(z)\mapsto
    \begin{cases}
      e_{i,i+1}(z+\frac{i-1}{2}) & \mbox{if\, } i<r \\
      e_{r-1,r+1}(z+\frac{r-2}{2}) & \mbox{if\, } i=r
    \end{cases} \, , \\
  & \sF_i(z)\mapsto
    \begin{cases}
      f_{i+1,i}(z+\frac{i-1}{2}) & \mbox{if\, } i<r \\
      f_{r+1,r-1}(z+\frac{r-2}{2}) & \mbox{if\, } i=r
    \end{cases} \, , \\
  & \sH_i(z)\mapsto
    \begin{cases}
      h_i(z+\frac{i-1}{2})^{-1}h_{i+1}(z+\frac{i-1}{2}) & \mbox{if\, } i<r \\
      h_{r-1}(z+\frac{r-2}{2})^{-1}h_{r+1}(z+\frac{r-2}{2}) & \mbox{if\, } i=r
    \end{cases} \, .
\end{split}
\end{equation}
\end{Thm}

\medskip

\begin{Rem}\label{rmk:comparison to JLM}
(a) We note that our $R$-matrix $R(z)$~(\ref{eq:R-matrix}) is related to the one
of~\cite[(2.6)]{jlm1}, to be denoted by $R^{\jlm}(z)$, via $R(z)=R^{\jlm}(-z)$.
Therefore, our matrix $T(z)$~(\ref{eq:T-matrix}) is related to the one of~\cite[(2.10)]{jlm1},
to be denoted by $T^{\jlm}(z)$, via $T(z)=T^{\jlm}(-z)$. This explains the sign difference
between our relations~\eqref{eY1}--\eqref{eY5.2} and those of~\cite[Theorem~5.14]{jlm1}.

\medskip
\noindent
(b) Accordingly, our formulas~\eqref{eq:explicit identification} agree with those of~\cite{jlm1},
once we identify the generating series $\sE_i(z), \sF_i(z), \sH_i(z)$ of $Y(\sso_{2r})$
with $\xi^-_i(-z),\xi^+_i(-z),\kappa_i(-z)$ of~\cite[(1.5)]{jlm1}, respectively.
\end{Rem}

\medskip

\begin{Rem}\label{rmk:center comparison}
Evoking the series $C_r(z)$ of~\eqref{eq:central C} and $\sz_N(z)$ of~\eqref{eq:z-series}, we note that:
\begin{equation}\label{eq:JLM-center}
  \sz_N(z)=\prod_{i=1}^{r-1}\frac{h_i(z+i-r)}{h_i(z+i-r+1)} \cdot h_r(z) h_{r+1}(z)=\Upsilon_0(C_r(z))
\end{equation}
with the first equality due to~\cite[Theorem 5.8]{jlm1}. Combining~\eqref{eq:JLM-center} with
Theorems~\ref{thm:Dr=RTT-unshifted-D},~\ref{thm:JLM Main thm}, Lemma~\ref{lem:embedding},
and the aforementioned isomorphism $ZX^\rtt(\sso_{2r})\simeq \BC[\{\sz_N^{(k)}\}_{k\geq 1}]$,
we see that the center of $Y(\sso_{2r})$ is trivial, while the center of $X(\sso_{2r})$
is a polynomial algebra in $\{C_r^{(k)}\}_{k\geq 1}$.
\end{Rem}

\medskip

    %%%%%%%%%%%%%%%%%%%%%%%%%%%%%%%%%%%%%%%%%%%%%%%%%%%%%%%%%%%%%%%%%%%%%%%%%%%%%%%
    %%%%%%%%%%%%%%%%%%%%%%%%%%%%%% From Drinfeld to RTT %%%%%%%%%%%%%%%%%%%%%%%%%%%
    %%%%%%%%%%%%%%%%%%%%%%%%%%%%%%%%%%%%%%%%%%%%%%%%%%%%%%%%%%%%%%%%%%%%%%%%%%%%%%%

\subsubsection{From Drinfeld to RTT realization}
\label{sssec: Drinfeld-to-RTT-D}
\

To simplify some of the upcoming formulas, let us introduce the following notations:
\begin{equation}\label{eq:generating ef}
\begin{split}
  & e_i(z)=\sum_{k\geq 1} e_i^{(k)}z^{-k}:=
    \begin{cases}
      e_{i,i+1}(z) & \mbox{if } i<r \\
      e_{r-1,r+1}(z) & \mbox{if } i=r
    \end{cases} \, , \\
  & f_i(z)=\sum_{k\geq 1} f_i^{(k)}z^{-k}:=
    \begin{cases}
      f_{i+1,i}(z) & \mbox{if } i <r \\
      f_{r+1,r-1}(z) & \mbox{if } i=r
    \end{cases} \, .
\end{split}
\end{equation}
According to Theorem~\ref{thm:Dr=RTT-unshifted-D}, the coefficients of
$\{e_i(z),f_i(z)\}_{i=1}^r \cup \{h_j(z)\}_{j=1}^{r+1}$ generate the algebra $X^\rtt(\sso_{2r})$.
In this Subsection, we record the explicit formulas (those of~\cite{jlm1} as well as some new ones)
for all other entries of the matrices $F(z), H(z), E(z)$ in~(\ref{eq:Gauss-D})--(\ref{eq:upper triangular}).

\medskip
\noindent
But first let us recall the key ingredient of~\cite{jlm1}: the embeddings
$X^\rtt(\sso_{2(r-s)})\hookrightarrow X^\rtt(\sso_{2r})$ for any $0\leq s<r$.
To this end, consider the following $(2r-2s)\times (2r-2s)$ submatrices:
\begin{equation}\label{eq:diagonal truncated}
  H^{[s]}(z)=\left(\begin{array}{cccc}
        h_{s+1}(z) & 0 & \cdots & 0 \\
        0 & h_{s+2}(z) & \cdots & 0 \\
        \vdots & \ddots & \ddots & \vdots \\
        0 & \cdots & 0 & h_{(s+1)'}(z)
       \end{array}\right) \, ,
\end{equation}
\begin{equation}\label{eq:lower truncated}
  F^{[s]}(z)=\left(\begin{array}{cccc}
        1 & 0 & \cdots & 0 \\
        f_{s+2,s+1}(z) & 1 & \ddots & \vdots \\
        \vdots & \ddots & \ddots & 0 \\
        f_{(s+1)',s+1}(z) & \cdots & f_{(s+1)',(s+2)'}(z) & 1
       \end{array}\right) \, ,
\end{equation}
\begin{equation}\label{eq:upper truncated}
  E^{[s]}(z)=\left(\begin{array}{cccc}
        1 & e_{s+1,s+2}(z) & \cdots & e_{s+1,(s+1)'}(z) \\
        0 & 1 & \ddots & \vdots \\
        \vdots & \ddots & \ddots & e_{(s+2)',(s+1)'}(z)\\
        0 & \cdots & 0 & 1
       \end{array}\right) \, .
\end{equation}
Then, according to~\cite[Proposition~4.1]{jlm1}, the matrix
\begin{equation}\label{eq:T truncated}
  T^{[s]}(z):=F^{[s]}(z)\cdot H^{[s]}(z)\cdot E^{[s]}(z) \,
\end{equation}
coincides with the image of the corresponding $T$-matrix of $X^\rtt(\sso_{2(r-s)})$ under the
embedding $X^\rtt(\sso_{2(r-s)})\hookrightarrow X^\rtt(\sso_{2r})$ of~\cite[Theorem~3.7]{jlm1}
constructed using the quasideterminants. While we omit the details of the latter construction,
let us record an important corollary that provides a powerful ``rank-reduction''
tool that will be used through the rest of this Subsection:

\medskip

\begin{Cor}[{\cite[Corollary 4.2]{jlm1}}]\label{cor:JLM reduction}
The subalgebra of $X^\rtt(\sso_{2r})$ generated by the coefficients of all matrix coefficients
of the matrix $T^{[s]}(z)$~\eqref{eq:T truncated} is isomorphic to $X^\rtt(\sso_{2(r-s)})$.
\end{Cor}

   %%%%%%%%%%%%%%%%%%%%%%%%%%%%%%%%%% H-factor %%%%%%%%%%%%%%%%%%%%%%%%%%%%%%%%%%%%%

\medskip
\noindent
$\bullet$ \underline{\emph{Matrix $H(z)$ explicitly}}.
\

\begin{Lem}\label{lem:all-H}
For $1\leq i\leq r-1$, we have:
\begin{equation}\label{eq:h-explicit 1}
  h_{i'}(z) =
  \frac{1}{h_i(z+i-r+1)}\cdot \prod_{j=i+1}^{r-1} \frac{h_j(z+j-r)}{h_j(z+j-r+1)} \cdot h_r(z) h_{r+1}(z) \, .
\end{equation}
\end{Lem}

\begin{proof}
For $i=1$, this follows from~\eqref{eq:JLM-center} combined with the equality
$h_{N}(z)=\frac{\sz_N(z)}{h_1(z-r+1)}$ of~\cite[(5.14)]{jlm1} (obtained by comparing the
$(N,N)$ matrix coefficients of both sides of the equality $T'(z-\kappa)=\sz_N(z)T(z)^{-1}$).
The general case follows now from Corollary~\ref{cor:JLM reduction}.
\end{proof}

   %%%%%%%%%%%%%%%%%%%%%%%%%%%%%%%%%% E-factor %%%%%%%%%%%%%%%%%%%%%%%%%%%%%%%%%%%%%

\medskip
\noindent
$\bullet$ \underline{\emph{Matrix $E(z)$ explicitly}}.

The following result is essentially due to~\cite{jlm1}:\footnote{Note
a sign and index errors in the equality from part (f) as stated in~\cite{jlm1}.}

\medskip

\begin{Lem}\label{lem:all-E-known}
(a) $e_{r,r+1}(z)=0$.

\medskip
\noindent
(b) $e_{r,(r-1)'}(z)=-e_{r}(z)$.

\medskip
\noindent
(c) $e_{(i+1)',i'}(z)=-e_i(z+i-r+1)$ for $1\leq i\leq r-1$.

\medskip
\noindent
(d) $e_{i,j+1}(z)=-[e_{i,j}(z),e_j^{(1)}]$ for $1\leq i<j\leq r-1$.

\medskip
\noindent
(e) $e_{i,j'}(z)=[e_{i,(j+1)'}(z),e_j^{(1)}]$ for $1\leq i<j \leq r-1$.

\medskip
\noindent
(f) $e_{i,r'}(z)=-[e_{i,r-1}(z),e_r^{(1)}]$ for $1\leq i\leq r-2$.

\medskip
\noindent
(g) $e_{i',j'}(z) = [e_{i',(j+1)'}(z),e_{j}^{(1)}]$ for $1\leq j\leq i-2 \leq r-2$.
\end{Lem}

\medskip

\begin{proof}
(a, b) follow from Corollary~\ref{cor:JLM reduction} and their validity for $r=2$
(the latter follow from the results of~\cite[\S4]{amr} in a straightforward way, see~\cite[Lemma 5.3]{jlm1}).

(c) is~\cite[Proposition 5.7]{jlm1} (due to Corollary~\ref{cor:JLM reduction}, it suffices
to prove it for $i=1$ case, in which case it follows by comparing the $(N-1,N)$ matrix coefficients
of both sides of the equality $T'(z-\kappa)=T(z)^{-1}\sz_N(z)$ and using the equality
$h_1(z)e_1(z)=e_1(z+1)h_1(z)$, a result of applying~\eqref{eq:rtt explicit} to the computation of
$[t_{11}(z),t_{12}(z+1)]=[h_1(z),h_1(z+1)e_{1,2}(z+1)]$).

(d, e, f) are~\cite[Lemma 5.15]{jlm1} (due to Corollary~\ref{cor:JLM reduction}, it suffices
to prove them for $i=1$, in which case they follow by evaluating the $w^{-1}$-coefficients in the expressions
$[t_{1j}(z),t_{j,j+1}(w)]$, $[t_{1,(j+1)'}(z),t_{(j+1)',j'}(w)]$, $[t_{1,r-1}(z),t_{r-1,r+1}(w)]$,
respectively, using~\eqref{eq:rtt explicit}, combined with the equalities $t_{1k}(z)=h_1(z)e_{1,k}(z)$
and $e_{(i+1)',i'}^{(1)}=-e_i^{(1)}$, the latter due to part (c)).

(g) follows immediately from~\cite[Proposition 5.6]{jlm1} (based on the observation that multiplying
the bottom-right $r\times r$ submatrices of $F(z),H(z),E(z)$ provides an $r\times r$ matrix satisfying
the RTT relation of type $A$) and the equality $e_{(j+1)',j'}^{(1)}=-e_j^{(1)}$ due to part (c).
\end{proof}

\medskip
\noindent
The remaining matrix coefficients of $E(z)$ from~\eqref{eq:upper triangular} are recovered via:

\medskip

\begin{Lem}\label{lem:all-E-new}
(a) $e_{i,i'}(z) = [e_{i,(i+1)'}(z),e_{i}^{(1)}]-e_{i}(z)e_{i,(i+1)'}(z)$ for $1\leq i\leq r-1$.

\medskip
\noindent
(b) $e_{i+1,i'}(z) = [e_{i+1,(i+1)'}(z),e_{i}^{(1)}] + e_{i}(z)e_{i+1,(i+1)'}(z)-e_{i,(i+1)'}(z)$
for $1\leq i\leq r-2$.

\medskip
\noindent
(c) $e_{i,j'}(z) = [e_{i,(j+1)'}(z),e_{j}^{(1)}]$ for $1\leq j\leq i-2 \leq r-2$.
\end{Lem}

\medskip

\begin{proof}
(a) Due to Corollary~\ref{cor:JLM reduction}, it suffices to establish this equality for $i=1$.
Comparing the $w^{-1}$-coefficients in the equality
  $[t_{1,2r-1}(z),t_{2r-1,2r}(w)]=\frac{t_{2r-1,2r-1}(z)t_{1,2r}(w)-t_{2r-1,2r-1}(w)t_{1,2r}(z)}{w-z}$
of~\eqref{eq:rtt explicit}, we get: $[t_{1,2r-1}(z),t_{2r-1,2r}^{(1)}]=-t_{1,2r}(z)$.
Note that $t_{2r-1,2r}^{(1)}=e_{2r-1,2r}^{(1)}=-e_{12}^{(1)}$, due to Lemma~\ref{lem:all-E-known}(c).
Combining this with the identities $t_{1k}(z)=h_1(z)e_{1,k}(z)$, we find:
\begin{equation}\label{eq:e-new-1a}
  [h_1(z),e_1^{(1)}]e_{1,2r-1}(z) + h_1(z)[e_{1,2r-1}(z),e_1^{(1)}] = h_1(z)e_{1,2r}(z) \, .
\end{equation}
On the other hand, we have $[t_{11}(z),t_{12}(w)]=\frac{t_{11}(w)t_{12}(z)-t_{11}(z)t_{12}(w)}{z-w}$,
so that $[h_1(z),e_{1}(w)]=\frac{h_1(z)\left(e_{1}(z)-e_{1}(w)\right)}{z-w}$. Comparing the
$w^{-1}$-coefficients of both sides of the latter equality, we get:
\begin{equation}\label{eq:e-new-1b}
  [h_1(z),e_{1}^{(1)}]=-h_1(z)e_{1}(z) \, .
\end{equation}
Combining the formulas~(\ref{eq:e-new-1a},~\ref{eq:e-new-1b}), we immediately obtain the desired equality:
\begin{equation}\label{eq:e-new-1}
  e_{1,2r}(z) = [e_{1,2r-1}(z),e_{1}^{(1)}]-e_{1}(z)e_{1,2r-1}(z) \, .
\end{equation}

\medskip
(b)  Due to Corollary~\ref{cor:JLM reduction}, it suffices to establish this equality for $i=1$.
To this end, let us compare the $w^{-1}$-coefficients in the equality
\begin{equation*}
  [t_{2,2r-1}(z),t_{2r-1,2r}(w)]=
  \frac{t_{2r-1,2r-1}(z)t_{2,2r}(w)-t_{2r-1,2r-1}(w)t_{2,2r}(z)}{w-z} +
  \frac{\sum_p t_{p,2r-1}(z)t_{p',2r}(w)}{z-w+r-1}
\end{equation*}
of~\eqref{eq:rtt explicit}, which together with the aforementioned equality
$t_{2r-1,2r}^{(1)}=-e_{1}^{(1)}$ implies:
\begin{equation}\label{eq:e-new-2a}
  [t_{2,2r-1}(z),e_{1}^{(1)}]=t_{2,2r}(z)+t_{1,2r-1}(z) \, .
\end{equation}
Note that
\begin{equation}\label{eq:e-new-2b}
  t_{2,2r-1}(z)=h_2(z)e_{2,2r-1}(z)+f_{1}(z)h_1(z)e_{1,2r-1}(z) \, .
\end{equation}
Comparing the $w^{-1}$-coefficients of both sides of
  $[t_{21}(z),t_{12}(w)]=\frac{t_{11}(z)t_{22}(w)-t_{11}(w)t_{22}(z)}{w-z}$,
we get $[f_{1}(z)h_1(z),e_{1}^{(1)}]=t_{11}(z)-t_{22}(z)=h_1(z)-t_{22}(z)$, so that:
\begin{equation}\label{eq:e-new-2c}
  [f_{1}(z)h_1(z)e_{1,2r-1}(z),e_{1}^{(1)}]=
  \left(h_1(z)-t_{22}(z)\right)e_{1,2r-1}(z)+t_{21}(z)[e_{1,2r-1}(z),e_{1}^{(1)}] \, .
\end{equation}
We also have $[h_2(z),e_{1}^{(1)}]=h_2(z)e_{1}(z)$, so that:
\begin{equation}\label{eq:e-new-2d}
  [h_2(z)e_{2,2r-1}(z),e_{1}^{(1)}]=
  h_2(z)\left(e_{1}(z)e_{2,2r-1}(z)+[e_{2,2r-1}(z),e_{1}^{(1)}]\right) \, .
\end{equation}

\medskip
\noindent
Combining the formulas~\eqref{eq:e-new-2a}--\eqref{eq:e-new-2d} with~\eqref{eq:e-new-1},
we immediately obtain the desired equality:
\begin{equation}\label{eq:e-new-2}
  e_{2,2r}(z) = [e_{2,2r-1}(z),e_{1}^{(1)}] + e_{1}(z)e_{2,2r-1}(z)-e_{1,2r-1}(z) \, .
\end{equation}

\medskip
(c) Due to Corollary~\ref{cor:JLM reduction}, it suffices to establish this equality for $j=1$.
We shall proceed by induction on $i$. Comparing the $w^{-1}$-coefficients in both parts of
  $[t_{i,2r-1}(z),t_{2r-1,2r}(w)]=\frac{t_{2r-1,2r-1}(z)t_{i,2r}(w)-t_{2r-1,2r-1}(w)t_{i,2r}(z)}{w-z}$
of~\eqref{eq:rtt explicit}, and evoking $t_{2r-1,2r}^{(1)}=-e_{1}^{(1)}$, we obtain:
\begin{equation}\label{eq:e-new-3a}
  [t_{i,2r-1}(z),e_{1}^{(1)}]=t_{i,2r}(z) \, .
\end{equation}
Note that the series featuring in~\eqref{eq:e-new-3a} are explicitly given by:
\begin{equation}\label{eq:e-new-3b}
\begin{split}
  & t_{i,2r}(z)=h_i(z)e_{i,2r}(z)+\sum_{j=1}^{i-1} f_{i,j}(z)h_j(z)e_{j,2r}(z) \, ,\\
  & t_{i,2r-1}(z)=h_i(z)e_{i,2r-1}(z)+\sum_{j=1}^{i-1} f_{i,j}(z) h_j(z) e_{j,2r-1}(z) \, .
\end{split}
\end{equation}
Comparing the $w^{-1}$-coefficients in both sides of
$[t_{i1}(z),t_{12}(w)]=\frac{t_{11}(z)t_{i2}(w)-t_{11}(w)t_{i2}(z)}{w-z}$,
we obtain $[t_{i1}(z),e_{1}^{(1)}]=-t_{i2}(z)=-f_{i,2}(z)h_2(z)-f_{i,1}(z)h_1(z)e_{1}(z)$,
so that:
\begin{equation}\label{eq:e-new-3c}
\begin{split}
  & [f_{i,1}(z)h_1(z)e_{1,2r-1}(z),e_{1}^{(1)}] =\\
  & f_{i,1}(z)h_1(z)\left([e_{1,2r-1}(z),e_{1}^{(1)}]-e_{1}(z)e_{1,2r-1}(z)\right) - f_{i,2}(z)h_2(z)e_{1,2r-1}(z) \, .
\end{split}
\end{equation}
For $j=2$, we have $[f_{i,2}(z),e_{1}^{(1)}]=0$ and $[h_2(z),e_{1}^{(1)}]=h_2(z)e_{1}(z)$, so that:
\begin{equation}\label{eq:e-new-3d}
  [f_{i,2}(z)h_2(z)e_{2,2r-1}(z),e_{1}^{(1)}]=
  f_{i,2}(z)h_2(z)\left(e_{1}(z)e_{2,2r-1}(z)+[e_{2,2r-1}(z),e_{1}^{(1)}]\right)\, .
\end{equation}
Finally, for $2<j\leq i-1$, we clearly have $[f_{i,j}(z),e_{1}^{(1)}]=0=[h_j(z),e_{1}^{(1)}]$, so that:
\begin{equation}\label{eq:e-new-3e}
  [f_{i,j}(z)h_j(z)e_{j,2r-1}(z),e_{1}^{(1)}]=f_{i,j}(z)h_j(z)[e_{j,2r-1}(z),e_{1}^{(1)}]=f_{i,j}(z)h_j(z)e_{j,2r}(z)
\end{equation}
with the last equality due to the induction assumption.

\medskip
\noindent
Combining the formulas~(\ref{eq:e-new-1},~\ref{eq:e-new-2},~\ref{eq:e-new-3a}--\ref{eq:e-new-3e}),
we immediately obtain the desired equality:
\begin{equation}\label{eq:e-new-3}
  e_{i,2r}(z) = [e_{i,2r-1}(z),e_{1}^{(1)}] \qquad \mathrm{for}\quad 3\leq i\leq r \, .
\end{equation}

\medskip
\noindent
This completes our proof of Lemma~\ref{lem:all-E-new}.
\end{proof}

\medskip
\noindent
Let us record the recursive relations that follow from the above two Lemmas:
\begin{equation}\label{eq:e-recursive}
\begin{split}
  & e_{i,j+1}(z)=[e_{j}^{(1)},[e_{j-1}^{(1)}, \,\cdots, [e_{i+2}^{(1)},[e_{i+1}^{(1)},e_{i}(z)]]\cdots]]
    \, , \qquad 1\leq i<j\leq r-1 \, , \\
  & e_{i,r'}(z)=[e_{r}^{(1)},[e_{r-2}^{(1)}, \,\cdots, [e_{i+2}^{(1)},[e_{i+1}^{(1)},e_{i}(z)]]\cdots]]
    \, , \qquad 1\leq i\leq r-2 \, , \\
  & e_{i,j'}(z)=[[\cdots [[e_{i,r'}(z),e_{r-1}^{(1)}],e_{r-2}^{(1)}], \,\cdots, e_{j+1}^{(1)}],e_{j}^{(1)}]
    \, , \qquad 1\leq i<j\leq r-1 \, , \\
  & e_{i,j'}(z)=[[\cdots [[e_{i,(i-1)'}(z),e_{i-2}^{(1)}],e^{(1)}_{i-3}], \,\cdots,e_{j+1}^{(1)}],e_{j}^{(1)}]
    \, , \qquad 1\leq j\leq i-2\leq r-2 \, , \\
  & e_{i',j'}(z)=[[\cdots [[e_{i',(i-1)'}(z),e_{i-2}^{(1)}],e^{(1)}_{i-3}], \,\cdots,e_{j+1}^{(1)}],e_{j}^{(1)}]
    \, , \qquad 1\leq j\leq i-2\leq r-2 \, .
\end{split}
\end{equation}

   %%%%%%%%%%%%%%%%%%%%%%%%%%%%%%%%%% F-factor %%%%%%%%%%%%%%%%%%%%%%%%%%%%%%%%%%%%%

\medskip
\noindent
$\bullet$ \underline{\emph{Matrix $F(z)$ explicitly}}.

The following result is essentially due to~\cite{jlm1}\footnote{Note a sign and index errors in
the equality from part (f) as stated in~\cite{jlm1}.} and is proved exactly as Lemma~\ref{lem:all-E-known}:

\medskip

\begin{Lem}\label{lem:all-F-known}
(a) $f_{r+1,r}(z)=0$.

\medskip
\noindent
(b) $f_{(r-1)',r}(z)=-f_{r}(z)$.

\medskip
\noindent
(c) $f_{i',(i+1)'}(z)=-f_{i}(z+i-r+1)$ for $1\leq i\leq r-1$.

\medskip
\noindent
(d) $f_{j+1,i}(z)=-[f_j^{(1)},f_{j,i}(z)]$ for $1\leq i<j\leq r-1$.

\medskip
\noindent
(e) $f_{j',i}(z)=[f_j^{(1)},f_{(j+1)',i}(z)]$ for $1\leq i<j \leq r-1$.

\medskip
\noindent
(f) $f_{r',i}(z)=-[f_r^{(1)},f_{r-1,i}(z)]$ for $1\leq i\leq r-2$.

\medskip
\noindent
(g) $f_{j',i'}(z) = [f_{j}^{(1)},f_{(j+1)',i'}(z)]$ for $1\leq j\leq i-2\leq r-2$.
\end{Lem}

\medskip
\noindent
The remaining matrix coefficients of $F(z)$~\eqref{eq:lower triangular} are recovered
via the analogue of Lemma~\ref{lem:all-E-new}:

\medskip

\begin{Lem}\label{lem:all-F-new}
(a) $f_{i',i}(z) = [f_{i}^{(1)},f_{(i+1)',i}(z)]-f_{(i+1)',i}(z)f_{i}(z)$
for $1\leq i\leq r-1$.

\medskip
\noindent
(b) $f_{i',i+1}(z) = [f_{i}^{(1)},f_{(i+1)',i+1}(z)] + f_{(i+1)',i+1}(z)f_{i}(z)-f_{(i+1)',i}(z)$
for $1\leq i\leq r-2$.

\medskip
\noindent
(c) $f_{j',i}(z) = [f_{j}^{(1)},f_{(j+1)',i}(z)]$ for $1\leq j\leq i-2 \leq r-2$.
\end{Lem}

\medskip
\noindent
Let us record the recursive relations that follow from the above two Lemmas:
\begin{equation}\label{eq:f-recursive}
\begin{split}
  & f_{j+1,i}(z)=[[\cdots [[f_i(z),f_{i+1}^{(1)}],f_{i+2}^{(1)}], \,\cdots,f_{j-1}^{(1)}],f_{j}^{(1)}]
    \, , \qquad 1\leq i<j\leq r-1 \, ,\\
  & f_{r',i}(z)=[[\cdots [[f_i(z),f_{i+1}^{(1)}],f_{i+2}^{(1)}], \,\cdots,f_{r-2}^{(1)}],f_{r}^{(1)}]
    \, , \qquad 1\leq i\leq r-2 \, , \\
  & f_{j',i}(z)=[f_j^{(1)},[f_{j+1}^{(1)}, \,\cdots, [f_{r-2}^{(1)},[f_{r-1}^{(1)},f_{r',i}(z)]]\cdots]]
    \, , \qquad 1\leq i<j\leq r-1 \, , \\
  & f_{j',i}(z)=[f_j^{(1)},[f_{j+1}^{(1)}, \,\cdots, [f^{(1)}_{i-3},[f_{i-2}^{(1)},f_{(i-1)',i}(z)]]\cdots ]]
    \, , \qquad 1\leq j\leq i-2\leq r-2 \, , \\
  & f_{j',i'}(z)=[f_j^{(1)},[f_{j+1}^{(1)}, \,\cdots, [f^{(1)}_{i-3},[f_{i-2}^{(1)},f_{(i-1)',i'}(z)]]\cdots ]]
    \, , \qquad 1\leq j\leq i-2\leq r-2 \, .
\end{split}
\end{equation}

\medskip

    %%%%%%%%%%%%%%%%%%%%%%%%%%%%%%%%%%%%%%%%%%%%%%%%%%%%%%%%%%%%%%%%%%%%%%%%%%%%%%%
    %%%%%%%%%%%%%%%%%%%%%%%%%%%%%%%%%% Shifted Story %%%%%%%%%%%%%%%%%%%%%%%%%%%%%%
    %%%%%%%%%%%%%%%%%%%%%%%%%%%%%%%%%%%%%%%%%%%%%%%%%%%%%%%%%%%%%%%%%%%%%%%%%%%%%%%

\subsection{Shifted story}
\label{ssec: shifted story}
\

    %%%%%%%%%%%%%%%%%%%%%%%%%%%%%%%%%%%%%%%%%%%%%%%%%%%%%%%%%%%%%%%%%%%%%%%%%%%%%%%
    %%%%%%%%%%%%%%%%%%%%%%%%%%% Shifted Drinfeld Yangians %%%%%%%%%%%%%%%%%%%%%%%%%
    %%%%%%%%%%%%%%%%%%%%%%%%%%%%%%%%%%%%%%%%%%%%%%%%%%%%%%%%%%%%%%%%%%%%%%%%%%%%%%%

\subsubsection{Shifted extended Drinfeld Yangians of $\sso_{2r}$}
\label{sssec shifted Drinfeld extended}
\

Consider the \emph{extended} lattice
  $\Lambda^\vee=\bigoplus_{j=1}^{r+1} \BZ\epsilon^\vee_j=\bar{\Lambda}^\vee\oplus \BZ\epsilon^\vee_{r+1}$,
endowed with the bilinear form via $(\epsilon^\vee_i,\epsilon^\vee_j)=\delta_{i,j}$.
We shall need the following family of elements $\{\hat{\alpha}^\vee_i\}_{i=1}^r$ of $\Lambda^\vee$:
\begin{equation}\label{eq:hat-alpha-vee}
  \hat{\alpha}^\vee_1=\epsilon^\vee_1-\epsilon^\vee_2\, ,\
  \hat{\alpha}^\vee_2=\epsilon^\vee_2-\epsilon^\vee_3\, ,\ \ldots\, ,\
  \hat{\alpha}^\vee_{r-1}=\epsilon^\vee_{r-1}-\epsilon^\vee_r\, ,\
  \hat{\alpha}^\vee_r=\epsilon^\vee_{r-1}-\epsilon^\vee_{r+1}\, .
\end{equation}
Let $\Lambda=\bigoplus_{j=1}^{r+1} \BZ\epsilon_j$ be the dual lattice with
$\epsilon^{\vee}_i(\epsilon_j)=\delta_{i,j}$. Identifying the dual space
$(\Lambda^\vee \otimes_{\BZ} \BC)^*$ with $\Lambda^\vee \otimes_{\BZ} \BC$ via
the form $(\cdot,\cdot)$, the lattice $\Lambda$ gets naturally identified with
$\Lambda^\vee$ via $\epsilon_i\leftrightarrow \epsilon^\vee_i$. We will also
need another $\BZ$-basis: $\Lambda=\bigoplus_{i=0}^{r} \BZ\varpi_i$ with
\begin{equation}\label{eq:varpis}
  \varpi_{r-1}:=-\epsilon_r \, ,\
  \varpi_r:=-\epsilon_{r+1}\, ,\
  \varpi_i=-\epsilon_{i+1}-\epsilon_{i+2}-\ldots-\epsilon_{r+1} \quad \mathrm{for}\ 0\leq i<r-1 \,.
\end{equation}
For $\mu\in \Lambda$, define $\unl{d}=\{d_j\}_{j=1}^{r+1}\in \BZ^{r+1}$
and $\unl{b}=\{b_i\}_{i=1}^{r}\in \BZ^{r}$ via:
\begin{equation}\label{extended D shifts}
  d_j:=\epsilon^\vee_j(\mu) \, ,
\end{equation}
\begin{equation}\label{extended D shifts 1}
  b_i:=\hat{\alpha}^\vee_i(\mu) \, ,
\end{equation}
so that:
\begin{equation}\label{extended D shifts 2}
  b_1=d_1-d_2 \, ,\ b_2=d_2-d_3 \, ,\ \ldots\, ,\ b_{r-1}=d_{r-1}-d_r \, ,\ b_r=d_{r-1}-d_{r+1} \, .
\end{equation}

\medskip
\noindent
For $\mu\in \Lambda$, define the \emph{shifted extended Drinfeld Yangian of $\sso_{2r}$},
denoted by $X_\mu(\sso_{2r})$, to be the associative $\BC$-algebra generated by
  $\{E_i^{(k)},F_i^{(k)}\}_{1\leq i\leq r}^{k\geq 1}\cup \{D_i^{(k_i)}\}_{1\leq i\leq r+1}^{k_i\geq d_i+1}$
with the defining relations~(\ref{eY0},~\ref{eY2.1}--\ref{eY7}) and the following replacement of~(\ref{eY1}):
\begin{equation}\label{eY1-shifted}
  [E_i(z), F_j(w)]=
  -\delta_{i,j}\, \frac{\underline{K_i(z)}-\underline{K_i(w)}}{z-w} \, ,
\end{equation}
where $E_i(z),E^\circ_i(z), F_i(z), F^\circ_i(z)$ are defined via~\eqref{eq:extended EF generating series},
$D_i(z),K_i(z)$ are defined via:
\begin{equation}\label{eq:shifted extended DK generating series}
\begin{split}
  & D_i(z):=\sum_{k\geq d_i} D_i^{(k)}z^{-k} = z^{-d_i} \, + \sum_{k\geq d_i+1} D_i^{(k)}z^{-k} \, , \\
  & K_i(z)\, =\sum_{k\geq -b_i} K_i^{(k)}z^{-k}:=
    \begin{cases}
      D_i(z)^{-1}D_{i+1}(z) & \mbox{if } i<r \\
      D_{r-1}(z)^{-1} D_{r+1}(z) & \mbox{if } i=r
    \end{cases} \, ,
\end{split}
\end{equation}
with the conventions
\begin{equation*}
  D_i^{(d_i)}=1=K_i^{(-b_i)} \, ,
\end{equation*}
and finally $\underline{K_i(z)}$ denotes the principal part of $K_i(z)$:
\begin{equation}\label{principal K}
  \underline{K_i(z)}\ :=\sum_{k\geq \max\{1,-b_i\}} K_i^{(k)}z^{-k} \, .
\end{equation}

\medskip

\begin{Rem}\label{rmk comparing to bk}
For $\mu=0$, we obviously get $X_0(\sso_{2r})\simeq X(\sso_{2r})$.
\end{Rem}

\medskip
\noindent
Similar to our proof of Lemma~\ref{lem:C is central}, we note that the coefficients $\{C_r^{(k)}\}_{k\geq d_r+d_{r+1}+1}$ of the series
\begin{equation}\label{eq:central C shifted}
  C_r(z)=z^{-d_r-d_{r+1}}\, + \sum_{k > d_r+d_{r+1}} C_r^{(k)}z^{-k}:=
  \prod_{i=1}^{r-1}\frac{D_i(z+i-r)}{D_i(z+i-r+1)} \cdot D_r(z) D_{r+1}(z)
\end{equation}
are central elements of $X_\mu(\sso_{2r})$, which is an immediate corollary of the relations~(\ref{eY2.1})--(\ref{eY3.2}).

\medskip
\noindent
Let $\bar{\Lambda}=\bigoplus_{i=1}^{r}\BZ\omega_i$ be the coweight lattice of $\sso_{2r}$,
where $\{\omega_i\}_{i=1}^{r}$ are the standard fundamental coweights of $\sso_{2r}$,
i.e.\ $\alphavee_i(\omega_j)=\delta_{i,j}$ for $1\leq i,j\leq r$.
There is a natural $\BZ$-linear projection:
\begin{equation}\label{non-ext coweight from ext}
  \Lambda\longrightarrow \bar{\Lambda}\, ,\quad \mu \, \mapsto \, \bar{\mu}
    \quad \mathrm{defined\ via}\quad
  \alpha^\vee_i(\bar{\mu})=\hat{\alpha}^\vee_i(\mu)\quad \mathrm{for}\ 1\leq i\leq r \, .
\end{equation}
Explicitly, we have:
\begin{equation*}
  \Lambda \ni \mu \ \mapsto \, \bar{\mu}=\sum_{i=1}^r b_i \omega_i \in \bar{\Lambda}\,
\end{equation*}
with $b_i=\hat{\alpha}^\vee_i(\mu)$, cf.~\eqref{extended D shifts 1}, so that:
\begin{equation*}
  \bar{\varpi}_0=0 \, ,\quad \bar{\varpi}_i=\omega_i \quad \mathrm{for}\quad 1\leq i\leq r \, .
\end{equation*}

\medskip
\noindent
The algebra $X_{\mu}(\sso_{2r})$ depends only on the associated $\sso_{2r}$--coweight $\bar{\mu}$,
up to an isomorphism:

\medskip

\begin{Lem}\label{identifying extended Yangians}
If $\mu_1,\mu_2\in \Lambda$ satisfy $\bar{\mu}_1=\bar{\mu}_2\in \bar{\Lambda}$, then the assignment
\begin{equation}\label{isom of shifted gl-Yangians}
  E^{(k)}_i \, \mapsto \, E^{(k)}_i \, ,\quad
  F^{(k)}_i \, \mapsto \, F^{(k)}_i \, ,\quad
  D^{(k_i)}_i \, \mapsto \, D^{(k_i-\epsilon^\vee_i(\mu_1-\mu_2))}_i
\end{equation}
gives rise to a $\BC$-algebra isomorphism
\begin{equation*}
  X_{\mu_1}(\sso_{2r})\iso X_{\mu_2}(\sso_{2r}) \, .
\end{equation*}
\end{Lem}

\medskip

\begin{proof}
The assignment~(\ref{isom of shifted gl-Yangians}) is clearly compatible with
the defining relations~(\ref{eY0},~\ref{eY2.1}--\ref{eY7},~\ref{eY1-shifted}),
giving rise to a $\BC$-algebra homomorphism $X_{\mu_1}(\sso_{2r})\to X_{\mu_2}(\sso_{2r})$.
Switching $\mu_1$ and $\mu_2$, we obtain the inverse homomorphism
$X_{\mu_2}(\sso_{2r})\to X_{\mu_1}(\sso_{2r})$. Hence, the result.
\end{proof}

\medskip
\noindent
Let us also recall the \emph{shifted Drinfeld Yangians of $\sso_{2r}$} introduced in~\cite[Definition~B.2]{bfnb}.
To this end, fix a coweight $\nu\in \bar{\Lambda}$ and set $b_i:=\alphavee_i(\nu)$ for $1\leq i\leq r$.
The \emph{shifted Drinfeld Yangian of $\sso_{2r}$}, denoted by $Y_\nu(\sso_{2r})$,
is the associative $\BC$-algebra generated by
  $\{\sE_i^{(k)},\sF_i^{(k)},\sH_i^{(\ell_i)}\}_{1\leq i\leq r}^{k\geq 1, \ell_i > -b_i}$
with the defining relations~(\ref{gY0},~\ref{gY2}--\ref{gY7}) and the following replacement of~(\ref{gY1}):
\begin{equation}\label{gY1-extended}
  [\sE_i(z), \sF_j(w)]=
  -\delta_{i,j}\, \frac{\underline{\sH_i(z)}-\underline{\sH_i(w)}}{z-w} \, ,
\end{equation}
where $\sE_i(z), \sF_i(z)$ are defined via~\eqref{eq:EFH series},
$\sH_i(z)$ are defined via:
\begin{equation*}
  \sH_i(z):=\sum_{k\geq -b_i} \sH_i^{(k)}z^{-k} = z^{b_i} \, + \sum_{k\geq 1-b_i} \sH_i^{(k)}z^{-k} \, ,
\end{equation*}
with the conventions $\sH_i^{(-b_i)}=1$, and finally $\underline{\sH_i(z)}$
denotes the principal part of $\sH_i(z)$:
\begin{equation}\label{principal H}
  \underline{\sH_i(z)}\ :=\sum_{k\geq \max\{1,-b_i\}} \sH_i^{(k)}z^{-k} \, .
\end{equation}

\medskip
\noindent
The explicit relation between the shifted Yangians $X_\mu(\sso_{2r})$ and $Y_\nu(\sso_{2r})$ is as follows:

\medskip

\begin{Prop}\label{prop:relating shifted yangians}
For any $\mu\in \Lambda$, the assignment~\eqref{eq:iota-null explicitly} gives rise to a $\BC$-algebra embedding
\begin{equation}\label{eq:iota-mu}
  \iota_\mu\colon Y_{\bar{\mu}}(\sso_{2r})\hookrightarrow X_\mu(\sso_{2r}) \, .
\end{equation}
Furthermore, we have a tensor product algebra decomposition:
\begin{equation}\label{eq:shifted decomposition}
  X_\mu(\sso_{2r})\simeq Y_{\bar{\mu}}(\sso_{2r})\otimes_{\BC} \BC[\{C_r^{(k)}\}_{k\geq d_r+d_{r+1}+1}] \, .
\end{equation}
\end{Prop}

\medskip

\begin{Rem}\label{classical embedding Yangians}
For $\mu=0$, this exactly recovers~(\ref{eq:iota-null}) and Lemma~\ref{lem:embedding}.
\end{Rem}

\medskip

\begin{proof}
The proof is completely analogous to that of Lemma~\ref{lem:embedding} treating the special case $\mu=0$
(while Lemma~\ref{lem:embedding} follows from the results of~\cite{jlm1} combined with the
isomorphism~\eqref{eq:unshifted isomorphism}, let us stress right away that our proof was only using
the defining relations~(\ref{eY0})--(\ref{eY7})).

%\medskip
%\noindent
The compatibility of the assignment~\eqref{eq:iota-null explicitly} with the defining relations of $Y_{\bar{\mu}}(\sso_{2r})$
is straightforward, giving rise to a $\BC$-algebra homomorphism $\iota_\mu\colon Y_{\bar{\mu}}(\sso_{2r})\to X_\mu(\sso_{2r})$.
The injectivity of $\iota_\mu$ as well as the tensor product algebra decomposition~(\ref{eq:shifted decomposition})
are immediate after switching from the coefficients of the generating Cartan series $\{D_i(z)\}_{i=1}^{r+1}$ to
the coefficients of the central Cartan series $\{C_r(z)\}$ of~\eqref{eq:central C shifted} and the series
$\{D_i(z)^{-1}D_{i+1}(z)\}_{i=1}^{r-1}\cup\{D_{r-1}(z)^{-1}D_{r+1}(z)\}$, as in our proof of Lemma~\ref{lem:embedding}.
\end{proof}

\medskip

\begin{Cor}\label{cor:sub and quotient}
$Y_{\bar{\mu}}(\sso_{2r})$ may be realized both as a subalgebra of $X_\mu(\sso_{2r})$
via~(\ref{eq:iota-mu}) as well as a quotient of $X_\mu(\sso_{2r})$ by the central ideal
$(C_r^{(k)}-c_k)_{k>d_r+d_{r+1}}$ for any $c_k\in \BC$.
\end{Cor}

\medskip
\noindent
Similar to Remark~\ref{rmk:center comparison} and~\cite[Lemma 2.26]{fpt}, we have:

\medskip

\begin{Lem}\label{center of shifted yangians}
(a) The center of the shifted Yangian $Y_{\nu}(\sso_{2r})$ is trivial for any $\nu\in \bar{\Lambda}$.

\medskip
\noindent
(b) The center of the shifted extended Yangian $X_\mu(\sso_{2r})$ is
$\BC[\{C_r^{(k)}\}_{k>d_r+d_{r+1}}]$ for any $\mu\in \Lambda$.
\end{Lem}

\medskip

\begin{proof}
Part (a) is a general result which follows from~\cite{w} as explained in~\cite[Remark 2.81]{fpt}.
Part (b) follows from (a), the decomposition~(\ref{eq:shifted decomposition}), and the series $C_r(z)$ being central.
\end{proof}

\medskip

    %%%%%%%%%%%%%%%%%%%%%%%%%%%%%%%%%%%%%%%%%%%%%%%%%%%%%%%%%%%%%%%%%%%%%%%%%%%%%%%
    %%%%%%%%%%%%%%%%%%%%%%%% Extended GKLO/BFN Homomorphisms %%%%%%%%%%%%%%%%%%%%%%
    %%%%%%%%%%%%%%%%%%%%%%%%%%%%%%%%%%%%%%%%%%%%%%%%%%%%%%%%%%%%%%%%%%%%%%%%%%%%%%%

\subsubsection{Homomorphisms $\Psi_D$}
\label{sssec extended homoorphisms}
\

In this Subsection, we generalize~\cite[Theorem B.15]{bfnb} for the type $D_r$
Dynkin diagram with arrows pointing $i\to i+1$ for $1\leq i\leq r-2$ and $r\to r-2$
(see~Fig.~1), by replacing $Y_{\bar{\mu}}(\sso_{2r})$ of \emph{loc.cit.}\ with $X_\mu(\sso_{2r})$.
We closely follow the presentation of~\cite[\S2.2]{fpt} for type~$A$.

%%%%%%%%%%%%%%%%%%%%%%%%%%%%%%%%%%%%%%%%%%%%%%%%%%%%%%%%%%%%%%%%%%%%%%%%%%%%%
%%%%%%%%%%%%%%%%%%%%%%%%%%%%% PICTURE %%%%%%%%%%%%%%%%%%%%%%%%%%%%%%%%%%%%%%%
%%%%%%%%%%%%%%%%%%%%%%%%%%%%%%%%%%%%%%%%%%%%%%%%%%%%%%%%%%%%%%%%%%%%%%%%%%%%%
 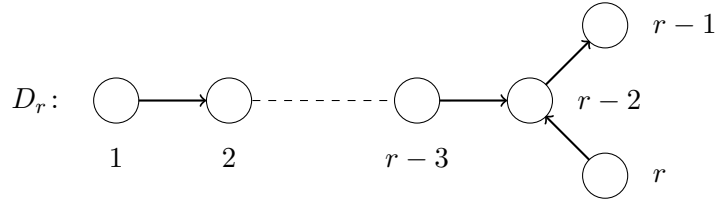
\begin{figure}[h!]
 \begin{center}
 \begin{tikzpicture}
       \draw[black,dashed] (0.1,0)--(3,0);
      \draw[black,thick,->] (-1.5,0)--(-0.3,0);
      \draw[black,thick,->] (2.5,0)--(3.7,0);
      \draw[black,thick,->] (4,0)--(4.8,0.8);
      \draw[black,thick,->] (4.8,-0.8)--(4.2,-0.2);
      \draw[fill=white] (-1.5,0) circle (0.3cm);
      \draw[fill=white] (0,0) circle (0.3cm);
      \draw[fill=white] (2.5,0) circle (0.3cm);
      \draw[fill=white] (4,0) circle (0.3cm);
      \draw[fill=white] (5,1) circle (0.3cm);
      \draw[fill=white] (5,-1) circle (0.3cm);
       \node [left] at (-2,0) {$D_r\colon$};
       \node [below] at (-1.5,-0.5) {$1$};
       \node [below] at (0,-0.5) {$2$};
       \node [below] at (2.5,-0.5) {$r-3$};
       \node [right] at (4.5,0) {$r-2$};
       \node [right] at (5.5,1) {$r-1$};
       \node [right] at (5.5,-1) {$r$};
\end{tikzpicture}
\end{center}
\caption{Oriented Dynkin diagram of type $D_r$}
\label{fig:dr}
\end{figure}
%%%%%%%%%%%%%%%%%%%%%%%%%%%%%%%%%%%%%%%%%%%%%%%%%%%%%%%%%%%%%%%%%%%%%%%%%%%%%

\medskip

\begin{Rem}\label{rmk:about other orientations}
While similar generalizations exist for all orientations of $D_r$ Dynkin diagram,
it suffices to consider only the above one for the purposes of this paper,
see Remark~\ref{rmk:other orientations}.
\end{Rem}

\medskip
\noindent
An element $\lambda\in \Lambda$ will be called \emph{dominant}, denoted by $\lambda\in \Lambda^+$,
if the corresponding $\sso_{2r}$--coweight $\bar{\lambda}$~\eqref{non-ext coweight from ext}
is dominant: $\bar{\lambda}\in \bar{\Lambda}^+$. Thus, $\sum_{i=0}^{r} c_i\varpi_i$ is dominant iff
$c_i\in \BN$ for $1\leq i\leq r$.

\medskip
\noindent
A \emph{$\Lambda$-valued divisor $D$ on $\BP^1$, $\Lambda^+$-valued outside $\{\infty\}\in \BP^1$},
is a formal sum:
\begin{equation}\label{eq:divisor def1}
  D\ =\sum_{1\leq s\leq N} \gamma_s\varpi_{i_s} [x_s] + \mu [\infty]
\end{equation}
with
  $N\in \BN,\, 0\leq i_s\leq r,\, x_s\in \BC,\,
   \gamma_s=\begin{cases}
     1 & \text{if } i_s\ne 0 \\
     \pm 1 & \text{if } i_s=0
   \end{cases}\, ,$
and $\mu\in \Lambda$. We will write
\begin{equation}\label{eq:mu from divisor}
  \mu=D|_\infty\, .
\end{equation}
If $\mu\in \Lambda^+$, we call $D$ a \emph{$\Lambda^+$-valued divisor on $\BP^1$}.
It will be convenient to present $D$ also as:
\begin{equation}\label{eq:divisor def2}
  D\ =\sum_{x\in \BP^1\backslash\{\infty\}} \lambda_x [x] + \mu [\infty]
  \quad \mathrm{with}\ \lambda_x\in \Lambda^+ \,,
\end{equation}
related to~(\ref{eq:divisor def1}) via
  $\lambda_x=D|_x:=\sum_{1\leq s\leq N}^{x_s=x} \gamma_s\varpi_{i_s}$.
Define $\lambda\in \Lambda^+$ via:
\begin{equation}\label{eq:lambda from divisor}
  \lambda\, :=\sum_{1\leq s \leq N} \gamma_s\varpi_{i_s}\ =\sum_{x\in \BP^1\backslash\{\infty\}}D|_x \, .
\end{equation}

\medskip
\noindent
Let  $\{\alpha_i\}_{i=1}^r\subset \bar{\Lambda}$ denote the simple coroots of $\sso_{2r}$,
explicitly given by:
\begin{equation}\label{eq:coroots D}
  \alpha_1=\epsilon_1-\epsilon_2 \, ,\ \ldots \, ,\
  \alpha_{r-2}=\epsilon_{r-2}-\epsilon_{r-1} \, ,\
  \alpha_{r-1}=\epsilon_{r-1}-\epsilon_r \, ,\ \alpha_r=\epsilon_{r-1}+\epsilon_r \, .
\end{equation}
We also consider the following family of elements $\{\hat{\alpha}_i\}_{i=1}^r\subset \Lambda$ given by:
\begin{equation}\label{eq:lifted coroots D}
  \hat{\alpha}_1=\epsilon_1-\epsilon_2 \, ,\ \ldots \, ,\
  \hat{\alpha}_{r-2}=\epsilon_{r-2}-\epsilon_{r-1}\, ,\
  \hat{\alpha}_{r-1}=\epsilon_{r-1}-\epsilon_r+\epsilon_{r+1} \, ,\
  \hat{\alpha}_r=\epsilon_{r-1}+\epsilon_r-\epsilon_{r+1} \, ,
\end{equation}
which are the ``lifts'' of $\alpha_i$ from~(\ref{eq:coroots D}) in the sense
of~\eqref{non-ext coweight from ext}, that is:
\begin{equation}\label{eq:lift}
  \bar{\hat{\alpha}}_i=\alpha_i\qquad \mathrm{for}\ 1\leq i\leq r \, .
\end{equation}

\medskip
\noindent
Following~\cite{bfnb}, we make the following
\begin{equation}\label{eq:assumption}
  \textbf{Assumption}:\qquad
  \lambda+\mu=a_1\hat{\alpha}_1+\ldots+a_{r}\hat{\alpha}_{r}\quad \mathrm{with}\ a_i\in \BN \, .
\end{equation}
Let us record the explicit formulas for the coefficients $a_i$ of~\eqref{eq:assumption}:
\begin{equation}\label{eq:a explicitly}
\begin{split}
   & a_k = (\epsilon^\vee_1+\ldots+\epsilon^\vee_k)(\lambda+\mu) \qquad \mathrm{for}\quad 1\leq k\leq r-2 \, , \\
   & a_{r-1} = \frac{(\epsilon^\vee_1+\ldots+\epsilon^\vee_{r-1} - \epsilon^\vee_r)(\lambda+\mu)}{2} \, , \\
   & a_{r} = \frac{(\epsilon^\vee_1+\ldots+\epsilon^\vee_{r-1} + \epsilon^\vee_r)(\lambda+\mu)}{2} \, .
\end{split}
\end{equation}

\medskip

\begin{Rem}\label{rmk:assumption explicitly}
Note that $D$ of~\eqref{eq:divisor def1} satisfies the assumption~(\ref{eq:assumption})
if and only if all quantities in the right-hand sides of~\eqref{eq:a explicitly} are
non-negative integers and $(\epsilon^\vee_{r}+\epsilon^\vee_{r+1})(\lambda+\mu)=0$.
\end{Rem}

\medskip
\noindent
Consider the associative $\BC$-algebra
\begin{equation}\label{algebra A}
   \CA = \BC\, \Big\langle p_{i,k}, e^{\pm q_{i,k}}, (p_{i,k}-p_{i,\ell}+m)^{-1}\Big\rangle_{1\leq i\leq r, m\in \BZ}^{1\leq k\ne \ell\leq a_i}
\end{equation}
with the defining relations:
\begin{equation*}
  [e^{\pm q_{i,k}},p_{j,\ell}]=\mp \delta_{i,j}\delta_{k,\ell} e^{\pm q_{i,k}} \, ,\quad
  [p_{i,k},p_{j,\ell}]=0=[e^{q_{i,k}},e^{q_{j,\ell}}] \, ,\quad
  e^{\pm q_{i,k}} e^{\mp q_{i,k}}=1 \, .
\end{equation*}

\medskip

\begin{Rem}\label{additive difference operators}
(a) This algebra $\CA$ can be represented in the algebra of difference operators
with rational coefficients on functions of $\{p_{i,k}\}_{1\leq i\leq r}^{1\leq k\leq a_i}$
by taking $e^{\mp q_{i,k}}$ to be a difference operator $\mathsf{D}^{\pm 1}_{i,k}$ that acts as
  $(\mathsf{D}^{\pm 1}_{i,k} \mathsf{\Psi})(p_{1,1},\ldots,p_{i,k},\ldots, p_{r,a_{r}}) =
   \mathsf{\Psi}(p_{1,1},\ldots,p_{i,k}\pm 1, \ldots, p_{r,a_{r}})$.

\medskip
\noindent
(b) The \emph{total number of pairs of $(p,q)$-oscillators in the algebra~$\CA$}
will refer to the sum $\sum_{i=1}^r a_i$.
\end{Rem}

\medskip
\noindent
For $0\leq i\leq r$ and $1\leq j\leq r$, we define:
\begin{equation}\label{ZW-series}
\begin{split}
  & P_j(z):=\prod_{k=1}^{a_j} (z-p_{j,k}) \, ,\quad
    P_{j,\ell}(z)\, :=\prod_{1\leq k\leq a_j}^{k\ne \ell} (z-p_{j,k}) \, ,\\
  & Z_i(z):=\prod_{1\leq s\leq N}^{i_s=i} (z-x_s)^{\gamma_s} \ =
    \prod_{x\in \BP^1\backslash\{\infty\}} (z-x)^{\wt{\alpha}^\vee_i(\lambda_x)} \, ,
\end{split}
\end{equation}
where $\{\wt{\alpha}^\vee_i\}_{i=0}^r$ is a $\BZ$-basis of $\Lambda^\vee$
dual to the $\BZ$-basis $\{\varpi_i\}_{i=0}^r$ of $\Lambda$.
Explicitly, we have:
\begin{equation}\label{eq:tildealphas}
  \wt{\alpha}^\vee_0=-\epsilon^\vee_1 \, ,\
  \wt{\alpha}^\vee_1=\epsilon^\vee_1-\epsilon^\vee_2 \, ,\
  \wt{\alpha}^\vee_2=\epsilon^\vee_2-\epsilon^\vee_3 \, ,\ \ldots \, ,\
  \wt{\alpha}^\vee_{r-1}=\epsilon^\vee_{r-1}-\epsilon^\vee_r \, ,\
  \wt{\alpha}^\vee_{r}=\epsilon^\vee_{r-1}-\epsilon^\vee_{r+1} \, .
\end{equation}
We also set:
\begin{equation}\label{eq:aP conventions}
  a_0:=0 \, ,\quad a_{r+1}:=0 \, ,\quad P_0(z):=1\, ,\quad P_{r+1}(z):=1 \, .
\end{equation}

\medskip
\noindent
The following result generalizes the $D_r$-case of~\cite[Theorem B.15]{bfnb}
stated for semisimple Lie algebras $\fg$ (preceded by~\cite{gklo} for the
trivial shift and by~\cite{kwwy} for dominant shifts):

\medskip

\begin{Thm}\label{thm:extended homom D}
Let $D$ be as in~\eqref{eq:divisor def1}, satisfying the assumption~\eqref{eq:assumption},
and set $\mu=D|_\infty$. There is a unique $\BC$-algebra homomorphism
\begin{equation}\label{eq:homom psi}
  \Psi_D\colon X_{-\mu}(\sso_{2r})\longrightarrow \CA \, ,
\end{equation}
determined by the following assignment:
\begin{equation}\label{eq:homom assignment}
\begin{split}
   & E_i(z)\mapsto
     \begin{cases}
       \sum_{k=1}^{a_i}\frac{P_{i-1}(p_{i,k}-1)}{(z-p_{i,k})P_{i,k}(p_{i,k})} e^{q_{i,k}} & \mbox{if\, } i\leq r-3 \\
       \sum_{k=1}^{a_{r-2}}\frac{P_{r-3}(p_{r-2,k}-1) P_{r}(p_{r-2,k})}{(z-p_{r-2,k})P_{r-2,k}(p_{{r-2},k})} e^{q_{r-2,k}} & \mbox{if\, } i=r-2 \\
       \sum_{k=1}^{a_{r-1}}\frac{P_{{r-2}}(p_{{r-1},k}-1)}{(z-p_{{r-1},k})P_{{r-1},k}(p_{{r-1},k})} e^{q_{r-1,k}} & \mbox{if\, } i=r-1 \\
       \sum_{k=1}^{a_{r}}\frac{1}{(z-p_{{r},k})P_{{r},k}(p_{{r},k})} e^{q_{r,k}} & \mbox{if\, } i=r
     \end{cases} \, , \\
   & F_i(z)\mapsto
     \begin{cases}
       -\sum_{k=1}^{a_i}\frac{Z_i(p_{i,k}+1)P_{i+1}(p_{i,k}+1)}{(z-p_{i,k}-1)P_{i,k}(p_{i,k})} e^{-q_{i,k}} & \mbox{if\, } i\leq r-2 \\
       -\sum_{k=1}^{a_{r-1}}\frac{Z_{r-1}(p_{{r-1},k}+1)}{(z-p_{{r-1},k}-1)P_{{r-1},k}(p_{{r-1},k})} e^{-q_{r-1,k}} & \mbox{if\, } i=r-1 \\
       -\sum_{k=1}^{a_{r}}\frac{Z_{r}(p_{{r},k}+1)P_{r-2}(p_{r,k})}{(z-p_{{r},k}-1)P_{{r},k}(p_{{r},k})} e^{-q_{r,k}} & \mbox{if\, } i=r
     \end{cases} \, , \\
   & D_i(z)\mapsto
     \begin{cases}
       \frac{P_i(z)}{P_{i-1}(z-1)} \cdot \prod_{k=0}^{i-1} Z_k(z) & \mbox{if\, } i\leq r-2 \\
       \frac{P_{r-1}(z)P_r(z)}{P_{r-2}(z-1)} \cdot \prod_{k=0}^{r-2} Z_k(z) & \mbox{if\, } i=r-1 \\
       \frac{P_r(z)}{P_{r-1}(z-1)} \cdot \prod_{k=0}^{r-1} Z_k(z) & \mbox{if\, } i=r \\
       \frac{P_{r-1}(z)}{P_{r}(z-1)} \cdot \prod_{k=0}^{r-2} Z_k(z)\cdot Z_r(z) & \mbox{if\, } i=r+1
     \end{cases} \quad =\\
   & \qquad \quad \prod_{x\in \BP^1\backslash\{\infty\}}(z-x)^{-\epsilon^\vee_i(\lambda_x)} \cdot
     \begin{cases}
       \frac{P_i(z)}{P_{i-1}(z-1)} & \mbox{if\, } i\leq r-2 \\
       \frac{P_{r-1}(z)P_r(z)}{P_{r-2}(z-1)} & \mbox{if\, } i=r-1 \\
       \frac{P_r(z)}{P_{r-1}(z-1)} & \mbox{if\, } i=r \\
       \frac{P_{r-1}(z)}{P_{r}(z-1)} & \mbox{if\, } i=r+1
     \end{cases} \, .
\end{split}
\end{equation}
\end{Thm}

\medskip

\begin{Rem}\label{rmk:relating to BFNb homom}
To compare this with~\cite[\S B(ii)]{bfnb}, let us identify $\CA$ with $\tilde{\CA}$ of~\emph{loc.cit.}
and the points $x_s$ with the parameters $z_s$ of~\emph{loc.cit.}\ (assigned to the summands of
$\bar{\lambda} = \sum_{1\leq s\leq N}^{i_s\ne 0} \omega_{i_s}$)~via:
\begin{equation}\label{eq:pq vs wu}
\begin{split}
  & p_{i,k}\leftrightarrow
    \begin{cases}
      w_{i,k}+\frac{i-1}{2} & \mbox{if } i<r\\
      w_{r,k}+\frac{r-2}{2} & \mbox{if } i=r
    \end{cases} \, , \\
  & e^{\pm q_{i,k}}\leftrightarrow \sfu_{i,k}^{\mp 1} \, , \\
  & x_s \leftrightarrow
    \begin{cases}
      z_s+\frac{i_s}{2} & \mbox{if } 1\leq i_s<r \\
      z_s+\frac{r-1}{2} & \mbox{if } i_s=r
    \end{cases} \, .
\end{split}
\end{equation}
Then, the (restriction) composition
\begin{equation}\label{eq:BFN via ours}
  Y_{-\bar{\mu}}(\sso_{2r})\xrightarrow{\iota_{-\mu}} X_{-\mu}(\sso_{2r}) \xrightarrow{\Psi_D} \CA
\end{equation}
is explicitly given by:
\begin{equation}\label{eq:BFN formulas}
\begin{split}
  & \sE_i(z) \, \mapsto \,
    \sum_{k=1}^{a_i}\frac{\prod_{h\in Q:\mathrm{i}(h)=i}W_{\mathrm{o}(h)}(w_{i,k}-\frac{1}{2})}{(z-w_{i,k})W_{i,k}(w_{i,k})} \sfu^{-1}_{i,k} \, , \\
  & \sF_i(z) \, \mapsto \,
    -\sum_{k=1}^{a_i}\frac{\sZ_i(w_{i,k}+1)\prod_{h\in Q:\mathrm{o}(h)=i}W_{\mathrm{i}(h)}(w_{i,k}+\frac{1}{2})}{(z-w_{i,k}-1)W_{i,k}(w_{i,k})} \sfu_{i,k} \, , \\
  & \sH_i(z) \, \mapsto \,
    \frac{\sZ_i(z)\prod_{h\in Q\cup \bar Q:\mathrm{o}(h)=i} W_{\mathrm{i}(h)}(z-\frac{1}{2})}{W_i(z)W_{i}(z-1)} \, ,
\end{split}
\end{equation}
where $Q$ (resp.\ $\bar{Q}$) denotes the set of oriented (resp.\ oppositely oriented) edges
of the Dynkin diagram from Fig.~1, the notation $\mathrm{i}(h)=i$ (resp.\ $\mathrm{o}(h)=i$) for
an edge $h\in Q$ (or $h\in Q\cup \bar{Q}$) is to indicate that $h$ points towards (resp.\ away from)
the $i$-th node, and the generating series in~\eqref{eq:BFN formulas} are defined via:
\begin{equation*}
  W_{i}(z)=\prod_{k=1}^{a_i}(z-w_{i,k})\, ,\
  W_{i,\ell}(z)\, =\prod_{1\leq k\leq a_i}^{k\ne \ell}(z-w_{i,k}) \, ,\
  \sZ_i(z)\, =\prod_{1\leq s\leq N}^{i_s=i} \left(z-z_s-\tfrac{1}{2}\right) \, .
\end{equation*}
Thus, the composition
\begin{equation*}
  \Psi_D\circ \iota_{-\mu}\colon Y_{-\bar{\mu}}(\sso_{2r})\longrightarrow \CA
\end{equation*}
essentially coincides with the version of the homomorphism $\Phi^{\bar{\lambda}}_{-\bar{\mu}}$
of~\cite[Theorem~B.15]{bfnb}, where the signs of all $E_i(z)$ and $F_i(z)$ are reversed, and the
$\sZ_i(w_{i,k})$-factors in $E_i(z)$-currents are now replaced with the $\sZ_i(w_{i,k}+1)$-factors
in $F_i(z)$-currents, cf.~\cite[Remark C.3]{ft1}.
\end{Rem}

\medskip

\begin{proof}[Proof of Theorem~\ref{thm:extended homom D}]
First, let us verify that under the above assignment~(\ref{eq:homom assignment}),
the image of $D_i(z)$ is of the form $z^{d_i}+(\mathrm{lower\ order\ terms\ in}\ z)$
for all $1\leq i\leq r+1$. Let $\deg_i$ denote the leading power of $z$ in the image
of $D_i(z)$ (clearly the coefficient of $z^{\deg_i}$ equals $1$). Then, we have:
\begin{equation}\label{eq:degrees 1}
  \deg_i \ = \, -\sum_{x\in \BP^1\backslash\{\infty\}} \epsilon^\vee_i(\lambda_x) \, +\,
  \begin{cases}
    a_i-a_{i-1} & \mbox{if } i\ne r\pm 1 \\
    a_{r-1}+a_r-a_{r-2} & \mbox{if } i=r-1 \\
    a_{r-1}-a_r & \mbox{if } i=r+1
  \end{cases} \, .
\end{equation}
Note that
  $\sum_{x\in \BP^1\backslash\{\infty\}} \lambda_x +\mu=\lambda + \mu =
   a_1\hat{\alpha}_1 + \ldots +a_r\hat{\alpha}_r$~\eqref{eq:assumption},
so that:
\begin{equation}\label{eq:degrees 2}
  \sum_{x\in \BP^1\backslash\{\infty\}} \epsilon^\vee_i(\lambda_x) \, +\,  \epsilon^\vee_i(\mu)=
  \epsilon^\vee_i(a_1\hat{\alpha}_1 + \ldots +a_r\hat{\alpha}_r)=
  \begin{cases}
    a_i-a_{i-1} & \mbox{if } i\ne r\pm 1 \\
    a_{r-1}+a_r-a_{r-2} & \mbox{if } i=r-1 \\
    a_{r-1}-a_r & \mbox{if } i=r+1
  \end{cases} \, .
\end{equation}
Combining~(\ref{eq:degrees 1},~\ref{eq:degrees 2}), we thus obtain the desired equality:
\begin{equation}\label{eq:degrees match}
  \deg_i=\epsilon^\vee_i(\mu)=d_i \, .
\end{equation}

\medskip
\noindent
Evoking the algebra decomposition~(\ref{eq:shifted decomposition})
\begin{equation*}
  X_{-\mu}(\sso_{2r})\simeq Y_{-\bar{\mu}}(\sso_{2r})\otimes_{\BC} \BC[\{C_r^{(k)}\}_{k> -d_r-d_{r+1}}] \, ,
\end{equation*}
it suffices to prove that the restrictions of the assignment~(\ref{eq:homom assignment})
to the subalgebras $Y_{-\bar{\mu}}(\sso_{2r})$ and $\BC[\{C^{(k)}_r\}_{k>-d_r-d_{r+1}}]$ determine
algebra homomorphisms, whose images commute. The former is clear for the restriction to
$Y_{-\bar{\mu}}(\sso_{2r})$, due to Theorem B.15 of~\cite{bfnb} combined with
Remark~\ref{rmk:relating to BFNb homom} above. On the other hand, we have:
\begin{equation}\label{image of C-series}
  \Psi_D(C_r(z)) \, =\, \prod_{i=0}^{r-2} \Big(Z_i(z)Z_i(z+i-r+1)\Big)\cdot Z_{r-1}(z)Z_r(z) \, .
\end{equation}
Thus, the restriction of $\Psi_D$ to the polynomial algebra $\BC[\{C^{(k)}_r\}_{k>-d_r-d_{r+1}}]$
defines an algebra homomorphism, whose image is central in $\CA$.
This completes our proof of Theorem~\ref{thm:extended homom D}.
\end{proof}

\medskip

\begin{Rem}
Our choice of $\hat{\alpha}_i\in \Lambda$ in~\eqref{eq:lifted coroots D} ``lifting''
$\alpha_i\in \bar{\Lambda}$ of~\eqref{eq:coroots D} in the sense of~\eqref{eq:lift}
is exactly to guarantee the equality~\eqref{eq:degrees 2}; moreover, the latter
determines $\hat{\alpha}_i$ uniquely.
\end{Rem}

\medskip

    %%%%%%%%%%%%%%%%%%%%%%%%%%%%%%%%%%%%%%%%%%%%%%%%%%%%%%%%%%%%%%%%%%%%%%%%%%%%%%%
    %%%%%%%%%%%%%%%%%%%%%%%% Shifted RTT Extended Yangians %%%%%%%%%%%%%%%%%%%%%%%%
    %%%%%%%%%%%%%%%%%%%%%%%%%%%%%%%%%%%%%%%%%%%%%%%%%%%%%%%%%%%%%%%%%%%%%%%%%%%%%%%

\subsubsection{Antidominantly shifted extended RTT Yangians of $\sso_{2r}$}
\label{sssec RTT extended Yangians D}
\

Fix $\mu\in \Lambda^+$. Define the \emph{antidominantly shifted extended RTT Yangian of $\sso_{2r}$},
denoted by $X^\rtt_{-\mu}(\sso_{2r})$, to be the associative $\BC$-algebra generated by
$\{t^{(k)}_{ij}\}_{1\leq i,j\leq 2r}^{k\in \BZ}$ subject to the following two families of relations:

\medskip
\noindent
$\bullet$
The RTT relation~\eqref{eq:rtt} with
  $T(z)\in X^\rtt_{-\mu}(\sso_{2r})[[z,z^{-1}]]\otimes_\BC \End\ \BC^{2r}$
defined via:
\begin{equation}\label{eq:shifted T-matrix}
  T(z)=\sum_{i,j} t_{ij}(z)\otimes E_{ij}\qquad \mathrm{with}\qquad
  t_{ij}(z):=\sum_{k\in \BZ} t^{(k)}_{ij}z^{-k} \, .
\end{equation}

\medskip
\noindent
$\bullet$
The second family of relations encodes the fact that $T(z)$ admits the Gauss decomposition:
\begin{equation}\label{eq:shifted Gauss product}
  T(z)=F(z)\cdot H(z)\cdot E(z) \, ,
\end{equation}
where
  $F(z),H(z),E(z)\in X^\rtt_{-\mu}(\sso_{2r})((z^{-1}))\otimes_\BC \End\ \BC^{2r}$
are of the form
\begin{equation*}
  F(z)=\sum_{i} E_{ii}+\sum_{i<j} f_{j,i}(z)\otimes E_{ji} \, ,\
  H(z)=\sum_{i} h_i(z)\otimes E_{ii} \, ,\
  E(z)=\sum_{i} E_{ii}+\sum_{i<j} e_{i,j}(z)\otimes E_{ij}
\end{equation*}
with the matrix coefficients having the following expansions in $z$:
\begin{equation}\label{eq:t-modes shifted}
\begin{split}
  & e_{i,j}(z)=\sum_{k\geq 1} e^{(k)}_{i,j}z^{-k} \, , \quad
    f_{j,i}(z)=\sum_{k\geq 1} f^{(k)}_{j,i}z^{-k} \, \quad \mathrm{for}\ 1\leq i<j\leq 2r \, ,\\
  & h_i(z)=z^{d_i}\, + \sum_{k\geq 1-d_i} h^{(k)}_i z^{-k} \, ,\quad
    h_{i'}(z)=z^{d'_i}\, + \sum_{k\geq 1-d'_i} h^{(k)}_{i'} z^{-k} \, \quad \mathrm{for}\ 1\leq i\leq r \, ,
\end{split}
\end{equation}
with $i'=2r+1-i$ as in~\eqref{eq:N,kappa,prime} and $d'_i\in \BZ$ defined via:
\begin{equation}\label{eq:d-prime}
  d'_i:=d_r+d_{r+1}-d_i \qquad \mathrm{for}\ 1\leq i\leq r \, .
\end{equation}
Note that $d'_r=d_{r+1}$. We also note that $\mu\in \Lambda^+$ implies the following inequalities:
\begin{equation}\label{eq:d-inequalities}
  d_1\geq d_2\geq \dots\geq d_{r-1}\geq \max\{d_{r},d'_{r}\}\geq
  \min\{d_{r},d'_{r}\} \geq d'_{r-1}\geq \dots \geq d'_1 \, .
\end{equation}

\medskip

\begin{Rem}\label{rmk:zero shift}
(a) For $\mu=0$, the second family of relations~(\ref{eq:shifted Gauss product},~\ref{eq:t-modes shifted})
is equivalent to the relations $t_{ij}^{(k)}=0$ for $k<0$ and $t_{ij}^{(0)}=\delta_{i,j}$, so that
$X^\rtt_0(\sso_{2r})\simeq X^\rtt(\sso_{2r})$.

\medskip
\noindent
(b) If $\mu_1,\mu_2\in \Lambda^+$ satisfy $\bar{\mu}_1=\bar{\mu}_2\in \bar{\Lambda}$, that is,
$\mu_2=\mu_1+c\varpi_0$ with $c\in \BZ$, then the assignment
$$T(z)\mapsto z^{c}T(z)$$
gives rise to a $\BC$-algebra isomorphism $X^\rtt_{-\mu_1}(\sso_{2r})\iso X^\rtt_{-\mu_2}(\sso_{2r})$,
cf.\ Lemma~\ref{identifying extended Yangians}.
\end{Rem}

\medskip
\noindent
Similar to the $\mu=0$ case, $X^\rtt_{-\mu}(\sso_{2r})$ is generated by
\begin{equation}\label{eq:rtt generators}
  e_{i,i+1}^{(k)}\, ,\ e_{r-1,r+1}^{(k)}\, ,\
  f_{i+1,i}^{(k)} \, ,\ f_{r+1,r-1}^{(k)} \, ,\
  h_j^{(s_j)}
\end{equation}
for all $1\leq i\leq r-1,\, 1\leq j\leq r+1,\, k\geq 1,\, s_j\geq 1-d_j$. Furthermore,
all the other generators $e_{i,j}^{(k)}, f_{j,i}^{(k)}, h_i^{(k)}$ of~\eqref{eq:t-modes shifted}
are expressed via~\eqref{eq:rtt generators} by exactly the same formulas as in the $\mu=0$ case,
treated in details in Subsection~\ref{sssec: Drinfeld-to-RTT-D}.
This immediately implies the following result:

\medskip

\begin{Prop}\label{prop:epimorphism of shifted Yangians}
For any $\mu\in \Lambda^+$, there is a unique $\BC$-algebra epimorphism
\begin{equation*}
  \Upsilon_{-\mu}\colon X_{-\mu}(\sso_{2r})\twoheadrightarrow X^\rtt_{-\mu}(\sso_{2r})
\end{equation*}
defined by the formulas~(\ref{eq:explicit identification 1},~\ref{eq:explicit identification 2}).
\end{Prop}

\medskip
\noindent
One of our key results (the proof of which is deferred to Subsection~\ref{sssec Proof via Lax matrices D})
is the following generalization of Theorem~\ref{thm:Dr=RTT-unshifted-D} (corresponding to the case $\mu=0$):

\medskip

\begin{Thm}\label{thm:Dr=RTT-shifted-D}
$\Upsilon_{-\mu}\colon X_{-\mu}(\sso_{2r})\iso X^\rtt_{-\mu}(\sso_{2r})$ is
a $\BC$-algebra isomorphism for any $\mu\in \Lambda^+$.
\end{Thm}

\medskip

    %%%%%%%%%%%%%%%%%%%%%%%%%%%%%%%%%%%%%%%%%%%%%%%%%%%%%%%%%%%%%%%%%%%%%%%%%%%%%%%
    %%%%%%%%%%%%%%%%%%%%%%%%%%% Coproduct Homomorphisms %%%%%%%%%%%%%%%%%%%%%%%%%%%
    %%%%%%%%%%%%%%%%%%%%%%%%%%%%%%%%%%%%%%%%%%%%%%%%%%%%%%%%%%%%%%%%%%%%%%%%%%%%%%%

\subsubsection{Coproduct homomorphisms}
\label{sssec coproduct D}
\

One of the key benefits of the RTT realization is that it immediately endows the (extended)
Yangian of $\sso_{2r}$ with the Hopf algebra structure, in particular, the coproduct homomorphism:
\begin{equation}\label{eq:unshifted rtt-coproduct}
  \Delta^\rtt\colon X^\rtt(\sso_{2r})\longrightarrow X^\rtt(\sso_{2r})\otimes X^\rtt(\sso_{2r})
  \, , \qquad T(z)\mapsto T(z)\otimes T(z) \, .
\end{equation}

\medskip
\noindent
The main observation of this Subsection is that~\eqref{eq:unshifted rtt-coproduct}
naturally admits a shifted version:

\medskip

\begin{Prop}\label{prop:shifted rtt coproduct}
For any $\mu_1,\mu_2\in \Lambda^+$, there is a unique $\BC$-algebra homomorphism
\begin{equation}\label{eq:shifted rtt-coproduct}
  \Delta^\rtt_{-\mu_1,-\mu_2}\colon X^\rtt_{-\mu_1-\mu_2}(\sso_{2r})\longrightarrow
  X^\rtt_{-\mu_1}(\sso_{2r})\otimes X^\rtt_{-\mu_2}(\sso_{2r})
\end{equation}
defined by:
\begin{equation}\label{eq:RTT coproduct}
  \Delta^\rtt_{-\mu_1,-\mu_2}(T(z))=T(z)\otimes T(z) \, .
\end{equation}
\end{Prop}

\medskip

\begin{proof}
The proof is completely analogous to its type $A$ counterpart established
in~\cite[Proposition 2.136]{fpt}: the arguments of~\emph{loc.cit.}\ apply on the nose,
due to~\eqref{eq:d-inequalities} as well as $e_{r,r+1}(z)=0=f_{r+1,r}(z)$
(to treat the possible case $d_r<d'_r$), cf.~Lemmas~\ref{lem:all-E-known}(a),~\ref{lem:all-F-known}(a).
\end{proof}

\medskip
\noindent
Similar to~\cite[Corollary 2.141]{fpt}, we note that $\Delta^\rtt_{\ast,\ast}$~\eqref{eq:shifted rtt-coproduct}
satisfy the natural coassociativity:

\medskip

\begin{Cor}\label{cor:coassociativity}
For any $\mu_1,\mu_2,\mu_3\in \Lambda^+$, the following diagram is commutative:
\begin{equation*}
 \begin{CD}
 X^\rtt_{-\mu_1-\mu_2-\mu_3}(\sso_{2r})
    @>{\Delta^\rtt_{-\mu_1,-\mu_2-\mu_3}}>>
 X^\rtt_{-\mu_1}(\sso_{2r})\otimes X^\rtt_{-\mu_2-\mu_3}(\sso_{2r})\\
 @V{\Delta^\rtt_{-\mu_1-\mu_2,-\mu_3}}VV   @VV{\on{Id}\otimes\, \Delta^\rtt_{-\mu_2,-\mu_3}}V\\
 X^\rtt_{-\mu_1-\mu_2}(\sso_{2r})\otimes X^\rtt_{-\mu_3}(\sso_{2r})
    @>>{\Delta^\rtt_{-\mu_1,-\mu_2}\otimes\, \on{Id}}>
 X^\rtt_{-\mu_1}(\sso_{2r})\otimes X^\rtt_{-\mu_2}(\sso_{2r})\otimes X^\rtt_{-\mu_3}(\sso_{2r})
 \end{CD}
\end{equation*}
\end{Cor}

\medskip
\noindent
Evoking the key isomorphism
  $\Upsilon_{-\mu}\colon X_{-\mu}(\sso_{2r})\iso X^\rtt_{-\mu}(\sso_{2r})$
of Theorem~\ref{thm:Dr=RTT-shifted-D} for $\mu=\mu_1$, $\mu_2$, $\mu_1+\mu_2$,
we conclude that $\Delta^\rtt_{-\mu_1,-\mu_2}$ of~\eqref{eq:shifted rtt-coproduct}
gives rise to the $\BC$-algebra homomorphism
\begin{equation}\label{eq:Dr-coproduct-D}
  \Delta_{-\mu_1,-\mu_2}\colon
  X_{-\mu_1-\mu_2}(\sso_{2r})\longrightarrow X_{-\mu_1}(\sso_{2r})\otimes X_{-\mu_2}(\sso_{2r}) \, .
\end{equation}

\medskip

\begin{Prop}\label{prop:shifted Drinfeld coproduct D-ext}
For any $\mu_1,\mu_2\in \Lambda^+$, the above $\BC$-algebra homomorphism~\eqref{eq:Dr-coproduct-D}
\begin{equation*}
  \Delta_{-\mu_1,-\mu_2}\colon
  X_{-\mu_1-\mu_2}(\sso_{2r})\longrightarrow X_{-\mu_1}(\sso_{2r})\otimes X_{-\mu_2}(\sso_{2r})
\end{equation*}
is uniquely determined by specifying the image of the central series $C_r(z)$ of~\eqref{eq:central C shifted} via:
\begin{equation}\label{coproduct Yangian center}
  C_r(z)\mapsto C_r(z)\otimes C_r(z) \,,
\end{equation}
and the following formulas (for any $1\leq i\leq r$ and $1\leq j\leq r+1$):
\begin{equation}\label{coproduct Yangian generators}
\begin{split}
  & F^{(k)}_i \, \mapsto \, F^{(k)}_i\otimes 1
    \qquad \mathrm{for}\quad 1\leq k\leq \hat{\alpha}^\vee_i(\mu_1) \, ,\\
  & F^{(\hat{\alpha}^\vee_i(\mu_1)+1)}_i \, \mapsto \,
    F^{(\hat{\alpha}^\vee_i(\mu_1)+1)}_i\otimes 1 \, +\, 1\otimes F^{(1)}_i \, ,\\
  & E^{(k)}_i \, \mapsto \, 1\otimes E^{(k)}_i
    \qquad \mathrm{for}\quad 1\leq k\leq \hat{\alpha}^\vee_i(\mu_2) \, ,\\
  & E^{(\hat{\alpha}^\vee_i(\mu_2)+1)}_i \, \mapsto \,
    1\otimes E^{(\hat{\alpha}^\vee_i(\mu_2)+1)}_i \, +\, E^{(1)}_i\otimes 1 \, ,\\
  & D^{(-\epsilon^\vee_j(\mu_1+\mu_2)+1)}_j \, \mapsto \,
    D^{(-\epsilon^\vee_j(\mu_1)+1)}_j \otimes 1 \, +\, 1\otimes D^{(-\epsilon^\vee_j(\mu_2)+1)}_j \, ,\\
  & D^{(-\epsilon^\vee_j(\mu_1+\mu_2)+2)}_j \, \mapsto \,
    D^{(-\epsilon^\vee_j(\mu_1)+2)}_j \otimes 1 \, +\, 1\otimes D^{(-\epsilon^\vee_j(\mu_2)+2)}_j \, + \\
  & \ \ \ \ \ \ \ \ \ \ \ \ \ \ \ \ \ \  \ \ \ \ \ \
    D^{(-\epsilon^\vee_j(\mu_1)+1)}_j\otimes D^{(-\epsilon^\vee_j(\mu_2)+1)}_j \, +\,
    \sum_{\gamma^\vee\in \Delta^+} (\tilde{\epsilon}^\vee_j,\gamma^\vee) E^{(1)}_{\gamma^\vee}\otimes F^{(1)}_{\gamma^\vee} \, ,
\end{split}
\end{equation}
with
\begin{equation}\label{eq:tilde-epsilon type D}
  \tilde{\epsilon}^\vee_j=\epsilon^\vee_j \quad \mathrm{for}\ j\leq r\, ,
  \qquad \tilde{\epsilon}^\vee_{r+1}=-\epsilon^\vee_r \, ,
\end{equation}
and the root generators $\{E^{(1)}_{\gamma^\vee}, F^{(1)}_{\gamma^\vee}\}_{\gamma\in \Delta^+}$
defined via (cf.~(\ref{eq:e-recursive},~\ref{eq:f-recursive})):
\begin{equation}\label{eq:higher root generators D}
\begin{split}
  & E^{(1)}_{\epsilon^\vee_i - \epsilon^\vee_j} = [E^{(1)}_{j-1},[E^{(1)}_{j-2},[E^{(1)}_{j-3}, \,\cdots, [E^{(1)}_{i+1},E^{(1)}_i]\cdots]]] \, , \\
  & F^{(1)}_{\epsilon^\vee_i - \epsilon^\vee_j} = [[[\cdots[F_i^{(1)},F_{i+1}^{(1)}], \,\cdots, F^{(1)}_{j-3}],F^{(1)}_{j-2}],F_{j-1}^{(1)}] \, , \\
  & E^{(1)}_{\epsilon^\vee_i + \epsilon^\vee_j} =
    [\cdots[[E^{(1)}_{r},[E^{(1)}_{r-2},[E^{(1)}_{r-3}, \,\cdots, [E^{(1)}_{i+1},E^{(1)}_i]\cdots]]], E^{(1)}_{r-1}], \,\cdots,E^{(1)}_{j}] \, , \\
  & F^{(1)}_{\epsilon^\vee_i + \epsilon^\vee_j} =
    [F^{(1)}_j, \,\cdots, [F^{(1)}_{r-1},[[[\cdots[F_i^{(1)},F_{i+1}^{(1)}], \,\cdots, F^{(1)}_{r-3}],F^{(1)}_{r-2}],F_{r-1}^{(1)}]]\cdots]\,
\end{split}
\end{equation}
for $1\leq i<j\leq r$, where
  $\Delta^+=\Big\{\epsilon^\vee_i \pm \epsilon^\vee_j\Big\}_{1\leq i<j\leq r}$
is the set of positive roots of $\sso_{2r}$.
\end{Prop}

\medskip
\noindent
The proof of this result is completely analogous to that of~\cite[Proposition 2.143]{fpt} with
the only non-trivial computation of $\Delta_{-\mu_1,-\mu_2}(D^{(-\epsilon^\vee_j(\mu_1+\mu_2)+2)}_j)$
based on the identifications:
\begin{equation*}
  \Upsilon_{-\mu}\colon\
  E^{(1)}_{\epsilon^\vee_i - \epsilon^\vee_j} \mapsto e^{(1)}_{i,j} \, ,\
  E^{(1)}_{\epsilon^\vee_i + \epsilon^\vee_j} \mapsto e^{(1)}_{i,j'} \, ,\
  F^{(1)}_{\epsilon^\vee_i - \epsilon^\vee_j} \mapsto f^{(1)}_{j,i} \, ,\
  F^{(1)}_{\epsilon^\vee_i + \epsilon^\vee_j} \mapsto f^{(1)}_{j',i} \, \quad \forall\, 1\leq i<j\leq r
\end{equation*}
and the equalities $e^{(1)}_{i,j}=-e^{(1)}_{j',i'},\, f^{(1)}_{j,i}=-f^{(1)}_{i',j'}$ for $1\leq i<j\leq 2r$,
cf.~(\ref{eq:E-linear-skew},~\ref{eq:F-linear-skew},~\ref{eq:C-linear-skew},~\ref{eq:B-linear-skew}).
Let us note that the formula~\eqref{coproduct Yangian center} is a direct corollary of
the formulas $\sz_N(z)=\Upsilon_{-\mu}(C_r(z))$ and $T(z)T'(z-\kappa)=\sz_N(z)\ID_N$
established in Lemma~\ref{lem:shifted z-vs-C} and Proposition~\ref{prop:crossymmetry D} below.

\medskip
\noindent
The above result provides a conceptual and elementary proof of~\cite[Theorem~4.8]{fkp}:

\medskip

\begin{Prop}\label{prop:shifted Drinfeld coproduct D}
(a) For any $\nu_1,\nu_2\in \bar{\Lambda}^+$, there is a unique $\BC$-algebra homomorphism
\begin{equation}\label{fkprw homom}
  \Delta_{-\nu_1,-\nu_2}\colon
  Y_{-\nu_1-\nu_2}(\sso_{2r})\longrightarrow Y_{-\nu_1}(\sso_{2r})\otimes Y_{-\nu_2}(\sso_{2r})
\end{equation}
such that the following diagram is commutative for any $\mu_1,\mu_2\in \Lambda^+$:
\begin{equation}\label{compatibility with FKPRW}
  \begin{CD}
  Y_{-\bar{\mu}_1-\bar{\mu}_2}(\sso_{2r}) @>{\Delta_{-\bar{\mu}_1,-\bar{\mu}_2}}>> Y_{-\bar{\mu}_1}(\sso_{2r})\otimes Y_{-\bar{\mu}_2}(\sso_{2r})\\
  @V{\iota_{-\mu_1-\mu_2}}VV   @VV{\iota_{-\mu_1}\otimes\, \iota_{-\mu_2}}V\\
  X_{-\mu_1-\mu_2}(\sso_{2r}) @>{\Delta_{-\mu_1,-\mu_2}}>> X_{-\mu_1}(\sso_{2r})\otimes X_{-\mu_2}(\sso_{2r})
  \end{CD}
\end{equation}

\medskip
\noindent
(b) The homomorphism $\Delta_{-\nu_1,-\nu_2}$
is uniquely determined by the following formulas:
\begin{equation}\label{coproduct Yangian generators sl}
\begin{split}
  & \sF^{(k)}_i \, \mapsto \, \sF^{(k)}_i\otimes 1
    \qquad \mathrm{for}\quad 1\leq k\leq \alphavee_i(\nu_1) \, ,\\
  & \sF^{(\alphavee_i(\nu_1)+1)}_i \, \mapsto \,
    \sF^{(\alphavee_i(\nu_1)+1)}_i\otimes 1 \, +\, 1\otimes \sF^{(1)}_i \, ,\\
  & \sE^{(k)}_i \, \mapsto \, 1\otimes \sE^{(k)}_i
    \qquad \mathrm{for}\quad 1\leq k\leq \alphavee_i(\nu_2) \, ,\\
  & \sE^{(\alphavee_i(\nu_2)+1)}_i \, \mapsto \,
    1\otimes \sE^{(\alphavee_i(\nu_2)+1)}_i \, +\, \sE^{(1)}_i\otimes 1 \, ,\\
  & \sH^{(\alphavee_i(\nu_1+\nu_2)+1)}_i \, \mapsto \,
    \sH^{(\alphavee_i(\nu_1)+1)}_i \otimes 1 \, +\, 1\otimes \sH^{(\alphavee_i(\nu_2)+1)}_i \, ,\\
  & \sH^{(\alphavee_i(\nu_1+\nu_2)+2)}_i \, \mapsto \,
    \sH^{(\alphavee_i(\nu_1)+2)}_i \otimes 1 \, +\, 1\otimes \sH^{(\alphavee_i(\nu_2)+2)}_i \, +\\
  & \ \ \ \ \ \ \ \ \ \ \ \ \ \ \ \ \ \  \ \ \ \
    \sH^{(\alphavee_i(\nu_1)+1)}_i\otimes \sH^{(\alphavee_i(\nu_2)+1)}_i \, -\,
    \sum_{\gamma^\vee\in \Delta^+} (\alphavee_i,\gamma^\vee) \sE^{(1)}_{\gamma^\vee}\otimes \sF^{(1)}_{\gamma^\vee} \, ,
\end{split}
\end{equation}
with the root generators $\{\sE^{(1)}_{\gamma^\vee}, \sF^{(1)}_{\gamma^\vee}\}_{\gamma\in \Delta^+}$
defined exactly as in~\eqref{eq:higher root generators D}, but using $\sE^{(1)}_{i}$ and $\sF^{(1)}_i$
instead of $E^{(1)}_i$ and $F^{(1)}_i$, respectively.
\end{Prop}

\medskip

\begin{proof}
This follows immediately from the formulas~(\ref{coproduct Yangian generators})
of Proposition~\ref{prop:shifted Drinfeld coproduct D-ext} combined with the
formulas~(\ref{eq:iota-null explicitly}) for the embedding
$\iota_{-\mu}\colon Y_{-\bar{\mu}}(\sso_{2r})\hookrightarrow X_{-\mu}(\sso_{2r})$
of Proposition~\ref{prop:relating shifted yangians}.
In particular, the proof of the last formula in~\eqref{coproduct Yangian generators sl}
uses the equality
  $\alpha^\vee_i=
   \begin{cases}
     \tilde{\epsilon}^\vee_i-\tilde{\epsilon}^\vee_{i+1} & \mbox{if } i<r \\
     \tilde{\epsilon}^\vee_{r-1}-\tilde{\epsilon}^\vee_{r+1} & \mbox{if } i=r
   \end{cases}$
with $\tilde{\epsilon}^\vee_j$ defined in~\eqref{eq:tilde-epsilon type D}.
\end{proof}

\medskip

\begin{Rem}\label{rmk:coproduct coincidence D}
(a) As our formulas~\eqref{coproduct Yangian generators sl} coincide with those
of~\cite[Theorem 4.8]{fkp}, this provides a confirmative answer to the question raised
in the end of~\cite[\S8]{cgy}, in type $D$.

\medskip
\noindent
(b) A simple argument (see~\cite[Theorem 4.12]{fkp}) shows that the coproduct homomorphisms
$\Delta_{-\nu_1,-\nu_2}$ of~\eqref{fkprw homom} with $\nu_1,\nu_2\in \bar{\Lambda}^+$
give rise to a family of coproduct homomorphisms
  $\Delta_{\nu_1,\nu_2}\colon Y_{\nu_1+\nu_2}(\sso_{2r})\to Y_{\nu_1}(\sso_{2r})\otimes Y_{\nu_2}(\sso_{2r})$
for any pair of $\sso_{2r}$--coweights $\nu_1,\nu_2\in \bar{\Lambda}$.
However, let us note that $\Delta_{\nu_1,\nu_2}\ (\nu_1,\nu_2\in \bar{\Lambda})$ are
not coassociative, in contrast to Corollary~\ref{cor:coassociativity}.
\end{Rem}

\medskip

    %%%%%%%%%%%%%%%%%%%%%%%%%%%%%%%%%%%%%%%%%%%%%%%%%%%%%%%%%%%%%%%%%%%%%%%%%%%%%%%
    %%%%%%%%%%%%%%%%%%%%%%%%%%%%%%%% Lax Matrices %%%%%%%%%%%%%%%%%%%%%%%%%%%%%%%%%
    %%%%%%%%%%%%%%%%%%%%%%%%%%%%%%%%%%%%%%%%%%%%%%%%%%%%%%%%%%%%%%%%%%%%%%%%%%%%%%%

\subsection{Lax matrices}
\

    %%%%%%%%%%%%%%%%%%%%%%%%%%%%%%%%%%%%%%%%%%%%%%%%%%%%%%%%%%%%%%%%%%%%%%%%%%%%%%%
    %%%%%%%%%%%%%%%%%%%%%%%%% Construction and key property %%%%%%%%%%%%%%%%%%%%%%%
    %%%%%%%%%%%%%%%%%%%%%%%%%%%%%%%%%%%%%%%%%%%%%%%%%%%%%%%%%%%%%%%%%%%%%%%%%%%%%%%

\subsubsection{Motivation, explicit construction, and the normalized limit description}
\label{sssec Lax}
\

Consider a $\Lambda^+$-valued divisor $D$ on $\BP^1$, see~\eqref{eq:divisor def1}, satisfying the assumption~(\ref{eq:assumption}).
Note that $\mu:=D|_\infty\in \Lambda^+$. Assuming the validity of Theorem~\ref{thm:Dr=RTT-shifted-D},
let us compose $\Psi_D\colon X_{-\mu}(\sso_{2r})\to \CA$ of~(\ref{eq:homom psi}) with
$\Upsilon_{-\mu}^{-1}\colon X^\rtt_{-\mu}(\sso_{2r})\iso X_{-\mu}(\sso_{2r})$
to get a homomorphism:
\begin{equation}\label{eq:Theta-homom}
  \Theta_D=\Psi_D\circ \Upsilon_{-\mu}^{-1}\colon X^\rtt_{-\mu}(\sso_{2r})\longrightarrow \CA \, .
\end{equation}
Such a homomorphism is uniquely determined by
  $T_D(z)\in \CA((z^{-1}))\otimes_\BC \End\ \BC^{2r}$
defined via:
\begin{equation}\label{eq:Lax motivation}
  T_D(z):=\Theta_D(T(z))=\Theta_D(F(z))\cdot \Theta_D(H(z))\cdot \Theta_D(E(z)) \, .
\end{equation}

\medskip
\noindent
While the above definition~\eqref{eq:Lax motivation} of $T_D(z)$ is based on yet unproved
Theorem~\ref{thm:Dr=RTT-shifted-D}, we can combine the formulas~\eqref{eq:homom assignment} for $\Psi_D$
with those of Subsection~\ref{sssec: Drinfeld-to-RTT-D} to recover the explicit sought-after images
$f^D_{j,i}(z),e^D_{i,j}(z),h^D_i(z)\in \CA((z^{-1}))$ of the generating series $f_{j,i}(z),e_{i,j}(z),h_i(z)$,
the matrix coefficients of $F(z),H(z),E(z)$ in~\eqref{eq:shifted Gauss product}.
Thus, we amend~\eqref{eq:Lax motivation} and define:
\begin{equation}\label{eq:Lax definition}
  T_D(z):=F^D(z)\cdot H^D(z)\cdot E^D(z) \,
\end{equation}
with $F^D(z), E^D(z), H^D(z)$ being the lower-triangular, upper-triangular, and diagonal
matrices with matrix coefficients $f^D_{j,i}(z),e^D_{i,j}(z),h^D_i(z)$ obtained from
the explicit formulas~\eqref{eq:homom assignment} for the images of
$\{e_i(z),f_i(z)\}_{i=1}^r\cup \{h_j(z)\}_{j=1}^{r+1}$ combined with
Lemmas~\ref{lem:all-H},~\ref{lem:all-E-known},~\ref{lem:all-E-new},~\ref{lem:all-F-known},~\ref{lem:all-F-new}.
The explicit formulas for $f^D_{j,i}(z),e^D_{i,j}(z),h^D_i(z)$ are presented in
Appendix~\ref{Appendix A: Lax explicitly}, cf.~\cite[\S2.4.1]{fpt}.
Therefore, the matrix coefficients of $T_D(z)$ are given by:
\begin{equation}\label{eq:Lax explicit}
  T_D(z)_{\alpha,\beta}\ =
  \sum_{i=1}^{\min\{\alpha,\beta\}}
  f^D_{\alpha,i}(z)\cdot h^D_i(z)\cdot e^D_{i,\beta}(z)
\end{equation}
for any $1\leq \alpha,\beta\leq 2r$, with the conventions
$f^D_{\alpha,\alpha}(z)=1=e^D_{\beta,\beta}(z)$.

\medskip

\begin{Def}
For an associative algebra $\CB$, a $\CB((z^{-1}))$-valued $2r\times 2r$ matrix $\mathsf{T}(z)$
is called \textbf{Lax} (of type $D_r$) if it satisfies the RTT relation~\eqref{eq:rtt} with
the $R$-matrix $R(z)$ of~\eqref{eq:R-matrix}.
\end{Def}

\medskip
\noindent
Following the arguments of~\cite[\S2.4.2]{fpt}, let us show that $T_D(z)\in \CA((z^{-1}))$
\eqref{eq:Lax explicit} are Lax. To this end, consider a $\Lambda^+$-valued divisor
  $D=\sum_{s=1}^{N} \gamma_s\varpi_{i_s} [x_s] + \mu [\infty]$.
As the point $x_N$ tends to $\infty$ ($x_N\to \infty$), we obtain another $\Lambda^+$-valued divisor
  $D'=\sum_{s=1}^{N-1} \gamma_s\varpi_{i_s} [x_s] + (\mu+\gamma_N\varpi_{i_N}) [\infty]$.
Similar to~\cite[Proposition 2.75]{fpt}, the matrix $T_{D'}(z)$ of~\eqref{eq:Lax explicit}
is related to $T_D(z)$ via:
\begin{equation}\label{eq:limit relation}
  T_{D'}(z)\ =\underset{x_N\to \infty}{\lim}\, \Big\{(-x_N)^{\gamma_N\varpi_{i_N}}\cdot\, T_D(z)\Big\} \, ,
\end{equation}
where $x^\nu$ (with $x\in \BC^\times, \nu\in \Lambda$) denotes the following
$2r\times 2r$ diagonal $z$-independent matrix:
\begin{equation}\label{eq:normalization factor}
  x^{\nu}=
  \mathrm{diag}\left(x^{\epsilon^\vee_1(\nu)},\ldots, x^{\epsilon^\vee_{r-1}(\nu)}, x^{\epsilon^\vee_r(\nu)},
  x^{{\epsilon'}^{\vee}_{r}(\nu)}, x^{{\epsilon'}^{\vee}_{r-1}(\nu)},\ldots, x^{{\epsilon'}^{\vee}_{1}(\nu)} \right)
\end{equation}
with
\begin{equation*}
  {\epsilon'}^{\vee}_i:=\epsilon^\vee_r+\epsilon^\vee_{r+1}-\epsilon^\vee_i
  \qquad \mathrm{for}\ 1\leq i\leq r \, .
\end{equation*}

\medskip

\begin{Rem}
In contrast to~\cite{fpt}, we note that the normalization factor $(-x_N)^{\gamma_N\varpi_{i_N}}$ appears
on the left of $T_D(z)$ in~\eqref{eq:limit relation}, due to our present choice~\eqref{eq:homom assignment}
of using $Z_k(z)$-factors in the $\Psi_D$-images of $F_k(z)$-currents rather than $E_k(z)$-currents,
cf.~Remark~\ref{rmk:relating to BFNb homom}.
\end{Rem}

\medskip
\noindent
In view of~\eqref{eq:limit relation}, $T_{D'}(z)$ can be constructed as a
\emph{normalized limit} of $T_D(z)$, hence we get:

\medskip

\begin{Cor}\label{cor:rational Lax via unshifted}
For any $\Lambda^+$-valued divisor $D$ on $\BP^1$ satisfying~(\ref{eq:assumption}),
the matrix $T_{D}(z)$ of~\eqref{eq:Lax explicit} is a normalized limit of $T_{\bar{D}}(z)$
with a $\Lambda^+$-valued divisor $\bar{D}$ satisfying $\bar{D}|_\infty=0$.
\end{Cor}

\medskip
\noindent
Note that the condition $\bar{D}|_\infty=0$ corresponds to the unshifted case ($\mu=0$),
in which case Theorem~\ref{thm:Dr=RTT-shifted-D} holds due to Theorem~\ref{thm:Dr=RTT-unshifted-D}.
Therefore, $T_{\bar{D}}(z)$ defined via~\eqref{eq:Lax explicit} can also be recovered
via~(\ref{eq:Theta-homom},~\ref{eq:Lax motivation}), hence, $T_{\bar{D}}(z)$ is Lax. Since multiplication by
$x^{\nu}$ preserves~(\ref{eq:rtt}) (due to Corollary~\ref{cor:YBE invariance} and the equality
$\epsilon^\vee_1+{\epsilon'}^\vee_1=\epsilon^\vee_2+{\epsilon'}^\vee_2=\ldots=\epsilon^\vee_r+{\epsilon'}^\vee_r$),
we finally obtain:

\medskip

\begin{Prop}\label{prop:preserving RTT}
For any $\Lambda^+$-valued divisor $D$ on $\BP^1$ satisfying the assumption~(\ref{eq:assumption}), the matrix
$T_D(z)$ defined via~(\ref{eq:Lax explicit}) is Lax, i.e.\ it satisfies the RTT relation~(\ref{eq:rtt}).
\end{Prop}

\medskip

    %%%%%%%%%%%%%%%%%%%%%%%%%%%%%%%%%%%%%%%%%%%%%%%%%%%%%%%%%%%%%%%%%%%%%%%%%%%%%%%
    %%%%%%%%%%%%%%%%%%%%%%%% Proof of Drinfeld=RTT in shifted %%%%%%%%%%%%%%%%%%%%%
    %%%%%%%%%%%%%%%%%%%%%%%%%%%%%%%%%%%%%%%%%%%%%%%%%%%%%%%%%%%%%%%%%%%%%%%%%%%%%%%

\subsubsection{Proof of the key isomorphism}
\label{sssec Proof via Lax matrices D}
\

Reversing the argument from the previous Subsection, we note that the Lax matrix
$T_D(z)$ (Proposition~\ref{prop:preserving RTT}) gives rise to the algebra homomorphism
$\Theta_D\colon X^\rtt_{-\mu}(\sso_{2r})\to \CA$, whose composition with the epimorphism
  $\Upsilon_{-\mu}\colon X_{-\mu}(\sso_{2r})\twoheadrightarrow X^\rtt_{-\mu}(\sso_{2r})$
of Proposition~\ref{prop:epimorphism of shifted Yangians} coincides with the homomorphism
$\Psi_D$~(\ref{eq:homom psi}). Thus, for $\mu\in \Lambda^+$ and any $\Lambda^+$-valued
divisor $D$ on $\BP^1$, see~\eqref{eq:divisor def1}, satisfying~(\ref{eq:assumption})
and $D|_\infty=\mu$, the homomorphism $\Psi_D$ does factor through $\Upsilon_{-\mu}$.
This observation immediately implies the injectivity of $\Upsilon_{-\mu}$,
due to the recent result of~\cite{w}:\footnote{Actually, we need the extended version
of Theorem~\ref{thm:alex's theorem} (now with the points $z_i$, cf.~\eqref{eq:pq vs wu},
ranging over all $\BC$), which nevertheless follows immediately from the
algebra decomposition~\eqref{eq:shifted decomposition} and the formula~\eqref{image of C-series}.}

\medskip

\begin{Thm}[\cite{w}]\label{thm:alex's theorem}
For any coweight $\nu$ of a semisimple Lie algebra $\fg$, the intersection of kernels
of the homomorphisms $\Phi^{\ast}_{-\nu}$ of~\cite[Theorem B.15]{bfnb} is zero:
  $\bigcap_{\lambda}\ \mathrm{Ker}(\Phi^{\lambda}_{-\nu})=0$,
where $\lambda$ ranges through all dominant coweights of $\fg$ such that
$\lambda+\nu=\sum a_i\alpha_i$ with $a_i\in \BN$, $\alpha_i$ being simple coroots
of $\fg$, and points $\{z_i\}$ of~\emph{loc.cit.}\ specialized to arbitrary complex parameters.
\end{Thm}

\medskip
\noindent
This completes our proof of Theorem~\ref{thm:Dr=RTT-shifted-D}.
Combining this with Lemma~\ref{center of shifted yangians}(b), we obtain:

\medskip

\begin{Lem}\label{lem:shifted z-vs-C}
For any $\mu\in \Lambda^+$, the center of $X^\rtt_{-\mu}(\sso_{2r})$ is
a polynomial algebra in the coefficients $\{\sz_N^{(k)}\}_{k>d_r+d_{r+1}}$ of the series:
\begin{equation}\label{eq:shifted z-central}
  \sz_N(z)\ =\sum_{k\geq d_r+d_{r+1}}\sz_N^{(k)}z^{-k}=\Upsilon_{-\mu}(C_r(z))=\\
  \prod_{i=1}^{r-1}\frac{h_i(z+i-r)}{h_i(z+i-r+1)} \cdot h_r(z) h_{r+1}(z) \, .
\end{equation}
\end{Lem}

\medskip
\noindent
The above argument can be also used to establish the \emph{crossing relation} for $X^\rtt_{-\mu}(\sso_{2r})$:

\medskip

\begin{Prop}\label{prop:crossymmetry D}
For any $\mu\in \Lambda^+$, the matrix $T(z)$ of~\eqref{eq:shifted T-matrix} satisfies:
\begin{equation}\label{eq:crossymetry}
  T(z)T'(z-\kappa)=T'(z-\kappa)T(z)=\sz_N(z)\ID_N \, .
\end{equation}
\end{Prop}

\medskip

\begin{proof}
According to (the extended version of) Theorem~\ref{thm:alex's theorem}, it suffices to verify:
\begin{equation}\label{eq:crossymetry Lax}
  T_D(z)T'_D(z-\kappa)=T'_D(z-\kappa)T_D(z)=\Theta_D(\sz_N(z))\ID_N
\end{equation}
for any $\Lambda^+$-valued divisor $D$ on $\BP^1$ satisfying~\eqref{eq:assumption} and $D|_\infty=\mu$.
According to~\eqref{eq:zenter}, the equality~\eqref{eq:crossymetry Lax} obviously holds for $D$ such that $D|_\infty=0$.
Therefore, the validity of~\eqref{eq:crossymetry Lax} for any $D$ follows now from the ``normalized limit''
construction~\eqref{eq:limit relation}. To this end, using the notations of~\emph{loc.cit.},
the validity of~\eqref{eq:crossymetry Lax} for $D$ implies the one for $D'$ as follows from:
\begin{equation*}
\begin{split}
  & T_{D'}(z)T'_{D'}(z-\kappa) \ =
    \underset{x_N\to \infty}{\lim}\, \Big( (-x_N)^{\gamma_N\varpi_{i_N}}\cdot T_D(z)T'_D(z-\kappa)\cdot \left((-x_N)^{\gamma_N\varpi_{i_N}}\right)' \Big) = \\
  & \underset{x_N\to \infty}{\lim}\, \Big( (-x_N)^{\gamma_N\varpi_{i_N}}\cdot \left((-x_N)^{\gamma_N\varpi_{i_N}}\right)' \cdot \Theta_{D}(\sz_N(z)) \ID_N \Big) =
    \Theta_{D'}(\sz_N(z)) \ID_N\, ,
\end{split}
\end{equation*}
where the last equality follows from
\begin{equation*}
  x^{\varpi_i}\cdot \left(x^{\varpi_i}\right)' = x^{-2+\delta_{i,r-1}+\delta_{i,r}}\cdot\, \ID_N
\end{equation*}
and the explicit formulas~(\ref{image of C-series},~\ref{eq:shifted z-central}).
\end{proof}

\medskip

    %%%%%%%%%%%%%%%%%%%%%%%%%%%%%%%%%%%%%%%%%%%%%%%%%%%%%%%%%%%%%%%%%%%%%%%%%%%%%%%
    %%%%%%%%%%%%%%%%%%%%%%%%%%%%%%%%% Regularity of Lax %%%%%%%%%%%%%%%%%%%%%%%%%%%
    %%%%%%%%%%%%%%%%%%%%%%%%%%%%%%%%%%%%%%%%%%%%%%%%%%%%%%%%%%%%%%%%%%%%%%%%%%%%%%%

\subsubsection{Regularity of Lax matrices}
\label{sssec Lax regularity}
\

Consider the following normalized version of $T_D(z)$:
\begin{equation}\label{eq:renormalized Lax}
  \sT_D(z):=T_D(z)/Z_0(z) \,,
\end{equation}
with the normalization factor $Z_0(z)$ defined in~\eqref{ZW-series}.
The key property of these matrices is their regularity in (the spectral parameter) $z$:

\medskip

\begin{Thm}\label{Main Theorem 1}
We have $\sT_D(z)\in \CA[z]\otimes_\BC \End\, \BC^{2r}$.
\end{Thm}

\medskip
\noindent
This straightforward verification, based on the explicit formulas of Appendix~\ref{Appendix A: Lax explicitly},
is completely analogous (though is more tedious) to its type $A$ counterpart of~\cite[Theorem~2.67]{fpt}.

\medskip

\begin{Rem}\label{rmk:other orientations}
Similar to~\cite[Theorem B.15]{bfnb}, Theorem~\ref{thm:extended homom D} can be generalized
by constructing the homomorphisms $\Psi_D\colon X_{-\mu}(\sso_{2r})\to \CA$ for any orientation
of $D_r$ Dynkin diagram, so that $\Psi_D\circ \iota_{-\mu}$ is to $\Phi^{\bar{\lambda}}_{-\bar{\mu}}$
of~\cite{bfnb} as in Remark~\ref{rmk:relating to BFNb homom} (note that the images of $D_i(z)$
are independent of the orientation, hence, so is the image of $C_r(z)$,
see~(\ref{eq:central C shifted},~\ref{image of C-series})). However, extending $\CA$ to
its localization $\CA_{\loc}$ by the multiplicative set generated by
  $\{p_{i,k}-p_{j,\ell}+m\}_{k\leq a_i,\ell\leq a_{j}}^{m\in \BZ}$
with $(i,j)$ connected by an edge, these homomorphisms are compositions of~(\ref{eq:homom psi})
with algebra automorphisms of $\CA_{\mathrm{loc}}$. Thus, similar to~\cite[Remark 2.73]{fpt},
the resulting Lax matrices are equivalent, up to algebra automorphisms of $\CA_{\mathrm{loc}}$,
to the above $T_D(z)$, cf.~Remark~\ref{rmk:about other orientations}.
\end{Rem}

\medskip

    %%%%%%%%%%%%%%%%%%%%%%%%%%%%%%%%%%%%%%%%%%%%%%%%%%%%%%%%%%%%%%%%%%%%%%%%%%%%%%%
    %%%%%%%%%%%%%%%%%%%%%%%%%% Linear Lax Matrices %%%%%%%%%%%%%%%%%%%%%%%%%%%%%%%%
    %%%%%%%%%%%%%%%%%%%%%%%%%%%%%%%%%%%%%%%%%%%%%%%%%%%%%%%%%%%%%%%%%%%%%%%%%%%%%%%

\subsubsection{Linear Lax matrices}
\label{ssec linear Lax}
\

The regularity of Theorem~\ref{Main Theorem 1} provides a shortcut to the computation
of the Lax matrices $T_D(z)$ defined, in general, as a product of three complicated
matrices $F^{D}(z),H^D(z),E^D(z)$ in~\eqref{eq:Lax definition}. Let us illustrate this
in the case of the linear ones, i.e.\ those of degree $1$ in the spectral parameter $z$.
We shall use the following notations:
\begin{equation}\label{eq:modes}
  e^{D}_{i,j}(z)=\sum_{k\geq 1} e^{(D)k}_{i,j}z^{-k} \, ,\quad
  f^{D}_{j,i}(z)=\sum_{k\geq 1} f^{(D)k}_{j,i}z^{-k} \, ,\quad
  h^{D}_{i}(z)=\sum_{k\in \BZ} h^{(D)k}_{i}z^{-k} \, .
\end{equation}
Let us also recall the coefficients $a_i\in \BN$ from~(\ref{eq:assumption},~\ref{eq:a explicitly}).

\medskip

\begin{Prop}\label{prop:linear Lax D}
(a) The normalized Lax matrix $\sT_D(z)$ of~\eqref{eq:renormalized Lax} is linear iff $a_1=1$.

\medskip
\noindent
(b) Any linear normalized Lax matrix $\sT_D(z)$ is explicitly determined as follows:

\medskip
$\bullet$
The diagonal entries are:
\begin{equation}\label{eq:linear diagonal}
  \sT_{D}(z)_{i,i}=z\cdot h^{(D),-1}_i + h^{(D)0}_i\, , \qquad 1\leq i\leq 2r \, .
\end{equation}

\medskip
$\bullet$
The entries above the main diagonal are:
\begin{equation}\label{eq:linear upper}
  \sT_{D}(z)_{i,j}=
  \begin{cases}
     e^{(D)1}_{i,j} & \mbox{if } h^{(D),-1}_i \ne 0 \\
     0 & \mbox{otherwise}
  \end{cases} \, , \qquad 1\leq i<j\leq 2r \, .
\end{equation}

\medskip
$\bullet$
The entries below the main diagonal are:
\begin{equation}\label{eq:linear lower}
  \sT_{D}(z)_{j,i}=
  \begin{cases}
     f^{(D)1}_{j,i} & \mbox{if } h^{(D),-1}_i \ne 0 \\
     0 & \mbox{otherwise}
  \end{cases} \, , \qquad 1\leq i<j\leq 2r \, .
\end{equation}
\end{Prop}

\medskip

\begin{proof}
(a) As $\sT_D(z)_{1,1}=\frac{h^D_1(z)}{Z_0(z)}=P_1(z)$ is a polynomial in $z$ of degree $a_1$,
the condition $a_1\leq 1$ is necessary for $\sT_D(z)$ to be of degree $\leq 1$ in $z$.
On the other hand, $\deg_z h^{D}_1(z)\geq \deg_z h^{D}_i(z)$ for any $1<i\leq 2r$,
due to~\eqref{eq:d-inequalities}. Combining this with $\deg_z e^{D}_{i,j}(z),\deg_z f^D_{j,i}(z)<0$,
we conclude that $a_1\leq 1$ is also a sufficient condition for $\sT_D(z)$ to be of degree $\leq 1$ in $z$.
Moreover, $\sT_D(z)$ is actually of degree $<1$ in $z$ if and only if $a_1<1$.
This concludes our proof of part (a).

(b) This follows immediately from the regularity result of Theorem~\ref{Main Theorem 1} combined
with the formula~\eqref{eq:Lax explicit} and the observation that $\deg_z e^D_{i,j}(z),\deg_z f^D_{j,i}(z)<0$
for any $i<j$.
\end{proof}

\medskip
\noindent
Let us now describe all $\Lambda^+$-valued divisors $D$ on $\BP^1$ satisfying~(\ref{eq:assumption})
such that $\deg_z \sT_D(z)=1$. Define $\lambda,\mu\in \Lambda^+$ via~(\ref{eq:mu from divisor},~\ref{eq:lambda from divisor}),
so that $\lambda+\mu=\sum_{j=0}^r b_j\varpi_j$ with $b_0\in\BZ,\ b_1,\ldots,b_r\in \BN$.
Then, the assumption~\eqref{eq:assumption} implies:
\begin{equation}\label{eq:ab}
  \sum_{j=0}^{r} b_j\varpi_j\, =\, \sum_{i=1}^{r} a_i\hat{\alpha}_i
  \qquad \mathrm{with}\quad a_i\in \BN \, .
\end{equation}
Decomposing both sides of~\eqref{eq:ab} in the basis $\{\epsilon_i\}_{i=1}^{r+1}$,
we can express $b_j$'s via $a_i$'s:
\begin{equation}\label{eq: b via a}
\begin{split}
  & b_0=-a_1 \, ,\ b_1=2a_1-a_{2} \, ,\ b_i=-a_{i-1}+2a_i-a_{i+1} \quad \mathrm{for}\ 2\leq i\leq r-3 \, ,\\
  & b_{r-2}=-a_{r-3}+2a_{r-2}-a_{r-1}-a_{r} \, ,\ b_{r-1}=-a_{r-2}+2a_{r-1} \, ,\ b_{r}=-a_{r-2}+2a_{r} \, .
\end{split}
\end{equation}
Likewise, let us also express $a_i$'s via $b_j$'s:
\begin{equation}\label{eq:a via b}
\begin{split}
  & a_i = \sum_{k=0}^{i} (k-i)b_k =
    \sum_{k=1}^{i-1} kb_k+ i \sum_{k=i}^{r-2} b_k + \frac{i}{2}\left(b_{r-1}+b_r\right) \quad \mathrm{for}\ 1\leq i\leq r-2 \, ,\\
  & a_{r-1} = \frac{1}{2} \sum_{k=0}^{r-1} (k+2-r)b_k =
    \frac{1}{2}\left(\sum_{k=1}^{r-2}kb_k+\frac{r}{2}b_{r-1}+\frac{r-2}{2}b_{r}\right) \, ,\\
  & a_{r} = \frac{1}{2} \sum_{k=0}^{r} (k-r)b_k =
    \frac{1}{2}\left(\sum_{k=1}^{r-2}kb_k+\frac{r-2}{2}b_{r-1}+\frac{r}{2}b_{r}\right) \, ,
\end{split}
\end{equation}
where we used the equality (arising from the comparison of the coefficients of $\epsilon_r$ and $\epsilon_{r+1}$):
\begin{equation}\label{eq:b-null}
  b_0=-b_1-\ldots-b_{r-2}-\frac{1}{2}(b_{r-1}+b_r) \, .
\end{equation}
Note that~\eqref{eq:b-null} uniquely recovers $b_0$ in terms of $b_1,\ldots,b_r$
and forces $b_{r-1}+b_r$ to be even. We also note that the total number of pairs
of $(p,q)$-oscillators in the algebra $\CA$~\eqref{algebra A} equals:
\begin{equation}\label{eq:counting oscillators}
  \sum_{i=1}^r a_i = - \sum_{k=0}^{r} \frac{(r-k)(r-k-1)}{2} b_k =
  \sum_{k=1}^{r-2}k\left(r-\frac{k+1}{2}\right)b_k+\frac{r(r-1)}{4}\left(b_{r-1}+b_r\right) \, .
\end{equation}

\medskip
\noindent
Combining the above formulas~(\ref{eq:a via b},~\ref{eq:b-null}) with Proposition~\ref{prop:linear Lax D}(a),
we thus conclude that the normalized Lax matrix $\sT_D(z)$ is linear only for the
following configurations of $b_i$'s:
\begin{enumerate}
  \item[(1)] $b_0=-1,\ b_j=1,\ b_1=\ldots=b_{j-1}=b_{j+1}=\ldots=b_r=0$ for an even $1\leq j\leq r-2$,

  \item[(2)]
  $\begin{cases}
    b_0=-1,\ b_1=\ldots=b_{r-2}=0,\ b_{r-1}=b_r=1 & \mbox{if } r \ \mathrm{is\ odd} \\
    b_0=-1,\ b_1=\ldots=b_{r-2}=0,\ \{b_{r-1},b_r\}=\{0,2\} & \mbox{if } r \ \mathrm{is\ even}
  \end{cases}\, .$
\end{enumerate}
As $b_0$ is uniquely determined via~\eqref{eq:b-null} and does not affect the normalized
Lax matrix $\sT_D(z)$, we shall rather focus on the corresponding values of the dominant
$\sso_{2r}$-coweights $\bar{\lambda},\bar{\mu}\in \bar{\Lambda}^+$.

\medskip
\noindent
$\bullet$ \emph{\underline{Case (1)}:} $\bar{\lambda}+\bar{\mu}=\omega_j$ for even $1\leq j\leq r-2$.
\

In this case, we have $a_1=1,\ldots,a_{j-1}=j-1,a_j=\ldots=a_{r-2}=j,a_{r-1}=a_r=j/2$,
and the total number of pairs of $(p,q)$-oscillators is $\frac{j(2r-j-1)}{2}$,
see~(\ref{eq:a via b},~\ref{eq:counting oscillators}).

\medskip
\noindent
For $\bar{\lambda}=\omega_j,\bar{\mu}=0$ we get a \emph{non-degenerate}
Lax matrix with $z$ appearing on the entire diagonal:
\begin{equation}\label{eq:nondeg D1}
  \sT_{\varpi_j[x]-\varpi_0[y]}(z)=z(E_{11}+\ldots+E_{2r,2r})+O(1) \, ,
\end{equation}
depending on the additional parameter $x\in \BC$ (note that it is independent of the point $y\in \BP^1$).

\medskip
\noindent
The normalized limit~\eqref{eq:limit relation} of~\eqref{eq:nondeg D1} as $x\to \infty$ recovers
the Lax matrix corresponding to $\bar{\lambda}=0,\bar{\mu}=\omega_j$, which is \emph{degenerate}
as it contains $z$ only in the first $j$ diagonal entries:
\begin{equation}\label{eq:deg D1}
  \sT_{\varpi_j[\infty]-\varpi_0[y]}=z(E_{11}+\ldots+E_{jj})+O(1)
\end{equation}
and also satisfies:
\begin{equation}\label{eq:deg D1 others}
  \sT_{\varpi_j[\infty]-\varpi_0[y]}(z)_{k,k}=
  \begin{cases}
    1 & \mbox{if } j+1\leq k\leq (j+1)' \\
    0 & \mbox{if } j'\leq k\leq 1'
  \end{cases}\, .
\end{equation}

\medskip
\noindent
$\bullet$ \emph{\underline{Case (2) for odd $r$}:} $\bar{\lambda}+\bar{\mu}=\omega_{r-1}+\omega_r$.
\

In this case, we have $a_1=1,\ldots,a_{r-2}=r-2, a_{r-1}=a_r=\frac{r-1}{2}$,
and the total number of pairs of $(p,q)$-oscillators is $\frac{r(r-1)}{2}$,
see~(\ref{eq:a via b},~\ref{eq:counting oscillators}).

\medskip
\noindent
For $\bar{\lambda}=\omega_{r-1}+\omega_r,\bar{\mu}=0$ we get a \emph{non-degenerate}
Lax matrix with $z$ on the entire diagonal:
\begin{equation}\label{eq:nondeg D2}
  \sT_{\varpi_{r-1}[x_1]+\varpi_r[x_2]-\varpi_0[y]}(z)=z(E_{11}+\ldots+E_{2r,2r})+O(1) \, ,
\end{equation}
which depends on two additional parameters $x_1,x_2\in \BC$ (but is independent of $y\in \BP^1$).

\medskip
\noindent
The normalized limit~\eqref{eq:limit relation} of~\eqref{eq:nondeg D2} as $x_2\to \infty$
recovers the Lax matrix corresponding to $\bar{\lambda}=\omega_{r-1},\bar{\mu}=\omega_r$,
which is \emph{degenerate} as it contains $z$ only in the first $r$ diagonal entries:
\begin{equation}\label{eq:deg D2.1}
  \sT_{\varpi_{r-1}[x_1]+\varpi_r[\infty]-\varpi_0[y]}(z)=z(E_{11}+\ldots+E_{rr})+O(1)
\end{equation}
and also satisfies:
\begin{equation}\label{eq:deg D2.1 others}
  \sT_{\varpi_{r-1}[x_1]+\varpi_r[\infty]-\varpi_0[y]}(z)_{k,k}=1
  \qquad \mathrm{for}\quad r'\leq k\leq 1' \, .
\end{equation}
Likewise, the normalized limit of~\eqref{eq:nondeg D2} as $x_1\to \infty$ recovers
the Lax matrix corresponding to $\bar{\lambda}=\omega_{r},\bar{\mu}=\omega_{r-1}$,
which is \emph{degenerate} as it contains $z$ only in $r$ of its diagonal entries:
\begin{equation}\label{eq:deg D2.2}
  \sT_{\varpi_{r}[x_2]+\varpi_{r-1}[\infty]-\varpi_0[y]}(z)=z(E_{11}+\ldots+E_{r-1,r-1}+E_{r+1,r+1})+O(1)
\end{equation}
and also satisfies:
\begin{equation}\label{eq:deg D2.2 others}
  \sT_{\varpi_{r}[x_2]+\varpi_{r-1}[\infty]-\varpi_0[y]}(z)_{k,k}=1
  \qquad \mathrm{for}\quad k\in \{r,r+2,r+3,\ldots,2r\} \, .
\end{equation}
Finally, the normalized limit of~\eqref{eq:deg D2.1} as $x_1\to \infty$, or equivalently
of~\eqref{eq:deg D2.2} as $x_2\to \infty$, recovers the Lax matrix corresponding to
$\bar{\lambda}=0,\bar{\mu}=\omega_{r-1}+\omega_{r}$, which is even more \emph{degenerate}:
\begin{equation}\label{eq:deg D2.3}
  \sT_{\varpi_{r-1}[\infty]+\varpi_r[\infty]-\varpi_0[y]}(z)=z(E_{11}+\ldots+E_{r-1,r-1})+O(1)
\end{equation}
and also satisfying:
\begin{equation}\label{eq:deg D2.3 others}
  \sT_{\varpi_{r-1}[\infty]+\varpi_r[\infty]-\varpi_0[y]}(z)_{k,k}=
  \begin{cases}
    1 & \mbox{if } k=r,r' \\
    0 & \mbox{if } r'<k\leq 1'
  \end{cases}\, .
\end{equation}

\medskip
\noindent
$\bullet$ \emph{\underline{Case (2) for even $r$}:} $\bar{\lambda}+\bar{\mu}=2\omega_{r-1}$ or $2\omega_r$.
\

In this case, we have $a_1=1,\ldots,a_{r-2}=r-2$ and $\{a_{r-1},a_r\}=\{\frac{r}{2},\frac{r}{2}-1\}$,
and the total number of pairs of $(p,q)$-oscillators is again $\frac{r(r-1)}{2}$,
see~(\ref{eq:a via b},~\ref{eq:counting oscillators}).

\medskip
\noindent
For $\bar{\lambda}=2\omega_{r-1},\bar{\mu}=0$ we get a \emph{non-degenerate}
Lax matrix with $z$ on the entire diagonal:
\begin{equation}\label{eq:nondeg D3}
  \sT_{\varpi_{r-1}([x_1]+[x_2])-\varpi_0[y]}(z)=z(E_{11}+\ldots+E_{2r,2r})+O(1) \, ,
\end{equation}
which depends, in a symmetric way, on additional parameters $x_1,x_2\in \BC$, but not on $y\in \BP^1$.

\medskip
\noindent
The normalized limit~\eqref{eq:limit relation} of~\eqref{eq:nondeg D3} as $x_2\to \infty$ recovers
the Lax matrix corresponding to $\bar{\lambda}=\omega_{r-1},\bar{\mu}=\omega_{r-1}$, which is
\emph{degenerate} as it contains $z$ only in half of its diagonal entries:
\begin{equation}\label{eq:deg D3.1}
  \sT_{\varpi_{r-1}([x_1]+[\infty])-\varpi_0[y]}(z)=z(E_{11}+\ldots+E_{r-1,r-1}+E_{r+1,r+1})+O(1)
\end{equation}
and also satisfies:
\begin{equation}\label{eq:deg D3.1 others}
  \sT_{\varpi_{r-1}([x_1]+[\infty])-\varpi_0[y]}(z)_{k,k}=1
  \qquad \mathrm{for}\quad k\in \{r,r+2,r+3,\ldots,2r\} \, .
\end{equation}
Finally, the normalized limit~\eqref{eq:limit relation} of~\eqref{eq:deg D3.1} as $x_1\to \infty$
recovers the Lax matrix corresponding to $\bar{\lambda}=0,\bar{\mu}=2\omega_{r-1}$,
which also contains $z$ in half of its diagonal entries:
\begin{equation}\label{eq:deg D3.2}
  \sT_{2\varpi_{r-1}[\infty]-\varpi_0[y]}(z)=z(E_{11}+\ldots+E_{r-1,r-1}+E_{r+1,r+1})+O(1) \, ,
\end{equation}
but is more degenerate in the other diagonal entries:
\begin{equation}\label{eq:deg D3.2 others}
  \sT_{2\varpi_{r-1}[\infty]-\varpi_0[y]}(z)_{k,k}=0
  \qquad \mathrm{for}\quad k\in \{r,r+2,r+3,\ldots,2r\} \, .
\end{equation}

\medskip
\noindent
For $\bar{\lambda}=2\omega_{r},\bar{\mu}=0$ we get a \emph{non-degenerate}
Lax matrix with $z$ on the entire diagonal:
\begin{equation}\label{eq:nondeg D4}
  \sT_{\varpi_{r}([x_1]+[x_2])-\varpi_0[y]}(z)=z(E_{11}+\ldots+E_{2r,2r})+O(1) \, ,
\end{equation}
which depends, in a symmetric way, on additional parameters $x_1,x_2\in \BC$, but not on $y\in \BP^1$.

\medskip
\noindent
The normalized limit~\eqref{eq:limit relation} of~\eqref{eq:nondeg D4} as $x_2\to \infty$ recovers
the Lax matrix corresponding to $\bar{\lambda}=\omega_{r},\bar{\mu}=\omega_{r}$, which is
\emph{degenerate} as it contains $z$ only in half of its diagonal entries:
\begin{equation}\label{eq:deg D4.1}
  \sT_{\varpi_{r}([x_1]+[\infty])-\varpi_0[y]}(z)=z(E_{11}+\ldots+E_{rr})+O(1)
\end{equation}
and also satisfies:
\begin{equation}\label{eq:deg D4.1 others}
  \sT_{\varpi_{r}([x_1]+[\infty])-\varpi_0[y]}(z)_{k,k}=1
  \qquad \mathrm{for}\quad r'\leq k\leq 1' \, .
\end{equation}
Finally, the normalized limit~\eqref{eq:limit relation} of~\eqref{eq:deg D4.1} as $x_1\to \infty$
recovers the Lax matrix corresponding to $\bar{\lambda}=0,\bar{\mu}=2\omega_{r}$,
which also contains $z$ in half of its diagonal entries:
\begin{equation}\label{eq:deg D4.2}
  \sT_{2\varpi_{r}[\infty]-\varpi_0[y]}(z)=z(E_{11}+\ldots+E_{rr})+O(1) \,
\end{equation}
but is more degenerate in the other diagonal entries:
\begin{equation}\label{eq:deg D4.2 others}
  \sT_{2\varpi_{r}[\infty]-\varpi_0[y]}(z)_{k,k}=0
  \qquad \mathrm{for}\quad r'\leq k\leq 1' \, .
\end{equation}

\medskip
\noindent
We conclude this Subsection with the following important \emph{unitarity} property of the above
non-degenerate linear Lax matrices (recall the parameter $\kappa=r-1$, see~\eqref{eq:N,kappa,prime}),
cf.~\cite[(3.8)]{r}:

\medskip

\begin{Prop}\label{prop:property linear D}
(a) For any even $1\leq \jmath\leq r-2$, the corresponding non-degenerate linear Lax matrix
$L_\jmath(z):=\sT_{\varpi_\jmath[x]-\varpi_0[y]}\left(z+x+\frac{\kappa-\jmath}{2}\right)$ is unitary:
\begin{equation*}
  L_{\jmath}(z)L_{\jmath}(-z)=\left[ \left(\frac{\kappa-\jmath}{2}\right)^2 - \, z^2 \right] \ID_N \, .
\end{equation*}

\noindent
(b) Consider $D=\varpi_{\imath}[x_1]+\varpi_{\jmath}[x_2]-\varpi_0[y]$ with
$\imath=\jmath\in \{r-1,r\}$ if $r$ is even or $\{\imath,\jmath\}=\{r-1,r\}$ if $r$ is odd.
Then, the non-degenerate linear Lax matrix $L(z):=\sT_{D}\left(z+\frac{x_1+x_2}{2}\right)$ is unitary:
\begin{equation*}
  L(z)L(-z)=\left[ \left(\frac{x_1-x_2}{2}\right)^2 - \, z^2 \right] \ID_N \, .
\end{equation*}
\end{Prop}

\medskip

\begin{Rem}
We note that such unitarity property can be regarded as a consequence of the natural constraints
that arise for a linear solution when inserted into the RTT relation~\eqref{eq:rtt}, see~\cite[(18)]{Kirsch}.
\end{Rem}

\medskip

\begin{proof}
(a) Combining Theorem~\ref{Main Theorem 1} and Proposition~\ref{prop:linear Lax D} with the equalities
\begin{equation*}
  h^{(D)0}_{i}=-h^{(D)0}_{i'}-2x+\jmath \, ,\quad
  e^{(D)1}_{i,j}=-e^{(D)1}_{j',i'}\, ,\quad
  f^{(D)1}_{j,i}=-f^{(D)1}_{i',j'} \, ,
\end{equation*}
due to~(\ref{eq:H-part Lax},~\ref{eq:E-linear-skew},~\ref{eq:F-linear-skew}), we obtain:
\begin{equation}\label{eq:T vs T'}
  \sT'_{\varpi_\jmath[x]-\varpi_0[y]}(z)=-\sT_{\varpi_\jmath[x]-\varpi_0[y]}(-z+2x-\jmath) \, .
\end{equation}
According to the crossing relation of Proposition~\ref{prop:crossymmetry D},
see formula~\eqref{eq:crossymetry Lax}, we also have:
\begin{equation}\label{eq:TT' lin Lax}
  \sT_{\varpi_\jmath[x]-\varpi_0[y]}(z)\sT'_{\varpi_\jmath[x]-\varpi_0[y]}(z-r+1)=(z-x)(z-x-r+\jmath+1)\ID_N \, .
\end{equation}
The result now follows by combining~(\ref{eq:T vs T'},~\ref{eq:TT' lin Lax}).

(b) The proof is completely analogous to that of part (a) and is left to the interested reader
(in particular, $h^{(D)0}_{i}=-h^{(D)0}_{i'}+r-1-x_1-x_2$).
\end{proof}

\medskip

    %%%%%%%%%%%%%%%%%%%%%%%%%%%%%%%%%%%%%%%%%%%%%%%%%%%%%%%%%%%%%%%%%%%%%%%%%%%%%%%
    %%%%%%%%%%%%%%%%%%%%%%%%%% Examples in type D %%%%%%%%%%%%%%%%%%%%%%%%%%%%%%%%%
    %%%%%%%%%%%%%%%%%%%%%%%%%%%%%%%%%%%%%%%%%%%%%%%%%%%%%%%%%%%%%%%%%%%%%%%%%%%%%%%

\subsubsection{Examples}
\label{ssec examples}
\

In this Subsection, we explain how the type $D_r$ linear and quadratic Lax matrices
recently constructed by the first author in~\cite{f} arise as particular examples of
our general construction.

   %%%%%%%%%%%%%%%%%%%%%%%%%%%%%%%%%%%%%%%%%%%%%%%%%%%%%%%%%%%%%%%%%%%%%%%%%%%%%%%%%%%%%%%
   %%%%%%%%%%%%%%%%%%%%%%%%%%%%% Example 1 for Lax D-type %%%%%%%%%%%%%%%%%%%%%%%%%%%%%%%%
   %%%%%%%%%%%%%%%%%%%%%%%%%%%%%%%%%%%%%%%%%%%%%%%%%%%%%%%%%%%%%%%%%%%%%%%%%%%%%%%%%%%%%%%

\medskip
\noindent
$\bullet$ \emph{\underline{Example 1}:}
Consider the following $\Lambda^+$-valued divisor on $\BP^1$:
\begin{equation}\label{eq:divisor D1}
  D=
  \begin{cases}
    \varpi_{r-1}[x]+\varpi_r[\infty]-\varpi_0[y] & \mbox{if } r\ \mathrm{is\ odd} \\[0.2cm]
    \varpi_{r}[x]+\varpi_r[\infty]-\varpi_0[y] & \mbox{if } r\ \mathrm{is\ even}
  \end{cases}
\end{equation}
depending on $x\in \BC$ (note that $\sT_D(z)$ is independent of $y\in \BP^1$),
so that the total number of $(p,q)$-oscillators in the algebra $\CA$ equals
$a_1+\ldots+a_r=\frac{r(r-1)}{2}$.

\medskip
\noindent
According to~(\ref{eq:deg D2.1}--\ref{eq:deg D2.1 others},~\ref{eq:deg D4.1}--\ref{eq:deg D4.1 others}), the corresponding
normalized Lax matrix $\sT_D(z)$ has the block form:
\begin{equation}\label{Matrix Example D1}
\sT_D(z)=
\begin{pmatrix}
  z\ID_{r}+\sF & \sB \\
  \sC & \ID_{r}
\end{pmatrix},
\end{equation}
where $\sB,\sC,\sF$ are $z$-independent $r\times r$ matrices. We have the following properties of $\sB,\sC$:

\medskip

\begin{Lem}\label{lem:BC-rel type D}
(a) The matrices $\sB,\sC$ are skew-symmetric with respect to their antidiagonals:
\begin{equation*}
  \sB_{ij}=-\sB_{r+1-j,r+1-i}\, ,\qquad \sC_{ij}=-\sC_{r+1-j,r+1-i}\, .
\end{equation*}
In particular, $\sB_{i,r+1-i}=\sC_{i,r+1-i}=0$ for any $1\leq i\leq r$.

\medskip
\noindent
(b) The matrix coefficients $\{\sB_{ij}\}_{i,j=1}^r$ of the matrix $\sB$ pairwise commute.

\medskip
\noindent
(c) The matrix coefficients $\{\sC_{ij}\}_{i,j=1}^r$ of the matrix $\sC$ pairwise commute.

\medskip
\noindent
(d) The commutation among the matrix coefficients of $\sB$ and $\sC$ is given by:
\begin{equation*}
  [\sB_{ij},\sC_{k\ell}]=\delta_{i,\ell}\delta_{j,k}-\delta_{i,r+1-k}\delta_{j,r+1-\ell} \, .
\end{equation*}
\end{Lem}

\medskip

\begin{proof}
(a) According to Proposition~\ref{prop:linear Lax D}, we have
$\sB_{ij}=e^{(D)1}_{i,r+j}, \sB_{r+1-j,r+1-i}=e^{(D)1}_{r+1-j,2r+1-i}$.
Combining this with $e^{(D)1}_{i,r+j}=-e^{(D)1}_{r+1-j,2r+1-i}$,
due to~\eqref{eq:E-linear-skew}, we obtain the desired equality $\sB_{ij}=-\sB_{r+1-j,r+1-i}$.
The proof of $\sC_{ij}=-\sC_{r+1-j,r+1-i}$ is completely analogous.

(b, c) Those follow immediately from the ansatz~\eqref{Matrix Example D1} and
the RTT relation~\eqref{eq:rtt explicit} applied to the evaluation of
$[\sB_{ij},\sB_{k\ell}]=[t_{i,r+j}(z),t_{k,r+\ell}(w)]$ or
$[\sC_{ij},\sC_{k\ell}]=[t_{r+i,j}(z),t_{r+k,\ell}(w)]$.

(d) This also follows from the ansatz~\eqref{Matrix Example D1} and the RTT relation~\eqref{eq:rtt explicit}.
Indeed, evaluating $[\sB_{ij},\sC_{k\ell}]=[t_{i,r+j}(z),t_{r+k,\ell}(w)]$ via~\eqref{eq:rtt explicit},
the first summand is easily seen to equal $\delta_{i,\ell}\delta_{j,k}$, while computing
the leading term of the second summand, we get $-\delta_{i,r+1-k}\delta_{j,r+1-\ell}$.
\end{proof}

\medskip
\noindent
It will be convenient to relabel the matrices $\sB,\sC$ as $\ap,-\am$, respectively:
\begin{equation}\label{eq:ApAm D-type}
\begin{split}
  \sB=&\ap=
    \left(\begin{array}{cccc}
          \oad_{1,r'} & \cdots & \oad_{1,2'} & 0\\
          \vdots & \iddots & \iddots & -\oad_{1,2'}\\
          \oad_{r-1,r'} & \iddots & \iddots & \vdots\\
          0 & -\oad_{r-1,r'} & \cdots & -\oad_{1,r'}
    \end{array}\right)\, ,\\[0.15cm]
  -\sC=&\am=
    \left(\begin{array}{cccc}
              \oa_{r',1} & \cdots & \oa_{r',r-1} & 0 \\
              \vdots & \iddots & \iddots & -\oa_{r',r-1} \\
              \oa_{2',1} & \iddots & \iddots & \vdots \\
              0 & -\oa_{2',1} & \cdots & -\oa_{r',1}
    \end{array}\right)\, ,
\end{split}
\end{equation}
with the matrix coefficients satisfying the following relations:
\begin{equation}\label{eq:aa-osc example D1}
  [\oa_{i',j},\oad_{k,\ell'}]=\delta_{i,\ell}\delta_{j,k}\, ,\quad
  [\oa_{i',j},\oa_{k',\ell}]=0\, ,\quad
  [\oad_{i,j'},\oad_{k,\ell'}]=0 \, ,
\end{equation}
due to Lemma~\ref{lem:BC-rel type D}.
Then, a tedious straightforward calculation (cf.~\cite[Theorem~2.133]{fpt}) yields:
\begin{equation}\label{Matrix Example D1 a-osc}
  \sT_D(z)=
  \left(\begin{BMAT}[5pt]{c:c}{c:c}
    (z+x)\ID_r-\ap\am & \ap\\
    -\am & \ID_r
  \end{BMAT}
\right)\,
\end{equation}
which can also be written in the following factorized form:
\begin{equation}\label{Matrix Example D1 factorized}
  \sT_D(z)=
  \left(\begin{BMAT}[5pt]{c:c}{c:c}
    \ID_r & \ap_{}\\
    0 & \ID_r
  \end{BMAT}\right)
  \left(\begin{BMAT}[5pt]{c:c}{c:c}
    (z+x) \ID_r & 0\\
    0 & \ID_r
  \end{BMAT}\right)
  \left(\begin{BMAT}[5pt]{c:c}{c:c}
    \ID_r & 0\\
    -\am_{} & \ID_r
  \end{BMAT}\right)\, .
\end{equation}

\medskip
\noindent
The type $D_r$ Lax matrix of the form~(\ref{Matrix Example D1 a-osc},~\ref{Matrix Example D1 factorized})
was recently discovered in~\cite[(4.3)]{f}.

   %%%%%%%%%%%%%%%%%%%%%%%%%%%%%%%%%%%%%%%%%%%%%%%%%%%%%%%%%%%%%%%%%%%%%%%%%%%%%%%%%%%%%%%
   %%%%%%%%%%%%%%%%%%%%%%%%%%%%% Example 2 for Lax D-type %%%%%%%%%%%%%%%%%%%%%%%%%%%%%%%%
   %%%%%%%%%%%%%%%%%%%%%%%%%%%%%%%%%%%%%%%%%%%%%%%%%%%%%%%%%%%%%%%%%%%%%%%%%%%%%%%%%%%%%%%

\medskip
\noindent
$\bullet$ \emph{\underline{Example 2}:}
Consider the following $\Lambda^+$-valued divisor on $\BP^1$:
\begin{equation}\label{eq:divisor D2}
  D=
  \begin{cases}
     \varpi_{r-1}[x_1]+\varpi_r[x_2]-\varpi_0[y] & \mbox{if } r\ \mathrm{is\ odd} \\[0.2cm]
     \varpi_{r}[x_1]+\varpi_r[x_2]-\varpi_0[y] & \mbox{if } r\ \mathrm{is\ even}
  \end{cases}
\end{equation}
depending on $x_1,x_2\in \BC$ (while $\sT_D(z)$ does not depend of $y\in \BP^1$),
so that the total number of $(p,q)$-oscillators in the algebra $\CA$ equals
$a_1+\ldots+a_r=\frac{r(r-1)}{2}$.

\medskip
\noindent
We expect that the normalized non-degenerate linear Lax matrix $\sT_{D}(z)$,
see~(\ref{eq:nondeg D2},~\ref{eq:nondeg D4}), is equivalent, up to a (nontrivial)
canonical transformation, to the Lax matrix $\mathcal{L}(z)$ of~\cite[(5.4)]{f}
(cf.~\cite[(3.6)]{r}). The latter was defined via:
\begin{equation}\label{Matrix Example D2 a-osc 2}
  \mathcal{L}(z)=
  \left(\begin{BMAT}[5pt]{c:c}{c:c}
    (z+x_1) \ID_r -\ap_{}\am_{} & \ap_{}(x_2-x_1+\am_{}\ap_{}) \\
     -\am_{} & (z+x_2) \ID_r+\am_{}\ap_{}
  \end{BMAT} \right)
\end{equation}
with the matrices $\ap,\am$ as in~\eqref{eq:ApAm D-type} encoding $\frac{r(r-1)}{2}$ pairs
of oscillators~\eqref{eq:aa-osc example D1}, which can also be written in the following factorized form:
\begin{equation}\label{Matrix Example D2 a-osc 1}
  \mathcal{L}^{}(z)=
  \left(\begin{BMAT}[5pt]{c:c}{c:c}
    \ID_r & \ap_{}\\
    0 & \ID_r
  \end{BMAT}\right)
  \left(\begin{BMAT}[5pt]{c:c}{c:c}
    (z+x_1) \ID_r & 0\\
    -\am_{} & (z+x_2)\ID_r
  \end{BMAT}\right)
  \left(\begin{BMAT}[5pt]{c:c}{c:c}
    \ID_r & -\ap_{}\\
    0 & \ID_r
  \end{BMAT}\right)\, .
\end{equation}

   %%%%%%%%%%%%%%%%%%%%%%%%%%%%%%%%%%%%%%%%%%%%%%%%%%%%%%%%%%%%%%%%%%%%%%%%%%%%%%%%%%%%%%%
   %%%%%%%%%%%%%%%%%%%%%%%%%%%%% Example 3 for Lax D-type %%%%%%%%%%%%%%%%%%%%%%%%%%%%%%%%
   %%%%%%%%%%%%%%%%%%%%%%%%%%%%%%%%%%%%%%%%%%%%%%%%%%%%%%%%%%%%%%%%%%%%%%%%%%%%%%%%%%%%%%%

\medskip
\noindent
$\bullet$ \emph{\underline{Example 3}:}
Consider the following $\Lambda^+$-valued divisor on $\BP^1$:
\begin{equation}\label{eq:divisor D3}
  D=\varpi_1([x]+[\infty])-\varpi_0([y_1]+[y_2])
\end{equation}
depending on $x\in \BC$ (while $\sT_D(z)$ does not depend of $y_1,y_2\in \BP^1$),
so that the total number of $(p,q)$-oscillators in the algebra $\CA$ equals
$a_1+\ldots+a_r=2(r-1)$.

\medskip
\noindent
We expect that the normalized quadratic Lax matrix $\sT_{D}(z)$ is equivalent, up to a
(nontrivial) canonical transformation, to the Lax matrix $\mathsf{L}(z+x)$ of~\cite[(4.12)]{f}.
The latter was defined via:
\begin{equation}\label{Matrix Example D3 a-osc}
  \mathsf{L}(z)=
  \left(\begin{BMAT}[5pt]{c|c|c}{c|c|c}
    z^2+z(2-\tfrac{N}{2}-\wp\wm)+\tfrac{1}{4}\wp\idb\wp^t\wm^t\idb\wm&z\wp-\tfrac{1}{2}\wp\idb\wp^t\wm^t\idb&-\tfrac{1}{2}\wp\idb\wp^t\\\
    -z\wm+\tfrac{1}{2}\idb\wp^t\wm^t\idb\wm&z\ID-\idb\wp^t\wm^t\idb&-\idb\wp^t\\
    -\tfrac{1}{2}\wm^t\idb\wm&\wm^t\idb&1\\
  \end{BMAT}\right)
\end{equation}
with
\begin{equation}\label{eq:J-matrix}
  \ID=\ID_{N-2}=
  \left(\begin{array}{cccc}
              1 & \cdots & 0 & 0 \\
              \vdots & \ddots & \vdots & \vdots \\
              0 & \cdots & 1 & 0 \\
              0 &\cdots & 0 & 1
  \end{array}\right) \, , \qquad
  \idb=\idb_{N-2}=
  \left(\begin{array}{cccc}
              0 & \cdots & 0 & 1 \\
              \vdots & \iddots & \iddots & 0 \\
              0 & \iddots & \iddots & \vdots \\
              1 & 0 & \cdots & 0
  \end{array}\right) \, ,
\end{equation}
and $\wm,\wp$ encoding $N-2=2(r-1)$ pairs of oscillators:
\begin{equation}\label{eq:wmwp D2}
  \wp=(\oad_{2},\ldots,\oad_{r},\oad_{r'},\ldots,\oad_{2'})\, ,\qquad
  \wm=(\oa_{2},\ldots,\oa_{r},\oa_{r'},\ldots,\oa_{2'})^t \, ,
\end{equation}
so that
\begin{equation}\label{eq:oaoad D4}
  [\oa_i,\oad_j]=\delta_{i,j}\, ,\qquad [\oa_i,\oa_j]=0\, ,\qquad [\oad_i,\oad_j]=0 \, .
\end{equation}

\medskip
\noindent
The matrix $\mathsf{L}(z)$ of~\eqref{Matrix Example D3 a-osc} can also be written
in the following factorized form, see~\cite[(4.15)]{f}:
\begin{multline}\label{Matrix Example D3 a-osc factroized}
  \mathsf{L}(z)=\\
  \left(\begin{BMAT}[5pt]{c|c|c}{c|c|c}
    1 & \wp & -\tfrac{1}{2}\wp\idb\wp^t\\\
    0 & \ID & -\idb\wp^t\\
    0 & 0 & 1\\
  \end{BMAT}\right)
  \left(\begin{BMAT}[5pt]{c|c|c}{c|c|c}
    z(z-\tfrac{N}{2}+2) & 0 & 0\\\
    0 & z\ID & 0\\
    0 & 0 & 1\\
  \end{BMAT}\right)
  \left(\begin{BMAT}[5pt]{c|c|c}{c|c|c}
    1 & 0 & 0 \\
    -\wm & \ID & 0 \\
    -\tfrac{1}{2}\wm^t\idb\wm & \wm^t\idb & 1\\
  \end{BMAT}\right) \, .
\end{multline}

   %%%%%%%%%%%%%%%%%%%%%%%%%%%%%%%%%%%%%%%%%%%%%%%%%%%%%%%%%%%%%%%%%%%%%%%%%%%%%%%%%%%%%%%
   %%%%%%%%%%%%%%%%%%%%%%%%%%%%% Example 4 for Lax D-type %%%%%%%%%%%%%%%%%%%%%%%%%%%%%%%%
   %%%%%%%%%%%%%%%%%%%%%%%%%%%%%%%%%%%%%%%%%%%%%%%%%%%%%%%%%%%%%%%%%%%%%%%%%%%%%%%%%%%%%%%

\medskip
\noindent
$\bullet$ \emph{\underline{Example 4}:}
Consider the following $\Lambda^+$-valued divisor on $\BP^1$:
\begin{equation}\label{eq:divisor D4}
  D=\varpi_1([x_1]+[x_2])-\varpi_0([y_1]+[y_2])
\end{equation}
depending on $x_1,x_2\in \BC$ (while $\sT_D(z)$ does not depend of $y_1,y_2\in \BP^1$),
so that the total number of $(p,q)$-oscillators in the algebra $\CA$ equals
$a_1+\ldots+a_r=2(r-1)$.

\medskip
\noindent
We expect that the normalized quadratic Lax matrix $\sT_{D}(z)$ is equivalent, up to a (nontrivial)
canonical transformation, to the Lax matrix $\mathfrak{L}_{x_1,-x_2}(z+x_1)$ of~\cite[(5.36, 5.38)]{f}
(see Remark~\ref{rem:reation to Reshetikhin} where a relation to~\cite[(3.11)]{r} is discussed).
The latter was defined via:
\begin{equation}\label{Matrix Example D4 a-osc factroized}
  \mathfrak{L}_{x_1,x_2}(z)=
  \left(\begin{BMAT}[5pt]{c|c|c}{c|c|c}
    1 & \wp & -\tfrac{1}{2}\wp\idb\wp^t\\\
    0 & \ID & -\idb\wp^t\\
    0 & 0 & 1\\
  \end{BMAT}\right)
  \cdot\, D_{x_1,x_2}(z)\, \cdot
  \left(\begin{BMAT}[5pt]{c|c|c}{c|c|c}
    1 & -\wp & -\tfrac{1}{2}\wp\idb\wp^t\\\
    0 & \ID & \idb\wp^t\\
    0 & 0 & 1\\
  \end{BMAT}\right)
\end{equation}
with $\ID,\idb,\wm,\wp$ as in~(\ref{eq:J-matrix})--(\ref{eq:oaoad D4}) and the middle factor explicitly given by:
\begin{equation*}
  D_{x_1,x_2}(z)=
  \left(\begin{BMAT}[5pt]{c|c|c}{c|c|c}
    (z-x_1)(z-x_1-\tfrac{N}{2}+2) & 0 & 0\\\
    -\wm_{}(z-x_1) & (z-x_1)(z-x_2)\ID & 0\\
    -\tfrac{1}{2}\wm^t\idb\wm & \wm^t\idb(z-x_2) & (z-x_2)(z-x_2-\tfrac{N}{2}+2)\\
  \end{BMAT}\right) \, .
\end{equation*}

\medskip
\noindent
We conclude this Subsection with the following observation:

\medskip

\begin{Rem}
Note that the following degeneration phenomena observed in~\cite[(5.11, 5.42)]{f}:
\begin{itemize}
  \item[(1)]
    degeneration of the Lax matrix~\eqref{Matrix Example D2 a-osc 2}
    into the one of~\eqref{Matrix Example D1 a-osc}
  \item[(2)]
    degeneration of the Lax matrix~\eqref{Matrix Example D4 a-osc factroized}
    into the one of~\eqref{Matrix Example D3 a-osc}
\end{itemize}
exactly agree with our general normalized limit construction~\eqref{eq:limit relation}.
\end{Rem}

\medskip

    %%%%%%%%%%%%%%%%%%%%%%%%%%%%%%%%%%%%%%%%%%%%%%%%%%%%%%%%%%%%%%%%%%%%%%%%%%%%%%%
    %%%%%%%%%%%%%%%%%%%%%%%%%%%%%%%%%% Type C %%%%%%%%%%%%%%%%%%%%%%%%%%%%%%%%%%%%%
    %%%%%%%%%%%%%%%%%%%%%%%%%%%%%%%%%%%%%%%%%%%%%%%%%%%%%%%%%%%%%%%%%%%%%%%%%%%%%%%

\section{Type C}\label{sec Rational Lax matrices C-type}
\label{sec C-type}

The type $C_r$ is completely similar to the type $D_r$, which we considered in details above.
Thus, we'll be brief, only stating the key results and highlighting the few technical differences.

    %%%%%%%%%%%%%%%%%%%%%%%%%%%%%%%%%%%%%%%%%%%%%%%%%%%%%%%%%%%%%%%%%%%%%%%%%%%%%%%
    %%%%%%%%%%%%%%%%%%%%%%%%%%% Ushifted story - Type C %%%%%%%%%%%%%%%%%%%%%%%%%%%
    %%%%%%%%%%%%%%%%%%%%%%%%%%%%%%%%%%%%%%%%%%%%%%%%%%%%%%%%%%%%%%%%%%%%%%%%%%%%%%%

\subsection{Classical (unshifted) story}
\label{ssec unshifted story C}
\

We shall realize the simple positive roots $\{\alphavee_i\}_{i=1}^r$ of
the Lie algebra $\ssp_{2r}$ in $\bar{\Lambda}^\vee$ via:
\begin{equation}\label{eq:alpha-vee C}
  \alphavee_1=\epsilon^\vee_1-\epsilon^\vee_2\, ,\
  \alphavee_2=\epsilon^\vee_2-\epsilon^\vee_3\, ,\ \ldots\, ,\
  \alphavee_{r-1}=\epsilon^\vee_{r-1}-\epsilon^\vee_r\, ,\
  \alphavee_r=2\epsilon^\vee_r\, .
\end{equation}

\medskip
\noindent
The \emph{Drinfeld Yangian} of $\ssp_{2r}$, denoted by $Y(\ssp_{2r})$, is defined similarly to $Y(\sso_{2r})$:
it is generated by $\{\sE_i^{(k)},\sF_i^{(k)},\sH_i^{(k)}\}_{1\leq i\leq r}^{k\geq 1}$ subject to
the defining relations~\eqref{Y0}--\eqref{Y7}, with $\alphavee_i$ of~\eqref{eq:alpha-vee C}.
The \emph{extended Drinfeld Yangian} of $\ssp_{2r}$, denoted by $X(\ssp_{2r})$,
is defined alike $X(\sso_{2r})$: it is generated by
  $\{E_i^{(k)},F_i^{(k)}\}_{1\leq i\leq r}^{k\geq 1}\cup \{D_i^{(k)}\}_{1\leq i\leq r+1}^{k\geq 1}$
subject to~(\ref{eY0})--(\ref{eY7}) with the modification:
\begin{equation}\label{eY23-C}
\begin{split}
  & [D_{r+1}(z), E_{r-1}(w)] =
    \frac{D_{r+1}(z)(E_{r-1}(z-2)-E_{r-1}(w))}{z-w-2} \, ,\\
  & [D_{r+1}(z), F_{r-1}(w)] =
    -\frac{(F_{r-1}(z-2)-F_{r-1}(w))D_{r+1}(z)}{z-w-2} \, .
\end{split}
\end{equation}
The central elements $\{C_r^{(k)}\}_{k\geq 1}$ of $X(\ssp_{2r})$ are now defined via (cf.~\eqref{eq:central C}):
\begin{equation}\label{eq:central C type C}
  C_r(z)=1+\sum_{k\geq 1} C_r^{(k)}z^{-k}:=
  \prod_{i=1}^{r-1}\frac{D_i(z+i-r-2)}{D_i(z+i-r-1)} \cdot D_r(z-2) D_{r+1}(z) \, .
\end{equation}
Furthermore, a natural analogue of Lemma~\ref{lem:embedding} holds with
$\iota_0\colon Y(\ssp_{2r})\hookrightarrow X(\ssp_{2r})$ given by:
\begin{equation}\label{eq:iota-null explicitly C}
\begin{split}
  & \sE_i(z)\mapsto
    \begin{cases}
      E_{i}(z+\frac{i-1}{2}) & \mbox{if } i<r \\
      E_{r}(z+\frac{r}{2}) & \mbox{if } i=r
    \end{cases} \, , \qquad
    \sF_i(z)\mapsto
    \begin{cases}
      F_i(z+\frac{i-1}{2}) & \mbox{if } i<r \\
      F_r(z+\frac{r}{2}) & \mbox{if } i=r
    \end{cases} \, , \\
  & \sH_i(z)\mapsto
    \begin{cases}
      D_i(z+\frac{i-1}{2})^{-1}D_{i+1}(z+\frac{i-1}{2}) & \mbox{if } i<r \\
      D_{r}(z+\frac{r}{2})^{-1}D_{r+1}(z+\frac{r}{2}) & \mbox{if } i=r
    \end{cases} \, .
\end{split}
\end{equation}

\medskip
\noindent
The \emph{extended RTT Yangian} of $\ssp_{2r}$, denoted by $X^\rtt(\ssp_{2r})$, is defined similarly
to $X^\rtt(\sso_{2r})$: it is generated by $\{t_{ij}^{(k)}\}_{1\leq i,j\leq N}^{k\geq 1}$
($N=2r$) subject to the RTT relation~\eqref{eq:rtt} with the $R$-matrix $R(z)$ given
by~\eqref{eq:R-matrix}, but with the following modifications of $\kappa\in \BC$ and
$\Qop\in \mathrm{End}\, \BC^N \otimes\, \mathrm{End}\, \BC^N$:
\begin{equation}\label{eq:varepsilons}
  \kappa=r+1 \, , \qquad
  \Qop\ = \sum_{i,j=1}^{N} \varepsilon_i \varepsilon_j\, E_{ij}\otimes E_{i'j'} \,
  \quad \mathrm{with} \quad
  \varepsilon_i=
  \begin{cases}
     1 & \mbox{if } 1\leq i\leq r \\
     -1 & \mbox{if } r'\leq i\leq 1'
  \end{cases} \, .
\end{equation}
The \emph{RTT Yangian} of $\ssp_{2r}$, denoted by $Y^\rtt(\ssp_{2r})$, is defined alike $Y^\rtt(\sso_{2r})$:
it is the subalgebra of $X^\rtt(\ssp_{2r})$ consisting of the elements stable under the automorphisms~\eqref{eq:f-autom}.
However, it can be also realized as a quotient of $X^\rtt(\ssp_{2r})$ as in~\eqref{eq:killing center}, due to
the natural analogue of~\eqref{eq:extended to usual}, where the center $ZX^\rtt(\ssp_{2r})$
of $X^\rtt(\ssp_{2r})$ is explicitly described as a polynomial algebra in the coefficients
$\{\sz_N^{(k)}\}_{k\geq 1}$ of the series $\sz_N(z)=1+\sum_{k\geq 1} \sz_N^{(k)}z^{-k}$ determined from:
\begin{equation}\label{eq:zenter C}
  T'(z-\kappa)T(z)= T(z)T'(z-\kappa)=\sz_N(z)\ID_N \, ,
\end{equation}
where in the present setup the matrix transposition~\eqref{eq:prime} should be redefined via:
\begin{equation}\label{eq:prime C}
  (X')_{ij}=\varepsilon_i \varepsilon_j X_{j'i'}
  \quad \mathrm{for\ any} \quad
  N\times N\ \mathrm{matrix}\ X \, .
\end{equation}

\medskip
\noindent
In the notations of Subsection~\ref{sssec: RTT-to-Drinfeld-D}, the analogue of
Theorem~\ref{thm:Dr=RTT-unshifted-D} still holds, explicitly:
\begin{equation}\label{eq:explicit extended identification C}
  \Upsilon_0 \colon \quad
  E_k(z)\mapsto
  \begin{cases}
    e_{k,k+1}(z) & \mbox{if } k<r \\
    \frac{e_{r,r+1}(z)}{2} & \mbox{if } k=r
  \end{cases} \, , \quad
  F_i(z)\mapsto f_{i+1,i}(z) \, , \quad D_j(z)\mapsto h_j(z)
\end{equation}
for all $i\leq r$, $j\leq r+1$. Hence, a natural analogue of Theorem~\ref{thm:JLM Main thm} holds
with $\Upsilon_0\circ \iota_0$ given~by:
\begin{equation}\label{eq:explicit identification C}
\begin{split}
  & \sE_i(z)\mapsto
    \begin{cases}
      e_{i,i+1}(z+\frac{i-1}{2}) & \mbox{if } i<r \\
      \frac{1}{2} e_{r,r+1}(z+\frac{r}{2}) & \mbox{if } i=r
    \end{cases} \, , \qquad
    \sF_i(z)\mapsto
    \begin{cases}
      f_{i+1,i}(z+\frac{i-1}{2}) & \mbox{if } i<r \\
      f_{r+1,r}(z+\frac{r}{2}) & \mbox{if } i=r
    \end{cases} \, , \\
  & \sH_i(z)\mapsto
    \begin{cases}
      h_i(z+\frac{i-1}{2})^{-1}h_{i+1}(z+\frac{i-1}{2}) & \mbox{if } i<r \\
      h_{r}(z+\frac{r}{2})^{-1}h_{r+1}(z+\frac{r}{2}) & \mbox{if } i=r
    \end{cases} \, .
\end{split}
\end{equation}
We note that our conventions are to those of~\cite{jlm1} as in type $D_r$,
see Remark~\ref{rmk:comparison to JLM} for details.

\medskip
\noindent
Accordingly, the algebra $X^\rtt(\ssp_{2r})$ is generated by the coefficients
of $\{h_j(z)\}_{j=1}^{r+1}$ as well as of:
\begin{equation}\label{eq:generating ef C}
  e_i(z)=\sum_{k\geq 1} e_i^{(k)}z^{-k}:=e_{i,i+1}(z) \, , \quad
  f_i(z)=\sum_{k\geq 1} f_i^{(k)}z^{-k}:=f_{i+1,i}(z) \, , \quad 1\leq i\leq r \, .
\end{equation}
We shall now record the explicit formulas for all other entries of the matrices $F(z), H(z), E(z)$
from~\eqref{eq:Gauss-D}--\eqref{eq:upper triangular}.
The following result is essentially due to~\cite{jlm1}\footnote{Note the missing summands
in the equalities from parts (g, m) as stated in~\cite{jlm1}.}:

\medskip

\begin{Lem}\label{lem:known type C}
(a) $h_{i'}(z)=\frac{1}{h_i(z+i-r-1)}\cdot \prod_{j=i+1}^{r-1} \frac{h_j(z+j-r-2)}{h_j(z+j-r-1)}
               \cdot h_r(z-2) h_{r+1}(z)$ for $1\leq i\leq r-1$.

\medskip
\noindent
(b) $e_{(i+1)',i'}(z)=-e_i(z+i-r-1)$ for $1\leq i\leq r-1$.

\medskip
\noindent
(c) $e_{i,j+1}(z)=-[e_{i,j}(z),e_j^{(1)}]$ for $1\leq i<j\leq r-1$.

\medskip
\noindent
(d) $e_{i,j'}(z)=[e_{i,(j+1)'}(z),e_j^{(1)}]$ for $1\leq i<j \leq r-1$.

\medskip
\noindent
(e) $e_{i,r'}(z)=-\frac{1}{2}[e_{i,r}(z),e_r^{(1)}]$ for $1\leq i\leq r-1$.

\medskip
\noindent
(f) $e_{i',j'}(z) = [e_{i',(j+1)'}(z),e_{j}^{(1)}]$ for $1\leq j\leq i-2 \leq r-2$.

\medskip
\noindent
(g) $e_{i,i'}(z) = [e_{i,(i+1)'}(z),e_{i}^{(1)}]-e_{i}(z)e_{i,(i+1)'}(z)$ for $1\leq i\leq r-1$.

\medskip
\noindent
(h) $f_{i',(i+1)'}(z)=-f_{i}(z+i-r-1)$ for $1\leq i\leq r-1$.

\medskip
\noindent
(i) $f_{j+1,i}(z)=-[f_j^{(1)},f_{j,i}(z)]$ for $1\leq i<j\leq r-1$.

\medskip
\noindent
(j) $f_{j',i}(z)=[f_j^{(1)},f_{(j+1)',i}(z)]$ for $1\leq i<j \leq r-1$.

\medskip
\noindent
(k) $f_{r',i}(z)=-\frac{1}{2}[f_r^{(1)},f_{r,i}(z)]$ for $1\leq i\leq r-1$.

\medskip
\noindent
(l) $f_{j',i'}(z) = [f_{j}^{(1)},f_{(j+1)',i'}(z)]$ for $1\leq j\leq i-2\leq r-2$.

\medskip
\noindent
(m) $f_{i',i}(z) = [f_{i}^{(1)},f_{(i+1)',i}(z)]-f_{(i+1)',i}(z)f_{i}(z)$ for $1\leq i\leq r-1$.
\end{Lem}

\medskip
\noindent
The remaining matrix coefficients of $E(z)$ and $F(z)$ are recovered via the following analogues
of Lemmas~\ref{lem:all-E-new} and~\ref{lem:all-F-new}:

\medskip

\begin{Lem}\label{lem:new type C}
(a) $e_{i,j'}(z) = [e_{i,(j+1)'}(z),e_{j}^{(1)}]$ for $1\leq j\leq i-2 \leq r-2$.

\medskip
\noindent
(b) $e_{i+1,i'}(z) = [e_{i+1,(i+1)'}(z),e_{i}^{(1)}] + e_{i}(z)e_{i+1,(i+1)'}(z)-e_{i,(i+1)'}(z)$ for $1\leq i\leq r-1$.

\medskip
\noindent
(c) $f_{j',i}(z) = [f_{j}^{(1)},f_{(j+1)',i}(z)]$ for $1\leq j\leq i-2 \leq r-2$.

\medskip
\noindent
(d) $f_{i',i+1}(z) = [f_{i}^{(1)},f_{(i+1)',i+1}(z)] + f_{(i+1)',i+1}(z)f_{i}(z)-f_{(i+1)',i}(z)$ for $1\leq i\leq r-1$.
\end{Lem}

\medskip

    %%%%%%%%%%%%%%%%%%%%%%%%%%%%%%%%%%%%%%%%%%%%%%%%%%%%%%%%%%%%%%%%%%%%%%%%%%%%%%%
    %%%%%%%%%%%%%%%%%%%%%%%%%%% Shifted story - Type C %%%%%%%%%%%%%%%%%%%%%%%%%%%%
    %%%%%%%%%%%%%%%%%%%%%%%%%%%%%%%%%%%%%%%%%%%%%%%%%%%%%%%%%%%%%%%%%%%%%%%%%%%%%%%

\subsection{Shifted story}
\label{ssec shifted story C}
\

We shall use the same \emph{extended} lattice $\Lambda^\vee$,
but $\{\hat{\alpha}^\vee_i\}_{i=1}^r$ of $\Lambda^\vee$ are now defined via:
\begin{equation}\label{eq:hat-alpha-vee type C}
  \hat{\alpha}^\vee_i=\epsilon^\vee_i-\epsilon^\vee_{i+1} \qquad \mathrm{for} \quad 1\leq i\leq r \, .
\end{equation}
We shall also use the same notation for the dual lattice
  $\Lambda=\bigoplus_{j=1}^{r+1} \BZ\epsilon_j=\bigoplus_{i=0}^{r} \BZ\varpi_i$
with
\begin{equation}\label{eq:varpis type C}
  \varpi_i=-\epsilon_{i+1}-\epsilon_{i+2}-\ldots-\epsilon_{r+1}
  \qquad \mathrm{for} \quad 0\leq i\leq r \, .
\end{equation}
For $\mu\in \Lambda$, define $\unl{d}=\{d_j\}_{j=1}^{r+1}\in \BZ^{r+1}, \unl{b}=\{b_i\}_{i=1}^{r}\in \BZ^{r}$
via~(\ref{extended D shifts},~\ref{extended D shifts 1}); so $b_i=d_i-d_{i+1}\ \forall i$.

\medskip
\noindent
The \emph{shifted extended Drinfeld Yangian of $\ssp_{2r}$}, denoted by $X_\mu(\ssp_{2r})$,
is defined similarly: it is generated by
  $\{E_i^{(k)},F_i^{(k)}\}_{1\leq i\leq r}^{k\geq 1}\cup \{D_i^{(k_i)}\}_{1\leq i\leq r+1}^{k_i\geq d_i+1}$
subject to~(\ref{eY0},~\ref{eY2.1}--\ref{eY7},~\ref{eY1-shifted},~\ref{eY23-C}).
Like in Lemma~\ref{identifying extended Yangians}, $X_\mu(\ssp_{2r})$ depends only
on the image of $\mu$ under~\eqref{non-ext coweight from ext}, up to an isomorphism.

\medskip
\noindent
For $\nu\in \bar{\Lambda}$, the \emph{shifted Drinfeld Yangian of $\ssp_{2r}$},
denoted by $Y_\nu(\ssp_{2r})$, is defined likewise. It is related to $X_\mu(\ssp_{2r})$
via a natural analogue of Proposition~\ref{prop:relating shifted yangians} with
$\iota_\mu\colon Y_{\bar{\mu}}(\ssp_{2r})\hookrightarrow X_\mu(\ssp_{2r})$ determined
by~\eqref{eq:iota-null explicitly C} and the central elements
$\{C_r^{(k)}\}_{k\geq d_r+d_{r+1}+1}$ of $X_\mu(\ssp_{2r})$ defined via:
\begin{equation}\label{eq:central C shifted type C}
  C_r(z)=z^{-d_r-d_{r+1}}\, + \sum_{k > d_r+d_{r+1}} C_r^{(k)}z^{-k}:=
  \prod_{i=1}^{r-1}\frac{D_i(z+i-r-2)}{D_i(z+i-r-1)} \cdot D_r(z-2) D_{r+1}(z) \, .
\end{equation}
The natural analogues of Corollary~\ref{cor:sub and quotient} and
Lemma~\ref{center of shifted yangians} still hold in the present setup.

\medskip
\noindent
We shall use the same notations~(\ref{eq:divisor def1})--(\ref{eq:lambda from divisor}) for
\emph{$\Lambda$-valued divisors $D$ on $\BP^1$, $\Lambda^+$-valued outside $\{\infty\}\in \BP^1$}.
The simple coroots $\{\alpha_i\}_{i=1}^r\subset \bar{\Lambda}$ of $\ssp_{2r}$ are explicitly given by:
\begin{equation}\label{eq:coroots C}
  \alpha_1=\epsilon_1-\epsilon_2 \, ,\ \ldots \, ,\
  \alpha_{r-2}=\epsilon_{r-2}-\epsilon_{r-1} \, ,\
  \alpha_{r-1}=\epsilon_{r-1}-\epsilon_r \, ,\ \alpha_r=\epsilon_r \, .
\end{equation}
We also consider $\{\hat{\alpha}_i\}_{i=1}^r\subset \Lambda$, which are
the ``lifts'' of $\{\alpha_i\}$ from~\eqref{eq:coroots C} in the sense of~\eqref{eq:lift}:
\begin{equation}\label{eq:lifted coroots C}
  \hat{\alpha}_1=\epsilon_1-\epsilon_2 \, ,\ \ldots \, ,\
  \hat{\alpha}_{r-2}=\epsilon_{r-2}-\epsilon_{r-1}\, ,\
  \hat{\alpha}_{r-1}=\epsilon_{r-1}-\epsilon_{r}+\epsilon_{r+1} \, ,\
  \hat{\alpha}_r=\epsilon_{r}-\epsilon_{r+1} \, .
\end{equation}
From now on, we shall impose the following assumption on $D$ (cf.~\eqref{eq:assumption}):
\begin{equation}\label{eq:assumption type C}
  \textbf{Assumption}:\qquad
  \lambda+\mu=a_1\hat{\alpha}_1+\ldots+a_{r}\hat{\alpha}_{r}\quad \mathrm{with}\quad a_i\in \BN \, .
\end{equation}
The above coefficients $a_i$ in~\eqref{eq:assumption type C} are explicitly given by:
\begin{equation}\label{eq:a explicitly type C}
  a_i = (\epsilon^\vee_1+\ldots+\epsilon^\vee_i)(\lambda+\mu) \qquad \mathrm{for}\quad 1\leq i\leq r \, .
\end{equation}
Thus,~(\ref{eq:assumption type C}) is equivalent to $(\epsilon^\vee_r+\epsilon^\vee_{r+1})(\lambda+\mu)=0$
and $\sum_{k=1}^i\epsilon^\vee_k(\lambda+\mu)\in \BN$ for all $1\leq i\leq r$.

\medskip
\noindent
Consider the algebra $\CA$ defined as in~\eqref{algebra A} with the following important \underline{modification}:
\begin{equation}\label{eq:general pq relation}
  [e^{\pm q_{i,k}},p_{j,\ell}]=
  \mp \frac{(\alphavee_i,\alphavee_i)}{2}\delta_{i,j}\delta_{k,\ell} e^{\pm q_{i,k}} \, ,
\end{equation}
so that $[e^{\pm q_{r,k}},p_{r,k}]=\mp 2 e^{\pm q_{r,k}}$.
Then, as in Theorem~\ref{thm:extended homom D}, we have an algebra homomorphism
\begin{equation}\label{eq:homom psi type C}
  \Psi_D\colon X_{-\mu}(\ssp_{2r})\longrightarrow \CA \, ,
\end{equation}
determined by the following assignment (keeping the notations~(\ref{ZW-series},~\ref{eq:aP conventions})):
\begin{equation}\label{eq:homom assignment type C}
\begin{split}
   & E_i(z)\mapsto
     \begin{cases}
       \sum_{k=1}^{a_i}\frac{P_{i-1}(p_{i,k}-1)}{(z-p_{i,k})P_{i,k}(p_{i,k})} e^{q_{i,k}} & \mbox{if } i<r \\
       \sum_{k=1}^{a_{r}}\frac{P_{r-1}(p_{r,k}-1)P_{r-1}(p_{r,k}-2)}{2(z-p_{{r},k})P_{{r},k}(p_{{r},k})} e^{q_{r,k}} & \mbox{if } i=r
     \end{cases} \, , \\
   & F_i(z)\mapsto
     \begin{cases}
       -\sum_{k=1}^{a_i}\frac{Z_i(p_{i,k}+1)P_{i+1}(p_{i,k}+1)}{(z-p_{i,k}-1)P_{i,k}(p_{i,k})} e^{-q_{i,k}} & \mbox{if } i<r \\
       -\sum_{k=1}^{a_{r}}\frac{Z_{r}(p_{{r},k}+2)}{(z-p_{{r},k}-2)P_{{r},k}(p_{{r},k})} e^{-q_{r,k}} & \mbox{if } i=r
     \end{cases} \, , \\
   & D_i(z)\mapsto
     \begin{cases}
       \frac{P_i(z)}{P_{i-1}(z-1)} \cdot \prod_{k=0}^{i-1} Z_k(z) & \mbox{if } i\leq r \\
       \frac{P_{r-1}(z-2)}{P_{r}(z-2)} \cdot \prod_{k=0}^{r} Z_k(z) & \mbox{if } i=r+1
     \end{cases} \, .
\end{split}
\end{equation}
The proof is analogous to that of Theorem~\ref{thm:extended homom D} and is based on
the explicit formula
\begin{equation}\label{image of C-series type C}
  \Psi_D(C_r(z)) \, =\, \prod_{i=0}^{r-1} \Big(Z_i(z)Z_i(z+i-r-1)\Big)\cdot Z_r(z) \,
\end{equation}
as well as the comparison to the homomorphisms of~\cite{nw}. Precisely, identifying $\CA$ with
$\tilde{\CA}$ of~\emph{loc.cit.} and the points $x_s$ with the parameters $z_s$ of~\emph{loc.cit.}\ via:
\begin{equation*}
  p_{i,k}\leftrightarrow
  \begin{cases}
    w_{i,k}+\frac{i-1}{2} & \mbox{if } i<r\\
    w_{r,k}+\frac{r}{2} & \mbox{if } i=r
  \end{cases} \, , \qquad
  e^{\pm q_{i,k}}\leftrightarrow \sfu_{i,k}^{\mp 1} \, , \qquad
  x_s \leftrightarrow
  \begin{cases}
     z_s+\frac{i_s}{2} & \mbox{if } 1\leq i_s<r \\
     z_s+\frac{r+1}{2} & \mbox{if } i_s=r
  \end{cases} \, ,
\end{equation*}
the (restriction) composition
  $Y_{-\bar{\mu}}(\ssp_{2r})\xrightarrow{\iota_{-\mu}} X_{-\mu}(\ssp_{2r}) \xrightarrow{\Psi_D} \CA$
is given by the formulas~\eqref{eq:nw homom assignment} of Appendix~\ref{Appendix B: shuffle realization}
(applied to the type $C_r$ Dynkin diagram with the arrows pointing $i\to i+1$ for $1\leq i<r$),
which essentially coincide with the homomorphisms $\Phi^{\bar{\lambda}}_{-\bar{\mu}}$ of~\cite[Theorem 5.4]{nw}.

\medskip
\noindent
For $\mu\in \Lambda^+$, the \emph{antidominantly shifted extended RTT Yangian of $\ssp_{2r}$},
denoted by $X^\rtt_{-\mu}(\ssp_{2r})$, is defined similarly to $X^\rtt_{-\mu}(\sso_{2r})$:
it is generated by $\{t^{(k)}_{ij}\}_{1\leq i,j\leq 2r}^{k\in \BZ}$ subject to the
RTT relation~\eqref{eq:rtt} and the restriction~\eqref{eq:t-modes shifted} on the
matrix coefficients of the matrices $F(z), H(z), E(z)$ with $d'_i\in \BZ$ defined
as in~\eqref{eq:d-prime}. We note that $\mu\in \Lambda^+$ implies now the following inequalities:
\begin{equation}\label{eq:d-inequalities type C}
  d_1\geq d_2\geq \dots\geq d_{r-1}\geq d_r\geq d'_r \geq d'_{r-1}\geq \dots \geq d'_1 \, .
\end{equation}
One of our key results in the type $C_r$ is the natural analogue of Theorem~\ref{thm:Dr=RTT-shifted-D}:

\medskip

\begin{Thm}\label{thm:Dr=RTT-shifted-C}
For any $\mu\in \Lambda^+$, the assignment~(\ref{eq:explicit extended identification C})
gives rise to the algebra isomorphism
  $\Upsilon_{-\mu}\colon X_{-\mu}(\ssp_{2r})\iso X^\rtt_{-\mu}(\ssp_{2r})$.
\end{Thm}

\medskip
\noindent
Similarly to the type $D_r$, the assignment~\eqref{eq:RTT coproduct} gives rise to the coproduct homomorphisms
\begin{equation*}
  \Delta^\rtt_{-\mu_1,-\mu_2}\colon X^\rtt_{-\mu_1-\mu_2}(\ssp_{2r})\longrightarrow
  X^\rtt_{-\mu_1}(\ssp_{2r})\otimes X^\rtt_{-\mu_2}(\ssp_{2r})
\end{equation*}
for any $\mu_1,\mu_2\in \Lambda^+$,
coassociative in the sense of Corollary~\ref{cor:coassociativity}.
Evoking the isomorphism of Theorem~\ref{thm:Dr=RTT-shifted-C} and the algebra embedding
$\iota_{\mu}\colon Y_{\bar{\mu}}(\ssp_{2r}) \hookrightarrow X_{\mu}(\ssp_{2r})$,
we obtain the coproduct homomorphisms
\begin{equation}\label{fkprw homom type C}
  \Delta_{-\nu_1,-\nu_2}\colon Y_{-\nu_1-\nu_2}(\ssp_{2r})\longrightarrow
  Y_{-\nu_1}(\ssp_{2r})\otimes Y_{-\nu_2}(\ssp_{2r})
\end{equation}
for any $\nu_1,\nu_2\in \bar{\Lambda}^+$. Explicitly, the homomorphism~\eqref{fkprw homom type C}
is uniquely determined by the formulas~\eqref{coproduct Yangian generators sl} with the root generators
$\{\sE_{\gamma^\vee}^{(1)},\sF_{\gamma^\vee}^{(1)}\}_{\gamma^\vee\in \Delta^+}$ defined via:
\begin{equation}\label{eq:higher root generators C part 1}
\begin{split}
  & \sE^{(1)}_{\epsilon^\vee_i - \epsilon^\vee_j} =
    [\sE^{(1)}_{j-1},[\sE^{(1)}_{j-2},[\sE^{(1)}_{j-3}, \,\cdots, [\sE^{(1)}_{i+1},\sE^{(1)}_i]\cdots]]] \, , \\
  & \sF^{(1)}_{\epsilon^\vee_i - \epsilon^\vee_j} =
    [[[\cdots[\sF_i^{(1)},\sF_{i+1}^{(1)}], \,\cdots, \sF^{(1)}_{j-3}],\sF^{(1)}_{j-2}],\sF_{j-1}^{(1)}] \, , \\
  & \qquad \qquad \qquad \qquad \qquad \qquad \qquad \qquad \qquad \qquad \qquad 1\leq i<j\leq r
\end{split}
\end{equation}
and
\begin{equation}\label{eq:higher root generators C part 2}
\begin{split}
  & \sE^{(1)}_{\epsilon^\vee_i + \epsilon^\vee_j} =
    -[\cdots[[\sE^{(1)}_{r-1},[\sE^{(1)}_{r-2},[\sE^{(1)}_{r-3}, \,\cdots, [\sE^{(1)}_{i+1},\sE^{(1)}_i]\cdots]]], \sE^{(1)}_{r}], \,\cdots, \sE^{(1)}_{j}] \, , \\
  & \sF^{(1)}_{\epsilon^\vee_i + \epsilon^\vee_j} =
    -2^{-\delta_{i,j}}
    [\sF^{(1)}_j, \,\cdots, [\sF^{(1)}_{r},[[[\cdots[\sF_i^{(1)},\sF_{i+1}^{(1)}], \,\cdots, \sF^{(1)}_{r-3}],\sF^{(1)}_{r-2}],\sF_{r-1}^{(1)}]]\cdots]\, ,\\
  & \qquad \qquad \qquad \qquad \qquad \qquad \qquad \qquad \qquad \qquad \qquad \qquad \qquad 1\leq i\leq j\leq r  \, ,
\end{split}
\end{equation}
where
  $\Delta^+=\Big\{\epsilon^\vee_i - \epsilon^\vee_j\Big\}_{1\leq i<j\leq r}\cup\
   \Big\{\epsilon^\vee_i + \epsilon^\vee_j\Big\}_{1\leq i\leq j\leq r}$
is the set of positive roots of $\ssp_{2r}$.

\medskip

\begin{Rem}\label{rmk:coproduct coincidence C}
As our formulas~\eqref{coproduct Yangian generators sl} coincide with those of~\cite[Theorem 4.8]{fkp},
this provides a confirmative answer to the question raised in the end of~\cite[\S8]{cgy}, in the type $C_r$.
\end{Rem}

\medskip

    %%%%%%%%%%%%%%%%%%%%%%%%%%%%%%%%%%%%%%%%%%%%%%%%%%%%%%%%%%%%%%%%%%%%%%%%%%%%%%%
    %%%%%%%%%%%%%%%%%%%%%%%%%%%%% Lax Matrices - Type C %%%%%%%%%%%%%%%%%%%%%%%%%%%
    %%%%%%%%%%%%%%%%%%%%%%%%%%%%%%%%%%%%%%%%%%%%%%%%%%%%%%%%%%%%%%%%%%%%%%%%%%%%%%%

\subsection{Lax matrices}
\label{ssec Lax matrrices C}
\

Similar to the type $D_r$, the proof of Theorem~\ref{thm:Dr=RTT-shifted-C} goes through the
faithfulness result of~\cite{w}, see Theorem~\ref{thm:alex's theorem}, and the construction
of the Lax matrices $T_D(z)$. To this end, for any $\Lambda^+$-valued divisor $D$ on $\BP^1$
satisfying~(\ref{eq:assumption type C}), we construct the matrix $T_D(z)$
via~(\ref{eq:Lax definition},~\ref{eq:Lax explicit}) with the matrix coefficients
$f^D_{j,i}(z),e^D_{i,j}(z),h^D_i(z)$ obtained from the explicit formulas~\eqref{eq:homom assignment type C}
combined with Lemmas~\ref{lem:known type C},~\ref{lem:new type C}.
Using the ``normalized limit'' procedure~\eqref{eq:limit relation}, we conclude
that Corollary~\ref{cor:rational Lax via unshifted} applies in the present setup.
Combining this with $\Upsilon_0$ being an isomorphism~\cite{jlm1}, we conclude
(as in Proposition~\ref{prop:preserving RTT}) that $T_D(z)$ are Lax (of type~$C_r$).

\noindent
Similarly to Proposition~\ref{prop:crossymmetry D}, the matrix $T(z)$ (encoding all generators
of $X^\rtt_{-\mu}(\ssp_{2r})$) still satisfies the \emph{crossing} relation~\eqref{eq:crossymetry}
with the central series $\sz_N(z)$ defined via (cf.~\eqref{eq:shifted z-central}):
\begin{equation}\label{eq:shifted z-central C}
  \sz_N(z) = \Upsilon_{-\mu}(C_r(z)) =
  \prod_{i=1}^{r-1}\frac{h_i(z+i-r-2)}{h_i(z+i-r-1)} \cdot h_r(z-2) h_{r+1}(z) \, .
\end{equation}
In Appendix~\ref{Appendix B: shuffle realization} (see Theorem~\ref{thm:shuffle homomorphism},
Lemma~\ref{lem:EF explicit shuffle C}), we use the \emph{shuffle algebra} approach to derive
the uniform formulas for the matrix coefficients $e_{i,j}^D(z), f^D_{j,i}(z)$, which are rather
inaccessible if derived iteratively via Lemmas~\ref{lem:known type C},~\ref{lem:new type C}.
This allows to prove the analogue of Theorem~\ref{Main Theorem 1}:

\medskip

\begin{Thm}\label{thm: regularity C}
The Lax matrix $\sT_D(z)=\frac{T_D(z)}{Z_0(z)}$ is regular in $z$, i.e.\
$\sT_D(z)\in \CA[z]\otimes_\BC \End\, \BC^{2r}$.
\end{Thm}

\medskip
\noindent
Similar to type $D_r$, the result above provides a shortcut to the computation of the Lax matrices
$T_D(z)$ defined, in general, as a product of three complicated matrices $F^{D}(z),H^D(z),E^D(z)$.
In particular, the natural analogue of Proposition~\ref{prop:linear Lax D} holds. To this end,
let us now describe all $\Lambda^+$-valued divisors $D$ on $\BP^1$ satisfying~(\ref{eq:assumption type C})
such that $\deg_z \sT_D(z)=1$. Define $\lambda,\mu\in \Lambda^+$ via~(\ref{eq:mu from divisor},~\ref{eq:lambda from divisor}),
so that $\lambda+\mu=\sum_{j=0}^r b_j\varpi_j$ with $b_0\in\BZ,\ b_1,\ldots,b_r\in \BN$.
Then, the assumption~\eqref{eq:assumption type C} implies that the corresponding coefficients
$a_i\in \BN$ are related to $b_j$'s via:
\begin{equation}\label{eq:a via b type C}
  a_i = b_1+2b_2+\ldots+(i-1)b_{i-1}+i(b_i+\ldots+b_{r-1})+\frac{i}{2}b_r
  \, , \qquad 1\leq i\leq r \, ,
\end{equation}
and $b_0=-b_1-\ldots-b_{r-1}-\frac{b_r}{2}$, which uniquely recovers $b_0$ in terms of
$b_1,\ldots,b_r$ and forces $b_r$ to be even. We also note that the total number of pairs
of $(p,q)$-oscillators in $\CA$ equals:
\begin{equation}\label{eq:total osc C}
  \sum_{i=1}^{r} a_i = \sum_{k=1}^{r-1} \frac{k(2r-k+1)}{2} b_k+\frac{r(r+1)}{4}b_r \, .
\end{equation}
Combining the above formulas~\eqref{eq:a via b type C} with Proposition~\ref{prop:linear Lax D}(a)
in type $C_r$, we thus conclude that the normalized Lax matrix $\sT_D(z)$ is linear only for the
following configurations of $b_i$'s:
\begin{enumerate}
  \item[(1)] $b_0=-1,\ b_j=1,\ b_1=\ldots=b_{j-1}=b_{j+1}=\ldots=b_r=0$ for some $1\leq j\leq r-1$,

  \item[(2)] $b_0=-1,\ b_1=\ldots=b_{r-1}=0,\ b_r=2$.
\end{enumerate}
As $b_0$ is uniquely determined via $b_1,\ldots,b_r$ and does not affect the Lax matrix $\sT_D(z)$,
we shall rather focus on the corresponding values of the dominant
$\ssp_{2r}$-coweights $\bar{\lambda},\bar{\mu}\in \bar{\Lambda}^+$.

\medskip
\noindent
$\bullet$ \emph{\underline{Case (1)}:} $\bar{\lambda}+\bar{\mu}=\omega_j$ for $1\leq j\leq r-1$.

In this case, $a_1=1,\ldots,a_{j-1}=j-1,a_j=\ldots=a_{r}=j$, so that the total number of pairs
of $(p,q)$-oscillators is $\frac{j(2r-j+1)}{2}$, see~(\ref{eq:a via b type C},~\ref{eq:total osc C}).
We obtain two Lax matrices: the \emph{non-degenerate} one, depending on the additional parameter $x\in \BC$
(but independent of the parameter $y\in \BP^1$):
\begin{equation}\label{eq:nondeg C1}
  \sT_{\varpi_j[x]-\varpi_0[y]}(z)=z(E_{11}+\ldots+E_{2r,2r})+O(1) \, ,
\end{equation}
and its normalized limit as $x\to \infty$, which is \emph{degenerate} with $z$ in the first $j$ diagonal entries:
\begin{equation}\label{eq:deg C1}
  \sT_{\varpi_j[\infty]-\varpi_0[y]}(z)=z(E_{11}+\ldots+E_{jj})+O(1)
\end{equation}
and also satisfying:
\begin{equation*}
   \sT_{\varpi_j[\infty]-\varpi_0[y]}(z)_{k,k}=
   \begin{cases}
     1 & \mbox{if } j+1\leq k\leq (j+1)' \\
     0 & \mbox{if } j'\leq k\leq 1'
   \end{cases} \, .
\end{equation*}

\medskip
\noindent
$\bullet$ \emph{\underline{Case (2)}:} $\bar{\lambda}+\bar{\mu}=2\omega_r$.
\

\noindent
In this case, we have $a_1=1,\ldots,a_r=r$, and the total number of pairs of $(p,q)$-oscillators
is $\frac{r(r+1)}{2}$, see~(\ref{eq:a via b type C},~\ref{eq:total osc C}).
We thus obtain three Lax matrices: the \emph{non-degenerate} one, depending in a symmetric way
on additional parameters $x_1,x_2\in \BC$ (but independent of $y\in \BP^1$):
\begin{equation}\label{eq:nondeg C2}
  \sT_{\varpi_r([x_1]+[x_2])-\varpi_0[y]}(z)=z(E_{11}+\ldots+E_{2r,2r})+O(1) \, ,
\end{equation}
its normalized limit as $x_2\to \infty$, which is \emph{degenerate} with $z$ only in the first $r$ diagonal entries:
\begin{equation}\label{eq:deg C2.1}
  \sT_{\varpi_r[x_1]+\varpi_r[\infty]-\varpi_0[y]}(z)=z(E_{11}+\ldots+E_{rr})+O(1) \, ,
\end{equation}
and its further $x_1\to \infty$ normalized limit, which also contains $z$ only in $r$ of its diagonal entries:
\begin{equation}\label{eq:deg C2.2}
  \sT_{2\varpi_r[\infty]-\varpi_0[y]}(z)=z(E_{11}+\ldots+E_{rr})+O(1) \, .
\end{equation}
The diagonal $z$-independent entries of these degenerate Lax matrices are explicitly given by:
\begin{equation}\label{eq:deg C2.12}
   \sT_{\varpi_r[x_1]+\varpi_r[\infty]-\varpi_0[y]}(z)_{k,k}=1 \,, \quad
   \sT_{2\varpi_r[\infty]-\varpi_0[y]}(z)_{k,k}=0 \qquad \mathrm{for}\quad r'\leq k\leq 1' \, .
\end{equation}

\medskip
\noindent
Completely analogously to Proposition~\ref{prop:property linear D}, we have the following
\emph{unitarity} property of the corresponding non-degenerate linear Lax matrices
(recall the parameter $\kappa=r+1$, see~\eqref{eq:varepsilons}):

\medskip

\begin{Prop}\label{prop:property linear C}
The non-degenerate Lax matrices
  $L_\jmath(z):=\sT_{\varpi_\jmath[x]-\varpi_0[y]}\left(z+x+\frac{\kappa-\jmath}{2}\right)$
for $1\leq \jmath<r$, as well as
  $L_r(z):=\sT_{\varpi_r([x_1]+[x_2])-\varpi_0[\infty]}\left(z+\frac{x_1+x_2}{2}\right)$,
are unitary:
\begin{equation*}
  L_\jmath(z)L_\jmath(-z)=\Big[\Big(\frac{\kappa-\jmath}{2}\Big)^2 - \, z^2\Big]\ID_N \, ,\qquad
  L_r(z)L_r(-z)=\Big[\Big(\frac{x_1-x_2}{2}\Big)^2 - \, z^2\Big]\ID_N \, .
\end{equation*}
\end{Prop}

\medskip
\noindent
We conclude this Section by presenting a few interesting examples of the Lax matrices $\sT_D(z)$.

   %%%%%%%%%%%%%%%%%%%%%%%%%%%%%%%%%%%%%%%%%%%%%%%%%%%%%%%%%%%%%%%%%%%%%%%%%%%%%%%%%%%%%%%
   %%%%%%%%%%%%%%%%%%%%%%%%%%%%% Example 1 for Lax C-type %%%%%%%%%%%%%%%%%%%%%%%%%%%%%%%%
   %%%%%%%%%%%%%%%%%%%%%%%%%%%%%%%%%%%%%%%%%%%%%%%%%%%%%%%%%%%%%%%%%%%%%%%%%%%%%%%%%%%%%%%

\medskip
\noindent
$\bullet$ \emph{\underline{Example 1}:} $D=\varpi_1[\infty]-\varpi_0[y]$
(note that $\sT_D(z)$ is independent of $y\in \BP^1$, as before).
\

In this case, $a_1=\ldots=a_r=1$. To simplify our notations, let us relabel
$\{p_{i,1},e^{\pm q_{i,1}}\}_{i=1}^{r}$ by $\{p_i,e^{\pm q_i}\}_{i=1}^{r}$,
so that
  $[p_i,e^{q_j}]=\delta_{i,j}\cdot
   \begin{cases}
     e^{q_i} & \mbox{if } i<r \\
     2e^{q_r} & \mbox{if } i=r
   \end{cases}\, .$
Then, we immediately find:
\begin{equation}\label{Matrix Example C1}
  \sT_D(z)=
  \begin{pmatrix}
  z-p_1 & * & * & \cdots & * & * \\
  * & 1 & 0 & \cdots & 0 & 0\\
  * & 0 & 1 & \cdots & 0 & 0\\
  \vdots &  \vdots & \vdots & \ddots & \vdots & \vdots\\
  * & 0 & 0 & \cdots & 1 & 0\\
  * & 0 & 0 & \cdots & 0 & 0
\end{pmatrix}
\end{equation}
with the nontrivial entries, marked by $*$ above, explicitly given by:
\begin{equation*}
\begin{split}
  & \sT_D(z)_{1,j}=(-1)^{j}e^{\sum_{k=1}^{j-1} q_k}\, ,\
    \sT_D(z)_{1,j'}=(-1)^{j+1}(p_j-p_{j-1}-1)e^{\sum_{k=1}^r q_k + \sum_{k=j}^{r-1} q_k}\, ,\ 1<j\leq r\, , \\
  & \sT_D(z)_{j,1}=(-1)^{j}(p_j-p_{j-1}-1)e^{-\sum_{k=1}^{j-1} q_k}\, ,\
    \sT_D(z)_{j',1}=(-1)^{j}e^{-\sum_{k=1}^{r} q_k - \sum_{k=j}^{r-1} q_k}\, ,\ 1<j\leq r\, ,\\
  & \sT_D(z)_{1,1'}=e^{2\sum_{k=1}^{r-1}q_k+q_r}\, ,\ \sT_D(z)_{1',1}=-e^{-2\sum_{k=1}^{r-1} q_k - q_r}\, .
\end{split}
\end{equation*}
These entries may be written more invariantly as:
\begin{equation}\label{eq:C1-equivalent}
\begin{split}
  & \sT_D(z)_{1,j}=-\oad_1\oa_j \, ,\quad
    \sT_D(z)_{1,j'}=\oad_1\oad_j \, ,\quad 1<j\leq r \,, \\
  & \sT_D(z)_{j,1}=-\oad_1^{-1}\oad_j\, ,\quad
    \sT_D(z)_{j',1}=-\oad_1^{-1}\oa_j\, ,\quad 1<j\leq r\, ,\\
  & \sT_D(z)_{1,1'}=\oad_1^2 \, ,\quad \sT_D(z)_{1',1}=-\oad_1^{-2}\, ,\quad
    \sT_D(z)_{1,1}=z-1-\oad_1\oa_1 \, ,
\end{split}
\end{equation}
where we used the following canonical transformation
($[\oa_i,\oad_j]=\delta_{i,j}$, $[\oa_i,\oa_j]=0=[\oad_i,\oad_j]$):
\begin{equation}\label{eq:a-coordinates C1}
\begin{split}
  & \oad_1=-e^{\sum_{k=1}^{r-1} q_k + \frac{1}{2}q_r}\, ,\quad
    \oad_i=(-1)^i(p_i-p_{i-1}-1)e^{\sum_{k=i}^{r-1}q_k+\frac{1}{2}q_r}\, ,\quad 1<i\leq r\, ,\\
  & \oa_1=-p_1e^{-\sum_{k=1}^{r-1} q_k - \frac{1}{2}q_r}\, ,\quad
    \oa_i=(-1)^ie^{-\sum_{k=i}^{r-1} q_k - \frac{1}{2}q_r}\, ,\quad 1<i\leq r\, .
\end{split}
\end{equation}

   %%%%%%%%%%%%%%%%%%%%%%%%%%%%%%%%%%%%%%%%%%%%%%%%%%%%%%%%%%%%%%%%%%%%%%%%%%%%%%%%%%%%%%%
   %%%%%%%%%%%%%%%%%%%%%%%%%%%%% Example 2 for Lax C-type %%%%%%%%%%%%%%%%%%%%%%%%%%%%%%%%
   %%%%%%%%%%%%%%%%%%%%%%%%%%%%%%%%%%%%%%%%%%%%%%%%%%%%%%%%%%%%%%%%%%%%%%%%%%%%%%%%%%%%%%%

\medskip
\noindent
$\bullet$ \emph{\underline{Example 2}:} $D=\varpi_1[x]-\varpi_0[y]$ with $x\in \BC$
($\sT_D(z)$ is independent of $y\in \BP^1$, as before).
\

As in the previous example, we have $a_1=\ldots=a_r=1$, and we shall use $\{p_i,e^{\pm q_i}\}_{i=1}^{r}$
instead of $\{p_{i,1},e^{\pm q_{i,1}}\}_{i=1}^{r}$. Then, the corresponding Lax matrix becomes:
\begin{equation}\label{Matrix Example C2}
  \sT_D(z)=
  (z-x-1)\ID_{2r} +
  (\oad_{1},\ldots,\oad_{r},\oa_{r},\ldots,\oa_{1})^t\cdot (-\oa_{1},\ldots,-\oa_{r},\oad_{r},\ldots,\oad_{1})
\end{equation}
after the following canonical transformation ($[\oa_i,\oad_j]=\delta_{i,j}$, $[\oa_i,\oa_j]=0=[\oad_i,\oad_j]$):
\begin{equation}\label{eq:a-coordinates C2}
\begin{split}
  & \oad_1=-e^{\sum_{k=1}^{r-1} q_k + \frac{1}{2}q_r}\, ,\quad
    \oad_i=(-1)^i(p_i-p_{i-1}-1)e^{\sum_{k=i}^{r-1}q_k+\frac{1}{2}q_r}\, ,\quad 1<i\leq r\, ,\\
  & \oa_1=-(p_1-x)e^{-\sum_{k=1}^{r-1} q_k - \frac{1}{2}q_r}\, ,\quad
    \oa_i=(-1)^ie^{-\sum_{k=i}^{r-1} q_k - \frac{1}{2}q_r}\, ,\quad 1<i\leq r\, .
\end{split}
\end{equation}
The type $C_r$ Lax matrix of the form~\eqref{Matrix Example C2} first appeared in~\cite[(4.34)]{ikk}.

\medskip

\begin{Rem}\label{rem normalized limit C-pol}
In contrast to the natural ``normalized limit'' relation~\eqref{eq:limit relation} in the
$(p,q)$-oscillators, such construction in the ``polynomial'' $(\oa,\oad)$-oscillators is more involved.
In particular, to recover the Lax matrix~\eqref{eq:C1-equivalent} from~\eqref{Matrix Example C2},
one should first apply the canonical transformation $\oa_1\rightsquigarrow \oa_1-x\oad_1^{-1}$
(preserving all other $\oa,\oad$-oscillators), and only afterwards consider the $x\to \infty$ limit
of the product on the left with the diagonal factor $\mathrm{diag}(1,-x^{-1},\ldots,-x^{-1},x^{-2})$.
\end{Rem}

   %%%%%%%%%%%%%%%%%%%%%%%%%%%%%%%%%%%%%%%%%%%%%%%%%%%%%%%%%%%%%%%%%%%%%%%%%%%%%%%%%%%%%%%
   %%%%%%%%%%%%%%%%%%%%%%%%%%%%% Example 3 for Lax C-type %%%%%%%%%%%%%%%%%%%%%%%%%%%%%%%%
   %%%%%%%%%%%%%%%%%%%%%%%%%%%%%%%%%%%%%%%%%%%%%%%%%%%%%%%%%%%%%%%%%%%%%%%%%%%%%%%%%%%%%%%

\medskip
\noindent
$\bullet$ \emph{\underline{Example 3}:} $D=\varpi_r[x]+\varpi_r[\infty]-\varpi_0[y]$
with $x\in \BC$ ($\sT_D(z)$ is independent of $y\in \BP^1$).
\

In this case, $a_1=1,\ldots,a_r=r$. According to~(\ref{eq:deg C2.1},~\ref{eq:deg C2.12}), $\sT_D(z)$ has the block form:
\begin{equation}\label{Matrix Example C3}
  \sT_D(z)=
  \begin{pmatrix}
    z\ID_{r}+\sF & \sB \\
    \sC & \ID_{r}
  \end{pmatrix},
\end{equation}
where $\sB,\sC,\sF$ are $z$-independent $r\times r$ matrices.
The following properties of $\sB,\sC$ are established exactly as in Lemma~\ref{lem:BC-rel type D}:

\medskip

\begin{Lem}\label{lem:BC-rel typeC}
(a) The matrices $\sB$ and $\sC$ are symmetric with respect to their antidiagonals:
\begin{equation*}
  \sB_{ij}=\sB_{r+1-j,r+1-i}\, ,\qquad \sC_{ij}=\sC_{r+1-j,r+1-i}\, .
\end{equation*}

\noindent
(b) The matrix coefficients $\{\sB_{ij}\}_{i,j=1}^r$ of the matrix $\sB$ pairwise commute.

\medskip
\noindent
(c) The matrix coefficients $\{\sC_{ij}\}_{i,j=1}^r$ of the matrix $\sC$ pairwise commute.

\medskip
\noindent
(d) The commutation among the matrix coefficients of $\sB$ and $\sC$ is given by:
\begin{equation*}
  [\sB_{ij},\sC_{k\ell}]=\delta_{i,\ell}\delta_{j,k}+\delta_{i,r+1-k}\delta_{j,r+1-\ell} \, .
\end{equation*}
\end{Lem}

\medskip
\noindent
It will be convenient to relabel the matrices $\sB,\sC$ as $\ap,-\am$, respectively, cf.~\eqref{eq:ApAm D-type}:
\begin{equation*}%\label{eq:ApAm}
  \sB=\ap=
  \left(\begin{array}{cccc}
          \oad_{1,r'} & \cdots & \oad_{1,2'} & 2\oad_{1,1'}\\
          \vdots & \iddots & \iddots & \oad_{1,2'}\\
          \oad_{r-1,r'} & \iddots & \iddots & \vdots\\
          2\oad_{r,r'} & \oad_{r-1,r'} & \cdots & \oad_{1,r'}
  \end{array}\right)\, ,\
  -\sC=\am=
  \left(\begin{array}{cccc}
              \oa_{r',1} & \cdots & \oa_{r',r-1} & \oa_{r',r} \\
              \vdots & \iddots & \iddots & \oa_{r',r-1} \\
              \oa_{2',1} & \iddots & \iddots & \vdots \\
              \oa_{1',1} & \oa_{2',1} & \cdots & \oa_{r',1}
  \end{array}\right)\,
\end{equation*}
with the matrix coefficients satisfying the following relations:
\begin{equation}
  [\oa_{i',j},\oad_{k,\ell'}]=\delta_{i,\ell}\delta_{j,k}\, ,\quad
  [\oa_{i',j},\oa_{k',\ell}]=0\, , \quad
  [\oad_{i,j'},\oad_{k,\ell'}]=0 \, ,
\end{equation}
due to Lemma~\ref{lem:BC-rel typeC}. Then, a tedious straightforward calculation yields:
\begin{equation}\label{Matrix Example C3 a-osc}
  \sT_D(z)=
  \left(\begin{BMAT}[5pt]{c:c}{c:c}
    (z+x)\ID_r-\ap\am & \ap\\
    -\am & \ID_r
  \end{BMAT}
\right)\, .
\end{equation}

\medskip
\noindent
We note that~\eqref{Matrix Example C3 a-osc} is the exact $\ssp_{2r}$-analogue of
the type $D_r$ Lax matrix of~(\ref{Matrix Example D1 a-osc},~\ref{Matrix Example D1 factorized}).

   %%%%%%%%%%%%%%%%%%%%%%%%%%%%%%%%%%%%%%%%%%%%%%%%%%%%%%%%%%%%%%%%%%%%%%%%%%%%%%%%%%%%%%%
   %%%%%%%%%%%%%%%%%%%%%%%%%%%%% Example 4 for Lax C-type %%%%%%%%%%%%%%%%%%%%%%%%%%%%%%%%
   %%%%%%%%%%%%%%%%%%%%%%%%%%%%%%%%%%%%%%%%%%%%%%%%%%%%%%%%%%%%%%%%%%%%%%%%%%%%%%%%%%%%%%%

\medskip
\noindent
$\bullet$ \emph{\underline{Example 4}:} $D=\varpi_r([x_1]+[x_2])-\varpi_0[y]$ with $x_1,x_2\in \BC$
($\sT_D(z)$ is independent of $y\in \BP^1$).
\

Applying the arguments of~\cite{f} (see~\cite{fkt} for more details) to the Lax matrix
of~\eqref{Matrix Example C3 a-osc} and keeping the same notations for the matrices $\ap,\am$,
we immediately obtain the following non-degenerate Lax matrix of type $C_r$ (cf.~\cite[(3.7)]{r}):
\begin{equation}\label{Matrix Example C4 a-osc}
\begin{split}
  & \mathcal{L}(z)=
    \left(\begin{BMAT}[5pt]{c:c}{c:c}
      (z+x_1) \ID_r -\ap_{}\am_{} & \ap_{}(x_2-x_1+\am_{}\ap_{}) \\
       -\am_{} & (z+x_2) \ID_r+\am_{}\ap_{}
    \end{BMAT} \right) =\\
  & \qquad \quad
    \left(\begin{BMAT}[5pt]{c:c}{c:c}
      \ID_r & \ap_{}\\
      0 & \ID_r
    \end{BMAT}\right)
    \left(\begin{BMAT}[5pt]{c:c}{c:c}
      (z+x_1) \ID_r & 0\\
      -\am_{} & (z+x_2)\ID_r
    \end{BMAT}\right)
    \left(\begin{BMAT}[5pt]{c:c}{c:c}
      \ID_r & -\ap_{}\\
      0 & \ID_r
    \end{BMAT}\right)\, .
\end{split}
\end{equation}
The type $C_r$ Lax matrix of the form~(\ref{Matrix Example C4 a-osc}) was recently discovered
in~\cite[\S6.2]{kk} and can be viewed as the exact $\ssp_{2r}$-analogue of the type $D_r$ Lax matrix
of~(\ref{Matrix Example D2 a-osc 2},~\ref{Matrix Example D2 a-osc 1}). We expect that $\mathcal{L}(z)$
is equivalent, up to a canonical transformation, to $\sT_{\varpi_r([x_1]+[x_2])-\varpi_0[y]}(z)$.

\medskip

    %%%%%%%%%%%%%%%%%%%%%%%%%%%%%%%%%%%%%%%%%%%%%%%%%%%%%%%%%%%%%%%%%%%%%%%%%%%%%%%
    %%%%%%%%%%%%%%%%%%%%%%%%%%%%%%%%%% Type B %%%%%%%%%%%%%%%%%%%%%%%%%%%%%%%%%%%%%
    %%%%%%%%%%%%%%%%%%%%%%%%%%%%%%%%%%%%%%%%%%%%%%%%%%%%%%%%%%%%%%%%%%%%%%%%%%%%%%%

\section{Type B}\label{sec Rational Lax matrices B-type}
\label{sec B-type}

The type $B_r$ is also quite similar to the type $D_r$, which we considered in details above.
Thus, we'll be brief, only stating the key results and highlighting the few technical differences.

    %%%%%%%%%%%%%%%%%%%%%%%%%%%%%%%%%%%%%%%%%%%%%%%%%%%%%%%%%%%%%%%%%%%%%%%%%%%%%%%
    %%%%%%%%%%%%%%%%%%%%%%%%%%% Ushifted story - Type B %%%%%%%%%%%%%%%%%%%%%%%%%%%
    %%%%%%%%%%%%%%%%%%%%%%%%%%%%%%%%%%%%%%%%%%%%%%%%%%%%%%%%%%%%%%%%%%%%%%%%%%%%%%%

\subsection{Classical (unshifted) story}
\label{ssec unshifted story B}
\

We shall realize the simple positive roots $\{\alphavee_i\}_{i=1}^r$
of the Lie algebra $\sso_{2r+1}$ in $\bar{\Lambda}^\vee$ via:
\begin{equation}\label{eq:alpha-vee B}
  \alphavee_1=\epsilon^\vee_1-\epsilon^\vee_2 \, ,\
  \alphavee_2=\epsilon^\vee_2-\epsilon^\vee_3 \, ,\ \ldots \, ,\
  \alphavee_{r-1}=\epsilon^\vee_{r-1}-\epsilon^\vee_r \, ,\
  \alphavee_r=\epsilon^\vee_r \, .
\end{equation}

\medskip
\noindent
The \emph{Drinfeld Yangian} of $\sso_{2r+1}$, denoted by $Y(\sso_{2r+1})$, is defined similarly
to $Y(\sso_{2r})$: it is generated by $\{\sE_i^{(k)},\sF_i^{(k)},\sH_i^{(k)}\}_{1\leq i\leq r}^{k\geq 1}$
subject to the relations~\eqref{Y0}--\eqref{Y7}, with $\alphavee_i$ of~\eqref{eq:alpha-vee B}.
The \emph{extended Drinfeld Yangian} of $\sso_{2r+1}$, denoted by $X(\sso_{2r+1})$,
is defined alike $X(\sso_{2r})$: it is generated by
  $\{E_i^{(k)},F_i^{(k)}\}_{1\leq i\leq r}^{k\geq 1}\cup \{D_i^{(k)}\}_{1\leq i\leq r+1}^{k\geq 1}$
subject to~(\ref{eY0})--(\ref{eY7}) with the modification:
\begin{equation}\label{eY23-B}
\begin{split}
  & [D_{r+1}(z), E_j(w)] =
    \begin{cases}
      -\frac{D_{r+1}(z)(E_r(z)-E_r(w))}{2(z-w)} + \frac{(E_r(z+1)-E_r(w))D_{r+1}(z)}{2(z-w+1)} & \mbox{if } j=r \\
      0 & \mbox{if } j<r
    \end{cases} \, ,\\
  & [D_{r+1}(z), F_j(w)] =
    \begin{cases}
      \frac{D_{r+1}(z)(F_r(z)-F_r(w))}{2(z-w)} - \frac{(F_r(z+1)-F_r(w))D_{r+1}(z)}{2(z-w+1)} & \mbox{if } j=r \\
      0 & \mbox{if } j<r
    \end{cases} \, .
\end{split}
\end{equation}
The central elements $\{C_r^{(k)}\}_{k\geq 1}$ of $X(\sso_{2r+1})$ are now defined via
(cf.~\eqref{eq:central C}):
\begin{equation}\label{eq:central C type B}
  C_r(z)=1+\sum_{k\geq 1} C_r^{(k)}z^{-k}:=
  \prod_{i=1}^{r}\frac{D_i(z+i-r-\tfrac{1}{2})}{D_i(z+i-r+\tfrac{1}{2})} \cdot D_{r+1}(z) D_{r+1}(z+\tfrac{1}{2}) \, .
\end{equation}
%%%%%%%%%%%%%%%%%%%%%%%%% Technical Comment %%%%%%%%%%%%%%%%%%%%%%%%%%%%%%%%%
% The verification of $[C_r(z),E_r(w)]=0$ is similar in spirit to types D,C %
% but is more tedious: the details are recorded in write-up #25             %
%%%%%%%%%%%%%%%%%%%%%%%%%%%%%%%%%%%%%%%%%%%%%%%%%%%%%%%%%%%%%%%%%%%%%%%%%%%%%
Also, a natural analogue of Lemma~\ref{lem:embedding} holds with
$\iota_0\colon Y(\sso_{2r+1})\hookrightarrow X(\sso_{2r+1})$ defined via:
\begin{equation}\label{eq:iota-null explicitly B}
\begin{split}
  & \sE_i(z)\mapsto E_{i}\left(z+\tfrac{i-1}{2}\right) \, ,\quad
    \sF_i(z)\mapsto F_{i}\left(z+\tfrac{i-1}{2}\right) \, ,\\
  & \sH_i(z)\mapsto D_i\left(z+\tfrac{i-1}{2}\right)^{-1}D_{i+1}\left(z+\tfrac{i-1}{2}\right) \, ,
    \quad \mathrm{for\ any} \quad 1\leq i\leq r \, .
\end{split}
\end{equation}

\medskip
\noindent
Define $N$ and $\kappa$ in the present setup via:
\begin{equation}\label{eq:kappa B}
  N=2r+1 \, , \qquad \kappa=r-\tfrac{1}{2} \, .
\end{equation}
The \emph{extended RTT Yangian} of $\sso_{2r+1}$, denoted by $X^\rtt(\sso_{2r+1})$, is
defined alike $X^\rtt(\sso_{2r})$: it is generated by $\{t_{ij}^{(k)}\}_{1\leq i,j\leq N}^{k\geq 1}$
subject to the RTT relation~\eqref{eq:rtt} with the $R$-matrix $R(z)$ given by~\eqref{eq:R-matrix}.
The \emph{RTT Yangian} of $\sso_{2r+1}$, denoted by $Y^\rtt(\sso_{2r+1})$, is defined similarly to
$Y^\rtt(\sso_{2r})$: it is the subalgebra of $X^\rtt(\sso_{2r+1})$ consisting of the elements stable
under the automorphisms~\eqref{eq:f-autom}. However, it can be also realized as a quotient of
$X^\rtt(\sso_{2r+1})$ as in~\eqref{eq:killing center}, due to the natural analogue of~\eqref{eq:extended to usual},
where the center $ZX^\rtt(\sso_{2r+1})$ of $X^\rtt(\sso_{2r+1})$ is explicitly described as a polynomial
algebra in the coefficients $\{\sz_N^{(k)}\}_{k\geq 1}$ of the series
$\sz_N(z)=1+\sum_{k\geq 1} \sz_N^{(k)}z^{-k}$ determined from (keeping the notations~\eqref{eq:prime}):
\begin{equation}\label{eq:zenter B}
  T'(z-\kappa)T(z)= T(z)T'(z-\kappa)=\sz_N(z)\ID_N \, .
\end{equation}

\medskip
\noindent
In the notations of Subsection~\ref{sssec: RTT-to-Drinfeld-D},
the analogue of Theorem~\ref{thm:Dr=RTT-unshifted-D} still holds, explicitly:
\begin{equation}\label{eq:explicit extended identification B}
  \Upsilon_0 \colon \quad
  E_i(z)\mapsto e_{i,i+1}(z) \,, \quad
  F_i(z)\mapsto f_{i+1,i}(z) \,, \quad
  D_j(z)\mapsto h_j(z) \,
\end{equation}
for all $i\leq r$, $j\leq r+1$. Hence, a natural analogue of Theorem~\ref{thm:JLM Main thm} holds
with $\Upsilon_0\circ \iota_0$ given~by:
\begin{equation}\label{eq:explicit identification B}
\begin{split}
  & \sE_i(z)\mapsto e_{i,i+1}(z+\tfrac{i-1}{2}) \, ,\quad
    \sF_i(z)\mapsto f_{i+1,i}(z+\tfrac{i-1}{2}) \, ,\\
  & \sH_i(z)\mapsto h_i(z+\tfrac{i-1}{2})^{-1}h_{i+1}(z+\tfrac{i-1}{2}) \, ,
    \quad \mathrm{for\ any}\quad 1\leq i\leq r \, .
\end{split}
\end{equation}
We note that our conventions are to those of~\cite{jlm1} as in type $D_r$,
see Remark~\ref{rmk:comparison to JLM} for details.

\medskip
\noindent
Accordingly, $X^\rtt(\sso_{2r+1})$ is generated by the coefficients of $\{h_j(z)\}_{j=1}^{r+1}$
as well as of:
\begin{equation}\label{eq:generating ef B}
  e_i(z)=\sum_{k\geq 1} e_i^{(k)}z^{-k}:=e_{i,i+1}(z) \, , \quad
  f_i(z)=\sum_{k\geq 1} f_i^{(k)}z^{-k}:=f_{i+1,i}(z) \, , \quad 1\leq i\leq r \, .
\end{equation}
We shall now record the explicit formulas for all other entries of the matrices $F(z), H(z), E(z)$.
The following result, the $B$ type analogue of Lemmas~\ref{lem:all-E-known} and~\ref{lem:all-F-known},
is essentially due to~\cite{jlm1}:

\medskip

\begin{Lem}\label{lem:known type B}
(a) $h_{i'}(z)=\frac{1}{h_i(z+i-r+\tfrac{1}{2})}\cdot \prod_{j=i+1}^{r} \frac{h_j(z+j-r-\tfrac{1}{2})}{h_j(z+j-r+\tfrac{1}{2})}
               \cdot h_{r+1}(z) h_{r+1}(z+\tfrac{1}{2})$ for $1\leq i\leq r$.

\medskip
\noindent
(b) $e_{(i+1)',i'}(z)=-e_i(z+i-r+\tfrac{1}{2})$ for $1\leq i\leq r$.

\medskip
\noindent
(c) $e_{i,j+1}(z)=-[e_{i,j}(z),e_j^{(1)}]$ for $1\leq i<j\leq r$.

\medskip
\noindent
(d) $e_{i,j'}(z)=[e_{i,(j+1)'}(z),e_j^{(1)}]$ for $1\leq i<j \leq r$.

\medskip
\noindent
(e) $e_{i',j'}(z) = [e_{i',(j+1)'}(z),e_{j}^{(1)}]$ for $1\leq j\leq i-2 \leq r-2$.

\medskip
\noindent
(f) $f_{i',(i+1)'}(z)=-f_{i}(z+i-r+\tfrac{1}{2})$ for $1\leq i\leq r$.

\medskip
\noindent
(g) $f_{j+1,i}(z)=-[f_j^{(1)},f_{j,i}(z)]$ for $1\leq i<j\leq r$.

\medskip
\noindent
(h) $f_{j',i}(z)=[f_j^{(1)},f_{(j+1)',i}(z)]$ for $1\leq i<j \leq r$.

\medskip
\noindent
(i) $f_{j',i'}(z) = [f_{j}^{(1)},f_{(j+1)',i'}(z)]$ for $1\leq j\leq i-2\leq r-2$.
\end{Lem}

\medskip
\noindent
The remaining matrix coefficients of $E(z)$ and $F(z)$ are recovered via the following
analogues of Lemmas~\ref{lem:all-E-new} and~\ref{lem:all-F-new}:

\medskip

\begin{Lem}\label{lem:new type B}
(a) $e_{i,i'}(z) = [e_{i,(i+1)'}(z),e_{i}^{(1)}]-e_{i}(z)e_{i,(i+1)'}(z)$ for $1\leq i\leq r$.

\medskip
\noindent
(b) $e_{i+1,i'}(z) = [e_{i+1,(i+1)'}(z),e_{i}^{(1)}] + e_{i}(z)e_{i+1,(i+1)'}(z)-e_{i,(i+1)'}(z)$
for $1\leq i\leq r-1$.

\medskip
\noindent
(c) $e_{i,j'}(z) = [e_{i,(j+1)'}(z),e_{j}^{(1)}]$ for $1\leq j\leq i-2 \leq r-1$.

\medskip
\noindent
(d) $f_{i',i}(z) = [f_{i}^{(1)},f_{(i+1)',i}(z)]-f_{(i+1)',i}(z)f_{i}(z)$ for $1\leq i\leq r$.

\medskip
\noindent
(e) $f_{i',i+1}(z) = [f_{i}^{(1)},f_{(i+1)',i+1}(z)] + f_{(i+1)',i+1}(z)f_{i}(z)-f_{(i+1)',i}(z)$
for $1\leq i\leq r-1$.

\medskip
\noindent
(f) $f_{j',i}(z) = [f_{j}^{(1)},f_{(j+1)',i}(z)]$ for $1\leq j\leq i-2 \leq r-1$.
\end{Lem}

\medskip

    %%%%%%%%%%%%%%%%%%%%%%%%%%%%%%%%%%%%%%%%%%%%%%%%%%%%%%%%%%%%%%%%%%%%%%%%%%%%%%%
    %%%%%%%%%%%%%%%%%%%%%%%%%%% Shifted story - Type B %%%%%%%%%%%%%%%%%%%%%%%%%%%%
    %%%%%%%%%%%%%%%%%%%%%%%%%%%%%%%%%%%%%%%%%%%%%%%%%%%%%%%%%%%%%%%%%%%%%%%%%%%%%%%

\subsection{Shifted story}
\label{ssec shifted story B}
\

We shall use the same \emph{extended} lattice $\Lambda^\vee$,
but $\{\hat{\alpha}^\vee_i\}_{i=1}^r$ of $\Lambda^\vee$ are now defined via:
\begin{equation}\label{eq:hat-alpha-vee type B}
  \hat{\alpha}^\vee_i=\epsilon^\vee_i-\epsilon^\vee_{i+1}
  \qquad \mathrm{for} \quad 1\leq i\leq r \, .
\end{equation}
We shall also use the same notation for the dual lattice
  $\Lambda=\bigoplus_{j=1}^{r+1} \BZ\epsilon_j=\bigoplus_{i=0}^{r} \BZ\varpi_i$
with
\begin{equation}\label{eq:varpis type B}
  \varpi_i=-\epsilon_{i+1}-\epsilon_{i+2}-\ldots-\epsilon_{r+1}
  \qquad \mathrm{for} \quad 0\leq i\leq r \, .
\end{equation}
For $\mu\in \Lambda$, define $\unl{d}=\{d_j\}_{j=1}^{r+1}\in \BZ^{r+1}, \unl{b}=\{b_i\}_{i=1}^{r}\in \BZ^{r}$
via~(\ref{extended D shifts},~\ref{extended D shifts 1}); so $b_i=d_i-d_{i+1}\ \forall i$.

\medskip
\noindent
The \emph{shifted extended Drinfeld Yangian of $\sso_{2r+1}$}, denoted by $X_\mu(\sso_{2r+1})$,
is defined similarly: it is generated by
  $\{E_i^{(k)},F_i^{(k)}\}_{1\leq i\leq r}^{k\geq 1}\cup \{D_i^{(k_i)}\}_{1\leq i\leq r+1}^{k_i\geq d_i+1}$
subject to~(\ref{eY0},~\ref{eY2.1}--\ref{eY7},~\ref{eY1-shifted},~\ref{eY23-B}).
Up to an isomorphism, $X_\mu(\sso_{2r+1})$ depends only on the image of $\mu$
under~\eqref{non-ext coweight from ext}, cf.~Lemma~\ref{identifying extended Yangians}.

\medskip
\noindent
For $\nu\in \bar{\Lambda}$, the \emph{shifted Drinfeld Yangian of $\sso_{2r+1}$},
denoted by $Y_\nu(\sso_{2r+1})$, is defined likewise. We note that a
natural analogue of Proposition~\ref{prop:relating shifted yangians} holds with
$\iota_\mu\colon Y_{\bar{\mu}}(\sso_{2r+1})\hookrightarrow X_\mu(\sso_{2r+1})$
determined by~\eqref{eq:iota-null explicitly B} and the central elements
$\{C_r^{(k)}\}_{k\geq 2d_{r+1}+1}$ of $X_\mu(\sso_{2r+1})$ defined via:
\begin{equation}\label{eq:central C shifted type B}
  C_r(z)=z^{-2d_{r+1}}\, + \sum_{k > 2d_{r+1}} C_r^{(k)}z^{-k}:=
  \prod_{i=1}^{r}\frac{D_i(z+i-r-\tfrac{1}{2})}{D_i(z+i-r+\tfrac{1}{2})} \cdot D_{r+1}(z) D_{r+1}(z+\tfrac{1}{2}) \, .
\end{equation}
The natural analogues of Corollary~\ref{cor:sub and quotient} and
Lemma~\ref{center of shifted yangians} still hold in the present setup.

\medskip
\noindent
We shall use the same notations~(\ref{eq:divisor def1})--(\ref{eq:lambda from divisor}) for
\emph{$\Lambda$-valued divisors $D$ on $\BP^1$, $\Lambda^+$-valued outside $\{\infty\}\in \BP^1$}.
The simple coroots $\{\alpha_i\}_{i=1}^r\subset \bar{\Lambda}$ of $\sso_{2r+1}$ are explicitly given by:
\begin{equation}\label{eq:coroots B}
  \alpha_1=\epsilon_1-\epsilon_2 \, ,\ \ldots \, ,\
  \alpha_{r-2}=\epsilon_{r-2}-\epsilon_{r-1} \, ,\
  \alpha_{r-1}=\epsilon_{r-1}-\epsilon_r \, ,\
  \alpha_r=2\epsilon_r \, .
\end{equation}
We also consider $\{\hat{\alpha}_i\}_{i=1}^r\subset \Lambda$, which are the ``lifts''
of $\{\alpha_i\}$ from~\eqref{eq:coroots B} in the sense of~\eqref{eq:lift}:
\begin{equation}\label{eq:extended coroots B}
  \hat{\alpha}_1=\epsilon_1-\epsilon_2 \, ,\ \ldots \, ,\
  \hat{\alpha}_{r-2}=\epsilon_{r-2}-\epsilon_{r-1} \, ,\
  \hat{\alpha}_{r-1}=\epsilon_{r-1}-\epsilon_r \, ,\
  \hat{\alpha}_r=2\epsilon_r \, .
\end{equation}
From now on, we shall impose the following assumption on $D$ (cf.~\eqref{eq:assumption}):
\begin{equation}\label{eq:assumption type B}
  \textbf{Assumption}:\qquad
  \lambda+\mu=a_1\hat{\alpha}_1+\ldots+a_{r}\hat{\alpha}_{r}\quad \mathrm{with}\quad a_i\in \BN \, .
\end{equation}
The above coefficients $a_i$ are explicitly given by:
\begin{equation}\label{eq:a explicitly type B}
\begin{split}
  & a_i = (\epsilon^\vee_1+\ldots+\epsilon^\vee_i)(\lambda+\mu)
    \quad \mathrm{for}\quad 1\leq i\leq r-1 \, ,\\
  & a_r = \frac{(\epsilon^\vee_1+\ldots+\epsilon^\vee_r)(\lambda+\mu)}{2} \, .
\end{split}
\end{equation}
Thus,~(\ref{eq:assumption type B}) is equivalent to $\epsilon^\vee_{r+1}(\lambda+\mu)=0$
and $2^{-\delta_{i,r}} \sum_{k=1}^i\epsilon^\vee_k(\lambda+\mu)\in \BN$ for all $1\leq i\leq r$.

\medskip
\noindent
Consider the algebra $\CA$ defined as in~\eqref{algebra A} but with the modified
relation~\eqref{eq:general pq relation} in place, so that
$[e^{\pm q_{r,k}},p_{r,k}]=\mp \frac{1}{2} e^{\pm q_{r,k}}$.
Then, as in Theorem~\ref{thm:extended homom D}, we have an algebra homomorphism
\begin{equation}\label{eq:homom psi type B}
  \Psi_D\colon X_{-\mu}(\sso_{2r+1})\longrightarrow \CA \, ,
\end{equation}
determined by the following assignment (keeping the notations~(\ref{ZW-series},~\ref{eq:aP conventions})):
\begin{equation}\label{eq:homom assignment type B}
\begin{split}
   & E_i(z)\mapsto
     2^{\delta_{i,r}}\cdot\,
     \sum_{k=1}^{a_i}\frac{P_{i-1}(p_{i,k}-1)}{(z-p_{i,k})P_{i,k}(p_{i,k})} e^{q_{i,k}}\, , \\
   & F_i(z)\mapsto
     \begin{cases}
       -\sum_{k=1}^{a_i}\frac{Z_i(p_{i,k}+1)P_{i+1}(p_{i,k}+1)}{(z-p_{i,k}-1)P_{i,k}(p_{i,k})} e^{-q_{i,k}} & \mbox{if } i\leq r-2 \\
       -\sum_{k=1}^{a_{r-1}}\frac{Z_{r-1}(p_{r-1,k}+1)P_{r}(p_{r-1,k}+1)P_{r}(p_{r-1,k}+\frac{3}{2})}{(z-p_{r-1,k}-1)P_{r-1,k}(p_{r-1,k})} e^{-q_{r-1,k}}
         & \mbox{if } i=r-1 \\
       -\sum_{k=1}^{a_{r}}\frac{Z_{r}(p_{{r},k}+\frac{1}{2})}{(z-p_{{r},k}-\frac{1}{2})P_{{r},k}(p_{{r},k})} e^{-q_{r,k}} & \mbox{if } i=r
     \end{cases} \, , \\
   & D_i(z)\mapsto
     \begin{cases}
       \frac{P_i(z)}{P_{i-1}(z-1)} \cdot \prod_{k=0}^{i-1} Z_k(z) & \mbox{if } i\leq r-1 \\
       \frac{P_r(z)P_r(z+\frac{1}{2})}{P_{r-1}(z-1)} \cdot \prod_{k=0}^{r-1} Z_k(z) & \mbox{if } i=r \\
       \frac{P_{r}(z+\frac{1}{2})}{P_{r}(z-\frac{1}{2})} \cdot \prod_{k=0}^{r} Z_k(z) & \mbox{if } i=r+1
     \end{cases} \, .
\end{split}
\end{equation}
The proof is analogous to that of Theorem~\ref{thm:extended homom D} and is based on the explicit formula
\begin{equation}\label{image of C-series type B}
  \Psi_D(C_r(z)) \, =\, \prod_{i=0}^{r} \Big(Z_i(z)Z_i(z+i-r+\tfrac{1}{2})\Big) \,
\end{equation}
as well as the comparison to the homomorphisms of~\cite{nw}. Precisely, identifying $\CA$ with
$\tilde{\CA}$ of~\emph{loc.cit.} and the points $x_s$ with the parameters $z_s$ of~\emph{loc.cit.}\ via:
\begin{equation*}
  p_{i,k}\leftrightarrow w_{i,k}+\frac{i-1}{2}\, , \qquad
  e^{\pm q_{i,k}}\leftrightarrow \sfu_{i,k}^{\mp 1} \, , \qquad
  x_s \leftrightarrow z_s+\frac{i_s}{2} \, ,
\end{equation*}
the (restriction) composition
  $Y_{-\bar{\mu}}(\sso_{2r+1})\xrightarrow{\iota_{-\mu}} X_{-\mu}(\sso_{2r+1}) \xrightarrow{\Psi_D} \CA$
is given by the formulas~\eqref{eq:nw homom assignment} of Appendix~\ref{Appendix B: shuffle realization}
(applied to the type $B_r$ Dynkin diagram with the arrows pointing $i\to i+1$ for $1\leq i<r$),
which essentially coincide with the homomorphisms $\Phi^{\bar{\lambda}}_{-\bar{\mu}}$ of~\cite{nw}.

\medskip
\noindent
The \emph{antidominantly shifted extended RTT Yangian of $\sso_{2r+1}$}, denoted by
$X^\rtt_{-\mu}(\sso_{2r+1})$ (with $\mu\in \Lambda^+$), is defined similarly to $X^\rtt_{-\mu}(\sso_{2r})$:
it is generated by $\{t^{(k)}_{ij}\}_{1\leq i,j\leq 2r+1}^{k\in \BZ}$ subject to the
RTT relation~\eqref{eq:rtt} and the restriction~\eqref{eq:t-modes shifted} on
the matrix coefficients of the matrices $F(z), H(z), E(z)$ with $d'_i\in \BZ$
defined in the present setup via:
\begin{equation}\label{eq:d-prime type B}
  d'_i:=2d_{r+1}-d_i \qquad \mathrm{for}\quad 1\leq i\leq r \, .
\end{equation}
We note that $\mu\in \Lambda^+$ implies now the following inequalities:
\begin{equation}\label{eq:d-inequalities type B}
  d_1\geq d_2\geq \dots\geq d_{r-1}\geq d_r \geq d_{r+1}\geq d'_r \geq d'_{r-1}\geq \dots \geq d'_1 \, .
\end{equation}
One of our key results in the type $B_r$ is the natural analogue of Theorem~\ref{thm:Dr=RTT-shifted-D}:

\medskip

\begin{Thm}\label{thm:Dr=RTT-shifted-B}
For any $\mu\in \Lambda^+$, the assignment~(\ref{eq:explicit extended identification B}) gives rise to
the algebra isomorphism $\Upsilon_{-\mu}\colon X_{-\mu}(\sso_{2r+1})\iso X^\rtt_{-\mu}(\sso_{2r+1})$.
\end{Thm}

\medskip
\noindent
Similarly to the type $D_r$, the assignment~\eqref{eq:RTT coproduct} gives rise to the coproduct homomorphisms
\begin{equation*}
  \Delta^\rtt_{-\mu_1,-\mu_2}\colon X^\rtt_{-\mu_1-\mu_2}(\sso_{2r+1})\longrightarrow
  X^\rtt_{-\mu_1}(\sso_{2r+1})\otimes X^\rtt_{-\mu_2}(\sso_{2r+1})
  \, \qquad \forall\, \mu_1,\mu_2\in \Lambda^+ \, ,
\end{equation*}
coassociative in the sense of Corollary~\ref{cor:coassociativity}.
Evoking the isomorphism of Theorem~\ref{thm:Dr=RTT-shifted-B} and the algebra embedding
  $\iota_{\mu}\colon Y_{\bar{\mu}}(\sso_{2r+1}) \hookrightarrow X_{\mu}(\sso_{2r+1})$,
we obtain the coproduct homomorphisms
\begin{equation}\label{fkprw homom type B}
  \Delta_{-\nu_1,-\nu_2}\colon Y_{-\nu_1-\nu_2}(\sso_{2r+1})\longrightarrow
  Y_{-\nu_1}(\sso_{2r+1})\otimes Y_{-\nu_2}(\sso_{2r+1})
\end{equation}
for any $\nu_1,\nu_2\in \bar{\Lambda}^+$. Explicitly, the homomorphism~\eqref{fkprw homom type B}
is uniquely determined by the formulas \eqref{coproduct Yangian generators sl} with the root generators
$\{\sE_{\gamma^\vee}^{(1)},\sF_{\gamma^\vee}^{(1)}\}_{\gamma\in \Delta^+}$ defined via:
\begin{equation}\label{eq:higher root generators B}
\begin{split}
  & \sE^{(1)}_{\epsilon^\vee_i - \epsilon^\vee_j} = [\sE^{(1)}_{j-1},[\sE^{(1)}_{j-2},[\sE^{(1)}_{j-3}, \,\cdots, [\sE^{(1)}_{i+1},\sE^{(1)}_i]\cdots]]] \, , \\
  & \sF^{(1)}_{\epsilon^\vee_i - \epsilon^\vee_j} = [[[\cdots[\sF_i^{(1)},\sF_{i+1}^{(1)}], \,\cdots, \sF^{(1)}_{j-3}],\sF^{(1)}_{j-2}],\sF_{j-1}^{(1)}] \, , \\
  & \sE^{(1)}_{\epsilon^\vee_i} = [\sE^{(1)}_{r},[\sE^{(1)}_{r-1},[\sE^{(1)}_{r-2}, \,\cdots, [\sE^{(1)}_{i+1},\sE^{(1)}_i]\cdots]]] \, , \\
  & \sF^{(1)}_{\epsilon^\vee_i} = [[[\cdots[\sF_i^{(1)},\sF_{i+1}^{(1)}], \,\cdots, \sF^{(1)}_{r-2}],\sF^{(1)}_{r-1}],\sF_{r}^{(1)}] \, , \\
  & \sE^{(1)}_{\epsilon^\vee_i + \epsilon^\vee_j} =
    [\cdots[[\sE^{(1)}_{r},[\sE^{(1)}_{r-1},[\sE^{(1)}_{r-2}, \,\cdots, [\sE^{(1)}_{i+1},\sE^{(1)}_i]\cdots]]], \sE^{(1)}_{r}], \,\cdots,\sE^{(1)}_{j}] \, , \\
  & \sF^{(1)}_{\epsilon^\vee_i + \epsilon^\vee_j} =
    [\sF^{(1)}_j, \,\cdots, [\sF^{(1)}_{r},[[[\cdots[\sF_i^{(1)},\sF_{i+1}^{(1)}], \,\cdots, \sF^{(1)}_{r-2}],\sF^{(1)}_{r-1}],\sF_{r}^{(1)}]]\cdots]\, ,\\
  & \qquad \qquad \qquad \qquad \qquad \qquad \qquad \qquad \qquad \qquad \qquad \qquad \qquad 1\leq i<j\leq r  \, ,
\end{split}
\end{equation}
where
  $\Delta^+=
   \Big\{\epsilon^\vee_i \pm \epsilon^\vee_j\Big\}_{1\leq i<j\leq r}\cup\
   \Big\{\epsilon^\vee_i\Big\}_{1\leq i\leq r}$
is the set of positive roots of $\sso_{2r+1}$.

\medskip

\begin{Rem}
We note that the last formula of~\eqref{coproduct Yangian generators} holds with the
following update of the formula~\eqref{eq:tilde-epsilon type D}:
\begin{equation}\label{eq:tilde-epsilon type B}
  \tilde{\epsilon}^\vee_j=\epsilon^\vee_j \quad \mathrm{for}\quad j\leq r\, ,\qquad
  \tilde{\epsilon}^\vee_{r+1}=0 \, .
\end{equation}
To this end, let us point out that the $j=r+1$ case of the last formula of~\eqref{coproduct Yangian generators}
is due to the equalities $e^{(1)}_{r+1,r+i}=-e^{(1)}_{r+2-i,r+1}$ and $f^{(1)}_{r+i,r+1}=-f^{(1)}_{r+1,r+2-i}$
which follow from~\eqref{eq:B-linear-skew}.
\end{Rem}

\medskip

\begin{Rem}\label{rmk:coproduct coincidence B}
As our formulas~\eqref{coproduct Yangian generators sl} coincide with those of~\cite[Theorem 4.8]{fkp},
this provides a confirmative answer to the question raised in the end of~\cite[\S8]{cgy}, in the type $B_r$.
\end{Rem}

\medskip

    %%%%%%%%%%%%%%%%%%%%%%%%%%%%%%%%%%%%%%%%%%%%%%%%%%%%%%%%%%%%%%%%%%%%%%%%%%%%%%%
    %%%%%%%%%%%%%%%%%%%%%%%%%%%%% Lax Matrices - Type B %%%%%%%%%%%%%%%%%%%%%%%%%%%
    %%%%%%%%%%%%%%%%%%%%%%%%%%%%%%%%%%%%%%%%%%%%%%%%%%%%%%%%%%%%%%%%%%%%%%%%%%%%%%%

\subsection{Lax matrices}
\label{ssec Lax matrrices B}
\

Similar to type $D_r$, the proof of Theorem~\ref{thm:Dr=RTT-shifted-B} goes through
the faithfulness result of~\cite{w}, see Theorem~\ref{thm:alex's theorem}, and the construction
of the Lax matrices $T_D(z)$. To this end, for any $\Lambda^+$-valued divisor $D$ on $\BP^1$
satisfying~(\ref{eq:assumption type B}), we construct the matrix $T_D(z)$
via~(\ref{eq:Lax definition},~\ref{eq:Lax explicit}) with the matrix coefficients
$f^D_{j,i}(z),e^D_{i,j}(z),h^D_i(z)$ obtained from the explicit formulas~\eqref{eq:homom assignment type B}
combined with Lemmas~\ref{lem:known type B},~\ref{lem:new type B}.
Using the same ``normalized limit'' procedure~\eqref{eq:limit relation}, we conclude
that Corollary~\ref{cor:rational Lax via unshifted} applies in the present setup.
Combining this with $\Upsilon_0$ being an isomorphism~\cite{jlm1}, we conclude
as in Proposition~\ref{prop:preserving RTT} that $T_D(z)$ are Lax (of type $B_r$).

\noindent
Similarly to Proposition~\ref{prop:crossymmetry D}, the matrix $T(z)$ (encoding all generators
of $X^\rtt_{-\mu}(\sso_{2r+1})$) still satisfies the \emph{crossing} relation~\eqref{eq:crossymetry}
with the central series $\sz_N(z)$ defined via (cf.~\eqref{eq:shifted z-central}):
\begin{equation}\label{eq:shifted z-central B}
  \sz_N(z) = \Upsilon_{-\mu}(C_r(z))=
  \prod_{i=1}^{r}\frac{h_i(z+i-r-\tfrac{1}{2})}{h_i(z+i-r+\tfrac{1}{2})} \cdot h_{r+1}(z) h_{r+1}(z+\tfrac{1}{2}) \, .
\end{equation}
In Appendix~\ref{Appendix B: shuffle realization} (see Theorem~\ref{thm:shuffle homomorphism},
Lemma~\ref{lem:EF explicit shuffle B}), we use the \emph{shuffle algebra} approach to derive
the uniform formulas for the matrix coefficients $e_{i,j}^D(z), f^D_{j,i}(z)$, which are
rather inaccessible if derived iteratively via Lemmas~\ref{lem:known type B},~\ref{lem:new type B}.
This allows to prove the analogue of Theorem~\ref{Main Theorem 1}:

\medskip

\begin{Thm}\label{thm: regularity B}
The Lax matrix $\sT_D(z)=\frac{T_D(z)}{Z_0(z)}$ is regular, i.e.~$\sT_D(z)\in \CA[z]\otimes_\BC \End\, \BC^{2r+1}$.
\end{Thm}

\medskip
\noindent
Similar to type $D_r$, the result above provides a shortcut to the computation of the Lax matrices
$T_D(z)$ defined, in general, as a product of three complicated matrices $F^{D}(z),H^D(z),E^D(z)$.
In particular, the natural analogue of Proposition~\ref{prop:linear Lax D} holds.
To this end, let us now describe all $\Lambda^+$-valued divisors $D$ on $\BP^1$
satisfying~(\ref{eq:assumption type B}) such that $\deg_z \sT_D(z)=1$. Define $\lambda,\mu\in \Lambda^+$
via~(\ref{eq:mu from divisor},~\ref{eq:lambda from divisor}), so that $\lambda+\mu=\sum_{j=0}^r b_j\varpi_j$
with $b_0\in\BZ,\ b_1,\ldots,b_r\in \BN$. Then, the assumption~\eqref{eq:assumption type B} implies
that the corresponding coefficients $a_i\in \BN$ are related to $b_j$'s via:
\begin{equation}\label{eq:a via b type B}
\begin{split}
  & a_i = b_1+2b_2+\ldots+(i-1)b_{i-1}+i(b_i+\ldots+b_{r}) \, , \qquad 1\leq i\leq r-1 \, ,\\
  & a_r = \frac{1}{2}\Big(b_1+2b_2+\ldots+(r-1)b_{r-1}+rb_{r}\Big) \, ,
\end{split}
\end{equation}
as well as $b_0=-b_1-\ldots-b_{r-1}-b_r$, which uniquely recovers $b_0$ in terms of $b_1,\ldots,b_r$.
We also note that the total number of pairs of $(p,q)$-oscillators in the algebra $\CA$ equals:
\begin{equation*}
  \sum_{i=1}^{r} a_i = \sum_{k=1}^{r} \frac{k(2r-k)}{2} b_k \, .
\end{equation*}
Combining the above formulas~\eqref{eq:a via b type B} with Proposition~\ref{prop:linear Lax D}(a)
in type $B_r$, we thus conclude that the normalized Lax matrix $\sT_D(z)$ is linear only for the
following configurations of $b_i$'s:
\begin{enumerate}
  \item[$\bullet$] $b_0=-1,\ b_j=1,\ b_1=\ldots=b_{j-1}=b_{j+1}=\ldots=b_r=0$ for an even $1\leq j\leq r$.
\end{enumerate}
As $b_0$ is uniquely determined via $b_1,\ldots,b_r$ and does not affect the Lax matrix $\sT_D(z)$,
we shall rather focus on the corresponding values of the dominant $\sso_{2r+1}$-coweights
$\bar{\lambda},\bar{\mu}\in \bar{\Lambda}^+$.

\medskip
\noindent
In the above case of $\bar{\lambda}+\bar{\mu}=\omega_j$,
we have $a_1=1,\ldots,a_{j-1}=j-1,a_j=\ldots=a_{r-1}=j,a_r=\frac{j}{2}$, and the total of $\frac{j(2r-j)}{2}$ pairs
of $(p,q)$-oscillators. We obtain two Lax matrices: the \emph{non-degenerate} one,
depending on the additional parameter $x\in \BC$ (but independent of $y\in \BP^1$):
\begin{equation}\label{eq:nondeg B1}
  \sT_{\varpi_j[x]-\varpi_0[y]}(z)=z(E_{11}+\ldots+E_{2r+1,2r+1})+O(1) \, ,
\end{equation}
and its normalized limit as $x\to \infty$, which is \emph{degenerate} with $z$ in the first $j$ diagonal entries:
\begin{equation}\label{eq:deg B1}
  \sT_{\varpi_j[\infty]-\varpi_0[y]}(z)=z(E_{11}+\ldots+E_{jj})+O(1)
\end{equation}
and also satisfying:
\begin{equation*}
   \sT_{\varpi_j[\infty]-\varpi_0[y]}(z)_{k,k}=
   \begin{cases}
     1 & \mbox{if } j+1\leq k\leq (j+1)' \\
     0 & \mbox{if } j'\leq k\leq 1'
   \end{cases} \, .
\end{equation*}

\medskip
\noindent
Completely analogously to Proposition~\ref{prop:property linear D}, we have the following
\emph{unitarity} property of the corresponding non-degenerate linear Lax matrices
(recall the parameter $\kappa=r-\frac{1}{2}$, see~\eqref{eq:kappa B}):

\medskip

\begin{Prop}\label{prop:property linear B}
For any even $1\leq \jmath\leq r$, the corresponding linear non-degenerate Lax matrix
  $L_\jmath(z):=\sT_{\varpi_\jmath[x]-\varpi_0[y]}\left(z+x+\frac{\kappa-\jmath}{2}\right)$
is unitary:
\begin{equation*}
  L_\jmath(z)L_\jmath(-z)=\Big[\Big(\frac{\kappa-\jmath}{2}\Big)^2 - \, z^2\Big] \ID_N \, .
\end{equation*}
\end{Prop}

\medskip
\noindent
Motivated by the Examples~3 and~4 in type $D_r$, we expect that
$\sT_{\varpi_1([x_1]+[x_2])-\varpi_0([y_1]+[y_2])}(z)$ and its normalized limit
$\sT_{\varpi_1([x]+[\infty])-\varpi_0([y_1]+[y_2])}(z)$ are equivalent, up to
canonical transformations, to the type $B_r$ quadratic Lax matrices given by
the formulas~(\ref{Matrix Example D4 a-osc factroized}) and~(\ref{Matrix Example D3 a-osc}),
respectively, with $\ID,\idb$ of~\eqref{eq:J-matrix} and $\wp,\wm$ encoding
$N-2=2r-1$ pairs of oscillators,~cf.~\eqref{eq:wmwp D2}:
\begin{equation}\label{eq:B-osc}
  \wp=(\oad_{2},\ldots,\oad_{r},\oad_{r+1},\oad_{r'},\ldots,\oad_{2'})\, ,\qquad
  \wm=(\oa_{2},\ldots,\oa_{r},\oa_{r+1},\oa_{r'},\ldots,\oa_{2'})^t \, .
\end{equation}

\medskip

\begin{Rem}\label{rem:reation to Reshetikhin}
Let us define $\mathfrak{L}_{x_1,x_2}(z)$ via~\eqref{Matrix Example D4 a-osc factroized}
with $\ID,\idb$ as in~(\ref{eq:J-matrix}) and $\wp,\wm$ as in~\eqref{eq:B-osc}.
Consider the expansion of the Lax matrix
\begin{equation}\label{eq:L vs fL}
  L_{x_1,x_2}(z)=\mathfrak{L}_{x_1,x_2}(z+a)=z^2+zM_{x_1,x_2}+G_{x_1,x_2}
\end{equation}
with the shift $a$ of the spectral parameter given by:
\begin{equation}\label{eq:a-mass}
  a=\frac{x_1+x_2-1}{2} \,.
\end{equation}
Using the equalities
\begin{equation*}
  \wm^t\wp^t=\wp\wm+N-2 \,,\qquad [\wp\idb\wp^t,\oa_i]=-2\oad_{i'} \,,
\end{equation*}
 cf.~\eqref{eq:B-osc}, it is straightforward to see from~\eqref{Matrix Example D4 a-osc factroized}
that the linear term in~\eqref{eq:L vs fL} reads:
\begin{equation*}
  M_{x_1,x_2} =
  \left(\begin{BMAT}[5pt]{c|c|c}{c|c|c}
    -x_1+x_2-\tfrac{N}{2}+1-\wp\wm&M_{[12]}&0\\\
    -\wm&\wm\wp-\idb\wp^t\wm^t\idb-\id&M_{[23]}\\
    0&\wm^t\idb&x_1-x_2+\tfrac{N}{2}-1+\wp\wm\\
  \end{BMAT}\right)
\end{equation*}
with the row $M_{[12]}$ and the column $M_{[23]}$ explicitly given by:
\begin{align}
  M_{[12]}&=\left(x_1-x_2+\frac{N}{2}-2+\wp\wm\right)\wp-\frac{1}{2}\wp\idb\wp^t\wm^t\idb \,,\\
  M_{[23]}&=-\left(x_1-x_2+\frac{N}{2}-2+\wp\wm\right)\idb\wp^t+\frac{1}{2}\wp\idb\wp^t\cdot \wm \,.
\end{align}
It is easily seen from the formulas above that the matrix coefficients $M_{ij}=(M_{x_1,x_2})_{i,j}$ satisfy
\begin{equation}\label{eq:transp}
  M_{ij}=-M_{j'i'}
\end{equation}
as well as obey the following commutation relations:
\begin{equation}\label{eq:comm-rel}
  [M_{ij},M_{k\ell}]=\delta_{i,\ell}M_{kj}-\delta_{j',\ell}M_{ki'}-\delta_{i,k'}M_{\ell'j}+\delta_{j,k}M_{\ell'i'} \,.
\end{equation}
Finally, we can show by direct computation that $M_{x_1,x_2}$ satisfies the characteristic identity:
\begin{equation}\label{eq:char id}
  (M_{x_1,x_2}+1) (2 M_{x_1,x_2}+N+2 x_1-2x_2-2) (2 M_{x_1,x_2}+N-2 x_1+2 x_2-2)=0 \,,
\end{equation}
while the free term $G_{x_1,x_2}$ in~\eqref{eq:L vs fL} is expressed via the linear term $M_{x_1,x_2}$ as follows:
\begin{equation}\label{eq:free via linear}
  G_{x_1,x_2}=\frac{1}{2}M_{x_1,x_2}^2+\frac{1}{4}(N-2)M_{x_1,x_2}+\frac{1}{4}\left(N-3-(x_1-x_2)^2\right)\ID_N \,.
\end{equation}

Let us further introduce the parameter $m$ via:
\begin{equation}
  x_1-x_2=1-m-\frac{N}{2} \,
\end{equation}
so that the characteristic identity~\eqref{eq:char id} for $M=M_{x_1,x_2}$ becomes
\begin{equation}
  (M-m)(M+N+m-2)(M+1)=0 \,,
\end{equation}
thus exactly coinciding with~\cite[(3.9)]{r}. Furthermore, after an additional shift in the spectral parameter,
the Lax matrix $L(z)=L_{x_1,x_2}(z)$ can be written as (taking~\eqref{eq:free via linear} into an account):
\begin{equation*}
  {L}\left(z+\frac{N-2}{4}\right)=z\left(z+\frac{N-2}{2}\right)+zM+\frac{1}{2}\left(M^2+(N-2)M-\frac{m(m+N-2)-N+3}{2}\right)
\end{equation*}
which coincides with~\cite[(3.11)]{r}. However, our oscillator realisation differs from that of~\cite{r}.
\end{Rem}

\medskip

   %%%%%%%%%%%%%%%%%%%%%%%%%%%%%%%%%%%%%%%%%%%%%%%%%%%%%%%%%%%%%%%%
   %%%%%%%%%%%%%%%%%%%%%%%% Further Directions %%%%%%%%%%%%%%%%%%%%
   %%%%%%%%%%%%%%%%%%%%%%%%%%%%%%%%%%%%%%%%%%%%%%%%%%%%%%%%%%%%%%%%

\section{Further directions}
\label{sec further directions}

   %%%%%%%%%%%%%%%%%%%%%%%%%%%%%%%%%%%%%%%%%%%%%%%%%%%%%%%%%%%%%%%%
   %%%%%%%%%%%%%%%%%%% Trigonometric Generalization %%%%%%%%%%%%%%%
   %%%%%%%%%%%%%%%%%%%%%%%%%%%%%%%%%%%%%%%%%%%%%%%%%%%%%%%%%%%%%%%%

\subsection{Trigonometric version}
\

The constructions and results of the present paper admit natural trigonometric counterparts.
To this end, recall the \emph{shifted quantum affine algebras} $U_{\nu^+,\nu^-}(L\fg)$,
introduced in~\cite[\S5]{ft1}, which are associative $\BC(\vv)$-algebras depending on a pair of
shifts $\nu^+,\nu^- \in \bar{\Lambda}$. Based on and generalizing the isomorphism between
the new Drinfeld and the RTT realizations of extended quantum affine algebras in the classical types
$B_r,C_r,D_r$ recently established in~\cite{jlm2,jlm3}, it turns out that the \emph{shifted extended quantum
affine algebras} $U^\mathrm{ext}_{-\mu^+,-\mu^-}(L\fg)$ with $\mu^+,\mu^-\in \Lambda^+$ admit the RTT realization
$U^\mathrm{ext}_{-\mu^+,-\mu^-}(L\fg)\iso U^{\rtt,\mathrm{ext}}_{-\mu^+,-\mu^-}(L\fg)$ alike~\eqref{eq:key isom}.
This can be viewed as a natural generalization of~\cite[Theorem 3.51]{fpt} for type $A_r$ and shall be addressed elsewhere.

\medskip
\noindent
As an immediate corollary, we obtain the following two important structures:

$\bullet$
  coproduct homomorphisms
  $\Delta_{\nu^+_1,\nu^-_1,\nu^+_2,\nu^-_2}\colon
   U_{\nu^+_1+\nu^+_2,\nu^-_1+\nu^-_2}(L\fg)\to U_{\nu^+_1,\nu^-_1}(L\fg)\otimes U_{\nu^+_2,\nu^-_2}(L\fg)$

$\bullet$
  $\BZ[\vv,\vv^{-1}]$ integral forms $\mathfrak{U}_{\nu^+,\nu^-}(L\fg)\subset U_{\nu^+,\nu^-}(L\fg)$
  compatible with $\Delta_{\nu^+_1,\nu^-_1,\nu^+_2,\nu^-_2}$

\noindent
for classical types $B_r,C_r,D_r$, thus generalizing the only known type $A_r$ case of~\cite{ft1,ft2}.

\medskip
\noindent
Combining the above RTT realization with~\cite[Theorem 7.1]{ft1}, one obtains trigonometric Lax matrices
$T^{\mathrm{trig}}_D(z)$ which can be degenerated to $T_{D+D|_0([\infty]-[0])}(z)$, cf.~\cite[Proposition~3.94]{fpt}.

\medskip

   %%%%%%%%%%%%%%%%%%%%%%%%%%%%%%%%%%%%%%%%%%%%%%%%%%%%%%%%%%%%%%%%
   %%%%%%%%%%%%%%%%%%%%%%% Integrable systems %%%%%%%%%%%%%%%%%%%%%
   %%%%%%%%%%%%%%%%%%%%%%%%%%%%%%%%%%%%%%%%%%%%%%%%%%%%%%%%%%%%%%%%

\subsection{Integrable systems}
\

As yet another important application of our key isomorphism~\eqref{eq:key isom} and its aforementioned trigonometric version,
the RTT presentation provides (cf.~\cite{mm}) interesting algebraic quantum integrable systems that appear
on the corresponding quantized ($K$-theoretic) Coulomb branches of $4d$ supersymmetric $\mathcal{N}=2$ quiver
gauge theories, cf.~\cite{np,nps} and~\cite{bfna,bfnb}.
To this end, note that the $\BC[\hbar]$-version of the homomorphisms~\eqref{eq:nw homom} factor through
the quantized Coulomb branches~\cite[Theorem 5.6]{nw}, cf.~\cite[Theorem 8.5]{ft1} in the trigonometric~case.

\medskip
\noindent
Let us also note that $T_D(z)T'_D(-z)$ satisfy the reflection equation~\cite[(4.1)]{gr},
thus giving rise to shifted versions of reflection algebras (aka twisted extended Yangians~\cite[Theorem 4.2]{gr})
of types $B,C,D$. We expect the latter to be related to integrable systems with boundary.

\medskip

   %%%%%%%%%%%%%%%%%%%%%%%%%%%%%%%%%%%%%%%%%%%%%%%%%%%%%%%%%%%%%%%%
   %%%%%%%%%%%%% Polynomial Solutions and Q-operators %%%%%%%%%%%%%
   %%%%%%%%%%%%%%%%%%%%%%%%%%%%%%%%%%%%%%%%%%%%%%%%%%%%%%%%%%%%%%%%

\subsection{Polynomial solutions and $Q$-operators}
\

As mentioned in the introduction (with more details provided in Subsections~\ref{ssec examples},~\ref{ssec Lax matrrices C},~\ref{ssec Lax matrrices B}),
some of the simplest examples of our Lax matrices $\sT_D(z)$ are equivalent (up to highly nontrivial canonical transformations)
to the \emph{polynomial} (as they take values in non-localized oscillator algebras) Lax matrices constructed quite recently in the physics literature.
A very interesting question is to understand which of our Lax matrices $\sT_D(z)$ can be transformed (up to canonical transformations) to the polynomial ones.
We note that one of the advantages of our construction is a natural limit procedure~\eqref{eq:limit relation} which becomes highly nontrivial
for the polynomial Lax matrices, see e.g.~Remark~\ref{rem normalized limit C-pol}. However, the polynomial Lax matrices have an obvious advantage
of allowing to take traces, thus leading to $Q$-operators as discussed~below.

\medskip
\noindent
As outlined in~\cite{f}, the polynomial solutions for $D_r$-type can be used to construct $Q$-operators following~\cite{blz,bflms}.
The corresponding $QQ$-system for $D_r$-type spin chains has been recently proposed in~\cite{ffk}, see also~\cite{esv} for a different approach.
We remark that only a subset of the $Q$-operators is constructed in~\cite{f}, namely those corresponding to the end nodes
of the $D_r$ Dynkin diagram for which the evaluation map does exist. While the remaining $Q$-operators are determined by
the $QQ$-system, a direct construction for those is not known. The asymptotic behaviour of the $Q$-operators in the spectral parameter
can be extracted from the algebraic Bethe ansatz, cf.~\cite{r1}, and solutions with the appropriate asymptotic behaviour are obtained
in the Table~\ref{tab:D} below,  suggesting that a construction from our Lax matrices may be possible.  However, the Lax matrices
$\sT_D(z)$ of the present paper are not polynomial in the oscillators and a trace prescription remains to be found.

\medskip
\noindent
The situation for $B_r$ and $C_r$ types is similar. The Lax matrices for $Q$-operators corresponding to the nodes of the Dynkin diagram
where the evaluation map exists are presented here in the polynomial form, see~\cite{fkt} for more details.
For the non-polynomial (in oscillators) solutions $\sT_D(z)$ of the present paper we are in the same position
as for $D_r$-type discussed above. The expected  asymptotic behaviour of the corresponding $Q$-operators is spelled out in
Tables~\ref{tab:C} and~\ref{tab:B} below. A study of the $QQ$-system for the spin chains of type $B_r$ and $C_r$ is outstanding.

\medskip

    %%%%%%%%%%%%%%%%%%%%%%%%%%%%%%%%%%%%%%%%%%%%%%%%%%%%%%%%%%%%%%%%%%%%%%%%%%%%%%%%%%%%%%%
    %%%%%%%%%%%%%%%%%%%%%%%%%%%%%%%%%%%%%% TABLES %%%%%%%%%%%%%%%%%%%%%%%%%%%%%%%%%%%%%%%%%
    %%%%%%%%%%%%%%%%%%%%%%%%%%%%%%%%%%%%%%%%%%%%%%%%%%%%%%%%%%%%%%%%%%%%%%%%%%%%%%%%%%%%%%%

\begin{table}[h!]
\def\arraystretch{1.5}
{\small
\centering
 \begin{tabular}{|c|| c| c| c|c|}
 \hline
 $Q_i$ & $\vec{a}=(a_1,\ldots,a_r)$ & $\vec{b}=(b_1,\ldots,b_r)$ & $\vec{Z}=(\deg Z_1,\ldots,\deg Z_r)$ &\#osc. \\ [0.5ex]
 \hline\hline
 $1\leq i\leq r-2$ & $(\underbrace{2,4,6,\ldots,2i}_i,\underbrace{2i,\ldots,2i}_{r-i-2},i,i)$ & $(\underbrace{0,\ldots,0,2}_i,\underbrace{0,\ldots,0}_{r-i})$ & $(\underbrace{0,\ldots,0,1}_i,\underbrace{0,\ldots,0}_{r-i})$&$i(2r-i-1)$\\
 \hline $i=r-1$ &\cellcolor{lightgray!30} &\cellcolor{lightgray!30} &\cellcolor{lightgray!30} &\cellcolor{lightgray!30} \\
 \hline
 $\begin{array}{l}
                                      \text{even $r$ }\\
                                       \text{odd $r$}\end{array}$ &$\begin{array}{l}
                                      (1,2,3,\ldots,r-2,\frac{r}{2},\frac{r}{2}-1)\\
                                      (1,2,3,\ldots,r-2,\frac{r-1}{2},\frac{r-1}{2})
                                      \end{array}$& $\begin{array}{l}
                                       (0,\ldots,0,2,0)\\
                                     (0,\ldots,0,1,1)
                                      \end{array}$
 & $(0,\ldots,0,1,0)$&$\frac{r(r-1)}{2}$\\
 \hline
 $i=r$ &\cellcolor{lightgray!30}&\cellcolor{lightgray!30} &\cellcolor{lightgray!30} &\cellcolor{lightgray!30} \\
 \hline
 $\begin{array}{l}
                                      \text{even $r$}\\
                                       \text{odd  $r$}\end{array}$  & $\begin{array}{l}
                                       (1,2,3,\ldots,r-2,\frac{r}{2}-1,\frac{r}{2})\\
                                        (1,2,3,\ldots,r-2,\frac{r-1}{2},\frac{r-1}{2})
                                      \end{array}$& $\begin{array}{l}
                                     (0,\ldots,0,0,2)\\
                                     (0,\ldots,0,1,1)
                                      \end{array}$ & $(0,\ldots,0,0,1)$&$\frac{r(r-1)}{2}$\\ [1ex]
 \hline
 \end{tabular}
 }
 \caption{Solutions of $D_r$-type with the expected asymptotic behavior and \newline\protect\phantom{TABLE 2\,} number of oscillator pairs for $Q$-operator at the node $i$}
 \label{tab:D}
\par
\par
\def\arraystretch{1.5}{
\centering
 \begin{tabular}{|c|| c| c| c|c|}
 \hline
 $Q_i$ & $\vec{a}=(a_1,\ldots,a_r)$ & $\vec{b}=(b_1,\ldots,b_r)$ & $\vec{Z}=(\deg Z_1,\ldots,\deg Z_r)$ &\#osc. \\ [0.5ex]
 \hline\hline
 $1\leq i<r$ & $(\underbrace{2,4,6,\ldots,2i}_i,\underbrace{2i,\ldots,2i}_{r-i})$ & $(\underbrace{0,\ldots,0,2}_i,\underbrace{0,\ldots,0}_{r-i})$ & $(\underbrace{0,\ldots,0,1}_i,\underbrace{0,\ldots,0}_{r-i})$&$i(2r-i+1)$\\
 \hline
 $i= r$ & $(1,2,3,\ldots,r)$ & $(0,\ldots,0,2)$ & $(0,\ldots,0,1)$&$\frac{r(r+1)}{2}$\\
 \hline
 \end{tabular}
 }
 \caption{Solutions of $C_r$-type with the expected asymptotic behavior and \newline\protect\phantom{TABLE 2\,} number of oscillator pairs for $Q$-operator at the node $i$}
 \label{tab:C}
\par
\par
\def\arraystretch{1.5}
{
\centering
 \begin{tabular}{|c|| c| c| c|c|}
 \hline
 $Q_i$ & $\vec{a}=(a_1,\ldots,a_r)$ & $\vec{b}=(b_1,\ldots,b_r)$ & $\vec{Z}=(\deg Z_1,\ldots,\deg Z_r)$ &\#osc. \\ [0.5ex]
 \hline\hline
 $1\leq i<r$ & $(\underbrace{2,4,6,\ldots,2i}_i,\underbrace{2i,\ldots,2i}_{r-i-1},i)$ & $(\underbrace{0,\ldots,0,2}_i,\underbrace{0,\ldots,0}_{r-i})$ & $(\underbrace{0,\ldots,0,1}_i,\underbrace{0,\ldots,0}_{r-i})$&$i(2r-i)$\\ [1ex]
 \hline
 $i=r$ & $(2,4,6,\ldots,2(r-1),r)$& $(0,\ldots,0,2)$ & $(0,\ldots,0,1)$&$r^2$\\
 \hline
 \end{tabular}
 }
 \caption{Solutions of $B_r$-type with the expected asymptotic behavior and \newline\protect\phantom{TABLE 2\,} number of oscillator pairs for $Q$-operator at the node $i$}
 \label{tab:B}
\end{table}

   %%%%%%%%%%%%%%%%%%%%%%%%%%%%%%%%%%%%%%%%%%%%%%%%%%%%%%%%%%%%%%%%
   %%%%%%%%%%%%%%%%%%%%%%%%%%%% Appendix %%%%%%%%%%%%%%%%%%%%%%%%%%
   %%%%%%%%%%%%%%%%%%%%%%%%%%%%%%%%%%%%%%%%%%%%%%%%%%%%%%%%%%%%%%%%

\appendix

    %%%%%%%%%%%%%%%%%%%%%%%%%%%%%%%%%%%%%%%%%%%%%%%%%%%%%%%%%%%%%%%%%%%%%%%%%%%%%%%
    %%%%%%%%%%%%%%%%%%%%%%%%% Explicit formulas for Lax %%%%%%%%%%%%%%%%%%%%%%%%%%%
    %%%%%%%%%%%%%%%%%%%%%%%%%%%%%%%%%%%%%%%%%%%%%%%%%%%%%%%%%%%%%%%%%%%%%%%%%%%%%%%

\section{Explicit formulas in type D}
\label{Appendix A: Lax explicitly}

In this Appendix, we record the explicit formulas for the matrices $F^D(z),H^D(z),E^D(z)$
the product of which recovers the Lax matrices $T_D(z)$ in type $D_r$, see~(\ref{eq:Lax definition},~\ref{eq:Lax explicit}).
All proofs are straightforward and are based on
Lemmas~\ref{lem:all-H},~\ref{lem:all-E-known},~\ref{lem:all-E-new},~\ref{lem:all-F-known},~\ref{lem:all-F-new}
of Subsection~\ref{sssec: Drinfeld-to-RTT-D}.

   %%%%%%%%%%%%%%%%%%%%%%%%%%%%%%%%%%%%%%%%%%%%%%%%%%%%%%%%%%%%%%%%%%%%%%%%%%%%%%%%%%%%%%%%
   %%%%%%%%%%%%%%%%%%%%%%%%%%%%%%%%%%%%%% H-factor %%%%%%%%%%%%%%%%%%%%%%%%%%%%%%%%%%%%%%%%
   %%%%%%%%%%%%%%%%%%%%%%%%%%%%%%%%%%%%%%%%%%%%%%%%%%%%%%%%%%%%%%%%%%%%%%%%%%%%%%%%%%%%%%%%

\medskip
\noindent
$\bullet$ \emph{\underline{Matrix $H^D(z)$ explicitly}.}
\begin{equation}\label{eq:H-part Lax}
\begin{split}
  & h^D_i(z)=
     \begin{cases}
       \frac{P_i(z)}{P_{i-1}(z-1)} \cdot \prod_{k=0}^{i-1} Z_k(z) & \mbox{if } i\leq r-2 \\
       \frac{P_{r-1}(z)P_r(z)}{P_{r-2}(z-1)} \cdot \prod_{k=0}^{r-2} Z_k(z) & \mbox{if } i=r-1 \\
       \frac{P_r(z)}{P_{r-1}(z-1)} \cdot \prod_{k=0}^{r-1} Z_k(z) & \mbox{if } i=r
     \end{cases} \\
  & h^D_{i'}(z)=
     \begin{cases}
       \frac{P_{i-1}(z+i-r)}{P_{i}(z+i-r)} \cdot \prod_{k=0}^{r} Z_k(z) \prod_{k=i}^{r-2} Z_k(z+k-r+1) & \mbox{if } i\leq r-2 \\
       \frac{P_{r-2}(z-1)}{P_{r-1}(z-1)P_r(z-1)} \cdot \prod_{k=0}^{r} Z_k(z) & \mbox{if } i=r-1 \\
       \frac{P_{r-1}(z)}{P_{r}(z-1)} \cdot \prod_{k=0}^{r-2} Z_k(z)\cdot Z_r(z) & \mbox{if } i=r
     \end{cases}
\end{split}
\end{equation}

   %%%%%%%%%%%%%%%%%%%%%%%%%%%%%%%%%%%%%%%%%%%%%%%%%%%%%%%%%%%%%%%%%%%%%%%%%%%%%%%%%%%%%%%%
   %%%%%%%%%%%%%%%%%%%%%%%%%%%%%%%%%%%%%% E-factor %%%%%%%%%%%%%%%%%%%%%%%%%%%%%%%%%%%%%%%%
   %%%%%%%%%%%%%%%%%%%%%%%%%%%%%%%%%%%%%%%%%%%%%%%%%%%%%%%%%%%%%%%%%%%%%%%%%%%%%%%%%%%%%%%%

\medskip
\noindent
$\bullet$ \emph{\underline{Matrix $E^D(z)$ explicitly}.}
\

\noindent
For $1\leq i<j\leq r-2$, we get:
\begin{small}
\begin{equation*}
   e^D_{i,j}(z)=(-1)^{j-i-1}
   \sum_{\substack{1\leq k_i\leq a_i\\\cdots\\ 1\leq k_{j-1}\leq a_{j-1}}}
     \frac{P_{i-1}(p_{i,k_i}-1)\cdot \prod_{s=i}^{j-2}P_{s,k_s}(p_{s+1,k_{s+1}}-1)}
          {(z-p_{i,k_i})\prod_{s=i}^{j-1} P_{s,k_s}(p_{s,k_s})}\cdot
     e^{\sum_{s=i}^{j-1} q_{s,k_s}}
\end{equation*}
\end{small}

\medskip
\noindent
For $1\leq i\leq r-2$, we get:
\begin{small}
\begin{equation*}
   e^D_{i,r-1}(z)=(-1)^{r-i}
   \sum_{\substack{1\leq k_i\leq a_i\\\cdots\\ 1\leq k_{r-2}\leq a_{r-2}}}
     \frac{P_{i-1}(p_{i,k_i}-1)\cdot \prod_{s=i}^{r-3}P_{s,k_s}(p_{s+1,k_{s+1}}-1)\cdot P_r(p_{r-2,k_{r-2}})}
          {(z-p_{i,k_i})\prod_{s=i}^{r-2} P_{s,k_s}(p_{s,k_s})}\cdot
     e^{\sum_{s=i}^{r-2} q_{s,k_s}}
\end{equation*}
\end{small}

\medskip
\noindent
For $1\leq i\leq r-2$, we get:
\begin{small}
\begin{equation*}
   e^D_{i,r}(z)=(-1)^{r-i-1}
   \sum_{\substack{1\leq k_i\leq a_i\\\cdots\\ 1\leq k_{r-1}\leq a_{r-1}}}
     \frac{P_{i-1}(p_{i,k_i}-1)\cdot \prod_{s=i}^{r-2}P_{s,k_s}(p_{s+1,k_{s+1}}-1)\cdot P_r(p_{r-2,k_{r-2}})}
          {(z-p_{i,k_i})\prod_{s=i}^{r-1} P_{s,k_s}(p_{s,k_s})}\cdot
     e^{\sum_{s=i}^{r-1} q_{s,k_s}}
\end{equation*}
\end{small}
while $e^D_{r-1,r}(z)$ is given by the same formula (with $i=r-1$) but with $P_r(p_{r-2,k_{r-2}})$ omitted.

\medskip
\noindent
For $1\leq i<j\leq r-2$, we get:
\begin{small}
\begin{equation*}
  e^D_{j',i'}(z) = (-1)^{j-i}
  \sum_{\substack{1\leq k_{i}\leq a_{i}\\\cdots\\ 1\leq k_{j-1}\leq a_{j-1}}}
     \frac{P_{i-1}(p_{i,k_{i}}-1)\cdot \prod_{s=i}^{j-2}P_{s,k_s}(p_{s+1,k_{s+1}}-1)}
          {(z-p_{j-1,k_{j-1}}+j-r)\prod_{s=i}^{j-1} P_{s,k_s}(p_{s,k_s})}\cdot
     e^{\sum_{s=i}^{j-1} q_{s,k_s}}
\end{equation*}
\end{small}

\medskip
\noindent
For $1\leq i\leq r-2$, we get:
\begin{small}
\begin{multline*}
   e^D_{(r-1)',i'}(z)=\\
   (-1)^{r-i+1}
   \sum_{\substack{1\leq k_{i}\leq a_{i}\\\cdots\\ 1\leq k_{r-2}\leq a_{r-2}}}
     \frac{P_{i-1}(p_{i,k_{i}}-1)\cdot \prod_{s=i}^{r-3}P_{s,k_s}(p_{s+1,k_{s+1}}-1)\cdot P_r(p_{r-2,k_{r-2}})}
          {(z-p_{r-2,k_{r-2}}-1)\prod_{s=i}^{r-2} P_{s,k_s}(p_{s,k_s})}\cdot
   e^{\sum_{s=i}^{r-2} q_{s,k_s}}
\end{multline*}
\end{small}

\medskip
\noindent
For $1\leq i\leq r-2$, we get:
\begin{small}
\begin{equation*}
   e^D_{r',i'}(z)=
   (-1)^{r-i}
   \sum_{\substack{1\leq k_{i}\leq a_{i}\\\cdots\\ 1\leq k_{r-1}\leq a_{r-1}}}
     \frac{P_{i-1}(p_{i,k_{i}}-1)\cdot \prod_{s=i}^{r-2}P_{s,k_s}(p_{s+1,k_{s+1}}-1)\cdot P_r(p_{r-2,k_{r-2}})}
          {(z-p_{r-1,k_{r-1}})\prod_{s=i}^{r-1} P_{s,k_s}(p_{s,k_s})}\cdot
   e^{\sum_{s=i}^{r-1} q_{s,k_s}}
\end{equation*}
\end{small}
while $e^D_{r',(r-1)'}(z)$ is given by the same formula (with $i=r-1$) but with $P_r(p_{r-2,k_{r-2}})$ omitted.

\medskip
\noindent
For $1\leq i\leq r-2$, we get:
\begin{small}
\begin{multline*}
   e^D_{i,r'}(z)=\\
   (-1)^{r-i}
   \sum_{\substack{1\leq k_{i}\leq a_{i}\\ \cdots\\ 1\leq k_{r-2}\leq a_{r-2}\\ 1\leq k_{r}\leq a_{r}}}
     \frac{P_{i-1}(p_{i,k_{i}}-1)\cdot \prod_{s=i}^{r-3}P_{s,k_s}(p_{s+1,k_{s+1}}-1)\cdot P_{r,k_r}(p_{r-2,k_{r-2}})}
          {(z-p_{i,k_{i}})\prod_{s=i}^{r-2} P_{s,k_s}(p_{s,k_s})\cdot P_{r,k_r}(p_{r,k_r})}\cdot
   e^{\sum_{s=i}^{r-2} q_{s,k_s} +\, q_{r,k_r}}
\end{multline*}
\end{small}
while $e^D_{r-1,r'}(z)$ equals $\Psi_D(E_r(z))$ specified in~\eqref{eq:homom assignment}.

\medskip
\noindent
For $1\leq i \leq r-2$, we get:
\begin{small}
\begin{equation*}
   e^D_{i,(r-1)'}(z)=(-1)^{r-i}
   \sum_{\substack{1\leq k_{i}\leq a_{i}\\ \cdots\\ 1\leq k_{r}\leq a_{r}}}
     \frac{P_{i-1}(p_{i,k_{i}}-1)\cdot \prod_{s=i}^{r-2}P_{s,k_s}(p_{s+1,k_{s+1}}-1)\cdot P_{r,k_r}(p_{r-2,k_{r-2}})}
          {(z-p_{i,k_{i}})\prod_{s=i}^{r} P_{s,k_s}(p_{s,k_s})}\cdot e^{\sum_{s=i}^{r} q_{s,k_s}}
\end{equation*}
\end{small}

\medskip
\noindent
For $1\leq i < j\leq r-2$, we get:
\begin{small}
\begin{multline*}
   e^D_{i,j'}(z)=(-1)^{j-i}\ \cdot\\
   \sum^{|I_{s}|=1+\delta_{s\in \{j,\ldots,r-2\}}}_{\substack{I_i\subset \{1,\ldots,a_i\} \\ \cdots\\ I_r\subset \{1,\ldots,a_r\}}}
     \frac{P_{i-1}(p_{i,k_{i}}-1)\cdot \prod_{i\leq s\leq r-2}^{k\in I_{s+1}}P_{s,I_s}(p_{s+1,k}-1)\cdot \prod_{k\in I_{r-2}}P_{r,I_r}(p_{r-2,k})}
          {(z-p_{i,k_{i}})\prod_{i\leq s\leq r}^{k\in I_s} P_{s,I_s}(p_{s,k})}\cdot
   e^{\sum_{i\leq s\leq r}^{k\in I_s} q_{s,k}} \, ,
\end{multline*}
\end{small}
the sum taken over all subsets $I_s\subset \{1,\ldots,a_s\}, i\leq s\leq r,$ of cardinality
$|I_i|=\ldots=|I_{j-1}|=1$, $|I_j|=\ldots=|I_{r-2}|=2, |I_{r-1}|=|I_r|=1$, and the natural generalization
of~\eqref{ZW-series} being used:
\begin{equation}\label{eq:P-generalized}
  P_{s,I_s}(z)\, :=\prod_{1\leq k\leq a_s}^{k\notin I_s} (z-p_{s,k}) \, .
\end{equation}

\medskip
\noindent
For $1\leq i \leq r-3$, we get:
\begin{small}
\begin{multline*}
   e^D_{i,i'}(z)=\\
   - \sum^{|I_{s}|=1+\delta_{s\in \{i,\ldots,r-2\}}}_{\substack{I_{i}\subset \{1,\ldots,a_i\} \\ \cdots\\ I_r\subset \{1,\ldots,a_r\}}}
     \frac{\prod_{k\in I_{i}}P_{i-1}(p_{i,k}-1)\cdot \prod_{i\leq s\leq r-2}^{k\in I_{s+1}}P_{s,I_s}(p_{s+1,k}-1)\cdot \prod_{k\in I_{r-2}}P_{r,I_r}(p_{r-2,k})}
          {\prod_{k\in I_{i}}(z-p_{i,k})\cdot \prod_{i\leq s\leq r}^{k\in I_s} P_{s,I_s}(p_{s,k})}\cdot
   e^{\sum_{i\leq s\leq r}^{k\in I_s} q_{s,k}}
\end{multline*}
\end{small}
while $e^D_{r,r'}(z)=0$, due to Lemma~\ref{lem:all-E-known}(a), and $e^D_{r-1,(r-1)'}(z)$ is given by:
\begin{small}
\begin{equation*}
   e^D_{r-1,(r-1)'}(z) \, =\,
   - \sum_{\substack{1\leq k_{r-1}\leq a_{r-1} \\ 1\leq k_r\leq a_r}}
     \frac{P_{r-2}(p_{r-1,k_{r-1}}-1)}
          {(z-p_{r-1,k_{r-1}})(z-p_{r,k_r})\prod_{s=r-1}^{r}P_{s,k_s}(p_{s,k_s})}\cdot e^{q_{r-1,k_{r-1}}+\, q_{r,k_r}}
\end{equation*}
\end{small}

\medskip
\noindent
For $2\leq i\leq r-2$, we get:
\begin{small}
\begin{multline*}
    e^D_{i,(i-1)'}(z)\ =
    \sum^{|I_{s}|=1+\delta_{s\in \{i,\ldots,r-2\}}}_{\substack{I_{i-1}\subset \{1,\ldots,a_{i-1}\} \\ \cdots\\ I_r\subset \{1,\ldots,a_r\}}}
     \frac{z-p_{i-1,k_{i-1}}-1}{\prod_{k\in I_{i}}(z-p_{i,k})}\, \times\\
    \frac{P_{i-2}(p_{i-1,k_{i-1}}-1)\cdot \prod_{i-1\leq s\leq r-2}^{k\in I_{s+1}}P_{s,I_s}(p_{s+1,k}-1)\cdot \prod_{k\in I_{r-2}}P_{r,I_r}(p_{r-2,k})}
          {\prod_{i-1\leq s\leq r}^{k\in I_s} P_{s,I_s}(p_{s,k})}\cdot
   e^{\sum_{i-1\leq s\leq r}^{k\in I_s} q_{s,k}}
\end{multline*}
\end{small}
while the $i=r-1,r$ counterparts of this formula are as follows:
\begin{small}
\begin{multline*}
   e^D_{r-1,(r-2)'}(z)\ =\
     -\sum_{\substack{1\leq k_{r-2}\leq a_{r-2}\\ 1\leq k_{r-1}\leq a_{r-1}\\ 1\leq k_{r}\leq a_{r}}}
     \frac{z-p_{r-2,k_{r-2}}-1}{(z-p_{r-1,k_{r-1}})(z-p_{r,k_r})}\, \times \\
     \frac{P_{r-3}(p_{r-2,k_{r-2}}-1) P_{r-2,k_{r-2}}(p_{r-1,k_{r-1}}-1) P_{r,k_r}(p_{r-2,k_{r-2}})}
          {\prod_{s=r-2}^{r} P_{s,k_s}(p_{s,k_s})}\cdot
   e^{\sum_{s=r-2}^{r} q_{s,k_s}}
\end{multline*}
\end{small}
and
\begin{small}
\begin{equation*}
   e^D_{r,(r-1)'}(z) \, =\,
   -\sum_{1\leq k_r\leq a_r}  \frac{1}{(z-p_{r,k_r})P_{r,k_r}(p_{r,k_r})}\cdot e^{q_{r,k_r}}
\end{equation*}
\end{small}

\medskip
\noindent
For $1\leq j\leq i-2 \leq r-4$, we get:
\begin{small}
\begin{multline*}
   e^D_{i,j'}(z) \ =\, (-1)^{i-j+1}
   \sum^{|I_{s}|=1+\delta_{s\in \{i,\ldots,r-2\}}}_{\substack{I_{j}\subset \{1,\ldots,a_{j}\} \\ \cdots\\ I_r\subset \{1,\ldots,a_r\}}}
   \frac{z-p_{i-1,k_{i-1}}-1}{\prod_{k\in I_{i}}(z-p_{i,k})}\, \times \\
   \frac{P_{j-1}(p_{j,k_j}-1)\cdot \prod_{j\leq s\leq r-2}^{k\in I_{s+1}} P_{s,I_s}(p_{s+1,k}-1)\cdot \prod_{k\in I_{r-2}} P_{r,I_r}(p_{r-2,k})}
          {\prod_{j\leq s\leq r}^{k\in I_s} P_{s,I_s}(p_{s,k})}\cdot
   e^{\sum_{j\leq s\leq r}^{k\in I_s} q_{s,k}}
\end{multline*}
\end{small}
while the $i=r-1,r$ counterparts of this formula are as follows:
\begin{small}
\begin{multline*}
   e^D_{r-1,j'}(z)\ =\, (-1)^{r-j+1}
   \sum_{\substack{1\leq k_j\leq a_{j} \\ \cdots\\ 1\leq k_r\leq a_r}}
   \frac{(z-p_{r-2,k_{r-2}}-1)}{(z-p_{r-1,k_{r-1}})(z-p_{r,k_r})}\, \times \\
   \frac{P_{j-1}(p_{j,k_j}-1)\cdot \prod_{s=j}^{r-2} P_{s,k_s}(p_{s+1,k_{s+1}}-1)\cdot P_{r,k_r}(p_{r-2,k_{r-2}})}
          {\prod_{s=j}^{r} P_{s,k_s}(p_{s,k_s})}\cdot e^{\sum_{s=j}^{r} q_{s,k_s}}
\end{multline*}
\end{small}
and
\begin{small}
\begin{multline*}
   e^D_{r,j'}(z) = \\
   (-1)^{r-j+1}
   \sum_{\substack{1\leq k_j\leq a_{j} \\ \cdots\\ 1\leq k_{r-2}\leq a_{r-2}\\ 1\leq k_r\leq a_{r}}}
   \frac{P_{j-1}(p_{j,k_j}-1)\cdot \prod_{s=j}^{r-3} P_{s,k_s}(p_{s+1,k_{s+1}}-1)\cdot P_{r,k_r}(p_{r-2,k_{r-2}})}
        {(z-p_{r,k_r})\prod_{s=j}^{r-2}P_{s,k_s}(p_{s,k_s})\cdot P_{r,k_r}(p_{r,k_r})}\cdot
   e^{\sum_{s=j}^{r-2} q_{s,k_s}+\, q_{r,k_r}}
\end{multline*}
\end{small}

\medskip

\begin{Rem}\label{rmk:E-skew}
In the notations $e^{D}_{i,j}(z)=\sum_{k\geq 1} e^{(D)k}_{i,j}z^{-k}$, see~\eqref{eq:modes},
the above formulas imply:
\begin{equation}\label{eq:E-linear-skew}
  e^{(D)1}_{i,j}=-e^{(D)1}_{j',i'} \,, \qquad \forall\, 1\leq i<j\leq 2r \, .
\end{equation}
\end{Rem}

   %%%%%%%%%%%%%%%%%%%%%%%%%%%%%%%%%%%%%%%%%%%%%%%%%%%%%%%%%%%%%%%%%%%%%%%%%%%%%%%%%%%%%%%%
   %%%%%%%%%%%%%%%%%%%%%%%%%%%%%%%%%%%%%% F-factor %%%%%%%%%%%%%%%%%%%%%%%%%%%%%%%%%%%%%%%%
   %%%%%%%%%%%%%%%%%%%%%%%%%%%%%%%%%%%%%%%%%%%%%%%%%%%%%%%%%%%%%%%%%%%%%%%%%%%%%%%%%%%%%%%%

\medskip
\noindent
$\bullet$ \emph{\underline{Matrix $F^D(z)$ explicitly}.}

\noindent
For $1\leq i<j\leq r-1$, we get:
\begin{small}
\begin{equation*}
\begin{split}
   & f^D_{j,i}(z)=\\
   & (-1)^{j-i}
   \sum_{\substack{1\leq k_i\leq a_i\\\cdots\\ 1\leq k_{j-1}\leq a_{j-1}}}
     \frac{\prod_{s=i}^{j-1}Z_s(p_{s,k_s}+1)\cdot \prod_{s=i+1}^{j-1}P_{s,k_s}(p_{s-1,k_{s-1}}+1)\cdot P_{j}(p_{j-1,k_{j-1}}+1)}
          {(z-p_{i,k_i}-1)\prod_{s=i}^{j-1} P_{s,k_s}(p_{s,k_s})}\cdot
     e^{-\sum_{s=i}^{j-1} q_{s,k_s}}
\end{split}
\end{equation*}
\end{small}

\medskip
\noindent
For $1\leq i\leq r-1$, we get:
\begin{small}
\begin{equation*}
   f^D_{r,i}(z)\, =\, (-1)^{r-i}
   \sum_{\substack{1\leq k_i\leq a_i\\\cdots\\ 1\leq k_{r-1}\leq a_{r-1}}}
     \frac{\prod_{s=i}^{r-1}Z_s(p_{s,k_s}+1)\cdot \prod_{s=i+1}^{r-1}P_{s,k_s}(p_{s-1,k_{s-1}}+1)}
          {(z-p_{i,k_i}-1)\prod_{s=i}^{r-1} P_{s,k_s}(p_{s,k_s})}\cdot
     e^{-\sum_{s=i}^{r-1} q_{s,k_s}}
\end{equation*}
\end{small}

\medskip
\noindent
For $1\leq i<j\leq r-1$, we get:
\begin{small}
\begin{multline*}
  f^D_{i',j'}(z) =\\
  (-1)^{j-i-1}
     \sum_{\substack{1\leq k_i\leq a_i\\\cdots\\ 1\leq k_{j-1}\leq a_{j-1}}}
     \frac{\prod_{s=i}^{j-1}Z_s(p_{s,k_s}+1)\cdot \prod_{s=i+1}^{j-1}P_{s,k_s}(p_{s-1,k_{s-1}}+1)\cdot P_{j}(p_{j-1,k_{j-1}}+1)}
          {(z-p_{j-1,k_{j-1}}+j-r-1)\prod_{s=i}^{j-1} P_{s,k_s}(p_{s,k_s})}\cdot
     e^{-\sum_{s=i}^{j-1} q_{s,k_s}}
\end{multline*}
\end{small}

\medskip
\noindent
For $1\leq i\leq r-1$, we get:
\begin{small}
\begin{equation*}
   f^D_{i',r'}(z)\, =\,
   (-1)^{r-i+1}
   \sum_{\substack{1\leq k_i\leq a_i\\\cdots\\ 1\leq k_{r-1}\leq a_{r-1}}}
     \frac{\prod_{s=i}^{r-1}Z_s(p_{s,k_s}+1)\cdot \prod_{s=i+1}^{r-1}P_{s,k_s}(p_{s-1,k_{s-1}}+1)}
          {(z-p_{r-1,k_{r-1}}-1)\prod_{s=i}^{r-1} P_{s,k_s}(p_{s,k_s})}\cdot
     e^{-\sum_{s=i}^{r-1} q_{s,k_s}}
\end{equation*}
\end{small}

\medskip
\noindent
For $1\leq i\leq r-2$, we get:
\begin{small}
\begin{multline*}
   f^D_{r',i}(z) \ =\
   (-1)^{r-i+1}
   \sum_{\substack{1\leq k_{i}\leq a_{i}\\ \cdots\\ 1\leq k_{r-2}\leq a_{r-2}\\ 1\leq k_{r}\leq a_{r}}}
   P_{r-2,k_{r-2}}(p_{r,k_r})P_{r-1}(p_{r-2,k_{r-2}}+1)\, \times\\
   \frac{\prod_{i\leq s\leq r}^{s\ne r-1}Z_s(p_{s,k_s}+1) \cdot \prod_{s=i+1}^{r-2}P_{s,k_s}(p_{s-1,k_{s-1}}+1)}
          {(z-p_{i,k_{i}}-1)\prod_{s=i}^{r-2} P_{s,k_s}(p_{s,k_s})\cdot P_{r,k_r}(p_{r,k_r})}\cdot
   e^{-\sum_{s=i}^{r-2} q_{s,k_s} -\, q_{r,k_r}}
\end{multline*}
\end{small}
while $f^D_{r',r-1}(z)$ equals $\Psi_D(F_r(z))$ specified in~\eqref{eq:homom assignment}.

\medskip
\noindent
For $1\leq i \leq r-2$, we get:
\begin{small}
\begin{multline*}
   f^D_{(r-1)',i}(z)\, =\\
   (-1)^{r-i+1}
   \sum_{\substack{1\leq k_{i}\leq a_{i}\\ \cdots\\ 1\leq k_{r}\leq a_{r}}}
     \frac{\prod_{s=i}^{r}Z_s(p_{s,k_s}+1)\cdot \prod_{s=i+1}^{r-1}P_{s,k_s}(p_{s-1,k_{s-1}}+1)\cdot P_{r-2,k_{r-2}}(p_{r,k_{r}})}
          {(z-p_{i,k_{i}}-1)\prod_{s=i}^{r} P_{s,k_s}(p_{s,k_s})}\cdot e^{-\sum_{s=i}^{r} q_{s,k_s}}
\end{multline*}
\end{small}

\medskip
\noindent
For $1\leq i < j\leq r-2$, we get:
\begin{small}
\begin{multline*}
   f^D_{j',i}(z)=(-1)^{j-i+1}\ \cdot\\
   \sum^{|I_{s}|=1+\delta_{s\in \{j,\ldots,r-2\}}}_{\substack{I_i\subset \{1,\ldots,a_i\} \\ \cdots\\ I_r\subset \{1,\ldots,a_r\}}}
     \frac{\prod_{i\leq s\leq r}^{k\in I_s} Z_s(p_{s,k}+1)\cdot \prod_{i+1\leq s\leq r-1}^{k\in I_{s-1}}P_{s,I_s}(p_{s-1,k}+1)\cdot P_{r-2,I_{r-2}}(p_{r,k_r})}
          {(z-p_{i,k_{i}}-1)\prod_{i\leq s\leq r}^{k\in I_s} P_{s,I_s}(p_{s,k})}\cdot
   e^{-\sum_{i\leq s\leq r}^{k\in I_s} q_{s,k}}
\end{multline*}
\end{small}
where we use the above notation~\eqref{eq:P-generalized} and $k_r$ denotes the only element of $I_r$, i.e.\ $I_r=\{k_r\}$.

\medskip
\noindent
For $1\leq i \leq r-2$, we get:
\begin{small}
\begin{multline*}
   f^D_{i',i}(z)\, =\\
   - \sum^{|I_{s}|=1+\delta_{s\in \{i,\ldots,r-2\}}}_{\substack{I_{i}\subset \{1,\ldots,a_i\} \\ \cdots\\ I_r\subset \{1,\ldots,a_r\}}}
     \frac{\prod_{i\leq s\leq r}^{k\in I_s} Z_s(p_{s,k}+1)\cdot \prod_{i+1\leq s\leq r-1}^{k\in I_{s-1}}P_{s,I_s}(p_{s-1,k}+1)\cdot P_{r-2,I_{r-2}}(p_{r,k_r})}
          {\prod_{k\in I_{i}}(z-p_{i,k}-1)\cdot \prod_{i\leq s\leq r}^{k\in I_s} P_{s,I_s}(p_{s,k})}\cdot
   e^{-\sum_{i\leq s\leq r}^{k\in I_s} q_{s,k}}
\end{multline*}
\end{small}
with $I_r=\{k_r\}$, while $f^D_{r',r}(z)=0$, due to Lemma~\ref{lem:all-F-known}(a), and $f^D_{(r-1)',r-1}(z)$ is given by:
\begin{small}
\begin{equation*}
   f^D_{(r-1)',r-1}(z) \, =\,
   - \sum_{\substack{1\leq k_{r-1}\leq a_{r-1} \\ 1\leq k_r\leq a_r}}
     \frac{Z_{r-1}(p_{r-1,k_{r-1}}+1)Z_r(p_{r,k_r}+1)P_{r-2}(p_{r,k_r})}
          {\prod_{s=r-1}^{r}(z-p_{s,k_{s}}-1)\cdot \prod_{s=r-1}^{r}P_{s,k_s}(p_{s,k_s})}\cdot e^{-q_{r-1,k_{r-1}}-\, q_{r,k_r}}
\end{equation*}
\end{small}

\medskip
\noindent
For $2\leq i\leq r-2$, we get:
\begin{small}
\begin{multline*}
    f^D_{(i-1)',i}(z) \, =\,
    -\sum^{|I_{s}|=1+\delta_{s\in \{i,\ldots,r-2\}}}_{\substack{I_{i-1}\subset \{1,\ldots,a_{i-1}\} \\ \cdots\\ I_r\subset \{1,\ldots,a_r\}}}
    \frac{z-p_{i-1,k_{i-1}}-2}{\prod_{k\in I_{i}}(z-p_{i,k}-1)}\, \times\\
    \frac{\prod_{i-1\leq s\leq r}^{k\in I_s} Z_s(p_{s,k}+1)\cdot \prod_{i\leq s\leq r-1}^{k\in I_{s-1}}P_{s,I_s}(p_{s-1,k}+1)\cdot P_{r-2,I_{r-2}}(p_{r,k_r})}
         {\prod_{i-1\leq s\leq r}^{k\in I_s} P_{s,I_s}(p_{s,k})}\cdot
   e^{-\sum_{i-1\leq s\leq r}^{k\in I_s} q_{s,k}}
\end{multline*}
\end{small}
with $I_r=\{k_r\}$, while the $i=r-1,r$ counterparts of this formula are as follows:
\begin{small}
\begin{multline*}
   f^D_{(r-2)',r-1}(z)\ =\,
     \sum_{\substack{1\leq k_{r-2}\leq a_{r-2}\\ 1\leq k_{r-1}\leq a_{r-1}\\ 1\leq k_{r}\leq a_{r}}}
     \frac{z-p_{r-2,k_{r-2}}-2}{(z-p_{r-1,k_{r-1}}-1)(z-p_{r,k_r}-1)}\, \times \\
     \frac{\prod_{s=r-2}^r Z_s(p_{s,k_s}+1)\cdot P_{r-2,k_{r-2}}(p_{r,k_{r}}) P_{r-1,k_{r-1}}(p_{r-2,k_{r-2}}+1)}
          {\prod_{s=r-2}^{r} P_{s,k_s}(p_{s,k_s})}\cdot
   e^{-\sum_{s=r-2}^{r} q_{s,k_s}}
\end{multline*}
\end{small}
and
\begin{small}
\begin{equation*}
   f^D_{(r-1)',r}(z)\ =
   \sum_{1\leq k_r\leq a_r} \frac{Z_r(p_{r,k_r}+1)P_{r-2}(p_{r,k_r})}{(z-p_{r,k_r}-1)P_{r,k_r}(p_{r,k_r})}\cdot e^{-q_{r,k_r}}
\end{equation*}
\end{small}

\medskip
\noindent
For $1\leq j\leq i-2 \leq r-4$, we get:
\begin{small}
\begin{multline*}
   f^D_{j',i}(z) \ =\ (-1)^{i-j}
   \sum^{|I_{s}|=1+\delta_{s\in \{i,\ldots,r-2\}}}_{\substack{I_{j}\subset \{1,\ldots,a_{j}\} \\ \cdots\\ I_r\subset \{1,\ldots,a_r\}}}
   \frac{z-p_{i-1,k_{i-1}}-2}{\prod_{k\in I_{i}}(z-p_{i,k}-1)}\, \times \\
   \frac{\prod_{j\leq s\leq r}^{k\in I_s} Z_s(p_{s,k}+1) \cdot \prod_{j+1\leq s\leq r-1}^{k\in I_{s-1}} P_{s,I_s}(p_{s-1,k}+1)\cdot P_{r-2,I_{r-2}}(p_{r,k_r})}
        {\prod_{j\leq s\leq r}^{k\in I_s} P_{s,I_s}(p_{s,k})}\cdot
   e^{-\sum_{j\leq s\leq r}^{k\in I_s} q_{s,k}}
\end{multline*}
\end{small}
with $I_r=\{k_r\}$, while the $i=r-1,r$ counterparts of this formula are as follows:
\begin{small}
\begin{multline*}
   f^D_{j',r-1}(z) \, =\, (-1)^{r-j}
   \sum_{\substack{1\leq k_j\leq a_{j} \\ \cdots\\ 1\leq k_r\leq a_r}}
   \frac{z-p_{r-2,k_{r-2}}-2}{(z-p_{r-1,k_{r-1}}-1)(z-p_{r,k_r}-1)}\, \times \\
   \frac{\prod_{s=j}^{r} Z_s(p_{s,k_s}+1)\cdot \prod_{s=j+1}^{r-1} P_{s,k_s}(p_{s-1,k_{s-1}}+1)\cdot P_{r-2,k_{r-2}}(p_{r,k_{r}})}
          {\prod_{s=j}^{r} P_{s,k_s}(p_{s,k_s})}\cdot e^{-\sum_{s=j}^{r} q_{s,k_s}}
\end{multline*}
\end{small}
and
\begin{small}
\begin{multline*}
   f^D_{j',r}(z) \ =\
   (-1)^{r-j}
   \sum_{\substack{1\leq k_j\leq a_{j} \\ \cdots\\ 1\leq k_{r-2}\leq a_{r-2}\\ 1\leq k_r\leq a_{r}}}
   P_{r-2,k_{r-2}}(p_{r,k_{r}})P_{r-1}(p_{r-2,k_{r-2}}+1)\, \times \\
   \frac{\prod_{j\leq s\leq r}^{s\ne r-1} Z_s(p_{s,k_s}+1)\cdot \prod_{s=j+1}^{r-2} P_{s,k_s}(p_{s-1,k_{s-1}}+1)}
        {(z-p_{r,k_r}-1)\prod_{s=j}^{r-2} P_{s,k_s}(p_{s,k_s})\cdot P_{r,k_r}(p_{r,k_r})}\cdot
   e^{-\sum_{s=j}^{r-2} q_{s,k_s} -\, q_{r,k_r}}
\end{multline*}
\end{small}

\medskip

\begin{Rem}\label{rmk:F-skew}
In the notations $f^{D}_{j,i}(z)=\sum_{k\geq 1} f^{(D)k}_{j,i}z^{-k}$, see~\eqref{eq:modes},
the above formulas imply:
\begin{equation}\label{eq:F-linear-skew}
  f^{(D)1}_{j,i} = -f^{(D)1}_{i',j'} \,, \qquad \forall\, 1\leq i<j\leq 2r \, .
\end{equation}
\end{Rem}

\medskip

    %%%%%%%%%%%%%%%%%%%%%%%%%%%%%%%%%%%%%%%%%%%%%%%%%%%%%%%%%%%%%%%%%%%%%%%%%%%%%%%
    %%%%%%%%%%%%%%%% Shuffle realization of BFN/NW homomorphsism %%%%%%%%%%%%%%%%%%
    %%%%%%%%%%%%%%%%%%%%%%%%%%%%%%%%%%%%%%%%%%%%%%%%%%%%%%%%%%%%%%%%%%%%%%%%%%%%%%%

\section{Explicit formulas in types B and C}
\label{Appendix B: shuffle realization}

In this Appendix, we provide a shuffle realization of the key homomorphisms of our paper.
To simplify the exposition, we shall follow the uniform formulas of~\cite{nw}
in the $(w,\sfu)$-oscillators (generalizing those of~\cite{bfnb} to non-simply-laced
cases in the spirit of~\cite{gklo,ft1}).

    %%%%%%%%%%%%%%%%%%%%%%%%%%%%%%%%%%%%%%%%%%%%%%%%%%%%%%%%%%%%%%%%%%%%%%%%%%%%%%%
    %%%%%%%%%%%%%%%%%%%%%%%% Nakajima-Weekes homomorphsism %%%%%%%%%%%%%%%%%%%%%%%%
    %%%%%%%%%%%%%%%%%%%%%%%%%%%%%%%%%%%%%%%%%%%%%%%%%%%%%%%%%%%%%%%%%%%%%%%%%%%%%%%

\subsection{Homomorphisms $\Phi^{\bar{\lambda}}_{\bar{\mu}}$}
\label{ssec Nakajima-Weekes formulas}
\

Let $\fg$ be a simple Lie algebra of rank $r$, and let $\{\alphavee_i\}_{i=1}^r$
(resp.\ $\{\alpha_i\}_{i=1}^r$) be the simple roots (resp.\ simple coroots) of $\fg$.
Let $(\cdot,\cdot)$ denote the corresponding pairing on the root lattice, and set
$\sd_i:=\frac{(\alphavee_i,\alphavee_i)}{2}$. Let $(a_{ij})_{i,j=1}^r$ be the
Cartan matrix of $\fg$, so that $\sd_i a_{ij} = (\alphavee_i,\alphavee_j)$.

\medskip
\noindent
We also choose an orientation of the graph $\mathrm{Dyn}_{\fg}$ obtained from the
Dynkin diagram of $\fg$ by replacing all multiple edges with simple ones.
The notation $j-i$ (resp.\ $j\rightarrow i$ or $j \leftarrow i$) is to indicate an edge
(resp.\ oriented edge pointing towards $i$ or $j$) between the vertices $i,j\in\mathrm{Dyn}_{\fg}$.

\medskip
\noindent
Fix a coweight $\bar{\mu}$ of $\fg$, and let $Y_{\bar{\mu}}(\fg)$ denote the corresponding
shifted (Drinfeld) Yangian of $\fg$, cf.~\cite{bfnb,nw}, whose generators are encoded into
the series $\sE_i(z),\sF_i(z),\sH_i(z)$ as before. We also fix a dominant coweight
$\bar{\lambda}=\omega_{i_1}+\ldots+\omega_{i_N}$ ($\omega_k$ being the $k$-th fundamental coweight)
such that $\bar{\lambda}+\bar{\mu}=a_1\alpha_1+\ldots+a_{r}\alpha_{r}$ with $a_i\in \BN$,
and choose a collection of points $z_1,\ldots,z_N\in \BC$.

\medskip
\noindent
Consider the associative $\BC$-algebra (cf.~(\ref{algebra A},~\ref{eq:general pq relation}))
\begin{equation*}
  \tilde{\CA} = \BC\, \Big\langle w_{i,k}\, ,\ \sfu_{i,k}^{\pm 1}\, ,\ (w_{i,k}-w_{i,\ell}+m\sd_i)^{-1}\Big\rangle
                 _{1\leq i\leq r, m\in \BZ}^{1\leq k\ne \ell\leq a_i}
\end{equation*}
with the defining relations:
\begin{equation*}
  [\sfu_{i,k},w_{j,\ell}]=\sd_i\delta_{i,j}\delta_{k,\ell} \sfu_{i,k} \, , \quad
  [w_{i,k},w_{j,\ell}]=0=[\sfu_{i,k},\sfu_{j,\ell}] \, , \quad
  \sfu_{i,k}^{\pm 1} \sfu_{i,k}^{\mp 1} = 1 \, .
\end{equation*}
Set $a_0:=0,\, a_{r+1}:=0,\, W_0(z):=1,\, W_{r+1}(z)=1$. For $1\leq i\leq r$, we also define:
\begin{equation*}
  W_{i}(z):=\prod_{k=1}^{a_i}(z-w_{i,k}) \, , \quad
  W_{i,\ell}(z)\, :=\prod_{1\leq k\leq a_i}^{k\ne \ell}(z-w_{i,k}) \, , \quad
  \sZ_i(z)\, :=\prod_{1\leq s\leq N}^{i_s=i} (z-z_s-\tfrac{1}{2}) \, .
\end{equation*}

\medskip

\begin{Rem}
The shift by $-\frac{1}{2}$ above is purely historical~\cite{bfnb}, and can be absorbed into~$z_s$.
\end{Rem}

\medskip
\noindent
The following is a rational counterpart of~\cite[Theorem 7.1]{ft1} (cf.~\cite[Theorem 5.4]{nw}):

\medskip

\begin{Thm}\label{thm:nakajima-weekes thm}
There is a unique $\BC$-algebra homomorphism
\begin{equation}\label{eq:nw homom}
  \Phi^{\bar{\lambda}}_{\bar{\mu}}\colon Y_{\bar{\mu}}(\fg)\longrightarrow \tilde{\CA} \, ,
\end{equation}
determined by the following assignment:
\begin{equation}\label{eq:nw homom assignment}
\begin{split}
  & \sE_i(z) \, \mapsto \,
    \frac{1}{\sd_i}\sum_{k=1}^{a_i}\frac{\prod_{j\to i} \prod_{p=1}^{-a_{ji}} W_{j}(w_{i,k}-\frac{1}{2}(\alphavee_i,\alphavee_j)-p\sd_j)}{(z-w_{i,k})W_{i,k}(w_{i,k})} \sfu^{-1}_{i,k} \, , \\
  & \sF_i(z) \, \mapsto \,
    -\sum_{k=1}^{a_i}\frac{\sZ_i(w_{i,k}+\sd_i)\prod_{j\leftarrow i}\prod_{p=1}^{-a_{ji}} W_{j}(w_{i,k}+\sd_i-\frac{1}{2}(\alphavee_i,\alphavee_j)-p\sd_j)}{(z-w_{i,k}-\sd_i)W_{i,k}(w_{i,k})} \sfu_{i,k} \, , \\
  & \sH_i(z) \, \mapsto \,
    \frac{\sZ_i(z)\prod_{j-i}\prod_{p=1}^{-a_{ji}} W_{j}(z-\frac{1}{2}(\alphavee_i,\alphavee_j)-p\sd_j)}{W_i(z)W_{i}(z-\sd_i)} \, .
\end{split}
\end{equation}
\end{Thm}

\medskip

    %%%%%%%%%%%%%%%%%%%%%%%%%%%%%%%%%%%%%%%%%%%%%%%%%%%%%%%%%%%%%%%%%%%%%%%%%%%%%%%
    %%%%%%%%%%%%%%%%%%%%%%%%%%%%% Shuffle realization %%%%%%%%%%%%%%%%%%%%%%%%%%%%%
    %%%%%%%%%%%%%%%%%%%%%%%%%%%%%%%%%%%%%%%%%%%%%%%%%%%%%%%%%%%%%%%%%%%%%%%%%%%%%%%

\subsection{Shuffle algebra realization of $\Phi^{\bar{\lambda}}_{\bar{\mu}}$}
\label{ssec shuffle realization}
\

Let $Y^{+}(\fg)$ and $Y^-(\fg)$ denote the subalgebras of the Drinfeld Yangian $Y(\fg)$
generated by $\{\sE^{(k)}_{i}\}_{1\leq i\leq r}^{k\geq 1}$ and $\{\sF^{(k)}_{i}\}_{1\leq i\leq r}^{k\geq 1}$,
respectively. They can also be described as algebras generated by $\{\sE^{(k)}_{i}\}_{1\leq i\leq r}^{k\geq 1}$
and $\{\sF^{(k)}_{i}\}_{1\leq i\leq r}^{k\geq 1}$ subject to the relations~(\ref{Y4},~\ref{Y6})
and~(\ref{Y5},~\ref{Y7}), respectively.

\medskip

\begin{Rem}\label{rmk:halves are opposite}
Note the algebra isomorphisms $Y^-(\fg)\iso Y^{+}(\fg)^{\op}$ determined via
$\sF_{i}^{(k)}\mapsto \sE_{i}^{(k)}$ (given an algebra $A$, we use $A^\op$
to denote the algebra with the opposite multiplication).
\end{Rem}

\medskip
\noindent
For any coweight $\nu$ of $\fg$, we define the subalgebras $Y^\pm_{\nu}(\fg)$ of the shifted Yangian
$Y_\nu(\fg)$ likewise. According to~\cite[Corollary 3.15]{fkp}, we have algebra isomorphisms for any $\nu$:
\begin{equation}\label{eq:halves isomorphism}
\begin{split}
  & Y^{+}_\nu(\fg)\iso Y^{+}(\fg) \, , \qquad \sE^{(k)}_{i}\mapsto \sE^{(k)}_{i} \, ,\\
  & Y^{-}_\nu(\fg)\iso Y^{-}(\fg) \, , \qquad \sF^{(k)}_{i}\mapsto \sF^{(k)}_{i} \, .
\end{split}
\end{equation}

\medskip
\noindent
Consider an $\BN^{r}$-graded $\BC$-vector space
\begin{equation}\label{eq:ambient shuffle}
  \BS^{(\fg)}\ =\underset{\underline{k}=(k_1,\ldots,k_{r})\in \BN^{r}}\bigoplus\BS^{(\fg)}_{\underline{k}} \, ,
\end{equation}
where $\BS^{(\fg)}_{\unl{k}}$ consists of $\prod_{i=1}^r S(k_i)$-symmetric rational functions
in the variables $\{x_{i,k}\}_{1\leq i\leq r}^{1\leq k\leq k_i}$. We also fix a matrix of
rational functions $(\zeta_{ij}(z))_{i,j=1}^{r}$ via:
\begin{equation}\label{eq:shuffle factor}
  \zeta_{ij}(z)=1+\frac{(\alphavee_i,\alphavee_j)}{2z}=1+\frac{\sd_i a_{ij}}{2z} \, .
\end{equation}
Let us define the \emph{shuffle product} $\star$ on $\BS^{(\fg)}$:
given $F\in \BS^{(\fg)}_{\underline{k}}$, $G\in \BS^{(\fg)}_{\underline{\ell}}$,
define $F\star G\in \BS^{(\fg)}_{\underline{k}+\underline{\ell}}$ via
\begin{equation}\label{shuffle product}
\begin{split}
  & (F\star G)(x_{1,1},\ldots,x_{1,k_1+\ell_1};\ldots;x_{r,1},\ldots, x_{r,k_{r}+\ell_{r}}):=
    \frac{1}{\unl{k}!\cdot\unl{\ell}!}\times\\
  & \Sym
    \left(F\left(\{x_{i,k}\}_{1\leq i\leq r}^{1\leq k\leq k_i}\right) G\left(\{x_{i',k'}\}_{1\leq i'\leq r}^{k_{i'}<k'\leq k_{i'}+\ell_{i'}}\right)\cdot
    \prod_{1\leq i\leq r}^{1\leq i'\leq r}\prod_{k\leq k_i}^{k'>k_{i'}}\zeta_{ii'}(x_{i,k}-x_{i',k'})\right).
\end{split}
\end{equation}
Here, $\unl{k}!=\prod_{i=1}^{r} k_i!$, while the \emph{symmetrization} of
$f\in \BC(\{x_{i,1},\ldots,x_{i,m_i}\}_{1\leq i\leq r})$ is defined via:
\begin{equation*}
  \Sym\, (f)\left(\{x_{i,1},\ldots,x_{i,m_i}\}_{1\leq i\leq r}\right)\ :=
  \sum_{(\sigma_1,\ldots,\sigma_{r})\in S(m_1)\times \dots \times S(m_r)}
  f\Big(\{x_{i,\sigma_i(1)},\ldots,x_{i,\sigma_i(m_i)}\}_{1\leq i\leq r}\Big).
\end{equation*}
This endows $\BS^{(\fg)}$ with a structure of an associative $\BC$-algebra
with the unit $\textbf{1}\in \BS^{(\fg)}_{(0,\ldots,0)}$.

\medskip
\noindent
We are interested in a certain $\BC$-subspace of $\BS^{(\fg)}$
defined by the \emph{pole} and \emph{wheel conditions}:
\begin{itemize}
\item[$\bullet$]
  We say that $F\in \BS^{(\fg)}_{\underline{k}}$ satisfies the \emph{pole conditions} if
  \begin{equation}\label{pole conditions}
    F=\frac{f(x_{1,1},\ldots,x_{r,k_{r}})}{\prod_{i-j}^{\mathrm{unordered}} \prod_{k\leq k_i}^{k'\leq k_{j}}(x_{i,k}-x_{j,k'})},\
    \mathrm{where}\ f\in \Big(\BC[\{x_{i,k}\}_{1\leq i\leq r}^{1\leq k\leq k_i}]\Big)^{S(k_1)\times \dots \times S(k_r)} \,.
  \end{equation}
\item[$\bullet$]
  We say that $F\in \BS^{(\fg)}_{\underline{k}}$ satisfies the \emph{wheel conditions} if
  for any connected $i-j$, we have:
  \begin{equation}\label{wheel conditions}
    F\Big(\{x_{i,k}\}\Big)\Big|_
    {(x_{i1},x_{i2},x_{i3}, \dots, x_{i,1-a_{ij}}) \mapsto (w, w+\sd_i, w+2\sd_i, \dots, w+\sd_i a_{ij}),\,
      x_{j1} \mapsto w+\frac{\sd_i a_{ij}}{2}} =\, 0 \,.
  \end{equation}
\end{itemize}
Let $S^{(\fg)}_{\underline{k}}\subset \BS^{(\fg)}_{\underline{k}}$ denote the
subspace of all elements $F$ satisfying these two conditions and set
$$
  S^{(\fg)}:=\underset{\underline{k}\in \BN^{r}}\bigoplus S^{(\fg)}_{\underline{k}}\, .
$$

\medskip
\noindent
It is straightforward to check that $S^{(\fg)}\subset\BS^{(\fg)}$ is $\star$-closed. The resulting algebra
$\left(S^{(\fg)},\star\right)$ is called the \emph{(rational) shuffle algebra of type $\fg$}.
It is related to $Y^+(\fg)$ via the embedding:
\begin{equation}\label{eq:Psi1}
  \Upsilon\colon Y^{+}(\fg)\hookrightarrow S^{(\fg)}\, ,\qquad
  \sE_{i}^{(k)}\mapsto x_{i,1}^{k-1} \qquad \mathrm{for}\quad 1\leq i\leq r \, ,\ k\geq 1 \, .
\end{equation}
In view of Remark~\ref{rmk:halves are opposite}, we also get:
\begin{equation}\label{eq:Psi2}
  \Upsilon\colon Y^{-}(\fg)\hookrightarrow S^{(\fg),\op}\, ,\qquad
  \sF_{i}^{(k)}\mapsto x_{i,1}^{k-1} \qquad \mathrm{for}\quad 1\leq i\leq r \, ,\ k\geq 1 \, .
\end{equation}

\medskip

\begin{Rem}
The above embeddings $\Upsilon$ of~(\ref{eq:Psi1},~\ref{eq:Psi2}) are expected to be actually
algebra isomorphisms, similar to the trigonometric counterpart as was recently established in~\cite{nt}.
This has been proved in (super version of) the type $A$ in~\cite[\S6-7]{t}.
\end{Rem}

\medskip
\noindent
The key result of this Appendix is the construction of the algebra homomorphisms
\begin{equation}\label{eq:shuffle homomorphisms}
  \wt{\Phi}^{\bar{\lambda}}_{\bar{\mu}}\colon \quad
  S^{(\fg)} \longrightarrow \tilde{\CA} \, , \qquad
  S^{(\fg),\op} \longrightarrow \tilde{\CA} \, ,
\end{equation}
compatible with $\Phi^{\bar{\lambda}}_{\bar{\mu}}$~\eqref{eq:nw homom} with respect to
the isomorphisms~\eqref{eq:halves isomorphism} and embeddings~(\ref{eq:Psi1},~\ref{eq:Psi2}).
To this end,  for $1\leq i\leq r$ and $1\leq \ell\leq a_i$, we define:
\begin{equation}\label{eq:Y-factors}
\begin{split}
  & Y_{i,\ell}(z):=
    \frac{\prod_{j\to i} \prod_{p=1}^{-a_{ji}} W_{j}(z-\frac{1}{2}(\alphavee_i,\alphavee_j)-p\sd_j)}{\sd_i\cdot W_{i,\ell}(z)} \, ,\\
  & Y'_{i,\ell}(z):=
    -\frac{\sZ_i(z+\sd_i)\prod_{j\leftarrow i}\prod_{p=1}^{-a_{ji}} W_{j}(z+\sd_i-\frac{1}{2}(\alphavee_i,\alphavee_j)-p\sd_j)}{W_{i,\ell}(z)} \, .
\end{split}
\end{equation}

\medskip

\begin{Thm}\label{thm:shuffle homomorphism}
(a) The assignment
\begin{equation}\label{eq:explicit shuffle homom 1}
\begin{split}
  & S^{(\fg)}_{(k_1,\ldots,k_r)} \ni E\mapsto \\
  & \sum_{\substack{m^{(1)}_1+\ldots+m^{(1)}_{a_1}=k_1\\\cdots\\ m^{(r)}_1+\ldots+m^{(r)}_{a_{r}}=k_{r}}}^{m^{(i)}_k\in \BN}
    \left\{\prod_{i=1}^{r}\prod_{k=1}^{a_i}\prod_{p=1}^{m^{(i)}_k} Y_{i,k}\Big(w_{i,k}-(p-1)\sd_i\Big)\cdot
    E\left(\Big\{w_{i,k}-(p-1)\sd_i\Big\}_{\substack{1\leq i\leq r\\ 1\leq k\leq a_i\\ 1\leq p\leq m^{(i)}_k}}\right)\times\right.\\
  & \left.
    \prod_{i=1}^{r}\prod_{k=1}^{a_i}\prod_{1\leq p_1<p_2\leq m^{(i)}_k}
      \zeta^{-1}_{ii}\Big( (w_{i,k}-(p_1-1)\sd_i)\,-\, (w_{i,k}-(p_2-1)\sd_i)\Big)\, \times\right.\\
  & \left.
    \prod_{i=1}^{r}\prod_{1\leq k_1\neq k_2\leq a_i}\prod_{1\leq p_1\leq m^{(i)}_{k_1}}^{1\leq p_2\leq m^{(i)}_{k_2}}
      \zeta^{-1}_{ii}\Big( (w_{i,k_1}-(p_1-1)\sd_i) \,-\, (w_{i,k_2}-(p_2-1)\sd_i)\Big)\, \times\right.\\
  & \left.\prod_{j\to i}\prod_{1\leq k_1\leq a_{i}}^{1\leq k_2\leq a_{j}}\prod_{1\leq p_1\leq m^{(i)}_{k_1}}^{1\leq p_2\leq m^{(j)}_{k_2}}
      \zeta^{-1}_{ij}\Big((w_{i,k_1}-(p_1-1)\sd_i) \,-\, (w_{j,k_2}-(p_2-1)\sd_j)\Big) \cdot\,
    \prod_{i=1}^{r}\prod_{k=1}^{a_i} \sfu_{i,k}^{-m^{(i)}_k}\right\}
\end{split}
\end{equation}
gives rise to the algebra homomorphism
\begin{equation}\label{eq:Psi-tilde +}
  \wt{\Phi}^{\bar{\lambda}}_{\bar{\mu}}\colon S^{(\fg)}\longrightarrow \tilde{\CA} \, .
\end{equation}
Moreover, the composition
\begin{equation}\label{eq:homom extension 1}
  Y^+_{\bar{\mu}}(\fg)\overset{\eqref{eq:halves isomorphism}}{\iso} Y^+(\fg)
  \overset{\Upsilon}{\longrightarrow} S^{(\fg)}
  \overset{\wt{\Phi}^{\bar{\lambda}}_{\bar{\mu}}}{\longrightarrow}\tilde{\CA}
\end{equation}
coincides with the restriction of the homomorphism $\Phi^{\bar{\lambda}}_{\bar{\mu}}$~\eqref{eq:nw homom}
to the subalgebra $Y^+_{\bar{\mu}}(\fg)\subset Y_{\bar{\mu}}(\fg)$.

\medskip
\noindent
(b)  The assignment
\begin{equation}\label{eq:explicit shuffle homom 2}
\begin{split}
  & S^{(\fg),\op}_{(k_1,\ldots,k_r)} \ni F\mapsto \\
  & \sum_{\substack{m^{(1)}_1+\ldots+m^{(1)}_{a_1}=k_1\\\cdots\\ m^{(r)}_1+\ldots+m^{(r)}_{a_{r}}=k_{r}}}^{m^{(i)}_k\in \BN}
    \left\{\prod_{i=1}^{r}\prod_{k=1}^{a_i}\prod_{p=1}^{m^{(i)}_k} Y'_{i,k}\Big(w_{i,k}+(p-1)\sd_i\Big)\cdot
    F\left(\Big\{w_{i,k}+p\sd_i\Big\}_{\substack{1\leq i\leq r\\ 1\leq k\leq a_i\\ 1\leq p\leq m^{(i)}_k}}\right)\times\right.\\
  & \left.
    \prod_{i=1}^{r}\prod_{k=1}^{a_i}\prod_{1\leq p_1<p_2\leq m^{(i)}_k}
      \zeta^{-1}_{ii}\Big((w_{i,k}+p_2\sd_i) \,-\, (w_{i,k}+p_1\sd_i)\Big)\, \times\right.\\
  & \left.
    \prod_{i=1}^{r}\prod_{1\leq k_1\neq k_2\leq a_i}\prod_{1\leq p_1\leq m^{(i)}_{k_1}}^{1\leq p_2\leq m^{(i)}_{k_2}}
      \zeta^{-1}_{ii}\Big( (w_{i,k_2}+p_2\sd_i) \,-\, (w_{i,k_1}+p_1\sd_i)\Big)\, \times\right.\\
  & \left.\prod_{j\leftarrow i}\prod_{1\leq k_1\leq a_{i}}^{1\leq k_2\leq a_{j}}\prod_{1\leq p_1\leq m^{(i)}_{k_1}}^{1\leq p_2\leq m^{(j)}_{k_2}}
      \zeta^{-1}_{ji}\Big( (w_{j,k_2}+p_2\sd_j) \,-\, (w_{i,k_1}+p_1\sd_i) \Big) \cdot\,
    \prod_{i=1}^{r}\prod_{k=1}^{a_i} \sfu_{i,k}^{m^{(i)}_k}\right\}
\end{split}
\end{equation}
gives rise to the algebra homomorphism
\begin{equation}\label{eq:Psi-tilde -}
  \wt{\Phi}^{\bar{\lambda}}_{\bar{\mu}}\colon S^{(\fg),\op}\longrightarrow \tilde{\CA} \, .
\end{equation}
Moreover, the composition
\begin{equation}\label{eq:homom extension 2}
  Y^-_{\bar{\mu}}(\fg)\overset{\eqref{eq:halves isomorphism}}{\iso} Y^-(\fg)
  \overset{\Upsilon}{\longrightarrow} S^{(\fg),\op}
  \overset{\wt{\Phi}^{\bar{\lambda}}_{\bar{\mu}}}{\longrightarrow}\tilde{\CA}
\end{equation}
coincides with the restriction of the homomorphism $\Phi^{\bar{\lambda}}_{\bar{\mu}}$~\eqref{eq:nw homom}
to the subalgebra $Y^-_{\bar{\mu}}(\fg)\subset Y_{\bar{\mu}}(\fg)$.
\end{Thm}

\medskip
\noindent
The proof is straightforward and is left to the interested reader.
We note that a trigonometric type $A$ counterpart of this result played
a crucial role in~\cite{ft2}, see Theorem~4.11 of~\emph{loc.cit.}

\medskip

    %%%%%%%%%%%%%%%%%%%%%%%%%%%%%%%%%%%%%%%%%%%%%%%%%%%%%%%%%%%%%%%%%%%%%%%%%%%%%%%
    %%%%%%%%%%%%%%%%%%%%%%%%%%%%% Shuffle realization %%%%%%%%%%%%%%%%%%%%%%%%%%%%%
    %%%%%%%%%%%%%%%%%%%%%%%%%%%%%%%%%%%%%%%%%%%%%%%%%%%%%%%%%%%%%%%%%%%%%%%%%%%%%%%

\subsection{Application to the Lax matrices of types B and C}
\label{ssec BC explicit}
\

\medskip
\noindent
The key application of Theorem~\ref{thm:shuffle homomorphism} to the main subject of the
present paper is that it allows to obtain explicit formulas for the matrix coefficients of
$E^D(z),F^D(z)$ featuring in our definition of the Lax matrices $T_D(z)$~\eqref{eq:Lax definition}.
In type $D_r$ this recovers the formulas of Appendix~\ref{Appendix A: Lax explicitly}
(which were rather derived using the relations of
Lemmas~\ref{lem:all-E-known},~\ref{lem:all-E-new},~\ref{lem:all-F-known},~\ref{lem:all-F-new}),
while in types $C_r$ and $B_r$ this provides concise formulas (used in the proofs of
Theorems~\ref{thm: regularity C},~\ref{thm: regularity B}), which are quite inaccessible
if derived iteratively via Lemmas~\ref{lem:known type C},~\ref{lem:new type C}
or~\ref{lem:known type B},~\ref{lem:new type B}, respectively.

\medskip
\noindent
Let $\fg$ be either $\sso_N\ (N=2r,2r+1)$ or $\ssp_N\ (N=2r)$. Let $X^+(\fg)$ and $X^-(\fg)$
denote the subalgebras of the corresponding extended Drinfeld Yangian $X(\fg)$, generated by
$\{E_{i}^{(k)}\}_{1\leq i\leq r}^{k\geq 1}$ and $\{F_{i}^{(k)}\}_{1\leq i\leq r}^{k\geq 1}$,
respectively. Likewise, let $X^{\rtt,+}(\fg)$ and $X^{\rtt,-}(\fg)$ denote the subalgebras of
the corresponding extended RTT Yangian $X^\rtt(\fg)$, generated by $\{e_{i,j}^{(k)}\}_{1\leq i<j\leq N}^{k\geq 1}$
and $\{f_{j,i}^{(k)}\}_{1\leq i<j\leq N}^{k\geq 1}$, respectively.
Then, we have the following natural algebra isomorphisms:
\begin{equation}\label{eq:various halves}
  X^{\rtt,+}(\fg)\iso X^+(\fg)\iso Y^+(\fg)\, ,\qquad
  X^{\rtt,-}(\fg)\iso X^-(\fg)\iso Y^-(\fg)\, .
\end{equation}
Let $\{\sE^{(k)}_{ij}\}_{1\leq i<j\leq N}^{k\geq 1}$ and
$\{\sF^{(k)}_{ji}\}_{1\leq i<j\leq N}^{k\geq 1}$ denote the images
of $e^{(k)}_{i,j}$ and $f^{(k)}_{j,i}$ in $Y^+(\fg)$ and $Y^-(\fg)$
under the composition maps of~\eqref{eq:various halves}, respectively, and
consider their generating series:
\begin{equation*}
  \sE_{ij}(z):=\, \sum_{k\geq 1}\sE^{(k)}_{ij}z^{-k}\, ,\qquad
  \sF_{ji}(z):=\, \sum_{k\geq 1}\sF^{(k)}_{ji}z^{-k} \, .
\end{equation*}
We conclude this Appendix by presenting explicit formulas for $\Upsilon(\sE_{ij}(z))$
and $\Upsilon(\sF_{ji}(z))$ (combining which with Theorem~\ref{thm:shuffle homomorphism}
recovers the Lax matrices $T_D(z)$ of~\eqref{eq:Lax definition}).
In what follows, $\varsigma_i$ will denote the $i$-th coordinate vector:
$\varsigma_i=(0,\ldots,0,1,0,\ldots,0)\in \BN^r$ with $1$ at the spot $i$.

\medskip

\begin{Lem}\label{lem:EF explicit shuffle C}(Type $C_r$)
Define the polynomial $Q(z_1,z_2;w_1,w_2)$ via:\footnote{Note that $Q(w,w-1;w-1/2,z)=0$
in accordance with the wheel conditions~\eqref{wheel conditions}.}
\begin{equation}%\label{eq:Q-factor}
  Q(z_1,z_2;w_1,w_2)=2z_1z_2+2w_1w_2-(z_1+z_2)(w_1+w_2)+\tfrac{1}{2} \, .
\end{equation}

\medskip
\noindent
(a) We have the following equalities:

\begin{small}
\begin{equation*}
  \Upsilon \Big(\sE_{ij}(z)\Big) =
  \frac{1}{(z-\frac{i-1}{2}-x_{i,1})\prod_{k=i}^{j-2} (x_{k,1}-x_{k+1,1})}\in S^{(\ssp_{2r})}_{\varsigma_i+\ldots+\varsigma_{j-1}}
  \qquad  \mathrm{for}\ 1\leq i<j\leq r \, ,
\end{equation*}
\end{small}

\begin{small}
\begin{equation*}
  \Upsilon \Big(\sE_{ir'}(z)\Big) =
  \frac{2}{(z-\frac{i-1}{2}-x_{i,1})\prod_{k=i}^{r-1} (x_{k,1}-x_{k+1,1})}\in S^{(\ssp_{2r})}_{\varsigma_i+\ldots+\varsigma_{r}}
  \qquad  \mathrm{for}\ 1\leq i< r \, ,
\end{equation*}
\end{small}

\begin{small}
\begin{equation*}
\begin{split}
  & \Upsilon \Big(\sE_{ij'}(z)\Big) =
    \frac{2(2x_{j-1,1}-x_{j,1}-x_{j,2})\prod_{k=j}^{r-2}Q(x_{k,1},x_{k,2};x_{k+1,1},x_{k+1,2})}
          {(z-\frac{i-1}{2}-x_{i,1})\prod_{k=i}^{r-1} \prod_{p\leq 1+\delta_{j\leq k<r}}^{p'\leq 1+\delta_{j\leq k+1<r}} (x_{k,p}-x_{k+1,p'})}
    \in S^{(\ssp_{2r})}_{\varsigma_i+\ldots+\varsigma_{j-1}+2\varsigma_j+\ldots+2\varsigma_{r-1}+\varsigma_r} \\
  & \qquad \qquad \qquad \qquad \qquad \qquad \qquad \qquad \qquad \qquad \qquad \qquad \qquad  \qquad  \qquad \ \ \mathrm{for}\ 1\leq i<j< r\, ,
\end{split}
\end{equation*}
\end{small}

\begin{small}
\begin{equation*}
\begin{split}
  & \Upsilon \Big(\sE_{ii'}(z)\Big) =
    \frac{2(2z-i+2-x_{i,1}-x_{i,2})\prod_{k=i}^{r-2}Q(x_{k,1},x_{k,2};x_{k+1,1},x_{k+1,2})}
          {(z-\frac{i-1}{2}-x_{i,1})(z-\frac{i-1}{2}-x_{i,2})\prod_{k=i}^{r-1} \prod_{p\leq 1+\delta_{k<r}}^{p'\leq 1+\delta_{k+1<r}} (x_{k,p}-x_{k+1,p'})}
    \in S^{(\ssp_{2r})}_{2\varsigma_i+\ldots+2\varsigma_{r-1}+\varsigma_r} \\
  & \qquad \qquad \qquad \qquad \qquad \qquad \qquad \qquad \qquad \qquad \qquad \qquad \qquad  \qquad  \qquad \qquad  \mathrm{for}\ 1\leq i\leq r\, ,
\end{split}
\end{equation*}
\end{small}

\begin{small}
\begin{equation*}
\begin{split}
  & \Upsilon \Big(\sE_{ij'}(z)\Big) =
    \frac{2(Q(z-\frac{i-1}{2},x_{i-1,1};x_{i,1},x_{i,2})+\frac{1}{2}(2x_{i-1,1}-x_{i,1}-x_{i,2}))\prod_{k=i}^{r-2}Q(x_{k,1},x_{k,2};x_{k+1,1},x_{k+1,2})}
          {(z-\frac{i-1}{2}-x_{i,1})(z-\frac{i-1}{2}-x_{i,2})\prod_{k=j}^{r-1} \prod_{p\leq 1+\delta_{i\leq k<r}}^{p'\leq 1+\delta_{i\leq k+1<r}} (x_{k,p}-x_{k+1,p'})}
    \\
  & \qquad \qquad \qquad \qquad \qquad \qquad \qquad \qquad \qquad \qquad \in S^{(\ssp_{2r})}_{\varsigma_j+\ldots+\varsigma_{i-1}+2\varsigma_i+\ldots+2\varsigma_{r-1}+\varsigma_r}
    \qquad \mathrm{for}\ 1\leq j<i<r \, ,
\end{split}
\end{equation*}
\end{small}

\begin{small}
\begin{equation*}
  \Upsilon \Big(\sE_{rj'}(z)\Big) =
  \frac{2}{(z-\frac{r}{2}-x_{r,1})\prod_{k=j}^{r-1} (x_{k,1}-x_{k+1,1})} \in S^{(\ssp_{2r})}_{\varsigma_j+\ldots+\varsigma_{r}}
  \qquad  \mathrm{for}\ 1\leq j < r \, ,
\end{equation*}
\end{small}

\begin{small}
\begin{equation*}
  \Upsilon \Big(\sE_{i'j'}(z)\Big) =
  -\frac{1}{(z+\frac{i-2}{2}-r-x_{i-1,1})\prod_{k=j}^{i-2} (x_{k,1}-x_{k+1,1})}\in S^{(\ssp_{2r})}_{\varsigma_j+\ldots+\varsigma_{i-1}}
  \qquad  \mathrm{for}\ 1\leq j<i\leq r \, .
\end{equation*}
\end{small}

\medskip
\noindent
(b) For any $1\leq i<j\leq 2r$, $\Upsilon(\sF_{ji}(z))\in S^{(\ssp_{2r}),\op}[[z^{-1}]]$ is given
by the same formula (the expansion in $z^{-1}$ of the corresponding rational function in (a)) as
$\Upsilon(\sE_{ij}(z))\in S^{(\ssp_{2r})}[[z^{-1}]]$.
\end{Lem}

\medskip

\begin{Lem}\label{lem:EF explicit shuffle B}(Type $B_r$)
(a) We have the following equalities:

\begin{small}
\begin{equation*}
  \Upsilon \Big(\sE_{ij}(z)\Big) =
  \frac{1}{(z-\frac{i-1}{2}-x_{i,1})\prod_{k=i}^{j-2} (x_{k,1}-x_{k+1,1})}\in S^{(\sso_{2r+1})}_{\varsigma_i+\ldots+\varsigma_{j-1}}
  \qquad  \mathrm{for}\ 1\leq i<j\leq r+1 \, ,
\end{equation*}
\end{small}

\begin{small}
\begin{equation*}
\begin{split}
  & \Upsilon \Big(\sE_{ij'}(z)\Big) =
    -\frac{\prod_{k=j}^{r-1}(x_{k,1}-x_{k,2}-1)(x_{k,2}-x_{k,1}-1)}
          {(z-\frac{i-1}{2}-x_{i,1})\prod_{k=i}^{r-1} \prod_{p\leq 1+\delta_{k\geq j}}^{p'\leq 1+\delta_{k+1\geq j}} (x_{k,p}-x_{k+1,p'})}
    \in S^{(\sso_{2r+1})}_{\varsigma_i+\ldots+\varsigma_{j-1}+2(\varsigma_j+\ldots+\varsigma_{r})} \\
  & \qquad \qquad \qquad \qquad \qquad \qquad \qquad \qquad \qquad \qquad \qquad \qquad \qquad  \qquad  \qquad \ \ \mathrm{for}\ 1\leq i<j\leq r\, ,
\end{split}
\end{equation*}
\end{small}

\begin{small}
\begin{equation*}
\begin{split}
  & \Upsilon \Big(\sE_{ii'}(z)\Big) =
    -\frac{\prod_{k=i}^{r-1}(x_{k,1}-x_{k,2}-1)(x_{k,2}-x_{k,1}-1)}
          {(z-\frac{i-1}{2}-x_{i,1})(z-\frac{i-1}{2}-x_{i,2})\prod_{k=i}^{r-1} \prod_{p\leq 2}^{p'\leq 2} (x_{k,p}-x_{k+1,p'})}
    \in S^{(\sso_{2r+1})}_{2(\varsigma_i+\ldots+\varsigma_{r})} \\
  & \qquad \qquad \qquad \qquad \qquad \qquad \qquad \qquad \qquad \qquad \qquad \qquad \qquad  \qquad  \qquad \qquad  \mathrm{for}\ 1\leq i\leq r\, ,
\end{split}
\end{equation*}
\end{small}

\begin{small}
\begin{equation*}
\begin{split}
  & \Upsilon \Big(\sE_{ij'}(z)\Big) =
    \frac{(z-\frac{i}{2}-x_{i-1,1})\prod_{k=i}^{r-1}(x_{k,1}-x_{k,2}-1)(x_{k,2}-x_{k,1}-1)}
          {(z-\frac{i-1}{2}-x_{i,1})(z-\frac{i-1}{2}-x_{i,2})\prod_{k=j}^{r-1} \prod_{p\leq 1+\delta_{k\geq i}}^{p'\leq 1+\delta_{k+1\geq i}} (x_{k,p}-x_{k+1,p'})} \\
  & \qquad \qquad \qquad \qquad \qquad \qquad \qquad \qquad \qquad \qquad
    \in S^{(\sso_{2r+1})}_{\varsigma_j+\ldots+\varsigma_{i-1}+2(\varsigma_i+\ldots+\varsigma_{r})} \qquad \mathrm{for}\ 1\leq j<i\leq r\, ,
\end{split}
\end{equation*}
\end{small}

\begin{small}
\begin{equation*}
  \Upsilon \Big(\sE_{i'j'}(z)\Big) =
  -\frac{1}{(z+\frac{i+1}{2}-r-x_{i-1,1})\prod_{k=j}^{i-2} (x_{k,1}-x_{k+1,1})}\in S^{(\sso_{2r+1})}_{\varsigma_j+\ldots+\varsigma_{i-1}}
  \quad  \mathrm{for}\ 1\leq j<i\leq r+1 \, .
\end{equation*}
\end{small}

\medskip
\noindent
(b) For any $1\leq i<j\leq 2r+1$, $\Upsilon(\sF_{ji}(z))\in S^{(\sso_{2r+1}),\op}[[z^{-1}]]$ is given
by the same formula (the expansion in $z^{-1}$ of the corresponding rational function in (a))
as $\Upsilon(\sE_{ij}(z))\in S^{(\sso_{2r+1})}[[z^{-1}]]$.
\end{Lem}

\medskip
\noindent
Inspecting the explicit formulas above, we obtain
(cf.~Remarks~\ref{rmk:E-skew},~\ref{rmk:F-skew} for the type $D_r$):

\medskip

\begin{Cor}\label{cor:BC1-symmetry}
(a) In the type $C_r$, we have (with $\varepsilon_i \in \{\pm 1\}$ defined as in~\eqref{eq:varepsilons}):
\begin{equation}\label{eq:C-linear-skew}
  \sE^{(1)}_{ij}=-\varepsilon_i \varepsilon_j \sE^{(1)}_{j'i'}\, ,\qquad
  \sF^{(1)}_{ji}=-\varepsilon_i \varepsilon_j \sF^{(1)}_{i'j'}\, ,\qquad
  \forall\, 1\leq i<j\leq 2r \, ,
\end{equation}
which imply the corresponding equalities for the matrix coefficients of $E^D(z)$ and $F^D(z)$:
\begin{equation*}
  e^{(D)1}_{i,j}=-\varepsilon_i \varepsilon_j e^{(D)1}_{j',i'}\, ,\qquad
  f^{(D)1}_{j,i}=-\varepsilon_i \varepsilon_j f^{(D)1}_{i',j'}\, ,\qquad
  \forall\, 1\leq i<j\leq 2r \, .
\end{equation*}

\medskip
\noindent
(b)  In the type $B_r$, we have:
\begin{equation}\label{eq:B-linear-skew}
  \sE^{(1)}_{ij} = - \sE^{(1)}_{j'i'}\, ,\qquad
  \sF^{(1)}_{ji} = - \sF^{(1)}_{i'j'}\, ,\qquad
  \forall\, 1\leq i<j\leq 2r+1 \, ,
\end{equation}
which imply the corresponding equalities for the matrix coefficients of $E^D(z)$ and $F^D(z)$:
\begin{equation*}
  e^{(D)1}_{i,j}=-e^{(D)1}_{j',i'}\, ,\qquad
  f^{(D)1}_{j,i}=-f^{(D)1}_{i',j'}\, ,\qquad
  \forall\, 1\leq i<j\leq 2r+1 \, .
\end{equation*}
\end{Cor}

\medskip

    %%%%%%%%%%%%%%%%%%%%%%%%%%%%%%%%%%%%%%%%%%%%%%%%%%%%%%%%%%%%%%%%%%%%%%%%%
    %%%%%%%%%%%%%%%%%%%%%%%%%%%%%% BIBLIOGRAPHY %%%%%%%%%%%%%%%%%%%%%%%%%%%%%
    %%%%%%%%%%%%%%%%%%%%%%%%%%%%%%%%%%%%%%%%%%%%%%%%%%%%%%%%%%%%%%%%%%%%%%%%%

\end{document}